\documentclass[11pt]{amsart}

\usepackage{soul}
\usepackage{pifont}
\usepackage{amsfonts}
\usepackage{amscd,amssymb,amsmath,graphicx,verbatim,mathrsfs}
\usepackage[TS1,OT1,T1]{fontenc}
\usepackage{extarrows}
\usepackage[all]{xy}
\usepackage{dcpic,mathtools,pictexwd}
\usepackage{relsize}
\usepackage{rotating}
\theoremstyle{remark}
\usepackage{indentfirst}
\usepackage{amsthm} 
\usepackage[utf8]{inputenc}
\theoremstyle{plain}
\usepackage{tikz-cd}
\usepackage{tikz}
\usetikzlibrary{angles,quotes}

\usepackage{geometry}
\allowdisplaybreaks

\newtheorem{thm}{Theorem}[section]
\newtheorem{lem}[thm]{Lemma}
\newtheorem{cor}[thm]{Corollary}
\newtheorem{prop}[thm]{Proposition}

\theoremstyle{definition}

\newtheorem{defn}[thm]{Definition}

\newtheorem{rmk}[thm]{Remark}

\newtheorem{constr}[thm]{Construction}

\numberwithin{equation}{section}

\makeatletter
\newcommand{\biggg}{\bBigg@{3}}

\newcommand{\Biggg}{\bBigg@{3.5}}

\newcommand{\bigggg}{\bBigg@{4}}

\newcommand{\Bigggg}{\bBigg@{4.5}}

\makeatother

\begin{document}

	\title{\normalsize \uppercase{E}\lowercase{mbedding complexity into the universal} \uppercase{B}\lowercase{anach space and the strong} \uppercase{N}\lowercase{ovikov} \lowercase{conjecture} 
	}

	\author{Geng Tian}
	
	\address{Department of Mathematics, Liaoning University, Liaoning, China 110036}
	\email{gengtian.ncg@gmail.com}
	
	\thanks{The first author was supported in part by Grant for Excellent Young Scholars in Tianyuan Mathematics.}

	\author{Guoliang Yu  }
	\address{Department of Mathematics, Texas A\&M University, College Station, TX, USA}
	\email{guoliangyu@tamu.edu}
	\thanks{The second author is partially supported by NSF Grant DMS-2247322.}

	
	\dedicatory{  }
	

	\maketitle

	\begin{abstract}
		
		Brown--Guentner and Haagerup--Przybyszewska
		showed that every countable discrete group admits a proper affine isometric action on the universal Banach space
		$\bigoplus_{p=1}^{\infty} \ell^{2p}(\mathbb{N}),$
		taken as the $\ell^{2}$-direct sum, and hence admits a coarse embedding into this space.  
		They further asked whether such embeddings could be used to study the Novikov conjecture. In this paper, we address this question by proving that the strong Novikov conjecture holds for any countable discrete group that admits a coarse embedding with finite complexity into this universal Banach space.
	Finally, we study some non-sofic groups,  all of which satisfy the Novikov conjecture.
	\end{abstract}

	
	\tableofcontents

	\section{Introduction}

	In their papers  \cite{BG,Haagerup-Przybyszewska},
	Brown--Guentner and Haagerup--Przybyszewska  proved that every countable discrete group admits  a proper affine isometric action on the universal Banach space
	$$\bigoplus_{p=1}^\infty
	\ell^{2p}(\mathbb{N}),$$ 
	where the direct sum is taken in the $\ell^2$-sense. They also raised the natural question of how to establish the Novikov conjecture for discrete groups that coarsely embed into this Banach space.
	
	The main purpose of this article is to prove the strong Novikov conjecture for discrete groups that admit coarse embeddings with finite complexity into this universal Banach space. In fact, we establish the following more general theorem.

	\begin{thm}\label{main-result1}
		Let $\Gamma$ be a countable discrete group. 
		If 
		$\Gamma$ 
		admits a
		coarse embedding   with finite complexity
		into an
		$\ell^2$-direct sum of infinitely many Property 
		$(H)$ 
		Banach spaces,
		then the rational strong Novikov conjecture holds for $\Gamma$, i.e., the Baum--Connes--Kasparov higher index map
		\[
		\mu: K_*(B\Gamma)  \longrightarrow K_*(C^*_{\rm r} \Gamma) 
		\]  
		is rationally
		injective,  
		where $K_*(B\Gamma)$ denotes the $K$-homology   of the classifying space $B\Gamma$ for the group $\Gamma$ and $K_*(C^*_{\rm r} \Gamma)$ denotes the $K$-theory of the reduced group $C^\ast$-algebra $C^*_{\rm r} \Gamma$  associated to $\Gamma$. 
	\end{thm}
	
	We note that the rational strong Novikov conjecture implies both the Novikov conjecture on the homotopy invariance of higher signatures and the  Gromov--Lawson--Rosenberg conjecture concerning the nonexistence of positive scalar curvature metrics on closed aspherical manifolds.
	
	Next,  we introduce the notion of \emph{finite complexity} for coarse embeddings into Banach spaces. We begin by recalling that a 
	metric space 
	$X$
	is said to be coarsely embeddable into a metric
	space 
	$Y$ 
	(see Gromov \cite{G1}) 
	if there exists a map 
	$h:X \to {Y}$ 
	for which there exist
	non-decreasing functions 
	$\rho^{\pm}$ 
	from 
	$[0,+\infty)$ 
	to
	$[0,+\infty)$  
	such that
	\begin{itemize}
		\item [(1)] 
		$\rho^{-}
		\big( 
		d_X(x_{1},x_{2})
		\big)
		\leq
		d_Y
		\big(
		h(x_{1}),
		h(x_{2})
		\big)
		\leq
		\rho^{+}
		\big(
		d_X(x_{1},x_{2})
		\big)
		$ for all  
		$x_{1}, x_{2}\in {X}$; and
		\item [(2)] 
		$\lim\limits_{t\to\infty}
		\rho^{-}(t)=+\infty$.
	\end{itemize}

	A \emph{paving} of a Banach space $\mathscr{X}$ 
	is 	an increasing sequence
	$V_{1}\subseteq V_{2} \subseteq \cdots$ 
	of finite-dimensional subspaces of $\mathscr{X}$  
	whose union is dense in $\mathscr{X}$.
	Recall that 
	a real Banach space 
	$\mathscr{X}$ 
	is said to have \emph{Property 
		$(H)$}  \cite{KY}
	if there exist  pavings
	$\{V_{n}\}_{n\in\mathbb{N}}$ 
	of 
	$\mathscr{X}$ 
	and
	$\{W_{n}\}_{n\in\mathbb{N}}$
	of a real Hilbert space 
	$\mathscr{H}$, and a
	uniformly continuous map $\phi$ from the unit sphere $S(\mathscr{X})$ of the Banach space $ \mathscr{X}$ to the unit sphere $S(\mathscr{H})$ of the Hilbert space $\mathscr{H}$
	such that the restriction of $\phi$ to $S(V_n)$ is a homeomorphism onto $S(W_n)$ 
	for all $n\in\mathbb{N}$.
	The map	
	$\phi$ is called  a
	\emph{Property 
		$(H)$ 
		map}. 
	We remark that the class of Banach spaces with Property $(H)$ is quite rich. For example,  $\ell^{p}(\mathbb{N})$ has Property $(H)$, where the Property $(H)$ map $\phi: S(\ell^{p}(\mathbb{N}) ) \rightarrow S(\ell^{2}(\mathbb{N}))$ 
	is given by  the Mazur map: $$\phi(x)_n = \operatorname{sign}(x_n)\, |x_n|^{p/2}$$ for all $x=(x_1, \cdots, x_n, \cdots) $ in the unit sphere of $ \ell^{p}(\mathbb{N}),$ and the subspaces $V_n$ of $\ell^{p}(\mathbb{N})$ and $W_n$  of $ \ell^{2}(\mathbb{N})$  are the finite-dimensional spaces consisting of all vectors whose coordinates vanish from the $(n+1)$-st position onward.
	More generally,  every Banach space with  an unconditional basis and nontrivial  cotype  has 
	 Property $(H)$.
	 Important examples include the commutative (non-commutative) $L^p$-spaces,
	 and the  
	Banach space of all Schatten $p$-class
	operators on a Hilbert space.

	We are now ready to introduce the concept of \emph{coarse 
		embedding with finite complexity}.

	\begin{defn}\label{mostkeyidea1}
		Let $Y$ be a metric space, 
		and let $\mathscr{E}=\bigoplus_{p=1}^{\infty}
		\mathscr{X}_p$ 
		be the $\ell^2$-direct sum of infinitely many  Property 
		$(H)$
		Banach spaces with Property 
		$(H)$ maps 
		$\phi_p: S(\mathscr{X}_p)\to S(\mathscr{H})$.
		A map 
		$h: Y \to \mathscr{E}$ 
		is called a
		\emph{coarse 
			embedding with finite complexity}  
		if there exist (not necessarily non-decreasing) functions $\rho_p^{\pm}:[0,+\infty)\to [0,+\infty)$ 
		such that
		\begin{itemize}
			\item  [(1)]
			for all 
			$y_{1},y_{2}\in {Y}$ and 
			all $p\geq1$,
			$$\rho_p^{-}
			\big(d(y_{1},y_{2})\big)
			\leq
			\Big\|
			h_p(y_{1})-h_p(y_{2})
			\Big\|
			\leq
			\rho_p^{+}
			\big(d(y_{1},y_{2})\big), 
			$$
			where $h_p: Y \to \mathscr{X}_p$ is the $p$-th coordinate of $h$; 
			\item [(2)]
			$\lim\limits_{t \to \infty}
			\sum\limits_{p=1}^{\infty}
			\big(
			\rho_p^{-}(t)
			\big)^2=+\infty
			$;  and
			\item  [(3)]
			there exists an integer 
			$N \geq 0$ 
			such that for any  $p >  N$, the map $\phi_p$ is Lipschitz with the Lipschitz constant   $L_{\phi_p}$, and
			for each
			$t>0$, 
			$$
			\sum\limits_{p=N+1}^{\infty}
			\Big(
			( L_{\phi_p}+1 ) 
			\cdot 
			\rho_{p}^+(t)
			\Big)^2
			<+\infty.
			$$
		\end{itemize}
	\end{defn}

We emphasize that  no requirements are imposed on the first finitely many maps  $\phi_p$ in the third condition. 
Consequently,  the classical setting of coarse embeddings into
Banach spaces with  Property  $(H)$  is a special case of
Definition  \ref{mostkeyidea1}, and the main result of \cite{KY}
is recovered as a special case of Theorem \ref{main-result1}.
Moreover, the target spaces covered by our theorem include many
Banach spaces without Property $(H)$.
Indeed, Feng and Tian \cite{FengTian2026}  
showed that the
$\ell^2$-direct sum
\[
\bigoplus_{j=1}^{\infty}L^{p_j}(\Omega_{p_j},\nu_{p_j})
\]
does  not  have Property  $(H)$  whenever  $1\leq p_j<\infty$,
$\lim_{j\to\infty}p_j=+\infty$,  and each
$L^{p_j}(\Omega_{p_j},\nu_{p_j})$ is separable and infinite-dimensional.
In particular, this applies to the space
\[
\bigoplus_{p=1}^\infty
\ell^{2p}(\mathbb{N}),
\]
which is universal for coarse embeddings of all countable discrete groups.
Thus, Theorem  \ref{main-result1} treats, for the first time,
the strong Novikov conjecture in the case where the target
Banach space does not have Property  $(H)$.

	Next, we specialize to the universal Banach space, the $\ell^2$-direct sum of $\ell^p$-spaces. 
	We  introduce a notion of complexity rate for discrete groups, which measures the degree of complexity. This notion is a generalization of Property A  with respect to the universal Banach space. 
	Recall that a map $x \mapsto \tau(x)$ from a metric space $X$ to a Banach space
	is said to have $(R,\epsilon)$-\emph{variation} 
	if $d(x,y)\leq R$ implies 
	$\|\tau(x)-\tau(y)\|\leq\epsilon$.
	A discrete metric space $X$  with
	bounded geometry has \emph{Property A} if  for  every $R>0$ and every $\epsilon>0$, there exists a map 
	$\tau: X \to  S\big(\ell^1(X)\big)$ 
	and a constant $c>0$, 
	such that $\tau$ has $(R,\epsilon)$-variation
	and  the support of $\tau(x)$ is contained in the closed ball of radius $c$ centered at $x$.
	
	To simplify the notation, we restrict our attention to finitely generated groups.
	In 
	\cite{Haagerup-Przybyszewska}, 
	the authors strengthened a result of   Brown--Guentner \cite{BG} 
	by showing that  every finitely generated group $\Gamma$ admits a proper affine isometric action on 
	the universal Banach space 
	$$\bigoplus_{p=1}^\infty
	\ell^{2p}(\mathbb{N}).$$ 
	The cocycle
	$b:\Gamma \to \bigoplus_{p=1}^\infty
	\ell^{2p}(\mathbb{N})$
	associated to the affine isometric action 
	provides a coarse embedding of $\Gamma$ into this universal Banach space.  
	More precisely,
	this  cocycle  (or coarse embedding)  is induced by a sequence of maps 
	$$
	\Big\{\tau_{p}: 
	\Gamma \longrightarrow  \ell^{2p}(\mathbb{N}) \Big\}_{p\in\mathbb{N}},
	$$ 
	which satisfy the following conditions: 
	\begin{itemize}
		\item [(1)]  $\tau_{p}$ has $(1,\frac{C}{p})$-variation,
		where 
		$C>0$ is independent of $p$;  and
		\item [(2)] 
		$\big\|\tau_{p}(x)-\tau_{p}(y)\big\|\geq 1$
		whenever 
		$d(x,y)\geq 2p$. 
	\end{itemize}
This construction provides a quantitative weakening of Property A and thereby motivates the following notion of complexity rate.

	We set   
	\[
	\mathcal{D} =
	\Big\{
	(a_{0}, \cdots, a_{p}, \cdots )
	~ \Big|~ 
	a_{p}\geq0  
	\text{ for all } p,  \ 
	a_{p}=0 \text{ for all but finitely many } p 
	\Big\} 
	\cup 
	\Big\{ 
	(+\infty, 0, 0, \cdots)
	\Big\}.
	\]
	For  
	$\vec{a}=(a_{0}, \cdots, a_{p},\cdots)$ and $\vec{b}=(b_{0}, \cdots, b_{p},\cdots)$ in $\mathcal{D}$,
	we define  $\vec{a}<\vec{b}$ 
	if there exists $k\geq 0$
	such that 
	$a_{p} =b_{p}$ for all $0\leq p\leq k$,
	and $a_{k+1}< b_{k+1}$. 
	We write $\vec{a}\leq\vec{b}$ if $\vec{a}<\vec{b}$  
	or  $\vec{a}=\vec{b}$. 
	Then  $(\mathcal{D},\leq)$ is a totally ordered set. 
	Similarly, we can define a totally ordered set 
	$$\frac{1}{\mathcal{D}} = \left\{ \frac{1}{\vec{a}} ~\bigg|~ \vec{a} \in \mathcal{D}\right\},$$
	where the elements $\frac{1}{\vec{a}}$
	are formal symbols, and the order is given by
	\[ 
	\frac{1}{\vec{a}}  \geq \frac{1}{\vec{b}}
	\Longleftrightarrow 
	\vec{a} \leq \vec{b}\ 
	.
	\]
	We set 
	\[
	\pmb{0} = \frac{1}{(+\infty,0,0,\cdots)}
	\quad \text{and} \quad \pmb{\infty} = \frac{1}{(0,0,\cdots)}    
	,
	\]
	then 
	$\pmb{0}$ and $\pmb{\infty}$ are the minimal and maximal elements of $\frac{1}{\mathcal{D}}$
	respectively.
	Note that if $a_p \in \{0,1,\cdots,9\}$ for all $p$,
	then 
	$\mathcal{D}$ and $\frac{1}{\mathcal{D}}$,
	together with their order structures, are analogous to the rational numbers under the identification with decimal expansions 
	$a_{0\textbf{.}}a_{1}a_{2}\cdots$.

	For each $k \geq 1$, 
	define
	\[
	\mathfrak{ln}_k(x)=
	\begin{cases}
		\underbrace{\ln(\ln(\cdots \ln}\limits_{k}(x)\cdots )),  &  x>	\underbrace{\exp(\exp(\cdots \exp}\limits_{k}(1)\cdots)), 
		\\ 
		1, &  
		x \leq	\exp(\exp(\cdots \exp(1)\cdots)).
	\end{cases}
	\]
	When $k=1$, we write $\mathfrak{ln}$ 
	in place of  $\mathfrak{ln}_1$.

	\begin{defn}\label{klsdjflksjdlkf333}
		Let $\Gamma$ 
		be a finitely generated group endowed with a word length metric $d$. 
		Let 
		$\frac{1}{\vec{a}} > \pmb{0}$, 
		where  $\vec{a}=(a_{0},a_{1},\cdots,{a}_{k},0,\cdots) \in \mathcal{D}$.
		We say that $\Gamma$ has 
		\emph{complexity rate at most $\frac{1}{\vec{a}}$}
		if there exists  $C>0$
		and a sequence of maps
		$$
		\Big\{\tau_{p}: 
		\Gamma \longrightarrow  \ell^{2p}(\mathbb{N})
		\Big\}_{p\in\mathbb{N}}
		$$ 
		such that: 
		\begin{itemize}
			\item [(1)]
			$\tau_{p}$
			has
			$\Big(1,\ 
			\dfrac{C}{ p^{1+a_{0}} \cdot \big(\mathfrak{ln}(p)\big)^{a_{1}} \cdots \big( \mathfrak{ln}_k(p)\big)^{a_{k}}}
			\Big)$-variation;  and
			\item [(2)] there exists $N_{p}>0$ such that
			$\big\|\tau_{p}(x)-\tau_{p}(y)\big\|\geq 1$
			whenever 
			$d(x,y)\geq N_{p}$.
		\end{itemize}
		In this case, we write  
		$\text{CR}(\Gamma)\leq \frac{1}{\vec{a}}$.
			We write
		$\text{CR}(\Gamma)< \frac{1}{\vec{a}}$  
		if there exists $\vec{b}\in\mathcal{D}$ such that 
		$\text{CR}(\Gamma)\leq \frac{1}{\vec{b}}$
		and $\frac{1}{\vec{b}}<\frac{1}{\vec{a}}$.  
	We say that $\Gamma$ has a \emph{finite complexity rate} if  $\text{CR}(\Gamma) <\pmb{\infty}$.
	\end{defn}

Note  that the images of the maps
$\tau_{p}$ 
are not assumed to lie in the corresponding unit spheres or to consist of finitely supported functions.	
We remark that the order on $\frac{1}{\mathcal{D}}$
	is not redundant;
	it reflects the level of complexity in the sense that
	$\text{CR}(\Gamma)\leq \frac{1}{\vec{b}}$ 
	and 
	$\frac{1}{\vec{b}}\leq \frac{1}{\vec{a}}$
	imply
	$\text{CR}(\Gamma)\leq \frac{1}{\vec{a}}$.
The above definition is independent of the choice of word length metrics on $\Gamma$.

	Haagerup--Przybyszewska's result  implies that
	$\text{CR}(\Gamma)\leq\pmb{\infty}$  
	for every finitely generated group $\Gamma$.
	Thus, the complexity rates of all finitely generated  groups lie between
	$\pmb{0}$ and $\pmb{\infty}$. 
	On the other hand, Theorem 1.2.4 
	in \cite{Willett}  implies that  if $\Gamma$ has Property A, then   
	$\text{CR}(\Gamma)\leq \frac{1}{\vec{a}}$
	for every 
	$\frac{1}{\vec{a}} > \pmb{0}$;
	in other words,  groups with  Property A  admit arbitrarily small complexity rates.

	We emphasize that the notion of finite complexity rate is  substantially weaker than  Property A.
	Indeed,   
	Osajda (\cite{Osajda}) constructed a family of finitely generated groups with the Haagerup property (also known as Gromov's a-T-menability)  but without Property A.
	In fact,  finitely generated groups with the Haagerup property admit arbitrarily small complexity rates; 
	more generally, the same holds for those coarsely embeddable into  $\ell^{q}(\mathbb{N})$  for some $q\in[1,+\infty)$.    
	We give a brief proof.
	Let  $\Gamma$  be a finitely generated group that  admits a coarse embedding  into  $\ell^{q}(\mathbb{N})$ for some 
	$q\in [1,+\infty)$.  
	For any  integer
	$p\geq \frac{q}{2}$,    $\ell^{q}(\mathbb{N})$ admits a coarse embedding into  
	$\ell^{2p}(\mathbb{N})$ (see \cite{MN}). 
	Thus $\Gamma$ admits a coarse embedding into  $\ell^{2p}(\mathbb{N})$  for every integer
	$p\geq \frac{q}{2}$.    
	We denote this embedding by $h_{p}$.  
	Since $\Gamma$ is quasi-geodesic with respect to the word metric,  there exists $C_{p}>0$ such that 
	$\|h_{p}(x)-h_{p}(y)\|\leq C_{p}\cdot d(x,y)$ 
	for all $x,y\in\Gamma$.  Now for any $R>0$, the rescaled map $\tau_{p,R}:=\frac{1}{R}\cdot h_{p}$ has $(1,\frac{C_p}{R})$-variation.  Moreover, since $h_{p}$ is a coarse embedding,  there exists  $N_p>0$ such that 
	$\|\tau_{p,R}(x)-\tau_{p,R}(y)\| \geq 1$  whenever  $d(x,y)\geq N_p$. 
	Hence,  for every 
	$\frac{1}{\vec{a}} > \pmb{0}$,  we can find a  family of maps  
	$\big\{ \tau_{p}: \Gamma \to 
	\ell^{2p}(\mathbb{N}) \big\}_{p\in\mathbb{N}}$
	such that $\text{CR}(\Gamma)\leq \frac{1}{\vec{a}}$,
	where  the maps $\tau_{p}$ for $p\geq \frac{q}{2}$ are obtained from the above argument, 
	while the finitely many remaining maps are chosen as in Haagerup--Przybyszewska's construction.  
	Consequently,  Osajda's  examples  have  finite complexity rates but do not have Property A.

The following result shows that any finite complexity rate, no matter how large, can guarantee a good embedding into the 	 
universal Banach space  $\bigoplus_{p=1}^\infty
\ell^{2p}(\mathbb{N})$.

	\begin{prop}
		If  $\Gamma$  has a finite complexity rate, i.e.	$\text{CR}(\Gamma)<\pmb{\infty}$,
		then $\Gamma$ admits a coarse 
		embedding with finite complexity into $\bigoplus_{p=1}^\infty
		\ell^{2p}(\mathbb{N})$.
	\end{prop}	
	\begin{proof}
		
		Since 
		$\text{CR}(\Gamma)<\pmb{\infty}$,
		there exists 
		$\vec{a}=(a_{0}, a_{1}, \cdots, a_{k}, 0 \cdots)>(0,0,0\cdots)$ 
		such that $\text{CR}(\Gamma)\leq \frac{1}{\vec{a}}$. 
		Without loss of generality, we assume that $a_k>0$.
		By Definition \ref{klsdjflksjdlkf333}, there exists a family of maps 
		$$
		\Big\{\tau_{p}: \Gamma \longrightarrow \ell^{2p}(\mathbb{N})
		\Big\}_{p\in\mathbb{N}}
		$$
		satisfying conditions  
		$(1)$ and $(2)$. 
		For each $p\in\mathbb{N}$, 
		define  
		$h_p: \Gamma \to \ell^{2p}(\mathbb{N})$
		by 
		\begin{equation*}
			h_p(\gamma)=
			\dfrac{\tau_{p}(\gamma)-\tau_{p}(e)}{~\sqrt{p \cdot \mathfrak{ln}(p) \cdots  \mathfrak{ln}_k(p)}~}
		\end{equation*}
		for every $\gamma\in\Gamma$, 
		where $e$ is the identity element of  $\Gamma$.

		We claim that the map
		$h:=\oplus_{p=1}^{\infty} h_p$ is a coarse 
		embedding with finite complexity from  $\Gamma$ 
		into $\bigoplus_{p=1}^{\infty}
		\ell^{2p}(\mathbb{N})$.
	First, 
	since the Cayley graph of the finitely generated group 
	$\Gamma$ is geodesic with respect to the path metric,   
	condition $(1)$ implies that
	$$
	\big\|\tau_{p}(x)-\tau_{p}(y)\big\|\leq \frac{C\cdot d(x,y)}{~ p^{1+a_{0}} \cdot \big(\mathfrak{ln}(p)\big)^{a_{1}} \cdot\cdots \big( \mathfrak{ln}_k(p)\big)^{a_{k}}}
	$$
	for all   $x,y\in\Gamma$ and  $p\geq 1$,
	where 	$C>0$ is independent of $p$.
It follows that 
\[
\sum_{p=1}^{\infty} 
\|h_p(\gamma)\|^2 <+\infty 
\]
for every 
$\gamma\in\Gamma$, 
and hence the map $h$ is well-defined.	
Moreover,
define 
	$$
	\rho_{p}^{+}(t)=
	\dfrac{C\cdot t}{~
		p^{a_{0}+\frac{3}{2}} \cdot \big(\mathfrak{ln}(p)\big)^{a_{1}+\frac{1}{2}} \cdot\cdots \big( \mathfrak{ln}_k(p)\big)^{a_{k}+\frac{1}{2}}
	}
	$$
	for every
	$t \geq 0$.	
	Then for all 
	$x,y\in \Gamma$ 
	and  
	$p\geq1$,	
	$$
	\Big\|
	h_p(x)-h_p(y)
	\Big\|
	\leq
	\rho_p^{+}
	\big(d(x,y)\big). 
	$$
	Since  $a_j\geq0$ for $0\leq j\leq k-1$, and $a_k>0$,
	we know that for each
	$t\geq 0$, 
	\begin{eqnarray*}
		\sum\limits_{p=1}^{\infty}
		\Big(
		(p+1) \cdot 
		\rho_{p}^+(t)
		\Big)^2
		&\leq& \sum\limits_{p=1}^{\infty}
		\Big(
		2p \cdot 
		\rho_{p}^+(t)
		\Big)^2
		\\
		&=&
		\sum\limits_{p=1}^{\infty}
		\Big(
		\dfrac{2C t}{
			p^{a_{0}+\frac{1}{2}} \cdot \big(\mathfrak{ln}(p)\big)^{a_{1}+\frac{1}{2}} \cdot\cdots \big( \mathfrak{ln}_k(p)\big)^{a_{k}+\frac{1}{2}}
		}
		\Big)^2
		\\
		&=&
		4 C^2 t^2\cdot \sum\limits_{p=1}^{\infty}
		\dfrac{1}{
			p^{2a_{0}+1} \cdot \big(\mathfrak{ln}(p)\big)^{2a_{1}+1} \cdot\cdots \big( \mathfrak{ln}_k(p)\big)^{2a_{k}+1}
		}
		\\
		&\leq&
		4 C^2 t^2\cdot \sum\limits_{p=1}^{\infty}
		\dfrac{1}{
			p \cdot \mathfrak{ln}(p) \cdots
			\mathfrak{ln}_{k-1}(p) \cdot  \big( \mathfrak{ln}_k(p)\big)^{2a_{k}+1}
		}	
		\\
		&<& +\infty.
	\end{eqnarray*}
	Note that for each $p\in\mathbb{N}$,
	the Property $(H)$ map  (Mazur map)  of $\ell^{2p}(\mathbb{N})$  has  Lipschitz constant $p$.
	Thus, the third condition in Definition \ref{mostkeyidea1} holds.

On the other hand,  condition $(2)$ implies that for each 
		$p\geq 1$, 
		there exists $N_{p}>0$ such that  
		$$
		\big\|\tau_{p}(x)-\tau_{p}(y)\big\|\geq 1
		$$ 
		whenever  $d(x,y)\geq N_{p}$.
		Define 
		\[
		\rho_{p}^{-}(t)=
		\begin{cases}
			\dfrac{1}{
				\sqrt{p \cdot \mathfrak{ln}(p) \cdots  \mathfrak{ln}_k(p)}
			}, &  
			t \geq N_{p},
			\\
			\hspace{0.1cm} 
			0, &  
			0\leq  t < N_{p}.
		\end{cases}
		\]
		Then for all 
		$x,y\in \Gamma$ and 
		$p\geq1$,	
		$$
		\rho_p^{-}
		\big(d(x,y)\big)
		\leq
		\Big\|
		h_p(x)-h_p(y)
		\Big\|. 
		$$
		Moreover,  
		for every $n\in\mathbb{N}$ 
		and every
		$t \geq \max\{N_1, \cdots, N_n\}$, 	
		\begin{equation*}
			\sum\limits_{p=1}^{\infty}
			\big(
			\rho_{p}^{-}(t)
			\big)^2
			\geq
			\sum\limits_{p=1}^{n}
			\big(
			\rho_{p}^{-}(t)
			\big)^2
			\\
			=
			\sum\limits_{p=1}^{n}
			\dfrac{1}{
				p \cdot \mathfrak{ln}(p) \cdots  \mathfrak{ln}_k(p)
			}.
		\end{equation*}
		Since 
		$$\sum\limits_{p=1}^{\infty}
		\dfrac{1}{
			p \cdot \mathfrak{ln}(p) \cdots  \mathfrak{ln}_k(p)
		}=+\infty,
		$$ 
		it follows that   	
		$$
		\lim\limits_{t\to\infty}
		\sum\limits_{p=1}^{\infty}
		\big(
		\rho_{p}^{-}(t)
		\big)^2
		=+\infty.
		$$

		Therefore,
		$h$ is a coarse embedding with finite complexity. 
	\end{proof}

	As a consequence of the above proposition and  Theorem \ref{main-result1}, 
	we have the following result.
	
	\begin{cor}\label{theclassofcwodimayaijksdjlfjalksdjf}
		The rational strong  Novikov conjecture holds for all finitely generated groups $\Gamma$ with 
		$\text{CR}(\Gamma)<\pmb{\infty}$.
		Moreover, 
		let $\mathcal{C}$  denote the class of all groups that can be expressed as inductive limits of finitely generated 
		groups with finite complexity rates, 
		where the connecting homomorphisms are not assumed to be injective, then
		the  Novikov conjecture holds for every  $\Gamma\in\mathcal{C}$.  
	\end{cor}

The class  $\mathcal{C}$   includes  amenable groups, free groups,  Gromov's hyperbolic groups,  subgroups of connected Lie groups,
subgroups of the mapping class groups,
subgroups of the outer automorphism groups of the free groups, Gromov's monster groups, 	Arzhantseva--Tessera's monster groups, 
Kun--Thom's  non-sofic groups, 
and   groups that  admit coarse embeddings into the principal  Hilbert--Hadamard spaces studied in \cite{GWY}  
etc.  
Examples of non-sofic groups were discovered quite recently. 
In the final section,  we prove that the non-sofic groups constructed by
Kun and Thom have Property A, 
thereby  resolving Pestov's  question  (\cite[Open Question 9.6]{VladimirGPestov2008}) in the negative.

The result of Haagerup and Przybyszewska, together with the framework of finite complexity rates, suggests an affirmative answer to the  open question:
Is every finitely generated group in the class 
$\mathcal{C}$?

	The proof of the main result of our article  differs substantially from the classical Dirac-dual-Dirac method of Kasparov. In the Hilbert space setting, the Dirac-dual-Dirac method establishes the injectivity of the Baum--Connes assembly map by constructing both a Bott element and a Dirac element whose Kasparov product is the identity. In our setting, however, the target Banach space is no longer a Hilbert space, and there is no natural Dirac operator or Dirac element available. Consequently, the classical Dirac-dual-Dirac method is no longer applicable.
	
The first ingredient in our approach is the construction of a proper $\Gamma$-$C^*$-algebra associated to a coarse embedding with finite complexity. 
Our construction is based on a new method that avoids the technical language of continuous fields and instead works directly with a simple and explicit trivial bundle over the base space $Z$, 
where $Z$ is the $\Gamma$-invariant contractible metric subspace of
$\ell^1(\Gamma)$ consisting of  all non-negative functions with $\ell^1$-norm $1$.  
Since the target space for coarse embedding is an infinite direct sum, we decompose it into an increasing filtration by finite partial sums. Each term in this filtration is a finite direct sum of Banach spaces with Property $(H)$, and hence itself has Property $(H)$. 
For each term of the filtration, regarded as a single Banach space with Property $(H)$, we construct an associated $\Gamma$-$C^*$-algebra. 
The desired proper $\Gamma$-$C^*$-algebra in the general setting is then obtained by taking a suitable limit of these algebras. The second condition in the definition of a coarse embedding with finite complexity plays a crucial role in establishing the properness of the resulting $\Gamma$-$C^*$-algebra.

The second ingredient is the construction of the Bott element by means of $\Gamma$-equivariant asymptotic morphisms. This element induces maps from  $KK$-theory with coefficients in $C_0(\mathbb{R})$ 
to $KK$-theory with coefficients in the proper algebra. 
Establishing $\Gamma$-equivariance requires the asymptotic morphisms to be asymptotically  invariant under the group action, and the third condition in the definition of a coarse embedding with finite complexity  provides the quantitative bound needed for this purpose.

After these two constructions,
we obtain a commutative diagram  	
\begin{equation*}
	\begin{array}{ccccc}
		KK^\Gamma_*\big( {E}\Gamma, C_0(\mathbb{R}) \big)  & \xlongrightarrow{\pi_{*}} &  KK^\Gamma_*\big( \underline{E}\Gamma, C_0(\mathbb{R}) \big) & \xlongrightarrow{\mu} & K_*\big( C_0(\mathbb{R}) \rtimes_{\operatorname{r}} \Gamma \big) 
		\\
	\Big\downarrow  &&	\Big\downarrow   & & \Big\downarrow  
		\\
		KK^\Gamma_*({E}\Gamma, A) & \xlongrightarrow{\pi_{*}} & KK^\Gamma_*(\underline{E}\Gamma, A) & \xlongrightarrow{\mu} & K_*(A \rtimes_{\operatorname{r}} \Gamma),
	\end{array}
\end{equation*}
where $A$ is a proper $\Gamma$-$C^*$-algebra,
and $\pi: E\Gamma \to \underline{E}\Gamma$ is the natural $\Gamma$-equivariant continuous map (see the Appendix).
The properness of $A$ ensures that the bottom map $\mu$ is an isomorphism.
The homomorphism 
$\pi_*$ is rationally injective 
by Lemma \ref{jdskljfaljsdkfjsqqq} in the Appendix.
If the left vertical map (Bott map)
is  (rationally)  injective,
then the commutative diagram implies that the composition of the maps in the top row is rationally injective, as desired. 
The proper $\Gamma$-$C^*$-algebra constructed initially in Section 5 
can be used to establish the injectivity of the Bott map only in a special case.
The general case calls for a further key construction introduced in this paper:
a deformation of group actions from the original proper action to one that acts trivially on the fibres while moving only the base space. 
The reason for introducing this construction is that the $K$-theory of the proper algebra cannot be computed directly.
To overcome this difficulty, 
we enlarge the proper $\Gamma$-$C^*$-algebra in Section 5  so that the resulting  $C^*$-algebra  admits a continuous deformation 
of  $\Gamma$-actions,  
starting from an action that encodes the  original proper action and ending at an almost trivial action.
This deformation allows us to compare the Bott maps associated to the original proper action with those associated to a much simpler action that is trivial on the fibres. The latter can be analyzed by ordinary $K$-theoretic methods, and this comparison leads to the rational injectivity of the Bott maps.
Thus the proof replaces the classical Dirac-dual-Dirac method by a new strategy based on new proper coefficient algebras, Bott elements, and a deformation of group actions.

	This paper is organized as follows. In Section 2, we recall the Property $(H)$ maps for Lebesgue--Bochner  spaces and  then construct proper extensions of arbitrary Property $(H)$ maps.
	These extensions will play
	crucial roles in 
	the construction of the  $C^*$-algebra
	$\mathcal{A}
	\big(
	C_{b}(Y,\mathscr{X})
	\big)$  in Section 4  and of the  Bott elements in Sections 6 and 8. 
	In Section 3, we resolve the problem of constructing a proper affine isometric action of the transformation groupoid $Z\rtimes\Gamma$ on a trivial bundle associated to a coarse embedding $h:\Gamma\to\mathscr{X}$, where $\mathscr{X}$ is an arbitrary Banach space and $Z$ is the metric subspace of $\ell^1(\Gamma)$ consisting of all non-negative functions with $\ell^1$-norm $1$.
	In Section 4, we construct the $C^*$-algebra associated to the trivial bundle constructed in Section 3.
	In Sections 5 and 6, we construct a proper $\Gamma$-$C^*$-algebra together with its Bott element, under the assumption that $\Gamma$ admits a coarse embedding with finite complexity into an $\ell^2$-direct sum of infinitely many Property $(H)$ Banach spaces. We emphasize that these constructions already suffice to prove the strong Novikov conjecture in a special case where the $K$-theory of the proper $\Gamma$-$C^*$-algebra is computable in an appropriate sense.  
Moreover, these sections are devoted to illustrating  the main ideas underlying our construction and to providing a prototype for the more elaborate construction developed in the subsequent two sections.
	In Section 7, we construct a larger proper $\Gamma$-$C^*$-algebra that encodes  the information of the one in Section 5. 
	This larger algebra admits a continuous deformation of the original proper group action to an action that acts trivially on the fibres while moving only the base space. 
	This deformation plays a key role in establishing the rational injectivity of the Bott maps.
	In Section 8, we construct the (Bott element) Bott maps associated to the larger proper $\Gamma$-$C^*$-algebra. 
	In Section 9, we establish the rational injectivity of the Bott maps using the continuous deformation of group actions.
	In Section 10, we prove the main theorem of this paper.
	In Section 11, we collect several elementary facts from $KK$-theory that are used throughout the paper.
	Finally, in Section 12, we prove that  Kun--Thom's  non-sofic groups have Property A.  
	These are the first examples of Property A groups that are not sofic.
	This gives a negative answer to an open question posed by  Pestov  in \cite{VladimirGPestov2008}.
	We also show  that countable groups admitting coarse embeddings into the principal  Hilbert--Hadamard spaces studied in \cite{GWY} fall within the scope of our  main theorem and hence satisfy the Novikov conjecture.

	\emph{Acknowledgements.} We thank Florent Baudier,  Bill Johnson, Assaf Naor, Rufus Willett, and Jianchao Wu  for very inspiring discussions.
	We also thank Jintao Deng for his comments  that improved the readability of this paper.

	\section{Property $(H)$ map and its extensions}

	\subsection{Property $(H)$ maps for Lebesgue--Bochner  spaces}
	
The Banach spaces considered throughout this paper are primarily of the form $\ell^2(\Gamma,\mathscr{X})$ or $L^2([0,1], \mathscr{X})$, 
where $\mathscr{X}$ is a Banach space with  Property $(H)$.
Both are Lebesgue--Bochner spaces, 
as defined in the next paragraph.
For our purposes, it is essential that
these spaces inherit Property $(H)$ from 
$\mathscr{X}$.
In this section,
we recall this fact from \cite{CD2017},
where the general
Banach-valued Mazur maps provide the required Property $(H)$ maps for these spaces.

	Let $(\mathscr{X},\|\cdot\|)$ 
	be a Banach space and let $(\Omega,\mu)$ 
	be a measure space.
	Recall that the   
	\emph{Lebesgue--Bochner space} 
	$L^p(\Omega,\mu,\mathscr{X})$,
	where $1\leq p<+\infty$,
	is the Banach space of all (equivalence classes of) 
	Bochner-integrable functions
	$f:\Omega\rightarrow\mathscr{X}$
	endowed with the norm
	$$
	\|f\| =
	\Big(
	\mathlarger{\int}_{\Omega} 
	\big\|
	f(v)
	\big\|^pd\mu 
	\Big)^{\frac{1}{p}}.
	$$
	In particular, if $\mathscr{X}$ 
	is the scalar field,
	then $L^p(\Omega,\mu,\mathscr{X})$
	is the classical $L^p$-space
	$L^p(\Omega,\mu)$.
	If $\Omega$ is a countable discrete space and 
	$\mu$ is the counting measure, 
	then the space 
	$L^p(\Omega,\mu,\mathscr{X})$ 
	reduces to the Banach sequence space $\ell^p(\Omega,\mathscr{X})$.

	Let  $\mathscr{X}$, $\mathscr{Y}$ be  Banach spaces,  and let  $\phi: S(\mathscr{X}) \to S(\mathscr{Y})$ be a map between  their unit spheres.
	For $1 \leq  p,q < +\infty$,  
	define a map 
	\[F_{\phi,p,q}: L^p(\Omega,\mu,\mathscr{X})\longrightarrow L^q(\Omega,\mu,\mathscr{Y})\] 
	by 
	\begin{equation}
		\label{eq:extendedmapforBanachsequencespaces}
		\big( 
		F_{\phi,p,q}(f)
		\big)(w)=\|f(w)\|^{\frac{p}{q}}\cdot \phi\Big(\dfrac{f(w)}{\|f(w)\|} 
		\Big)
	\end{equation}
	for
	$f\in  L^p(\Omega,\mu,\mathscr{X})$
	and $w\in\Omega$, 
	where
	$\big( 
	F_{\phi,p,q}(f)
	\big)(w)$ 
	should be interpreted as $0$ if 
	$f(w)=0$.
	It follows directly that 
	$F_{\phi,p,q}$   
	maps 
	$S\big( 
	L^p(\Omega,\mu,\mathscr{X}) 
	\big)$ 
	into
	$S\big( 
	L^q(\Omega,\mu,\mathscr{Y}) 
	\big)$.
	We denote by 
	$\mathcal{M}_{\phi,p,q}$
	the restriction of 
	$F_{\phi,p,q}$
	to the unit spheres, 
	and call it the 
	\emph{Banach-valued Mazur map}.

	If $\mathscr{X}=\mathscr{Y}=\mathbb{C}$ and $\phi$ is the identity map, 
	then $\mathcal{M}_{\phi,p,q}$  
	reduces to the classical Mazur map from 
	$S\big( 
	L^p(\Omega,\mu)
	\big)$ 
	to 
	$S\big( 
	L^q(\Omega,\mu)
	\big)$.  
	In this case, 	
	we omit the symbol $\phi$ and write 
	$\mathcal{M}_{p,q}$. 
	Moreover, 
	$\mathcal{M}_{p,q}$
	is the inverse of 
	$\mathcal{M}_{q,p}$
	for all
	$1 \leq  p,q < +\infty$.
	If $\Omega$  is a countable discrete space equipped with the counting measure,
	then $\mathcal{M}_{p,q}$
	is given by
	\begin{eqnarray*}
		\mathcal{M}_{p,q}: S\big(\ell^p(\Omega)\big)
		&\longrightarrow& S\big(\ell^q(\Omega)\big)
		\\
		(a_j)_{j\in\Omega}
		&\longmapsto& 
		\big(
		{\rm sign}(a_j)\cdot  |a_j|^{\frac{p}{q}}
		\big)_{j\in\Omega},
	\end{eqnarray*} 
	where 
	${\rm sign}(a)=\frac{a}{|a|}$ if $a \ne 0$, and  ${\rm sign}(0)=0$.
	We now recall a basic property of the classical Mazur map.
	\begin{lem}[\cite{BL2000}]\label{ineaualityforMazurmap1}
		Let $p\geq q\geq 1$. 
		Then there exists a constant 
		$C=C(p,q)>0$
		such that 
		\begin{equation}
			\label{eq:constantformazurmap}
			C\cdot 
			\big\|
			f-g
			\big\|^{\frac{p}{q}}
			\leq
			\Big\|
			\mathcal{M}_{p,q}(f)-\mathcal{M}_{p,q}(g)
			\Big\|
			\leq \frac{~p~}{q}\cdot \big\|
			f-g
			\big\|
		\end{equation}
		for all 
		$f,g\in 
		S\big( 
		L^p(\Omega,\mu)
		\big)$.
	\end{lem}
	
	We remark that the classical $L^p$-spaces have Property $(H)$ with respect to the classical Mazur map.
	The following result shows that  Lebesgue--Bochner  spaces also have Property $(H)$ with respect to  the  Banach-valued Mazur map.

	\begin{prop} 
		[{\cite[Theorem 3.3]{CD2017}}]
		\label{PropertyHmapforLpspaces}
		Assume that $\mathscr{X}$ is a Property $(H)$ Banach space with a Property $(H)$ map
		$\phi: S(\mathscr{X}) \to S(\mathscr{H})$.
		Then for any 
		$1 \leq  p < +\infty$,   $L^p(\Omega,\mu,\mathscr{X})$
		is a Property $(H)$ space with respect to the Banach-valued Mazur map 
		$\mathcal{M}_{\phi,p,2}: 
		S\big( 
		L^p(\Omega,\mu,\mathscr{X}) 
		\big)\to 
		S\big( 
		L^2(\Omega,\mu,\mathscr{H}) 
		\big)$.
	\end{prop}
	
	In this paper, we consider only the case $p=2$, 
	and write 
	$\mathcal{M}_{\phi}$
	in place of 
	$\mathcal{M}_{\phi,2,2}$.

	\begin{rmk}\label{properHmapsforiteratedconstructionkdjlkjl}
		As a corollary of Proposition \ref{PropertyHmapforLpspaces}, 	
		if  $\mathscr{X}$ has  Property $(H)$, 
		then so do
		$\ell^2(\Omega, \mathscr{X})$ 
		and   
		$L^2 ( [0,1], m, 
		\mathscr{X} )$,
		where $\Omega$ is a countable discrete space equipped with
		the counting measure and $m$ is the Lebesgue measure on $[0,1]$. 
		For simplicity, we write  
		$L^2 ( [0,1],  
		\mathscr{X} )$ 
		in place of 
		$L^2 ( [0,1], m, 
		\mathscr{X} )$.
		These spaces, together with their iterated constructions 
		(for instance 
		$L^2 ( 
		[0,1],  
		\ell^2(\Omega, \mathscr{X})
		)$),
		will serve as the main objects of study in this paper.
		Moreover, if $\phi: S(\mathscr{X}) \to S(\mathscr{H})$ is a Property $(H)$ map for $\mathscr{X}$,
		then 
		$$
		\mathcal{M}_{\phi}: S\big(\ell^2(\Omega, \mathscr{X})\big) \longrightarrow  S\big(\ell^2(\Omega, \mathscr{X})\big)
		$$
		is a Property $(H)$ map for $\ell^2(\Omega, \mathscr{X})$.
		Applying Proposition \ref{PropertyHmapforLpspaces} once more yields that
		$$\mathcal{M}_{\mathcal{M}_{\phi}}: S\Big(L^2 
		\big( 
		[0,1],  
		\ell^2(\Omega, \mathscr{X})
		\big)\Big) \longrightarrow  
		S\Big(L^2 
		\big( 
		[0,1],  
		\ell^2(\Omega, \mathscr{X})
		\big)\Big)
		$$
		is a Property $(H)$ map for $L^2 
		\big( 
		[0,1],  
		\ell^2(\Omega, \mathscr{X})
		\big)$.
	\end{rmk}

	Next, 
	we show that the Property $(H)$ map $\mathcal{M}_{\phi}$ is Lipschitz  whenever $\phi$ is Lipschitz. 
	
	Given Banach spaces $\mathscr{X}$, $\mathscr{Y}$, a map 
	$\phi: S(\mathscr{X}) \to S(\mathscr{Y})$
	and a function $\eta:[0,+\infty)\to[0,+\infty)$ satisfying $\eta(1)=1$,
	we define a map 
	$\phi^{\eta}: \mathscr{X}\to \mathscr{Y}$ by 
	\[
	\phi^{\eta}(x)=
	\begin{cases}
		\eta(\|x\|) \cdot \phi
		\Big(\dfrac{x}{\|x\|}\Big), & x \ne 0,
		\\
		0, & x=0.
	\end{cases}
	\]
	We call $\phi^{\eta}$ the \emph{$\eta$-scalar extension} of $\phi$.
	In the special case $\eta(t)\equiv t$, 
	we simply call it the \emph{scalar extension} of $\phi$,
	and denote it by
	$\phi^{I}$.

	\begin{lem}\label{3lip}
		Let $\mathscr{X}$, $\mathscr{Y}$  be Banach spaces, and  let
		$\phi: S(\mathscr{X}) \to S(\mathscr{Y})$  be a map between their unit spheres. 
		If $\phi$ is  Lipschitz,
		then  $\phi^I$ is also 
		Lipschitz with Lipschitz constant at most  $2\cdot L_{\phi}+1$. 
	\end{lem}
	\begin{proof}
		For any nonzero vectors
		$v,w\in\mathscr{X}$,
		we have
		\begin{eqnarray*}
			&& \Big\| 
			\phi^I(v)- \phi^I(w)
			\Big\|
			\\
			&\leq& 
			\bigg\| 
			\|v\|\cdot 
			\phi
			\Big(
			\frac{v}{\|v\|}
			\Big)- \|v\|\cdot 
			\phi
			\Big(
			\frac{w}{\|w\|}
			\Big)\bigg\| +
			\bigg\| 
			\|v\|\cdot 
			\phi
			\Big(
			\frac{w}{\|w\|}
			\Big)- 
			\|w\|\cdot 
			\phi
			\Big(
			\frac{w}{\|w\|}
			\Big)
			\bigg\|
			\\
			&=&\|v\|\cdot 
			\bigg\| 
			\phi
			\Big(
			\frac{v}{\|v\|}
			\Big)- 
			\phi
			\Big(
			\frac{w}{\|w\|}
			\Big)\bigg\| +
			\Big| 
			\|v\|- \|w\|
			\Big|
			\\
			&\leq&
			L_{\phi} \cdot \|v\| \cdot 
			\bigg\| 
			\frac{v}{\|v\|}
			- 
			\frac{w}{\|w\|}
			\bigg\| +
			\| v - w \|
			\\
			&\leq&
			L_{\phi} \cdot \|v\| \cdot
			\Bigg(
			\bigg\|  \frac{v}{\|v\|}- \frac{w}{\|v\|}\bigg\|
			+\bigg\|\frac{w}{\|v\|}-  \frac{w}{\|w\|}\bigg\|
			\Bigg)
			+\|v - w\|
			\\
			&\leq &
			L_{\phi} \cdot \|v\| \cdot\frac{2\cdot 
				\|v - w\|}{\|v\|}+
			\|v - w\|
			\\
			&=&\big(2\cdot L_{\phi}+1\big) \cdot\|v-w\|.
		\end{eqnarray*}	
		If one of $v$ and $w$ is zero, 
		then  
		$$
		\Big\| 
		\phi^I(v)- \phi^I(w)
		\Big\| = \|v-w\|.
		$$ 
		This finishes the proof.
	\end{proof}

	\begin{prop}\label{LipofextenedMazurmap}
		Assume that $\mathscr{X}$ is a Property
		$(H)$ Banach space with a Lipschitz Property $(H)$ map
		$\phi: S(\mathscr{X}) \to S(\mathscr{H})$.
		Then for each 
		$1\leq p <\infty$,
		$F_{\phi,p,p}: 
		L^p(\Omega,\mu,\mathscr{X}) 
		\to 
		L^p(\Omega,\mu,\mathscr{H})$
		is  Lipschitz with  Lipschitz constant at most
		$2 \cdot  L_{\phi}+1$. 
		
		In particular,  
		the Property $(H)$ map 
		$\mathcal{M}_{\phi}: 
		S\big( 
		L^2(\Omega,\mu,\mathscr{X}) 
		\big)\to 
		S\big( 
		L^2(\Omega,\mu,\mathscr{H}) 
		\big)$ 
		for 
		$L^2(\Omega,\mu,\mathscr{X})$
		is Lipschitz with Lipschitz constant at most
		$2 \cdot  L_{\phi}+1$. 
	\end{prop}
	\begin{proof}
		
		In view of \eqref{eq:extendedmapforBanachsequencespaces}, 
		we know that 
		$F_{\phi,p,p}$ is defined via scalar extension. 
		Thus,  
		for all 
		$f,g\in L^p(\Omega,\mu,\mathscr{X})$,
		\begin{eqnarray*}
			&& 
			\big\| F_{\phi,p,p}(f)-F_{\phi,p,p}(g)
			\big\| 
			\\
			&=&  
			\Big(
			\mathlarger{\int}_{\Omega} 
			\Big\|
			\big( 
			F_{\phi,p,p}(f)
			\big)(w) -
			\big( 
			F_{\phi,p,p}(g)
			\big)(w)
			\Big\|^p d\mu 
			\Big)^{\frac{1}{p}}
			\\
			&\stackrel{Lemma \ref{3lip}}{\leq}& 	\Big(
			\mathlarger{\int}_{\Omega} (2\cdot L_{\phi}+1)^p \cdot \big\| 
			f(w)- g(w) 
			\big\|^p
			d\mu 
			\Big)^{\frac{1}{p}}
			\\
			&=& (2\cdot L_{\phi}+1) \cdot \|f-g\|.
		\end{eqnarray*}
		The result also holds for 
		$\mathcal{M}_{\phi}$,
		since it
		is  the restriction of 
		$F_{\phi,2,2}$. 
	\end{proof}

	\subsection{Proper extensions of Property $(H)$ maps} 
	
	In this subsection,
	we  study  proper extensions of  Property $(H)$ maps.
	Such extensions play  crucial roles in the construction of the  $C^*$-algebra
	$\mathcal{A}
	\big(
	C_{b}(Y,\mathscr{X})
	\big)$  in Section 4  and of the  Bott elements in Sections 6 and 8.

	Recall that the 
	\emph{modulus of continuity} for 
	a map
	$f: (X,d_X)\to (Y,d_Y)$ between two metric spaces
	is defined by
	$$
	\omega_{\scalebox{0.6}{$f$}}(\delta)=
	\sup
	\Big\{
	d_Y
	\big(
	f(x_{1}),f(x_{2})
	\big)
	~\Big|~ 
	d_X(x_{1},x_{2})\leq\delta\Big\}
	$$
	for 
	$\delta\geq0$.
	Obviously
	$\omega_{\scalebox{0.6}{$f$}}$
	is an increasing function from 
	$[0,+\infty)$ 
	to 
	$[0,+\infty]$.
	A map 
	$f$
	is uniformly continuous if and only if
	$\lim\limits_{\delta\rightarrow0^+}\omega_{\scalebox{0.6}{$f$}}(\delta)=0$.

	\begin{defn}
		\label{ProperextensionsofPropertyHspaces}
		Given two Banach spaces $\mathscr{X}$, $\mathscr{Y}$, and a uniformly continuous map 
		$\phi: S(\mathscr{X}) \to S(\mathscr{Y})$,
		a map  
		$\varPhi: \mathscr{X} \to \mathscr{Y}$ 
		is called  a 
		\emph{proper extension} of $\phi$ if 
		\begin{itemize}
			\item [(1)]  $\varPhi\big|_{S(\mathscr{X})} = \phi$;
			\item [(2)] 
			for each $\delta > 0$,
			$\omega_{
				\scalebox{0.6}{$\varPhi$}
			}(\delta)<+\infty$;
			\item [(3)] 
			$\varPhi$  
			is uniformly continuous on every bounded subset of $\mathscr{X}$; and
			\item [(4)] 
			$\lim\limits_{\|v\|\to\infty}
			\big\|\varPhi(v)\big\|=\infty$.
		\end{itemize}
	\end{defn}

	Proper extensions of Property $(H)$ maps play  crucial roles in the subsequent sections. 
	In  \cite{KY}, 
	there is an oversight concerning such extensions:
	scalar extensions are applied to arbitrary Property $(H)$ maps, although this is not valid in general. 
	Indeed, if we consider the space 
	$\ell^1(\mathbb{N})$ and the scalar extension $\mathcal{M}_{1,2}^{I}$ 
	of the Property $(H)$ map (the Mazur map) 
	$\mathcal{M}_{1,2}: S\big(\ell^1(\mathbb{N})\big)
	\to S\big(\ell^2(\mathbb{N})\big)$,
	then 
	$\omega_{\scalebox{0.6}{$\mathcal{M}_{1,2}^{I}$}}(\delta)=+\infty$ for every $\delta>0$ 
	(see \cite[Chapter 9]{BL2000}),
	and hence the second condition
	in Definition \ref{ProperextensionsofPropertyHspaces}
	is not satisfied.
	In general, a proper extension should be constructed using an $\eta$-scalar extension
	rather than the scalar extension.
	The latter is valid only in the Lipschitz setting.
	By Lemma \ref{3lip}, the scalar extension of a Lipschitz Property $(H)$ map is again Lipschitz and therefore provides a proper extension.
	
	We now turn to general Property $(H)$ maps and show that every such map admits a proper extension.

	\begin{lem}\label{generaluniformlycontinuousonfiniteBall}
		Let  $\mathscr{X}$,  $\mathscr{Y}$  be  Banach spaces,  and let
		$\phi: S(\mathscr{X})\to S(\mathscr{Y})$ be a uniformly continuous map.
		Then for any continuous function 
		$\eta: [0,+\infty) \to [0,+\infty)$ satisfying $\eta(0)=0$ and $\eta(1)=1$,
		the $\eta$-scalar extension $\phi^{\eta}$ of $\phi$   is uniformly continuous on every bounded subset of $\mathscr{X}$.
	\end{lem}
	\begin{proof}
		It suffices to show that 
		$\phi^{\eta}$
		is uniformly continuous on every bounded ball of $\mathscr{X}$.
		Fix $R>0$, and  
		denote by $B_{\mathscr{X}}(0,R)$ the closed ball in $\mathscr{X}$ centered at $0$ with radius $R$.
		Let $\eta_{_R}$ be the restriction of $\eta$ to $[0,R]$, and set
		$$
		M_R=\max\limits_{
			0 \leq  t \leq R}
		\eta(t).
		$$

		For any  fixed  
		$0<\delta\leq  R$ 
		and
		$v,w \in  B_{\mathscr{X}}(0,R)$
		with $\delta=\|v-w\|$,
		we distinguish two cases.

		Case 1:    
		Suppose that $\max\big\{
		\|v\|,\|w\|
		\big\} \geq \sqrt{\delta}$. 
		Without loss of generality, 
		we may assume that $\|v\| \geq  \sqrt{\delta}$.
		If  $w\neq0$,  
		then 
		\begin{eqnarray*}
			\big\| 
			\phi^{\eta}(v)
			-\phi^{\eta}(w)
			\big\| 
			& = &
			\bigg\| 
			\eta(\|v\|) \cdot 
			\phi	
			\Big( 
			\dfrac{v}{\|v\|}
			\Big)
			-
			\eta(\|w\|) \cdot 
			\phi	
			\Big( 
			\dfrac{w}{\|w\|}
			\Big)
			\bigg\| 
			\\
			&\leq&
			\bigg\| 
			\eta(\|v\|) \cdot 
			\phi	
			\Big( 
			\dfrac{v}{\|v\|}
			\Big)
			-
			\eta(\|v\|) \cdot 
			\phi	
			\Big( 
			\dfrac{w}{\|w\|}
			\Big)
			\bigg\| 
			\\
			&& +\bigg\| 
			\eta(\|v\|) \cdot 
			\phi	
			\Big( 
			\dfrac{w}{\|w\|}
			\Big)
			-
			\eta(\|w\|) \cdot 
			\phi	
			\Big( 
			\dfrac{w}{\|w\|}
			\Big)
			\bigg\| 
			\\
			&=&
			\eta(\|v\|) 
			\cdot 
			\bigg\| 
			\phi	
			\Big( 
			\dfrac{v}{\|v\|}
			\Big)
			-
			\phi	
			\Big( 
			\dfrac{w}{\|w\|}
			\Big)
			\bigg\| 
			+
			\Big| 
			\eta(\|v\|) 
			-
			\eta(\|w\|) 
			\Big| 
			\\
			&\leq& 
			\eta(\|v\|) 
			\cdot 
			\omega_{\scalebox{0.6}{$\phi$}}
			\bigg(
			\Big\|	
			\dfrac{v}{\|v\|}-	\dfrac{w}{\|w\|}
			\Big\|
			\bigg)
			+
			\omega_{\scalebox{0.6}{$\eta_{_R}$}} 
			\bigg(
			\Big| 
			\|v\|-\|w\|
			\Big|
			\bigg) 
			\\
			&\leq& 
			\eta(\|v\|) 
			\cdot 
			\omega_{\scalebox{0.6}{$\phi$}}
			\bigg(
			\dfrac{2\|v-w\|}{\|v\|}
			\bigg)
			+
			\omega_{\scalebox{0.6}{$\eta_{_R}$}}
			\big( 
			\|v-w\|
			\big)
			\\
			&\leq& 
			M_R \cdot 
			\omega_{\scalebox{0.6}{$\phi$}}
			\big(
			2\sqrt{\delta}
			\big)
			+
			\omega_{\scalebox{0.6}{$\eta_{_R}$}} 
			( 
			\delta
			).
		\end{eqnarray*}
		If $w=0$, 
		then 	
		\begin{eqnarray*}
			\big\| 
			\phi^{\eta}(v)
			-\phi^{\eta}(w)
			\big\| = \big\| 
			\phi^{\eta}(v)
			\big\| = \eta(\delta).
		\end{eqnarray*}
		
		Case 2:  
		Suppose that  
		$\max\big\{
		\|v\|,\|w\|
		\big\} < \sqrt{\delta}$.
		In this case, 
		\begin{eqnarray*}
			\big\| 
			\phi^{\eta}(v)
			-\phi^{\eta}(w)
			\big\| 
			&\leq & 
			\big\| 
			\phi^{\eta}(v)
			\big\|
			+
			\big\|
			\phi^{\eta}(w)
			\big\|
			\\
			& = &
			\eta(\|v\|)+
			\eta(\|w\|)
			\\
			&\leq& 
			2 \cdot \sup_{
				t \in [ 0,\sqrt{\delta} ] }
			\eta(t). 
		\end{eqnarray*}
		
		Let 
		$\phi^{\eta}_{R}$
		denote the restriction of 
		$\phi^{\eta}$ to  $B_{\mathscr{X}}(0,R)$.
		Combining the above estimates,
		we obtain
		$$
		\omega_{\scalebox{0.6}{$\phi^{\eta}_{R}$}}(\delta)
		\leq 
		M_R 
		\cdot 
		\omega_{\scalebox{0.6}{$\phi$}}
		\big(
		2\sqrt{\delta}
		\big)
		+
		\omega_{\scalebox{0.6}{$\eta_{_R}$}} 
		( 
		\delta
		)+ \eta(\delta)+
		2 \cdot \sup_{
			t \in [ 0,\sqrt{\delta} ] }
		\eta(t)
		$$
		for every 
		$0<\delta\leq R$. 
		Since 
		$\eta(0)=0$ and both $\phi$ and $\eta_{_R}$
		are
		uniformly continuous,
		it follows that 
		$$\lim\limits_{
			\delta \to 0^+}
		\omega_{\scalebox{0.6}{$\phi^{\eta}_{R}$}}(\delta)
		=0.
		$$
		This completes the proof.
	\end{proof}

	\begin{lem}\label{asmalltrickfunction}
		For any increasing function $f: (0,+\infty) \to (0,+\infty)$
		satisfying 
		$\lim\limits_{
			t\to+\infty} 
		f(t)
		=+\infty$,
		there exists a strictly increasing Lipschitz function $\eta: [0,+\infty) \to [0,+\infty)$
		such that 
		\begin{itemize}
			\item [(1)]	 
			there exists  
			$c > 0$ such that 
			$\eta(t) \leq  f(t)$ for all $t \geq  c$;
			\item [(2)] $\eta(t)=t$ for all $t\in[0,1]$; and
			\item [(3)] $\lim\limits_{
				t\to+\infty}
			\eta(t)=+\infty$.
		\end{itemize}
	\end{lem}
	\begin{proof}
		Since  
		$f$
		is an increasing function
		satisfying 
		$\lim\limits_{
			t\to+\infty} 
		f(t)
		=+\infty$,
		there exists a strictly increasing sequence 
		$\{n_{k}\}_{k\in\mathbb{N}}$
		such that 
		\begin{itemize}
			\item [(1)]
			$n_{1} \geq 2$, 
			$n_{k+1}-n_{k} \geq  n_{k}-n_{k-1}$
			for all 
			$k \geq 1$, where we set $n_{0}=1$; 
			\item [(2)] $\lim\limits_{
				k\to+\infty}
			n_{k}=+\infty$; and
			\item [(3)] $f(n_{k}) \geq  k+2$
			for all  $k \geq 1$.
		\end{itemize}
		Define $\eta$ by 
		\[
		\eta(t) =
		\begin{cases}
			\hspace{0.1cm} 
			t, &   t \in [0,1), 
			\\
			\dfrac{t-1}{n_{1}-1} + 1, &  t \in [1, n_{1}), 
			\\
			\hspace{1.2cm} 
			\vdots &\hspace{0.6cm} \vdots 
			\\
			\dfrac{t-n_{k}}{n_{k+1}-n_{k}} + k+1, &  t \in [n_{k}, n_{k+1}), 
			\\
			\hspace{1.2cm} 
			\vdots &\hspace{0.6cm} \vdots \hspace{1.5cm}.
		\end{cases}
		\]
		Then we have 
		\begin{center}
			$\eta(t) \leq  f(t)$ for all 
			$t \geq  n_{1}$,
			$L_{\eta}\leq 1$
			and 
			$\lim\limits_{
				t\to+\infty}
			\eta(t)=+\infty$.
		\end{center}
		This completes the proof.
	\end{proof}

	\begin{prop}\label{whatisthenergithikalikdkfslak}
		Let $\mathscr{X}$, $\mathscr{Y}$  be Banach spaces, 
		and let $\phi: S(\mathscr{X})\to S(\mathscr{Y})$ 
		be a uniformly continuous map
		such that  
		$\omega_{\phi}(t) \neq 0$ for all $t>0$.
		Then there exists a strictly increasing continuous function 
		$\eta: [0, +\infty) \to [0, +\infty)$
		satisfying
		$\eta(t)=t$ for  
		$t \in [0,1]$ 
		and 
		$\lim\limits_{t\to\infty} \eta(t)=+\infty$,
		such that 
		\begin{itemize}
			\item [(1)] 
			for every  $\delta\geq0$,
			$\omega_{
				\scalebox{0.6}{$\phi^{\eta}$}
			}(\delta)<+\infty$; and
			\item [(2)] 
			$\phi^{\eta}$  
			is uniformly continuous on every bounded subset of $\mathscr{X}$.
		\end{itemize}
	\end{prop}
	\begin{proof}
		It follows from the uniform continuity of $\phi$ that  
		$$
		\lim\limits_{t \rightarrow  0^+}
		\omega_{\scalebox{0.6}{$\phi$}}(t)=0.
		$$
		Hence, the function 
		$f: (0,+\infty) \to (0,+\infty)$ defined by
		$$
		f(t)= \frac{1}{\omega_{\scalebox{0.6}{$\phi$}}(t^{-1})}
		$$ 
		is  increasing   
		and  satisfies
		$$
		\lim\limits_{t \to +\infty} f(t)= +\infty.
		$$ 
		By Lemma \ref{asmalltrickfunction},
		there exists a Lipschitz, strictly increasing function $\eta: [0,+\infty) \to [0,+\infty)$ 
		such that
		\begin{itemize}
			\item [(1)]	 
			there exists  
			$c > 0$ such that 
			$\eta(t) \leq  f(t)$  for all  $t \geq  c$;
			\item [(2)] $\eta(t)=t$ for  $t\in[0,1]$; and
			\item [(3)] $\lim\limits_{
				t\to+\infty}
			\eta(t)=+\infty$.
		\end{itemize}

		Next, we distinguish two cases to show that  
		$\omega_{
			\scalebox{0.6}{$\phi^{\eta}$}
		}(\delta)<+\infty$
		for every $\delta\geq0$.

		Case 1:	
		Suppose that 
		$\max\big\{
		\|v\|,\|w\|
		\big\} >  c$.
		Without loss of generality, we may assume that 
		$\|v\| >  c$.
		If $w\neq0$,
		then 
		\begin{eqnarray*}
			\big\| 
			\phi^{\eta}(v)
			-\phi^{\eta}(w)
			\big\| 
			&\leq&
			\bigg\| 
			\eta(\|v\|) \cdot 
			\phi	
			\Big( 
			\dfrac{v}{\|v\|}
			\Big)
			-
			\eta(\|v\|) \cdot 
			\phi	
			\Big( 
			\dfrac{w}{\|w\|}
			\Big)
			\bigg\| 
			\\
			&& +
			\bigg\| 
			\eta(\|v\|) \cdot 
			\phi	
			\Big( 
			\dfrac{w}{\|w\|}
			\Big)
			-
			\eta(\|w\|) \cdot 
			\phi	
			\Big( 
			\dfrac{w}{\|w\|}
			\Big)
			\bigg\| 
			\\
			&=&
			\eta(\|v\|) 
			\cdot 
			\bigg\| 
			\phi	
			\Big( 
			\dfrac{v}{\|v\|}
			\Big)
			-
			\phi	
			\Big( 
			\dfrac{w}{\|w\|}
			\Big)
			\bigg\| 
			+
			\Big| 
			\eta(\|v\|) 
			-
			\eta(\|w\|) 
			\Big| 
			\\
			&\leq& 
			\eta(\|v\|) 
			\cdot 
			\omega_{\scalebox{0.6}{$\phi$}}
			\bigg(
			\Big\|	
			\dfrac{v}{\|v\|}-	\dfrac{w}{\|w\|}
			\Big\|
			\bigg)
			+
			\omega_{\scalebox{0.6}{$\eta$}} 
			\bigg(
			\Big| 
			\|v\|-\|w\|
			\Big| 
			\bigg)  
			\\
			&\leq& 
			\eta(\|v\|) 
			\cdot 
			\omega_{\scalebox{0.6}{$\phi$}}
			\bigg(
			\dfrac{2\|v-w\|}{\|v\|}
			\bigg)
			+
			\omega_{\scalebox{0.6}{$\eta$}} 
			\big( 
			\|v-w\| 
			\big)  
			\\
			&\leq& 
			\eta(\|v\|) 
			\cdot 
			\omega_{\scalebox{0.6}{$\phi$}}
			\bigg(
			\dfrac{\big\lceil 
				2\|v-w\| \big\rceil}{\|v\|}
			\bigg)
			+
			\omega_{\scalebox{0.6}{$\eta$}} 
			\big( 
			\|v-w\|
			\big) 
			\\
			&\leq& 
			\eta(\|v\|) 
			\cdot 
			\big\lceil  2\|v-w\| \big\rceil
			\cdot 
			\omega_{\scalebox{0.6}{$\phi$}}
			\bigg(
			\dfrac{1}{\|v\|}
			\bigg)
			+
			\omega_{\scalebox{0.6}{$\eta$}} 
			\big( 
			\|v-w\| 
			\big)  
			\\
			&\leq& 
			f(\|v\|) 
			\cdot 
			\big\lceil  2\|v-w\| \big\rceil
			\cdot 
			\omega_{\scalebox{0.6}{$\phi$}}
			\bigg(
			\dfrac{1}{\|v\|}
			\bigg)
			+
			\omega_{\scalebox{0.6}{$\eta$}} 
			\big( 
			\|v-w\| 
			\big) 
			\\
			&=& 
			\big\lceil  2\|v-w\| \big\rceil
			+
			\omega_{\scalebox{0.6}{$\eta$}} 
			\big( 
			\|v-w\| 
			\big) , 
		\end{eqnarray*}
		where 
		$\lceil \cdot \rceil$
		denotes the ceiling function.
		If $w=0$, 
		then
		\begin{eqnarray*}
			\big\| 
			\phi^{\eta}(v)
			-\phi^{\eta}(w)
			\big\| = 
			\big\| 
			\phi^{\eta}(v)
			\big\| = 
			\big|
			\eta(\|v\|)-\eta(0)
			\big| \leq \omega_{\scalebox{0.6}{$\eta$}}\big(\|v\|\big) = 
			\omega_{\scalebox{0.6}{$\eta$}} \big(\|v-w\|\big). 
		\end{eqnarray*}	
		
		Case 2: 
		Suppose that 
		$\max\big\{
		\|v\|,\|w\|
		\big\} \leq  c$.
		In this case,
		\begin{eqnarray*}
			\big\| 
			\phi^{\eta}(v)
			-\phi^{\eta}(w)
			\big\|
			\leq 
			\big\| 
			\phi^{\eta}(v)
			\big\|
			+\big\| 
			\phi^{\eta}(w)
			\big\|
			= 
			\eta(\|v\|) +	\eta(\|w\|) 
			\leq 
			2\cdot  M_{c},
		\end{eqnarray*}
		where 
		$M_{c}=\sup\limits_{
			0 \leq t \leq c}
		\eta(t)$.

		Combining the above estimates,
		we obtain 
		\begin{eqnarray*}
			\omega_{
				\scalebox{0.6}{$\phi^{\eta}$}
			}(\delta)
			\leq 
			\big\lceil  2\delta \big\rceil
			+
			\omega_{\scalebox{0.6}{$\eta$}} 
			(\delta) 
			+ 
			2\cdot  M_{c}
			< +\infty  
		\end{eqnarray*}
		for every $\delta\geq 0$.
		
		The second assertion follows directly from Lemma \ref{generaluniformlycontinuousonfiniteBall}.
	\end{proof}

	The following result 
	asserts that every Property $(H)$ map admits a proper extension.

	\begin{cor}\label{finitemodulusofcontinuity}
		Let 
		$\phi: S(\mathscr{X}) \to S(\mathscr{H})$ be a Property $(H)$ map. 
		Then there exists a strictly increasing continuous function 
		$\eta: [0, +\infty) \to [0, +\infty)$
		satisfying
		$\eta(t)=t$ for  
		$t \in [0,1]$ 
		and 
		$\lim\limits_{t\to\infty} \eta(t)=+\infty$,
		such that 
		\begin{itemize}
			\item [(1)] 
			for every  $\delta\geq0$,
			$\omega_{
				\scalebox{0.6}{$\phi^{\eta}$}
			}(\delta)<+\infty$; and
			\item [(2)] 
			$\phi^{\eta}$  
			is uniformly continuous on every bounded subset of $\mathscr{X}$.
		\end{itemize}
	\end{cor}
	\begin{proof}
		We observe that for every Property $(H)$ map $\phi$, $\omega_{\phi}(t) \neq 0$  for all $t>0$.
		The conclusion follows immediately from Proposition \ref{whatisthenergithikalikdkfslak}. 
	\end{proof}

	We call $\phi^{\eta}$ in the above corollary an 
	\emph{$\eta$-scalar proper extension}.
	
	\begin{rmk}\label{otheidrectsumofetascalarextensionsslaoproper}
		The direct sum of finitely many $\eta$-scalar proper extensions is again 
		a proper extension.       
		Let $\mathscr{X}_p$ ($1\leq p \leq k$)
		be  Property $(H)$ Banach spaces
		with  Property $(H)$ maps
		$\phi_p: S(\mathscr{X}_p) \to S(\mathscr{H})$.
		By Corollary \ref{finitemodulusofcontinuity}, 
		each $\phi_p$ admits an
		$\eta$-scalar proper extension $\phi_{p}^{\eta_p}$. 
		Since 
		$\eta_p(t)=t$ for all  
		$t \in [0,1]$ and $1\leq p \leq k$,
		the restriction of
		the direct sum map
		\[
		\begin{array}{cccc}
			\oplus_{p=1}^{k}\phi_{p}^{\eta_p}: & \bigoplus_{p=1}^{k}
			\mathscr{X}_p & \longrightarrow & \bigoplus_{p=1}^{k}
			\mathscr{H}
		\end{array} 
		\]
		to the unit spheres provides a Property $(H)$ map for 
		$\bigoplus_{p=1}^{k}
		\mathscr{X}_p$. 
		Moreover, it is straightforward to verify that this direct sum map is a proper extension
		of the resulting Property $(H)$ map on the unit spheres.
	\end{rmk}

	\section{Affine isometric actions associated to bornologous maps}

	In \cite{Y2000}, the second author proved that if a countable discrete group $\Gamma$ admits a coarse embedding into a Hilbert space, then it satisfies the coarse Baum--Connes conjecture and, consequently, the strong Novikov conjecture, provided that its classifying space $B\Gamma$ has the homotopy type of a finite CW complex. Subsequently, Skandalis, Tu, and Yu \cite{STY} removed the finiteness assumption on $B\Gamma$ by developing Higson's elegant descent technique.
	
	At the heart of the descent method is the observation that a coarse embedding of $\Gamma$ into a Hilbert space gives rise to a compact contractible space $X$ such that the transformation groupoid $X\rtimes\Gamma$ acts properly by affine isometries on a continuous field of Hilbert spaces over $X$. This construction relies crucially on the theory of negative type functions. For general Banach spaces, however, negative type functions are no longer available, and the Hilbert space argument breaks down.

	In this section, we resolve the above problem for arbitrary Banach spaces by giving an explicit construction. More precisely, 
	let $Z$ be the $\Gamma$-invariant contractible metric subspace of
	$\ell^1(\Gamma)$ consisting of  all non-negative functions with $\ell^1$-norm $1$.  
	We show that every bornologous map 
	$h: \Gamma \to \mathscr{X}$, 
	where $\mathscr{X}$ is a Banach space, gives rise to an affine isometric action of the transformation groupoid $Z\rtimes\Gamma$ on the trivial bundle 
	$Z\times\ell^2(\Gamma,\mathscr{X})$. 
	Furthermore,
	if 
	$h$ is a coarse embedding, then this action is proper. 
	In contrast to previous approaches based on continuous fields of Banach spaces, our construction is explicit and elementary.

	\begin{defn}\label{uniformlycontinuousalgebra}
		Let 
		$Y$ 
		be a metric space and 
		let $\mathscr{X}$
		be a Banach space.
		We define
		$C_{b}(Y,\mathscr{X})$
		to be the space of all
		uniformly continuous and bounded functions from  
		$Y$
		to
		$\mathscr{X}$.	  
	\end{defn}
	It is a Banach space under the norm
	$$
	\|f\| =  \sup\limits_{y\in{Y}} \big\|f(y)\big\|.
	$$
	Moreover, if 
	$\mathscr{X}$
	is a  $C^*$-algebra,
	then  
	$C_{b}(Y,\mathscr{X})$ is also a  $C^*$-algebra.

	Let 
	$\Gamma$ 
	be a countable discrete group with  
	a proper 
	\emph{right}-invariant metric
	$d$.
	Set  
	\begin{equation}	\label{eq:simplicityforthebasespaceZ}
		Z=\Big\{f\in\ell^1(\Gamma) ~ \Big|~ f\geq 0 \text{ and } \|f\|=1 \Big\}.
	\end{equation}
	Then
	$Z$ is a
	convex, bounded metric subspace of 
	$\ell^1(\Gamma)$.  
	Throughout the paper,  we write each 
	$z \in  Z$ in the form  
	$$
	z=\sum\limits_{\gamma\in\Gamma}
	z_{_\gamma} \cdot \gamma,
	$$
	where 
	$z_{\gamma}\geq 0$
	for all 
	$\gamma\in\Gamma$
	and
	$\sum\limits_{\gamma\in\Gamma}
	z_{_\gamma}=1$.
	The group 
	$\Gamma$
	admits
	an isometric action on 
	$Z$
	via
	\begin{equation*}
		\gamma\cdot
		\Big(
		\sum\limits_{\tilde{\gamma}\in\Gamma}
		z_{\tilde{\gamma}} \cdot 
		\tilde{\gamma}
		\Big) 
		=
		\sum\limits_{\tilde{\gamma}\in\Gamma}
		z_{\tilde{\gamma}} \cdot (\gamma\tilde{\gamma})_.
	\end{equation*}

	\begin{rmk}
		The metric space $Z$ is not compact in general, 
		but its convexity and boundedness are sufficient for our purposes.
		In the rest of this paper,
		we fix the notation $Z$ for the metric space defined in \eqref{eq:simplicityforthebasespaceZ}  associated to  $\Gamma$. 
	\end{rmk}

	\begin{defn}
		[{\cite[Definition~1.8]{Roe2003}}]
		\label{Bornologousdef}    
		A map $h$ from $\Gamma$ to a Banach space $\mathscr{X}$
		is called  \emph{bornologous} if there exists a function $\rho^{+}: [0,+\infty) \to [0,+\infty)$ 
		such that
		\begin{equation*}
			\Big\|
			h(\gamma_{1})-
			h(\gamma_{2})
			\Big\|
			\leq\rho^{+}
			\big(d(\gamma_{1},
			\gamma_{2})\big)
		\end{equation*}
		for all 
		$\gamma_{1},\gamma_{2}\in\Gamma$.
		Here $\rho^{+}$  is  not necessarily non-decreasing.
	\end{defn}

	For simplicity, we set $$\mathscr{X}_{\Gamma}=\ell^2(\Gamma, \mathscr{X}).$$ 
	Next,
	we construct a family of affine isometries on $\mathscr{X}_{\Gamma}$ associated to a bornologous map.
	This family will induce an affine isometric action of $\Gamma$ on 
	the function space
	$C_{b}(Z,\mathscr{X}_{\Gamma})$.

	Let 
	$h: \Gamma \to \mathscr{X}$ 
	be a  bornologous map. 
	For each $\gamma\in\Gamma$ and each
	$z=\sum\limits_{\scalebox{0.6}{$\tilde{\gamma}\in\Gamma$}} 
	z_{\tilde{\gamma}}\cdot
	\tilde{\gamma} \in {Z}$,
	define a function 
	$$
	\zeta^{\scalebox{0.6}{$(\gamma,z)$}}: \Gamma \longrightarrow  \mathscr{X}
	$$
	by
	\begin{equation}\label{eq:cocycles}
		\zeta^{\scalebox{0.6}{$(\gamma,z)$}}
		(\tilde{\gamma})=	
		\sqrt{z_{\scalebox{0.6}{$\gamma^{-1}\tilde{\gamma}$}}}
		\cdot \Big(h(\gamma^{-1}\tilde{\gamma})-
		h(\tilde{\gamma})
		\Big)
	\end{equation}
	for all
	$\tilde{\gamma}\in\Gamma$. 
	It follows from 
	\eqref{eq:cocycles}  
	and 
	the right-invariance of the metric
	$d$ on $\Gamma$
	that 	
	\begin{eqnarray*}
		\sum\limits_{
			\tilde{\gamma}\in\Gamma}
		\Big\|
		\zeta^{\scalebox{0.6}{$(\gamma,z)$}}
		(\tilde{\gamma})
		\Big\|^2 
		&=&	
		\sum\limits_{
			\tilde{\gamma}\in\Gamma}
		\bigg\|
		\sqrt{z_{\scalebox{0.6}{$\gamma^{-1}\tilde{\gamma}$}}}
		\cdot \Big(h(\gamma^{-1}\tilde{\gamma})-
		h(\tilde{\gamma})
		\Big)
		\bigg\|^2
		\\
		&=&   
		\sum\limits_{
			\tilde{\gamma}\in\Gamma}
		z_{\scalebox{0.6}{$\gamma^{-1}\tilde{\gamma}$}}
		\cdot
		\Big\|
		h(\gamma^{-1}\tilde{\gamma})-
		h(\tilde{\gamma})
		\Big\|^2
		\\
		&\leq&   
		\sum\limits_{
			\tilde{\gamma}\in\Gamma}
		z_{\scalebox{0.6}{$\gamma^{-1}\tilde{\gamma}$}}
		\cdot
		\rho^{+}
		\big(
		d(
		\gamma^{-1}\tilde{\gamma},
		\tilde{\gamma})
		\big)^2 
		\\
		&=&
		\rho^{+}
		\big(
		d(\gamma^{-1},e)
		\big)^2. 
	\end{eqnarray*}
	Hence,
	$\zeta^{\scalebox{0.6}{$(\gamma,z)$}} \in \mathscr{X}_{\Gamma}$
	and 
	\begin{equation}	\label{eq:upperboundforthecocycles}
		\Big\|
		\zeta^{\scalebox{0.6}{$(\gamma,z)$}}
		\Big\|
		\leq 
		\rho^{+}
		\big(
		d(\gamma^{-1},e)
		\big)
	\end{equation}
	for all 
	$\gamma\in\Gamma$
	and $z\in Z$.

	Moreover, for every $\gamma\in\Gamma$,
	define  	
	$$
	\lambda^{\scalebox{0.6}{$\gamma$}}: 	
	\mathscr{X}_{\Gamma}
	\longrightarrow 
	\mathscr{X}_{\Gamma}
	$$
	by 
	\begin{equation}	\label{eq:thedefinitionforthetranslation}
		\big( \lambda^{\scalebox{0.6}{$\gamma$}}(f)
		\big)(\tilde{\gamma})=
		f(\gamma^{-1}\tilde{\gamma})
	\end{equation}
	for 
	$f \in \mathscr{X}_{\Gamma}$ and 
	$\tilde{\gamma}\in\Gamma$.
	It is clear that each
	$\lambda^{\scalebox{0.6}{$\gamma$}}$
	is an isometric isomorphism.

	For
	each $\gamma\in\Gamma$ 
	and each
	$z\in {Z}$,
	define an affine isometry
	$$
	U^{\scalebox{0.6}{$(\gamma,z)$}}: 
	\mathscr{X}_{\Gamma}
	\longrightarrow 
	\mathscr{X}_{\Gamma}
	$$
	by 
	\begin{equation}	\label{eq:affineisometricesgeneralone}
		U^{\scalebox{0.6}{$(\gamma,z)$}}(f)=	\lambda^{\scalebox{0.6}{$\gamma$}}(f)+\zeta^{\scalebox{0.6}{$(\gamma,z)$}} 
	\end{equation}	
	for $f \in \mathscr{X}_{\Gamma}$.

	\begin{lem}\label{affineisometryproperty1}
		The following statements hold: 
		\begin{itemize}
			\item [(1)]	$U^{\scalebox{0.6}{$(e,z)$}}$
			is the 
			identity operator 
			for all 
			$z\in{Z}$, where $e$ is the identity element of $\Gamma$;
			\item [(2)]	$U^{\scalebox{0.6}{$(\gamma, z^\prime)$}} U^{\scalebox{0.6}{$(\gamma^\prime,z)$}}=U^{\scalebox{0.6}{$(\gamma\gamma^\prime,z)$}}$
			for all 
			$\gamma, \gamma^\prime\in\Gamma$ and $z,z^\prime\in {Z}$ 
			such that
			$z^\prime=\gamma^\prime z$. 
		\end{itemize}
	\end{lem}	
	\begin{proof}
		The first statement follows directly from the definition. 
		For the second one, 
		since $\lambda^{\scalebox{0.6}{$\gamma$}}\lambda^{\scalebox{0.6}{$\gamma^\prime$}}=\lambda^{\scalebox{0.6}{$\gamma\gamma^\prime$}}$
		for all 
		$\gamma, \gamma^\prime \in \Gamma$,
		it suffices to verify that
		\begin{equation*}
			\lambda^{\scalebox{0.6}{$\gamma$}}
			\big(
			\zeta^{\scalebox{0.6}{$(\gamma^\prime,z)$}}
			\big)
			+ \zeta^{\scalebox{0.6}{$(\gamma,z^\prime)$}}
			= \zeta^{\scalebox{0.6}{$(\gamma\gamma^\prime,z)$}},
		\end{equation*}
		where
		$z^\prime=\gamma^\prime z$. 
		
		Write 
		$z=\sum\limits_{\scalebox{0.6}{$\tilde{\gamma}\in\Gamma$}} 
		z_{\tilde{\gamma}}\cdot
		\tilde{\gamma}$.
		Then 
		$z^\prime=\gamma^\prime z=\sum\limits_{\scalebox{0.6}{$\tilde{\gamma}\in\Gamma$}} 
		z_{\tilde{\gamma}}\cdot
		(\gamma^\prime\tilde{\gamma})$   
		and for every $\tilde{\gamma}\in\Gamma$,
		we have
		\begin{eqnarray*}
			\Big(
			\lambda^{\scalebox{0.6}{$\gamma$}}
			\big(
			\zeta^{\scalebox{0.6}{$(\gamma^\prime,z)$}}
			\big)
			+ \zeta^{\scalebox{0.6}{$(\gamma,z^\prime)$}}
			\Big)(\tilde{\gamma})
			& = & 
			\zeta^{\scalebox{0.6}{$(\gamma^\prime,z)$}}
			(\gamma^{-1}\tilde{\gamma})
			+ \zeta^{\scalebox{0.6}{$(\gamma,z^\prime)$}}
			(\tilde{\gamma})
			\\
			&=&
			\sqrt{z_{\scalebox{0.6}{$\gamma^{\prime-1}\gamma^{-1}\tilde{\gamma}$}}}
			\cdot \Big(h(\gamma^{\prime-1}\gamma^{-1}\tilde{\gamma})-
			h(\gamma^{-1}\tilde{\gamma})
			\Big) +
			\\
			&& 
			\sqrt{z_{\scalebox{0.6}{$\gamma^{\prime-1}\gamma^{-1}\tilde{\gamma}$}}}
			\cdot 
			\Big(
			h(\gamma^{-1}\tilde{\gamma}) -
			h(\tilde{\gamma})
			\Big)
			\\
			&=& 
			\sqrt{z_{\scalebox{0.6}{$\gamma^{\prime-1}\gamma^{-1}\tilde{\gamma}$}}}
			\cdot \Big(h(\gamma^{\prime-1}\gamma^{-1}\tilde{\gamma})-
			h(\tilde{\gamma})
			\Big) 
			\\
			&=&
			\zeta^{\scalebox{0.6}{$(\gamma\gamma^{\prime},z)$}}(\tilde{\gamma}).
		\end{eqnarray*}
		This verifies the second statement. 
	\end{proof}

	We now construct a $\Gamma$-action 
	(in fact, an affine isometric action) on 
	$C_{b}(Z, \mathscr{X}_{\Gamma})$
	induced by the family of affine isometries 
	\[\Big\{
	U^{\scalebox{0.6}{$(\gamma,z)$}} ~\Big|~ 
	\gamma\in\Gamma, z \in  Z 
	\Big\}.\]
	For each
	$\gamma\in\Gamma$
	and each
	$f\in C_{b}(Z, \mathscr{X}_{\Gamma})$,
	define a function 
	$$
	\gamma\cdot f:
	Z \longrightarrow \mathscr{X}_{\Gamma}
	$$
	by
	\begin{equation}	\label{eq:theactionsofgammaonsections}
		\big(
		\gamma\cdot f
		\big)
		(z)
		=	
		U^{\scalebox{0.6}{$(\gamma,\gamma^{-1}z)$}}
		\big(
		f(\gamma^{-1}z)
		\big)
	\end{equation}
	for $z\in{Z}$.

	\begin{lem}\label{The-very-most-essentional-idea}
		$\gamma\cdot f
		\in
		C_{b}(Z, \mathscr{X}_{\Gamma})$ 
		for all
		$\gamma\in\Gamma$
		and 
		$f\in
		C_{b}(Z, \mathscr{X}_{\Gamma})$. 
	\end{lem}
	\begin{proof}
		
		Let 
		$\gamma\in\Gamma$ 
		and 
		$f\in
		C_{b}(Z, \mathscr{X}_{\Gamma})$
		be fixed.
		Then for any 
		$y,z\in Z$,  	
		\begin{eqnarray*}
			\Big\| 
			(\gamma\cdot f)
			(y)-
			(\gamma\cdot f)
			(z)
			\Big\| &=& 
			\Big\| U^{\scalebox{0.6}{$(\gamma,\gamma^{-1}y)$}}
			\big(
			f(\gamma^{-1}y)
			\big)
			-
			U^{\scalebox{0.6}{$(\gamma,\gamma^{-1}z)$}}
			\big(
			f(\gamma^{-1}z)
			\big)
			\Big\|
			\\
			&=& 
			\Big\| 
			\lambda^{\scalebox{0.6}{$\gamma$}}
			\big(
			f(\gamma^{-1}y)
			\big)
			+\zeta^{\scalebox{0.6}{$(\gamma,\gamma^{-1}y)$}}
			-
			\lambda^{\scalebox{0.6}{$\gamma$}}
			\big(
			f(\gamma^{-1}z)
			\big)
			-\zeta^{\scalebox{0.6}{$(\gamma,\gamma^{-1}z)$}}
			\Big\|
			\\
			&\leq& 
			\Big\| 
			\lambda^{\scalebox{0.6}{$\gamma$}}
			\big(
			f(\gamma^{-1}y)
			\big)
			-
			\lambda^{\scalebox{0.6}{$\gamma$}}
			\big(
			f(\gamma^{-1}z)
			\big)
			\Big\|
			+
			\Big\| \zeta^{\scalebox{0.6}{$(\gamma,\gamma^{-1}y)$}}
			-\zeta^{\scalebox{0.6}{$(\gamma,\gamma^{-1}z)$}}
			\Big\|
			\\
			&=&
			\Big\| 
			f(\gamma^{-1}y)
			-
			f(\gamma^{-1}z)
			\Big\|
			+
			\Big\| 
			\zeta^{\scalebox{0.6}{$(\gamma,\gamma^{-1}y)$}}
			-\zeta^{\scalebox{0.6}{$(\gamma,\gamma^{-1}z)$}}
			\Big\|.
		\end{eqnarray*}

		For the second term,
		we have
		\begin{eqnarray*}
			\Big\|
			\zeta^{\scalebox{0.6}{$(\gamma,\gamma^{-1}y)$}}-
			\zeta^{\scalebox{0.6}{$(\gamma,\gamma^{-1}z)$}}
			\Big\|
			&=&
			\sqrt{\sum_{\scalebox{0.7}{$\tilde{\gamma}\in\Gamma$}}
				\Bigg\|
				\sqrt{y_{\tilde{\gamma}}}
				\cdot \Big(h(\gamma^{^{-1}}\tilde{\gamma})-
				h(\tilde{\gamma})\Big)	-	
				\sqrt{z_{\tilde{\gamma}}}
				\cdot \Big(h(\gamma^{^{-1}}\tilde{\gamma})-
				h(\tilde{\gamma})\Big)
				\Bigg\|^2}
			\\ 	
			&=&
			\sqrt{\sum_{\scalebox{0.7}{$\tilde{\gamma}\in\Gamma$}}
				\big|
				\sqrt{y_{\tilde{\gamma}}}
				-	
				\sqrt{z_{\tilde{\gamma}}}
				\big|^2
				\cdot 
				\Big\|
				h(\gamma^{^{-1}}\tilde{\gamma})-
				h(\tilde{\gamma})
				\Big\|^2}
			\\ 	
			&\leq& 
			\rho^{+}
			\big(
			d(\gamma^{-1},e)
			\big)
			\cdot \sqrt{\sum_{\scalebox{0.7}{$\tilde{\gamma}\in\Gamma$}}
				\big|	\sqrt{y_{\tilde{\gamma}}}
				-
				\sqrt{z_{\tilde{\gamma}}}
				\big|^2} 
			\\
			&=& 
			\rho^{+}
			\big(
			d(\gamma^{-1},e)
			\big)
			\cdot 
			\Big\|
			\mathcal{M}_{1,2}(y)-\mathcal{M}_{1,2}(z)
			\Big\|
			\\
			&\leq&	
			\rho^{+}
			\big(
			d(\gamma^{-1},e)
			\big)
			\cdot 
			C
			\cdot
			\sqrt{\|y-z\|}, 
		\end{eqnarray*}
		where 
		$\mathcal{M}_{1,2}: S(\ell^1(\Gamma))\to  S(\ell^2(\Gamma))$
		is the Mazur map between the unit spheres, 
		and $C$ is the constant given
		by Lemma \ref{ineaualityforMazurmap1}.

		Combining the above estimates, 
		we obtain 
		\begin{eqnarray*}
			&& 
			\Big\| 
			(\gamma\cdot f)
			(y)-
			(\gamma\cdot f)
			(z)
			\Big\|
			\\ 
			&\leq&
			\Big\| 
			f(\gamma^{-1}y)
			-
			f(\gamma^{-1}z)
			\Big\|
			+
			\rho^{+}
			\big(
			d(\gamma^{-1},e)
			\big)
			\cdot 
			C
			\cdot
			\sqrt{\|y-z\|} 
		\end{eqnarray*}	
		for all  $y,z \in  Z$.
		Since 
		the $\Gamma$-action on $Z$
		is isometric and 
		$f$ is uniformly continuous,
		it follows that
		$\gamma\cdot f$
		is uniformly continuous on $Z$.
		
		On the other hand, 
		$\gamma\cdot f$ is bounded since
		\begin{eqnarray*}
			\sup\limits_{z\in Z}
			\Big\|
			(\gamma\cdot f)
			(z)
			\Big\| &=&
			\sup\limits_{z\in Z}
			\Big\| 
			U^{\scalebox{0.6}{$(\gamma,\gamma^{-1}z)$}}
			\big(
			f(\gamma^{-1}z)
			\big)
			\Big\|
			\\
			&=&
			\sup\limits_{z\in Z}
			\Big\|
			\lambda^{\scalebox{0.6}{$\gamma$}}
			\big(
			f(\gamma^{-1}z)
			\big)
			+ \zeta^{\scalebox{0.6}{$(\gamma,\gamma^{-1}z)$}}
			\Big\|
			\\
			&\stackrel{\eqref{eq:upperboundforthecocycles}}{\leq}&
			\|f\|+ 	\rho^{+}
			\big(
			d(\gamma^{-1},e)
			\big).
		\end{eqnarray*}
		This completes the proof.
	\end{proof}

	By Lemma \ref{affineisometryproperty1} and Lemma \ref{The-very-most-essentional-idea},
	\eqref{eq:theactionsofgammaonsections}  gives rise to a $\Gamma$-action on 
	$C_{b}(Z, \mathscr{X}_{\Gamma})$. 
	The following remark shows that this action is affine isometric, although this is not essential for our purposes.
	
	\begin{rmk}
		Let  
		$$
		\pi:\Gamma\longrightarrow \text{Isom}
		\big(
		C_{b}(Z, \mathscr{X}_{\Gamma})
		\big)
		$$
		be the 
		group homomorphism
		given by 
		\begin{equation*}
			\big(\pi(\gamma)f\big)(z)=	\lambda^{\scalebox{0.6}{$\gamma$}}
			\big(f(\gamma^{-1}z)\big) 
		\end{equation*}
		for all
		$\gamma\in\Gamma$,
		$f\in  C_{b}(Z, \mathscr{X}_{\Gamma})$
		and 
		$z\in Z$.
		For each 
		$\gamma\in\Gamma$,
		define 
		$\kappa^{\scalebox{0.6}{$\gamma$}}: Z\to \mathscr{X}_{\Gamma}$
		by 
		$\kappa^{\scalebox{0.6}{$\gamma$}}(z)=
		\zeta^{\scalebox{0.6}{$(\gamma,\gamma^{-1}z)$}}$.	
		It follows from \eqref{eq:affineisometricesgeneralone} and \eqref{eq:theactionsofgammaonsections}
		that
		\begin{eqnarray*}
			(\gamma\cdot f)
			(z)
			&=&	U^{\scalebox{0.6}{$(\gamma,\gamma^{-1}z)$}}
			\big(
			f(\gamma^{-1}z)
			\big)
			\\
			&=&
			\lambda^{\scalebox{0.6}{$\gamma$}}
			\big(
			f(\gamma^{-1}z)
			\big)
			+\zeta^{\scalebox{0.6}{$(\gamma,\gamma^{-1}z)$}}
			\\
			&=&
			\big(\pi(\gamma)f\big)(z)
			+
			\kappa^{\scalebox{0.6}{$\gamma$}}(z)
		\end{eqnarray*}
		for all 
		$\gamma\in\Gamma$,
		$f\in
		C_{b}(Z, \mathscr{X}_{\Gamma})$
		and 
		$z\in {Z}$.
		Thus
		$$
		\gamma\cdot f
		= 	
		\pi(\gamma)f
		+
		\kappa^{\scalebox{0.6}{$\gamma$}}.
		$$
		By Lemma \ref{The-very-most-essentional-idea}, $\kappa^{\scalebox{0.6}{$\gamma$}}
		=
		\gamma\cdot f
		-
		\pi(\gamma)f
		\in 
		C_{b}(Z, \mathscr{X}_{\Gamma})$
		for all
		$\gamma\in\Gamma$.  
		Therefore,
		the $\Gamma$-action
		on 
		$C_{b}(Z, \mathscr{X}_{\Gamma})$  
		is  affine isometric,
		with associated one-cocycles 
		$\{
		\kappa^{\scalebox{0.6}{$\gamma$}} ~|~ \gamma\in\Gamma
		\}$.
	\end{rmk}

	We emphasize that all the results above rely only on the assumption that $h$ is bornologous.
	In the following proposition, we impose the stronger assumption that $h$ is a coarse embedding in order to establish properness.
	The affine isometric $\Gamma$-action
	on 
	$C_{b}(Z, \mathscr{X}_{\Gamma})$  
	is called \emph{uniformly proper} if 
	$$	
	\lim\limits_{\gamma\rightarrow\infty}
	\inf\limits_{z\in{Z}}
	\big\|
	\kappa^{\scalebox{0.6}{$\gamma$}}(z)
	\big\|
	=+\infty. 
	$$
	Recall that 	
	$h: \Gamma \to \mathscr{X}$  
	is a coarse embedding if 
	$h$ is bornologous, and there exists  a   non-decreasing function 
	$\rho^{-}: [0,+\infty) \to [0,+\infty)$ 
	such that
	$\lim\limits_{t\to\infty}\rho^{-}(t)=+\infty$ and 
	\begin{equation}	\label{eq:lowerboundforcoarseembedding}
		\Big\|
		h(\gamma_{1})-
		h(\gamma_{2})
		\Big\|
		\geq\rho^{-}
		\big(d(\gamma_{1},
		\gamma_{2})\big)
	\end{equation}
	for all 
	$\gamma_{1},\gamma_{2}\in\Gamma$.

	We conclude this section by showing that  the $\Gamma$-action on 
	$C_{b}(Z, \mathscr{X}_{\Gamma})$
	is uniformly proper.

	\begin{prop}\label{strongproperness}
		Suppose that
		$h:\Gamma\rightarrow\mathscr{X}$ 
		is a coarse embedding.
		Then 
		\begin{equation*}
			\lim\limits_{\gamma\rightarrow\infty}
			\inf\limits_{z\in{Z}}
			\big\|
			\zeta^{\scalebox{0.6}{$(\gamma,z)$}}
			\big\|
			=+\infty
			~\bigg(
			\text{equivalently}~ 	
			\lim\limits_{\gamma\rightarrow\infty}
			\inf\limits_{z\in{Z}}
			\big\|
			\kappa^{\scalebox{0.6}{$\gamma$}}(z)
			\big\|
			=+\infty 
			\bigg).
		\end{equation*}
	\end{prop}
	\begin{proof}
		It follows from \eqref{eq:cocycles} and  \eqref{eq:lowerboundforcoarseembedding}
		that
		\begin{eqnarray*}
			\big\|
			\kappa^{\scalebox{0.6}{$\gamma$}}(z)
			\big\|
			&=&
			\big\|
			\zeta^{\scalebox{0.6}{$(\gamma,\gamma^{-1}z)$}}
			\big\|
			\\
			&=&
			\sqrt{
				\sum_{\scalebox{0.7}{$\tilde{\gamma}\in\Gamma$}}
				z_{\tilde{\gamma}}
				\cdot 
				\big\|
				h(\gamma^{-1}\tilde{\gamma})-
				h(\tilde{\gamma})
				\big\|^2
			}
			\\
			&\geq& 
			\rho^{-}
			\big(
			d(\gamma^{-1},e)
			\big)
		\end{eqnarray*}
		for all 
		$\gamma\in\Gamma$
		and 
		$z\in 
		Z$.
		Since $\lim\limits_{t\to\infty}\rho^{-}(t)=+\infty$, 
		the conclusion follows. 			
	\end{proof}

	\begin{rmk}
		The uniform properness above is equivalent to the classical notion of properness for the transformation groupoid action on  continuous fields of Banach spaces; 
		see 
		Definition 6.3 in \cite{STY}.
		Moreover, all results in this section remain valid upon replacing
		$\ell^2(\Gamma, \mathscr{X})$  by $\ell^p(\Gamma, \mathscr{X})$, and the square root by the 
		$p$-th root.  
		We believe that the action constructed in this section on the trivial bundle with Banach-space fibres may be of independent interest in the study of group dynamical systems. 
	\end{rmk}

	\section{The $C^*$-algebra 
		associated to a function space}
	
	In this section, 
	we construct a  $C^*$-algebra associated to a function space
	$C_{b}(Y,\mathscr{X})$,
	where $Y$ is a metric space and $\mathscr{X}$ is a Banach space with Property $(H)$. 
	$C^*$-algebras of this type, 
	with  $Y$  taken to be  $Z$ defined  in \eqref{eq:simplicityforthebasespaceZ} 
	 and $\mathscr{X}$  replaced by  
	$\mathscr{X}_{\Gamma} =\ell^2(\Gamma, \mathscr{X})$,
will serve as building blocks for the proper $\Gamma$-algebras constructed in subsequent sections.

	\subsection{Clifford $C^*$-algebras}
	Given a \emph{real} Hilbert space $\mathscr{H}$, 
	we write $\text{Cliff} _{\mathbb{C}}(\mathscr{H})$ 
	for the complex Clifford $C^*$-algebra generated by $\mathscr{H}$. 
	More precisely, 
	consider the \emph{antisymmetric Fock space} 
	\[
	\varLambda_{\mathbb{C}} ^*\mathscr{H}:= \bigoplus_{k=0}^\infty \varLambda_{\mathbb{C}}^k \mathscr{H},
	\]
	where $\varLambda_{\mathbb{C}}^k \mathscr{H}$, the \emph{$k$-th complex exterior power} of $\mathscr{H}$, is defined to be the complexification of the quotient of the real tensor vector space $\bigotimes_{j=1}^k\mathscr{H}$ by equating $w_1\otimes\cdots\otimes w_k$ with $\text{sgn}(\sigma)\ w_{\sigma(1)}\otimes\cdots\otimes w_{\sigma(k)}$ for any $k$-permutation $\sigma$, with the equivalence class denoted by $w_1\wedge\cdots\wedge w_k$. For each $v \in \mathscr{H}$, we may define its \emph{exterior multiplication operator} $\text{ext}(v): \varLambda_{\mathbb{C}} ^*\mathscr{H}\rightarrow\varLambda_{\mathbb{C}}^*\mathscr{H}$ 
	by
	\[
	\text{ext}(v)(w_1\wedge\cdots\wedge w_k) = v\wedge w_1\wedge\cdots\wedge w_k \; .
	\] Let $\text{int}(v)$ be the adjoint operator of $\text{ext}(v)$. 
	The \textit{Clifford multiplication operator}  $C(v)$ is  defined by
	$$C(v)= \text{ext}(v)+\text{int}(v)$$
	for each $v\in \mathscr{H}$. 
	We have the relation
	\begin{equation*}\label{cliffordrelation}
		C(v)C(w)+C(w)C(v) =2 \left\langle v,w \right\rangle
	\end{equation*}
	for any $v,w \in \mathscr{H}$. In particular,
	$C(v)^2=\|v\|^2$,
	which is a scalar multiplication. 
	Notice that $C$ is linear,  which means that 
	$C(a_1 v_1+a_2 v_2)=a_1 C(v_1)+a_2 C(v_2)$ 
	for all $a_1,a_2\in\mathbb{R}$ and $v_1,v_2\in \mathscr{H}$. 
	The \emph{(complex) Clifford $C^*$-algebra} $\text{Cliff}_\mathbb{C} (\mathscr{H})$ of $\mathscr{H}$ is the $*$-subalgebra of 
	$\mathcal{B}(\varLambda_{\mathbb{C}} ^*\mathscr{H})$ generated by $\{C(v) ~|~ v \in \mathscr{H}\}$. We remark that the assignment $\mathscr{H} \mapsto \text{Cliff}_\mathbb{C}(\mathscr{H})$ 
	is functorial with regard to isometric linear embeddings of Hilbert spaces and $*$-homomorphisms and it also preserves direct limits in the respective categories. In particular, the involutive isometry on $\mathscr{H}$ that takes each $v$ to $-v$ induces a distinguished involutive $*$-automorphism of 
	$\text{Cliff}_\mathbb{C}(\mathscr{H})$, 
	which turns the latter into a \emph{{graded} $C^*$-algebra}. All $*$-homomorphisms induced from isometric linear embeddings of Hilbert spaces also preserve the grading.

	\begin{rmk}\label{functionalcalculusforvectors} 
		We point out an elementary fact about functional calculus in Clifford $C^*$-algebras: 
		for any $v \in \mathscr{H}$ and any bounded continuous complex function $f$ on $\mathbb{R}$, since $C(v)$ is self-adjoint, we may apply functional calculus to obtain $f(C(v))$, and the result is rather simple: 
		\begin{enumerate}
			\item if $f$ is even, then $f(C(v))$ is equal to the scalar $f(\|v\|)$; 
			\item if $f$ is odd, then $f(C(v))$ is equal to $f(\|v\|) C\big(\frac{v}{\|v\|}\big)$ 
			(understood to be $0$ when $v = 0$), which is a rescaling of $C(v)$. 
		\end{enumerate}
	\end{rmk}

	\subsection{The $C^*$-algebra  associated to a  Banach space with Property 
		$(H)$}
	
	In this subsection, we recall the $C^*$-algebra  (\cite{KY})  associated to a  Banach space with Property 
	$(H)$.
	
	\begin{defn}
		Let $Y$ be a metric space with a fixed base point $y_0\in Y$, and let $A$ be a graded $C^*$-algebra. 
		Denote by 
		$\mathcal{F}(Y,A)$ 
		the graded $C^*$-algebra of all bounded functions from $Y$ to $A$, equipped with the supremum norm
		and with the grading induced from that of
		$A$.
		Moreover,  denote by $C_{0}(Y,A)$  the graded $C^*$-subalgebra of $\mathcal{F}(Y,A)$ 
		consisting of all
		bounded and uniformly continuous functions  $f$
		such that 
		$$
		\lim\limits_{
			d(y,y_0) \to \infty}
		f(y)=0.
		$$
	\end{defn}
	
	We remark that  $Y$ is not required to be locally compact, and that the definition of  $C_{0}(Y,A)$
	is independent of the choice of base point $y_0$.
	
	\begin{defn}\label{Cliffordmap1}
		Suppose that $\mathscr{X}$ is a Property 
		$(H)$
		Banach space with a Property 
		$(H)$
		map 
		$\phi: S(\mathscr{X})\to S(\mathscr{H})$. 	
		Let  $\varPhi: \mathscr{X} \to \mathscr{H}$
		be a proper extension of $\phi$.
		We define
		$\mathfrak{B} :\mathscr{X} \to
		\text{Cliff}_{\mathbb{C}}
		(\mathscr{H})$
		by $v \mapsto  C\big(\varPhi(v)\big)$.	
		Moreover,
		for any  
		$v_{0} \in \mathscr{X}$,  
		we define 
		$\mathfrak{B}^{v_{0}}: \mathscr{X} \to
		\text{Cliff}_{\mathbb{C}}
		(\mathscr{H})$  
		by 
		$\mathfrak{B}^{v_{0}}(v) = \mathfrak{B}(v-v_{0})$.
	\end{defn}

	Denote by $\varTheta$ the function of multiplication by $\theta$ on $\mathbb{R}$,
	viewed as a degree one, essentially self-adjoint, unbounded multiplier of 
	$\mathcal{S}=C_0(\mathbb{R})$  with domain the compactly supported functions in $\mathcal{S}$.
	We grade $\mathcal{S}$ according
	to even and odd functions, and write
	$\mathcal{S}_{\rm even}$ for the subspace of all even functions.
	
	The third and fourth conditions for proper extension  in Definition \ref{ProperextensionsofPropertyHspaces}
	guarantee that, for each $v\in\mathscr{X}$,
	$\mathfrak{B}^{v}$ is 
	a degree one, essentially self-adjoint,
	unbounded multiplier of 
	${C}_{0}
	\big(	
	\mathscr{X}, 
	\text{Cliff}_{\mathbb{C}}
	(\mathscr{H})
	\big)$ 
	with domain 
	${C}_{c}
	\big(	
	\mathscr{X}, 
	\text{Cliff}_{\mathbb{C}}
	(\mathscr{H})
	\big)$, 
	consisting of
	all bounded and uniformly continuous functions
	whose supports are bounded subsets of $\mathscr{X}$.

	Thus, 
	for each $v\in\mathscr{X}$, 
	the operator 
	$\varTheta \widehat{\otimes} 1
	+1\widehat{\otimes} 		\mathfrak{B}^{v}$ 
	is a degree one, essentially self-adjoint,
	unbounded multiplier of
	$\mathcal{S}\widehat{\otimes} {C}_{0}
	\big(	
	\mathscr{X}, 
	\text{Cliff}_{\mathbb{C}}
	(\mathscr{H})
	\big)$, 
	and 
	has compact resolvents in the sense of multiplier algebra theory \cite{KY}. 
	Here $\widehat{\otimes}$ denotes the graded tensor product. 
	Hence, by functional calculus, 
	we obtain a graded $*$-homomorphism 
	\begin{equation}	\label{eq:functionalcalculusone}
		\begin{array}{rcl}
			\beta_{v}:~ \mathcal{S} &\longrightarrow&
			\mathcal{S}\widehat{\otimes} {C}_{0}
			\big(\mathscr{X}, 
			\text{Cliff}_{\mathbb{C}}
			(\mathscr{H})
			\big)
			\\
			f &\longmapsto &
			f\big(
			\varTheta \widehat{\otimes} 1+1\widehat{\otimes} 		\mathfrak{B}^{v}
			\big).
		\end{array}
	\end{equation}
	For 
	$v=0$,
	we simply write
	$\beta$ for it.

	\begin{defn}[\cite{KY}]
		\label{basicproperalgebra-geng}
		Let  $\mathscr{X}$  be a Property $(H)$ Banach space with a Property 
		$(H)$
		map 
		$\phi: S(\mathscr{X})\to S(\mathscr{H})$. 	 
		Denote by 
		$\mathcal{A}
		(\mathscr{X})$
		the graded $C^*$-subalgebra of 
		$
		\mathcal{S}\widehat{\otimes}
		{C}_{0}
		\big(\mathscr{X}, 
		\text{Cliff}_{\mathbb{C}}
		(\mathscr{H})
		\big)
		$
		generated by the set 
		$$\Big\{ 		\beta_{v}(f) ~\Big|~
		f\in \mathcal{S},~ v\in\mathscr{X}
		\Big\},$$
		where $\beta_{v}$  encodes information about $\phi$ and  is defined by \eqref{eq:functionalcalculusone}.
		Moreover, for a subspace 
		$\mathscr{V}\subseteq\mathscr{X}$,
		we denote by 	
		$\mathcal{A}
		(\mathscr{X}, \mathscr{V})$
		the graded $C^*$-subalgebra of 
		$\mathcal{A}
		(\mathscr{X})$
		generated by  
		$$\Big\{ \beta_{v}(f) ~\Big|~
		f\in \mathcal{S},~ v\in \mathscr{V}
		\Big\}.$$
	\end{defn}

	For any (not necessarily continuous) function $\xi$ from a metric space $Y$
	to  
	$\mathscr{X}$, 
	\eqref{eq:functionalcalculusone}  
	induces a graded $*$-homomorphism 
	\begin{equation}
		\label{eq:functionalcalculustwo}
		\begin{array}{rcl}
			\beta_{\xi}: ~ \mathcal{S} &\longrightarrow &
			\mathcal{F}
			\big(
			Y, ~
			\mathcal{A}
			(\mathscr{X})
			\big)
			\\
			f &\longmapsto & 
			\Big\{
			y\mapsto 
			\beta_{\xi(y)}(f)
			\Big\}.
		\end{array}
	\end{equation}
	If 
	$\xi$ is the zero function,
	we write $\beta_{\bf 0}$
	in place of $\beta_\xi$.
	
Note that 
we \emph{disregard the gradings of these algebras} 
and  treat them as trivially graded  when computing their $K$-theory.

	\subsection{Uniform continuity}

	In this subsection, we show that, for any $f\in\mathcal{S}$,
	$\beta_{v}(f)$ 
	is uniformly continuous with respect to $v$  on every bounded subset of $\mathscr{X}$.
	For this purpose, we first introduce an analogue of the functional calculus in Section 4.2.

	\begin{lem}\label{Clifford-key-fact}
		Suppose that  $X$  is a metric space and $Y$ is  a locally compact metric space. 	
		Let $A$   
		be a graded $C^*$-algebra, and let $V$  be a finite-dimensional Euclidean space.
		Then there exists a graded $*$-isomorphism
		$$C_{0}\big(Y,\text{Cliff}_\mathbb{C}(V)\big)
		\widehat{\otimes}
		C_{0}(X,A) \cong C_{0}\big(Y\times X, \text{Cliff}_\mathbb{C}(V)\widehat{\otimes}A\big)$$ 
		mapping $f \widehat{\otimes} g$ to $\psi$,
		where $\psi$ is defined by
		$\psi(y,x)=f(y)  \widehat{\otimes} g(x)$ for 
		$(y,x)\in Y\times X$  and the metric on  
		$Y\times X$ is given by
		$d\big((y_{1},x_{1}),(y_{2},x_{2})\big)=\sqrt{d_{X}(x_{1},x_{2})^{2}+d_{Y}(y_{1},y_{2})^{2}}$. 
	\end{lem}
	\begin{proof}
		
		It is a classical result that
		$\text{Cliff}_\mathbb{C}(V)$ is isomorphic to a  $C^*$-algebra consisting of complex matrices, namely
		\begin{equation*} 
			\text{Cliff}_\mathbb{C}(V)
			\cong
			\begin{cases}
				M_{2^n}(\mathbb{C}),& \text{dim}(V)=2n, 
				\\
				M_{2^n}(\mathbb{C})\oplus 	M_{2^n}(\mathbb{C}),&\text{dim}(V)=2n+1.
			\end{cases}
		\end{equation*}
		This induces a graded  $*$-isomorphism  
		$$
		\text{Cliff}_\mathbb{C}(V)\widehat{\otimes} C_{0}(X,A)\cong
		C_{0}\big(X,\text{Cliff}_\mathbb{C}(V)\widehat{\otimes} A\big),
		$$  
		which sends 
		$b\widehat{\otimes} f$ to the function defined by
		$x\mapsto  b\widehat{\otimes} f(x)$.   
		It follows that 
		\begin{equation}
			\label{eq:elementaryisomorohdlkjflkdslkfjks}
			\begin{array}{ccc}
				C_{0}\big(Y,\text{Cliff}_\mathbb{C}(V)\big)
				\widehat{\otimes} 
				C_{0}(X,A) &\cong & 
				C_{0}\big(Y,~	\text{Cliff}_\mathbb{C}(V)
				\widehat{\otimes}C_{0}(X,A)\big) \\
				&\cong  & C_{0}\big( Y,~ C_{0}
				(X,\text{Cliff}_\mathbb{C}(V) \widehat{\otimes}A) \big).
			\end{array}
		\end{equation}

		Define  
		\begin{equation*}
			\tau: 
			C_{0}\big(Y,~ C_{0}(X,\text{Cliff}_\mathbb{C}(V) \widehat{\otimes}A) \big) \longrightarrow 
			C_{0}\big(Y\times X,\text{Cliff}_\mathbb{C}(V) \widehat{\otimes}A\big)
		\end{equation*}
		by 
		$$
		\big(\tau(f)\big)(y,x)=\big(f(y)\big)(x)
		$$
		for all $(y,x)\in Y\times X$.
		It is straightforward to verify that  
		$\tau$ is a $*$-isomorphism.

	The composition of $\tau$ with the isomorphism in
	 \eqref{eq:elementaryisomorohdlkjflkdslkfjks}  satisfies the required conditions.
	\end{proof}

	We remark that the above lemma is classical when $X$ is also locally compact.
	
	As a corollary, 
	for any  Banach space $\mathscr{X}$ with Property $(H)$, 
	we obtain a graded $*$-isomorphism
	\begin{equation*}
		C_{0}
		\big(
		\mathbb{R}^+,\text{Cliff}_\mathbb{C}(\mathbb{R})
		\big)
		\widehat{\otimes} 
		C_{0}
		\big(
		\mathscr{X},\text{Cliff}_\mathbb{C}(\mathscr{H})
		\big)
		\cong 
		C_{0}
		\big(
		\mathbb{R}^+\times\mathscr{X},  \text{Cliff}_\mathbb{C}(\mathbb{R}\oplus\mathscr{H})
		\big),
	\end{equation*} 
	sending $f \widehat{\otimes} g$ to $\psi$,
	where 
	$\psi(t,x)=f(t)  \widehat{\otimes} g(x)$ for 
	$(t,x)\in \mathbb{R}^+\times\mathscr{X}$. 
	Since $\mathcal{S}$ 
	is graded $*$-isomorphic to 
	the $C^*$-algebra (\cite{GWY})
	$$
	\Big\{
	f\in 
	C_{0}
	\big(
	\mathbb{R}^+,\text{Cliff}_\mathbb{C}(\mathbb{R})
	\big) ~\Big|~ 
	f(0)\in\mathbb{C} 
	\Big\},
	$$
	it follows that
	$\mathcal{S}\widehat{\otimes}	C_{0}
	\big(
	\mathscr{X},\text{Cliff}_\mathbb{C}(\mathscr{H})
	\big)$
	is graded $*$-isomorphic to 
	the $C^*$-algebra
	$$
	\Big\{	
	f\in 
	C_{0}
	\big(
	\mathbb{R}^+\times\mathscr{X},  \text{Cliff}_\mathbb{C}(\mathbb{R}\oplus\mathscr{H})
	\big) ~\Big|~ 
	f(0,v)\in \text{Cliff}_\mathbb{C}(\mathscr{H}) \text{ for all } v\in\mathscr{X} 
	\Big\}.
	$$
	Under this identification, 
	for any $f\in\mathcal{S}$,
	the element 
	$\beta_{v_0}(f)$
	in 
	$\mathcal{S}\widehat{\otimes}	
	C_{0}
	\big(
	\mathscr{X},\text{Cliff}_\mathbb{C}(\mathscr{H})
	\big)$
	corresponds to the function
	\begin{equation}
		\mathbb{R}^+\times\mathscr{X}\ni (\theta,v) \longmapsto 
		f\Big(
		C \big( 
		\theta, \varPhi(v-v_0) 
		\big)
		\Big) \in
		\text{Cliff}_\mathbb{C}(\mathbb{R}\oplus\mathscr{H})
	\end{equation}
	in 
	$C_{0}
	\big(
	\mathbb{R}^+\times\mathscr{X},  \text{Cliff}_\mathbb{C}(\mathbb{R}\oplus\mathscr{H})
	\big)$,
	where $f\big(C(\cdot)\big)$ is the functional calculus of vectors as in Remark \ref{functionalcalculusforvectors}. 
	We identify these two formulations throughout this section.

	\begin{defn}[\cite{GWY}]
		For  
		$f \in C_0(\mathbb{R})$  
		and 
		$r > 0$,  
		we define
		\[
		\Lambda_{f}(r) = r \cdot \sup \left\{ \frac{|f(t) - f(-t)|}{2t} ~\bigg|~ t \geq r \right\}.
		\]
	\end{defn}
	
	It is clear that  
	$\Lambda_{f}(r) \leq \|f\|$.
	In general, 
	$\Lambda_{f}(r)$ 
	is not monotone in $r\in(0,+\infty)$. 
	However, 
	$$
	\lim\limits_{r\rightarrow0^+}
	\Lambda_{f}(r)=0
	$$
	for all $f\in  C_0(\mathbb{R})$.
	See \cite{GWY} for details.

	It was shown in \cite{GWY} that 
	$\beta_{v}(f)$ 
	is uniformly continuous in 
	$v$ on the entire Hilbert--Hadamard space $M$.
	In contrast, for a general Banach space $\mathscr{X}$ 
	with Property $(H)$, $\beta_{v}(f)$ is uniformly continuous only on bounded subsets of $\mathscr{X}$.

	\begin{lem}\label{comparingbasepoint}
		Let $\mathscr{X}$  
		and 
		$\varPhi: \mathscr{X} \to \mathscr{H}$
		be as in  Definition \ref{Cliffordmap1}. 
		Let
		$f\in C_{c}(\mathbb{R})$ and
		$R>0$. 
		Then there exists $c>R$ 
		such that 
		for all  
		$v_{0},v_{1}\in B_{\scalebox{0.6}{$\mathscr{X}$}}(0,R)$,  
		\begin{equation*}
			\Big\|
			\beta_{v_{0}}(f)-\beta_{v_{1}}(f)
			\Big\|
			\leq 
			4\cdot\omega_{f}
			\Big(
			\omega_{\scalebox{0.6}{$\varPhi_{c}$}}
			\big(
			\|v_{0}-v_{1}\|
			\big)
			\Big) + \Lambda_{f}\Big(
			\omega_{\scalebox{0.6}{$\varPhi_{c}$}}
			\big(
			\|v_{0}-v_{1}\|
			\big)
			\Big),
		\end{equation*}
		where 
		$\varPhi_{c}$
		denotes the restriction of  $\varPhi$ to the closed ball $B_{\scalebox{0.6}{$\mathscr{X}$}}(0,c)$.
	\end{lem}
	\begin{proof}
		Suppose 
		${\rm Supp}(f)\subseteq [-N,N]$ for some $N>0$.
		Let  $f_0$ and $f_1$
		be the even and odd parts of 
		$f$, that is,
		\[
		f_0(t) = \frac{f(t) + f(-t)}{2}
		\]
		and
		\[
		f_1(t) = \frac{f(t) - f(-t)}{2}.
		\]
		For $i = 0, 1$ 
		and any 
		$s > 0$, 
		we have 
		$f_i \in C_0(\mathbb{R})$
		with 
		${\rm Supp}(f_i)\subseteq [-N,N]$, 
		$\|f_i\|\leq\|f\|$, 
		$\omega_{f_{i}}(s)\leq\omega_{f}(s)$  
		and
		\[
		\Lambda_{f}(s) = \Lambda_{f_1}(s) = s \cdot \sup\left\{ \frac{|f_1(t)|}{t} ~\bigg|~ t \geq s \right\}.
		\]

		Since $\varPhi$ is a proper extension of $\phi$, 
		it follows from the last condition of Definition \ref{ProperextensionsofPropertyHspaces}
		that
		there exists  
		$L>0$
		such that 
		$$
		\big\|
		\varPhi(v-w)
		\big\|
		\geq 
		N
		$$
		for all $v,w\in\mathscr{X}$ 
		satisfying 
		$\|w\| \leq  R$
		and  
		$\|v\|\geq L$.	
		By Remark \ref{functionalcalculusforvectors} and
		${\rm Supp}(f_i)\subseteq [-N,N]$,   
		we obtain
		\begin{eqnarray}
			\label{eq:firstonedkljflak}
			\beta_{w}(f_i)(\theta,v) = 	f_{i}\Big(
			C \big( 
			\theta, \varPhi(v-w) 
			\big)
			\Big) = 0
		\end{eqnarray}
		for all $i=0,1$, $\theta\in\mathbb{R}^+$, 
		and 
		$v,w\in\mathscr{X}$ 
		satisfying 
		$\|w\| \leq  R$
		and   
		$\|v\|\geq L$.	
		
		Let	
		$c=L+R$ 
		and 
		denote by $\varPhi_{c}$
		the restriction of  $\varPhi$ 
		to the closed ball $B_{\scalebox{0.6}{$\mathscr{X}$}}(0,c)$ in $\mathscr{X}$.
		Fix any two vectors
		$v_{0},v_{1}\in B_{\scalebox{0.6}{$\mathscr{X}$}}(0,R)$
		and set 
		$r=\omega_{\scalebox{0.6}{$\varPhi_{c}$}}
		\big(
		\|v_{0}-v_{1}\|
		\big)$.
		We consider two cases.

		Case of $f_{0}$. 
		By Remark \ref{functionalcalculusforvectors}, 
		we have 
		$\beta_{v_{i}}(f_0)(\theta,v) = f_0\big(
		\|C(\theta, \varPhi(v-v_{i}))\|
		\big)$ 
		for  
		$i = 0, 1$. 	
		Thus, for all
		$\theta\in\mathbb{R}^+$
		and $v\in\mathscr{X}$ 
		with 
		$\|v\|\leq L$, 	
		\begin{eqnarray*}
			\Big\|
			\Big(
			\beta_{v_{0}}(f_0) - \beta_{v_{1}}(f_0)
			\Big)
			(\theta,v)
			\Big\| 
			&=& 
			\Big|
			f_0\big(\|C(\theta, \varPhi(v-v_{0}))\|\big) - f_0\big(\|C(\theta, \varPhi(v-v_{1}))\|\big)
			\Big|
			\\
			&\leq& 
			\omega_{f_{0}}
			\bigg(
			\Big| 
			\|C(\theta, \varPhi(v-v_{0}))\|-
			\|C(\theta, \varPhi(v-v_{1}))\|
			\Big| 
			\bigg)
			\\
			&\leq& 
			\omega_{f_{0}}
			\bigg(
			\Big\| 
			\varPhi(v-v_{0})-
			\varPhi(v-v_{1})
			\Big\| 
			\bigg)
			\\
			&=& 
			\omega_{f_{0}}
			\bigg(
			\Big\| 
			\varPhi_{c}(v-v_{0})-
			\varPhi_{c}(v-v_{1})
			\Big\| 
			\bigg)
			\\
			&\leq& 
			\omega_{f_{0}}
			(r)
			\leq
			\omega_{f}(r).
		\end{eqnarray*}	
		For all $\theta\in\mathbb{R}^+$ and  
		$v\in\mathscr{X}$ 
		with 
		$\|v\|\geq L$, it follows from \eqref{eq:firstonedkljflak} 
		that 
		\begin{equation*}
			\label{eq:secondonedkljflak}
			\Big\|
			\Big(
			\beta_{v_{0}}(f_0) - \beta_{v_{1}}(f_0)
			\Big)
			(\theta,v)
			\Big\| 
			= 0.		
		\end{equation*}
		Hence, 
		\begin{equation}
			\label{eq:thirdonedkljflak}
			\Big\|
			\beta_{v_{0}}(f_0) - \beta_{v_{1}}(f_0)
			\Big\| 
			\leq
			\omega_{f}\Big(
			\omega_{\scalebox{0.6}{$\varPhi_{c}$}}
			\big(
			\|v_{0}-v_{1}\|
			\big)
			\Big).
		\end{equation}

		Case of $f_{1}$.
		By \eqref{eq:firstonedkljflak},
		we have 
		\begin{equation*}
			\Big\|
			\Big(
			\beta_{v_{0}}(f_1) - \beta_{v_{1}}(f_1)
			\Big)
			(\theta,v)
			\Big\| 
			= 0		
		\end{equation*}
		for all $\theta\in\mathbb{R}^+$ and  
		$v\in\mathscr{X}$ 
		satisfying 
		$\|v\|\geq L$.	
		Next, for 
		$\theta\in\mathbb{R}^+$ and  
		$v\in\mathscr{X}$ 
		with 
		$\|v\|\leq L$, 
		we distinguish two subcases.

		\noindent $(1)$  
		If one of 	
		$\big\|	C \big( 
		\theta, \varPhi(v-v_0) 
		\big)
		\big\|$ 
		and 
		$\big\|	C \big( 
		\theta, \varPhi(v-v_1) 
		\big)
		\big\|$ 
		is less than $r$, 
		then the other is bounded by
		$2r$.
		Hence observing that
		$f_1(0)=0$,
		we obtain 
		\begin{eqnarray*}
			\Big\|
			\Big(
			\beta_{v_{0}}(f_1) - \beta_{v_{1}}(f_1)
			\Big)
			(\theta,v)
			\Big\| 
			&=& 
			\Big\|
			f_1\big(C(\theta, \varPhi(v-v_{0}))\big) - f_1\big(C(\theta, \varPhi(v-v_{1}))\big)
			\Big\|
			\\
			&\leq& 
			\Big\|
			f_1\big(C(\theta, \varPhi(v-v_{0}))\big)
			\Big\| 
			+ 
			\Big\| 
			f_1\big(C(\theta, \varPhi(v-v_{1}))\big)
			\Big\|
			\\
			&=& 
			\Big|
			f_1\big(\|
			C(\theta, \varPhi(v-v_{0}))\|
			\big)
			\Big| 
			+ 
			\Big| 
			f_1\big(
			\|C(\theta, \varPhi(v-v_{1}))\|
			\big)
			\Big|
			\\
			&\leq& 
			\omega_{f_1}(r)+\omega_{f_1}(2r)
			\\
			&\leq& 
			3\cdot \omega_{f_1}(r)
			\leq
			3\cdot \omega_{f}(r).
		\end{eqnarray*}

		\noindent
		$(2)$ 
		Suppose that
		$$
		\min\Big\{
		\big\|	C \big( 
		\theta, \varPhi(v-v_0) 
		\big)
		\big\|,
		\big\|	C \big( 
		\theta, \varPhi(v-v_1) 
		\big)
		\big\|
		\Big\}\geq r.
		$$ 
		Let $\vartheta$ denote  
		the angle between
		the vectors
		$\big(\theta, \varPhi(v-v_{0})\big)$
		and 
		$\big(\theta, \varPhi(v-v_{1})\big)$
		in the real Hilbert space $\mathbb{R}\oplus\mathscr{H}$. 
		By Euclidean geometry,
		\begin{eqnarray*}
			\cos(\vartheta)=
			\dfrac{\|C(\theta, \varPhi(v-v_{0}))\|^2+\|C(\theta, \varPhi(v-v_{1}))\|^2-\|C(\theta, \varPhi(v-v_{0}))-C(\theta, \varPhi(v-v_{1}))\|^2}{2\cdot\|C(\theta, \varPhi(v-v_{0}))\|\cdot\|C(\theta, \varPhi(v-v_{1}))\|}. 
		\end{eqnarray*}
		Thus,  	
		\begin{eqnarray*}
			1-\cos(\vartheta)
			&=&
			\dfrac{
				\|C(\theta, \varPhi(v-v_{0}))-C(\theta, \varPhi(v-v_{1}))\|^2
				-
				\big(
				\|C(\theta, \varPhi(v-v_{0}))\|-\|C(\theta, \varPhi(v-v_{1}))\|
				\big)^2
			}{2\cdot\|C(\theta, \varPhi(v-v_{0}))\|\cdot\|C(\theta, \varPhi(v-v_{1}))\|}
			\\
			&\leq&
			\dfrac{r^2}{2\cdot\|C(\theta, \varPhi(v-v_{0}))\|\cdot\|C(\theta, \varPhi(v-v_{1}))\|}.
		\end{eqnarray*}
		Hence, writing 	
		$s_i=\big\|C(\theta, \varPhi(v-v_{i}))\big\|$
		for $i=0,1$  and using Remark \ref{functionalcalculusforvectors},
		we obtain 
		\begin{eqnarray*}
			\Big\|
			\Big(
			\beta_{v_{0}}(f_1) - \beta_{v_{1}}(f_1)
			\Big)
			(\theta,v)
			\Big\|^2 
			&=& 
			\Big\|
			f_1\big(C(\theta, \varPhi(v-v_{0}))\big) - f_1\big(C(\theta, \varPhi(v-v_{1}))\big)
			\Big\|^2
			\\
			&=& 
			f_1(s_0)^2+f_1(s_1)^2-2f_1(s_0)f_1(s_1)\cos(\vartheta)
			\\
			&=& 
			\big(
			f_1(s_0)-f_1(s_1)
			\big)^2+
			2f_1(s_0)f_1(s_1)
			\big(
			1-\cos(\vartheta)
			\big)
			\\
			&\leq& 
			\omega_{f_1}(r)^2+
			\frac{f_1(s_0)}{s_0}
			\cdot
			\frac{f_1(s_1)}{s_1}
			\cdot 
			r^2
			\\
			&\leq& 
			\omega_{f_1}(r)^2+
			\Lambda_{f_1}(r)^2
			\\
			&\leq& 
			\Big(
			\omega_{f_1}(r)+
			\Lambda_{f_1}(r)
			\Big)^2
			\\
			&\leq& 
			\Big( 
			\omega_{f}(r)+
			\Lambda_{f}(r)
			\Big)^2.
		\end{eqnarray*}
		The above estimates imply that
		\begin{equation}
			\label{eq:lastonedkljflak}
			\Big\|
			\beta_{v_{0}}(f_1) - \beta_{v_{1}}(f_1)
			\Big\| 
			\leq 
			3\cdot\omega_{f}
			\Big(
			\omega_{\scalebox{0.6}{$\varPhi_{c}$}}
			\big(
			\|v_{0}-v_{1}\|
			\big)
			\Big) + 	\Lambda_{f}
			\Big(
			\omega_{\scalebox{0.6}{$\varPhi_{c}$}}
			\big(
			\|v_{0}-v_{1}\|
			\big)
			\Big).
		\end{equation}

		Finally, since $f=f_0+f_1$, 
		combining \eqref{eq:thirdonedkljflak} and \eqref{eq:lastonedkljflak} 
		yields the desired conclusion.
	\end{proof}

	\begin{prop}\label{localuniformlycontinuous}
		Let $\mathscr{X}$  
		and 
		$\varPhi: 
		\mathscr{X}
		\to
		\mathscr{H}$ 
		be as in Definition \ref{Cliffordmap1}. 
		Let
		$f\in \mathcal{S}$ 
		and
		$R>0$. 
		Then for any 
		$\epsilon>0$,
		there exists 
		$\delta>0$
		such that 
		\begin{equation*}
			\Big\|
			\beta_{v_{0}}(f)-\beta_{v_{1}}(f)
			\Big\|
			<\epsilon
		\end{equation*}
		whenever   
		$v_{0},v_{1}\in B_{\scalebox{0.6}{$\mathscr{X}$}}(0,R)$
		satisfy
		$\|v_{0}-v_{1}\|<\delta$. 
	\end{prop}	
	\begin{proof}
		Since $\varPhi$ is a proper extension of $\phi$, 
		it follows from the third condition of Definition \ref{ProperextensionsofPropertyHspaces}
		that
		$\varPhi$ is uniformly continuous on every bounded subset of $\mathscr{X}$.
		Thus, by Lemma \ref{comparingbasepoint},  
		the result holds for all	 
		$f\in C_{c}(\mathbb{R})$. 
		
		The general case is obtained by an approximation argument.
	\end{proof}

	\subsection{The $C^*$-algebra associated to 
		$C_{b}(Y,\mathscr{X})$ 
		and its $K$-theory}

	In this subsection, 
	we construct a  $C^*$-algebra associated to a function space
	$C_{b}(Y,\mathscr{X})$,
	where $Y$ is a metric space and $\mathscr{X}$ is a Banach space with Property $(H)$. 
	The $K$-theory of this $C^*$-algebra can be computed  when $Y$ is a bounded convex subset of a Banach space.

	\begin{lem}
		Let $Y$ be a metric space and let $\mathscr{X}$ be a Property $(H)$ Banach space.
		Then for all $f\in\mathcal{S}$
		and 	
		$\xi\in 
		C_{b}(Y,\mathscr{X})$, 
		we have 
		$\beta_{\xi}
		(f)
		\in 
		C_{b}
		\big(
		Y, \mathcal{A}(
		\mathscr{X})
		\big)$,
		where $\beta_{\xi}$ is given by \eqref{eq:functionalcalculustwo}.
	\end{lem}
	\begin{proof}
		For every
		$\xi\in 
		C_{b}(Y,\mathscr{X})$, 
		its image is bounded in $\mathscr{X}$  and $\xi$ is  uniformly continuous.
		Combining this with
		Proposition \ref{localuniformlycontinuous},
		we conclude that 
		$\beta_{\xi}
		(f)$ 
		is uniformly continuous on $Y$. 
	\end{proof}

	The following algebra is a natural generalization of that in Definition \ref{basicproperalgebra-geng}.

	\begin{defn}\label{proper-algebra-by-geng-most-key}
		For the function space 
		$C_{b}(Y,\mathscr{X})$, where $\mathscr{X}$  is a Banach space with Property $(H)$,  
		we define 
		$\mathcal{A}
		\big(
		C_{b}(Y,\mathscr{X})
		\big)$
		to be  the graded $C^*$-subalgebra of
		$C_{b}
		\big(
		Y, \mathcal{A}(\mathscr{X})
		\big)$
		generated by the set
		$$\Big\{
		\beta_{\xi}
		(f) ~\Big|~ f\in\mathcal{S}, ~ 
		\xi\in 
		C_{b}(Y,\mathscr{X}) 
		\Big\}.$$
		Moreover, for $R>0$,
		we denote by
		$\mathcal{A}_{R}
		\big(
		C_{b}(Y,\mathscr{X})
		\big)$
		the graded
		$*$-subalgebra of
		$\mathcal{A}
		\big(
		C_{b}(Y,\mathscr{X})
		\big)$
		consisting of all  $f$ 
		such that 
		\begin{equation*}
			{\rm Supp} \big(f(y)\big) \subseteq 
			B_{\mathbb{R}\times\mathscr{X}}(0,R)
			\quad  \text{for all  } 
			y\in Y.
		\end{equation*}
	\end{defn}

	By definition, 
	\begin{equation*}
		\mathcal{A}
		\big(
		C_{b}(Y,\mathscr{X})
		\big)=
		\overline{\bigcup_{R>0} \mathcal{A}_{R}
			\big(
			C_{b}(Y,\mathscr{X})
			\big)}.
	\end{equation*}

	\begin{rmk}\label{generalcaseforembeddingalgebra}
		The algebra 
		$\mathcal{A}
		(\mathscr{X})$
		is identified with the subalgebra of constant functions in 
		$C_{b}
		\big(
		Y, \mathcal{A}(\mathscr{X})
		\big)$.
		Moreover, every constant function in
		$C_{b}
		\big(
		Y, \mathcal{A}(\mathscr{X})
		\big)$
		belongs to 
		$\mathcal{A}
		\big(
		C_{b}(Y,\mathscr{X})
		\big)$.
		Hence, we obtain a canonical embedding
		$$\iota:
		\mathcal{A}
		(\mathscr{X}) 
		\longrightarrow
		\mathcal{A}
		\big(
		C_{b}(Y,\mathscr{X})
		\big).$$
	\end{rmk}

	The algebra introduced in Definition \ref{proper-algebra-by-geng-most-key} 
	plays a central role in the sequel. 
	It provides building blocks for the proper
	$\Gamma$-algebra constructed  in the subsequent sections,  and its $K$-theory can be computed in the case where $Y$ is contractible.
	In the following lemma, we compute the $K$-theory of 
	$\mathcal{A}
	\big(
	C_{b}(Y,\mathscr{X})
	\big)$
	when $Y$ is a bounded convex subset of a Banach space, equipped with the induced metric. 
	In later sections, it suffices to restrict attention to the specific convex subspace
	$Z$ of $\ell^1(\Gamma)$ defined in \eqref{eq:simplicityforthebasespaceZ}.

	\begin{lem}\label{isomorphismbetweenKtheory-babycase}
		Let $Y$ be a bounded convex subset of a Banach space endowed with the induced metric. 
		Then for any Property $(H)$ Banach space $\mathscr{X}$, 
		the embedding 
		\begin{equation*}
			\iota:
			\mathcal{A}
			(\mathscr{X}) 
			\longrightarrow
			\mathcal{A}
			\big(
			C_{b}(Y,\mathscr{X})
			\big)
		\end{equation*} 
		induces an isomorphism
		\begin{eqnarray*}
			\iota_*:
			K_*\big(
			\mathcal{A}
			(\mathscr{X}) 
			\big)
			&\longrightarrow&
			K_*\Big(
			\mathcal{A}
			\big(
			C_{b}(Y,\mathscr{X})
			\big)
			\Big).
		\end{eqnarray*} 
	\end{lem}
	\begin{proof}
		Fix  
		$y_{0}\in Y$.	
		Since 	
		$\mathcal{A}
		\big(
		C_{b}(Y,\mathscr{X})
		\big)\subseteq
		C_{b}
		\big(
		Y, \mathcal{A}(\mathscr{X})
		\big)$,
		the evaluation map 
		$$
		{\rm ev}_{y_0}: 
		C_{b}
		\big(
		Y, \mathcal{A}(\mathscr{X})
		\big) 
		\longrightarrow 
		\mathcal{A}(\mathscr{X}), \quad 
		f\mapsto
		f(y_{0})
		$$
		restricts to
		a $*$-homomorphism
		\begin{equation*}
			{\rm ev}_{y_0}: 
			\mathcal{A}
			\big(
			C_{b}(Y,\mathscr{X})
			\big)
			\longrightarrow
			\mathcal{A}(\mathscr{X}).  
		\end{equation*}
		
		We claim that
		${\rm ev}_{y_0}$ is a homotopy inverse of $\iota$, i.e., 
		each of ${\rm ev}_{y_0}\circ\iota$
		and 
		$\iota\circ{\rm ev}_{y_0}$
		is homotopic to the identity. 
		The first one is immediate:
		${\rm ev}_{y_0}\circ\iota = {\rm id}$.
		It remains to construct a homotopy between $\iota\circ{\rm ev}_{y_0}$ and
		${\rm id}$.

		Define a map
		\begin{equation}
			\label{eq:Homotopymap}
			\varUpsilon: 	
			C_{b}
			\big(
			Y, \mathcal{A}(\mathscr{X})
			\big)
			\longrightarrow  
			C\Big(
			[0,1],~	
			C_{b}
			\big(
			Y, \mathcal{A}(\mathscr{X})
			\big)
			\Big)
		\end{equation}
		by
		\begin{equation*}
			\big(
			\varUpsilon(f)(t)
			\big)(y)= 
			f\big(
			t \cdot {y}_{0} + (1-t) \cdot {y}
			\big) 
		\end{equation*}
		for all
		$f\in
		C_{b}
		\big(
		Y, \mathcal{A}(\mathscr{X})
		\big)$,
		$t \in [0,1]$ and 
		$y \in {Y}$.
		We first check that it is well-defined.
		Since $f$ is bounded and uniformly continuous,
		it follows that 
		$\varUpsilon(f)(t)\in
		C_{b}
		\big(
		Y, \mathcal{A}(\mathscr{X})
		\big)$
		for all 
		$t \in [0,1]$.
		Moreover, for all $s,t \in [0,1]$ and $y \in Y$, 
		\begin{eqnarray*}
			\Big\|
			\big( 
			s \cdot {y}_{0}+(1-s) \cdot {y}
			\big)-
			\big(
			t \cdot {y}_{0}+(1-t) \cdot y
			\big)
			\Big\|
			&=&
			\Big\|
			(s-t) \cdot {y}_{0} + (t-s) \cdot {y}
			\Big\|
			\\
			&\leq&
			2L \cdot |s-t|,
		\end{eqnarray*}
		where 
		$L=
		\sup\limits_{y\in Y} 
		\|y\|$. 
		Hence, for all $s,t \in [0,1]$,
		\begin{eqnarray*}
			\Big\|
			\varUpsilon(f)(s)-
			\varUpsilon(f)(t)
			\Big\|
			&=& 
			\sup\limits_{y \in {Y}}
			\bigg\|
			\Big(
			\varUpsilon(f)(s)-
			\varUpsilon(f)(t)
			\Big)(y)
			\bigg\|
			\\
			&=&
			\sup\limits_{y \in {Y}}
			\bigg\|
			f\Big(
			s \cdot {y}_{0} + (1-s) \cdot {y}
			\Big)-
			f\Big(
			t \cdot {y}_{0} + (1-t) \cdot {y}
			\Big) 
			\bigg\|
			\\
			&\leq&
			\omega_{f}\big(
			2L \cdot |s-t|
			\big).
		\end{eqnarray*}
		Together with the uniform continuity of $f$, this implies that
		$t\mapsto \varUpsilon(f)(t)$ is continuous on $[0,1]$, and hence $\varUpsilon$ 
		is well-defined.

		By definition,
		$\varUpsilon$ 
		is a $*$-homomorphism.
		Moreover, when $\varUpsilon$ is restricted to
		$\mathcal{A}
		\big(
		C_{b}(Y,\mathscr{X})
		\big)$,
		the resulting $*$-homomorphism takes values in
		$C
		\Big(
		[0,1],~	
		\mathcal{A}
		\big(
		C_{b}(Y,\mathscr{X})
		\big)
		\Big)$.
		It suffices to verify this on generators. 
		For 
		$\beta_{\xi}
		(f)$, 
		where 
		$f\in\mathcal{S}$ and 
		$\xi\in 
		C_{b}(Y,\mathscr{X})$,
		it follows from  \eqref{eq:functionalcalculustwo}  and the definition of $\varUpsilon$  that
		$$
		\varUpsilon
		\big(
		\beta_{\xi}
		(f)
		\big)(t)=
		\beta_{\xi_t}
		(f), \quad 
		t\in[0,1],
		$$ 
		where 
		$\xi_t\in
		C_{b}(Y,\mathscr{X})$
		is defined by 
		$$\xi_t(y) = \xi \big( 
		t \cdot {y}_{0} + (1-t) \cdot {y} \big), \quad 
		y \in Y.
		$$
		Thus, for each $t\in[0,1]$, 
		$$
		\varUpsilon
		\big(
		\beta_{\xi}
		(f)
		\big)(t)\in
		\mathcal{A}
		\big(
		C_{b}(Y,\mathscr{X})
		\big)
		$$
		and therefore
		$$\varUpsilon
		\big(
		\beta_{\xi}
		(f)
		\big)\in 
		C\Big(
		[0,1],~	
		\mathcal{A}
		\big(
		C_{b}(Y,\mathscr{X})
		\big)
		\Big).$$
		We continue to denote this induced map by $\varUpsilon$.

		It is clear that 
		\begin{center}
			${\rm ev}_{0}\circ \varUpsilon={\rm id}$ \quad 
			and \quad 
			${\rm ev}_{1}\circ \varUpsilon=\iota\circ{\rm ev}_{y_0}$,
		\end{center}
		where
		\begin{eqnarray*}
			{\rm ev}_{t}: 
			C\Big(
			[0,1],~	
			\mathcal{A}
			\big(
			C_{b}(Y,\mathscr{X})
			\big)
			\Big)
			\longrightarrow 
			\mathcal{A}
			\big(
			C_{b}(Y,\mathscr{X})
			\big)
		\end{eqnarray*}
		denotes the evaluation map at $t\in[0,1]$.
		This completes the proof.
	\end{proof}

	\begin{rmk}
		Note that the proof of the above lemma also yields an isomorphism
		\begin{equation*}
			K_*\big(
			\mathcal{A}
			(\mathscr{X}) 
			\big)
			\cong
			K_*\Big(
			C_{b}
			\big(
			Y, \mathcal{A}(\mathscr{X})
			\big)
			\Big).
		\end{equation*} 
	\end{rmk}

	\section{Proper $\Gamma$-$C^*$-algebras}

	In this section, we shall construct the proper $\Gamma$-$C^*$-algebra for a countable discrete group $\Gamma$ that 
	admits a
	coarse embedding  with finite complexity
	into an
	$\ell^2$-direct sum of infinitely many Property 
	$(H)$ 
	Banach spaces. 
	The construction of the $\Gamma$-$C^*$-algebra proceeds as follows.
	We decompose the infinite direct sum into an increasing filtration by finite partial sums. 
	Each term in this filtration is a finite direct sum of Property 
	$(H)$ Banach spaces and hence also has Property $(H)$.
	For a single Banach space with Property 
	$(H)$, 
	an associated $\Gamma$-$C^*$-algebra will be constructed in Section  5.1. 
	We then take a suitable limit of these algebras to obtain the desired proper  $\Gamma$-$C^*$-algebra in the general setting. 
	This will be carried out in Section  5.2.
The second condition in the definition of coarse embedding  with finite complexity  (Definition \ref{mostkeyidea1})  
will play a crucial role in establishing  the properness of the $\Gamma$-action on the $C^*$-algebra.

	\subsection{Case for single Property $(H)$ Banach space}

	Suppose that
	$h: \Gamma \to \mathscr{X}$ 
	is bornologous (Definition \ref{Bornologousdef}), where  $\mathscr{X}$ is a Property $(H)$ Banach space with a Property 
	$(H)$
	map 
	$\phi: S(\mathscr{X})\to S(\mathscr{H})$.

	By Proposition \ref{PropertyHmapforLpspaces}, 
	$\mathscr{X}_{\Gamma}=\ell^2(\Gamma, \mathscr{X})$
	is a Property $(H)$  Banach space with respect to the induced 
	Property $(H)$ map  
	$\mathcal{M}_{\phi}:
	S(\mathscr{X}_{\Gamma})\to 
	S(\mathscr{H}_{\Gamma})$.
	Let 
	$$
	\mathcal{M}_{\phi}^{\eta}:
	\mathscr{X}_{\Gamma}
	\longrightarrow 
	\mathscr{H}_{\Gamma}
	$$ 
	be an  $\eta$-scalar proper extension of $\mathcal{M}_{\phi}$
	given by Corollary \ref{finitemodulusofcontinuity}.
	We emphasize that the choice of the $\eta$-scalar proper extension, rather than an arbitrary proper extension,
	ensures the commutativity of the diagram \eqref{eq:commutativeofriowithcliffalgebrariojdklk} below.

	\begin{defn}\label{Cliffordmap}
		We define
		$\mathfrak{B}: \mathscr{X}_{\Gamma} \rightarrow
		\text{Cliff}_{\mathbb{C}}
		(\mathscr{H}_{\Gamma})$
		by 
		$v\mapsto C\big(\mathcal{M}_{\phi}^{\eta}(v)\big)$. 
		Moreover,
		for any  
		$v_{0}\in \mathscr{X}_{\Gamma}$,  
		we define $\mathfrak{B}^{v_{0}}: \mathscr{X}_{\Gamma} \rightarrow
		\text{Cliff}_{\mathbb{C}}
		(\mathscr{H}_{\Gamma})$ 
		by
		$\mathfrak{B}^{v_{0}}(v) = \mathfrak{B}(v-v_{0})$.
	\end{defn}		
	The same notation as in Definition \ref{Cliffordmap1} is used,
	and this causes no ambiguity.

	By \eqref{eq:functionalcalculusone} and \eqref{eq:functionalcalculustwo},
	we obtain the graded $*$-homomorphisms 
	\begin{equation}
		\label{eq:secondfunctionalcalculusforlsdkjlfa}
		\begin{array}{rcl}
			\beta_{v}: ~ \mathcal{S} &\longrightarrow&
			\mathcal{S}\widehat{\otimes} {C}_{0}
			\big(\mathscr{X}_{\Gamma}, 
			\text{Cliff}_{\mathbb{C}}
			(\mathscr{H}_{\Gamma})
			\big)
			\\
			f &\longmapsto & 
			f\big(
			\varTheta \widehat{\otimes} 1+1\widehat{\otimes} 		\mathfrak{B}^{v}
			\big),
		\end{array}
	\end{equation}
	and
	\begin{equation}
		\label{eq:secondfunctionalcalculusforsections}
		\begin{array}{rcl}
			\beta_{\xi}: ~ \mathcal{S}&\longrightarrow &
			\mathcal{F}
			\big(
			Z, ~
			\mathcal{A}(\mathscr{X}_{\Gamma})
			\big)
			\\
			f &\longmapsto & 
			\Big\{
			y\mapsto 
			\beta_{\xi(y)}(f)
			\Big\},
		\end{array}
	\end{equation}
	where  
	$v\in\mathscr{X}_{\Gamma}$,
	$\xi\in 
	C_{b}(Z, \mathscr{X}_{\Gamma})$,
	and 
	$Z$ is the metric space defined in \eqref{eq:simplicityforthebasespaceZ}.

	Following Section 4.4,
	we have the graded $C^*$-algebra 
	$$
	\mathcal{A}
	\big(
	C_{b}(Z, \mathscr{X}_{\Gamma})
	\big).
	$$
	We now show that the $\Gamma$-action on  
	$C_{b}(Z, \mathscr{X}_{\Gamma})$,  
	constructed in Section 3,
	induces a 
	$\Gamma$-action on 	
	$$
	\mathcal{A}
	\big(
	C_{b}(Z, \mathscr{X}_{\Gamma})
	\big).
	$$
	By the definition of the isometric isomorphism $\lambda^{\scalebox{0.6}{$\gamma$}}$,
	the Property $(H)$ map  
	$\mathcal{M}_{\phi}$, and the
	$\eta$-scalar extension,  
	there exists a 
	natural 
	$*$-isomorphism
	$\lambda^{\scalebox{0.6}{$\gamma$}}_*$
	such that the following diagram commutes:
	\begin{equation}
		\label{eq:commutativeofriowithcliffalgebrariojdklk}
		\begin{tikzcd}[row sep=1cm, column sep=1cm]
			\mathscr{X}_{\Gamma}
			\arrow[r,"\lambda^{\scalebox{0.6}{$\gamma$}}"] \arrow[d,"\mathfrak{B}"] &
			\mathscr{X}_{\Gamma}
			\arrow[d,"\mathfrak{B}"] 
			\\
			\text{Cliff}_{\mathbb{C}}
			(\mathscr{H}_{\Gamma}) \arrow[r,"\lambda^{\scalebox{0.6}{$\gamma$}}_*"]
			& \text{Cliff}_{\mathbb{C}}
			(\mathscr{H}_{\Gamma}).
		\end{tikzcd}
	\end{equation}
	For each $\gamma\in\Gamma$  
	and each
	$z \in {Z}$, 
	define a map
	\begin{equation*}
		\mathcal{U}^{\scalebox{0.6}{$(\gamma,z)$}}:
		\mathcal{S}\widehat{\otimes} {C}_{0}
		\big(\mathscr{X}_{\Gamma}, 
		\text{Cliff}_{\mathbb{C}}
		(\mathscr{H}_{\Gamma})
		\big)
		\longrightarrow \mathcal{S}\widehat{\otimes} {C}_{0}
		\big(\mathscr{X}_{\Gamma}, 
		\text{Cliff}_{\mathbb{C}}
		(\mathscr{H}_{\Gamma})
		\big)
	\end{equation*}
	by 
	\begin{equation}
		\label{eq:isometricsforalgebras}
		\big(
		\mathcal{U}^{\scalebox{0.6}{$(\gamma,z)$}}(f)
		\big)(\theta,v)
		=
		(id\widehat{\otimes}\lambda^{\scalebox{0.6}{$\gamma$}}_*)
		\bigg(
		f\Big(
		\theta,		U^{\scalebox{0.5}{$(\gamma^{-1},\gamma{z})$}}(v)
		\Big)	
		\bigg)
	\end{equation}
	for
	$f\in\mathcal{S}\widehat{\otimes} {C}_{0}
	\big(\mathscr{X}_{\Gamma}, 
	\text{Cliff}_{\mathbb{C}}
	(\mathscr{H}_{\Gamma})
	\big)$ 
	and  
	$(\theta,v)\in \mathbb{R} \times \mathscr{X}_{\Gamma}$.
	It is a $*$-isomorphism between 
	$C^*$-algebras.

	The following result follows immediately from Lemma \ref{affineisometryproperty1}.
	\begin{lem}\label{affineisometryproperty22}
		The following statements hold: 
		\begin{itemize}
			\item [(1)]	$\mathcal{U}^{^{\scalebox{0.6}{$(e,z)$}}}$
			is the 
			identity map 
			for all 
			$z\in{Z}$;
			\item [(2)]	$\mathcal{U}^{\scalebox{0.6}{$(\gamma, z^\prime)$}} \mathcal{U}^{\scalebox{0.6}{$(\gamma^\prime,z)$}}=\mathcal{U}^{\scalebox{0.6}{$(\gamma\gamma^\prime,z)$}}$
			for all 
			$\gamma, \gamma^\prime\in\Gamma$ and $z,z^\prime\in {Z}$ 
			such that
			$z^\prime=\gamma^\prime z$. 
		\end{itemize}
	\end{lem}

	\begin{lem}
		$
		\mathcal{U}^{\scalebox{0.6}{$(\gamma,z)$}}\big(
		\mathcal{A}(\mathscr{X}_{\Gamma})
		\big)
		\subseteq
		\mathcal{A}(\mathscr{X}_{\Gamma})
		$
		for all
		$\gamma\in\Gamma$
		and
		$z\in {Z}$.      
	\end{lem}
	\begin{proof}
		For simplicity, we write 
		$\varPhi=\mathcal{M}_{\phi}^{\eta}$.
		To complete the proof, it suffices to show that
		\begin{equation}
			\label{eq:rulesforisometricforalgebras}
			\mathcal{U}^{\scalebox{0.6}{$(\gamma,z)$}}
			\big(	
			\beta_{v_{0}}(f)
			\big)
			=
			\beta_{U^{\scalebox{0.5}{$(\gamma,z)$}}				(v_{0})}(f)
		\end{equation}
		for all 
		$\gamma\in\Gamma$,
		$z\in {Z}$, 
		$v_{0} \in \mathscr{X}_{\Gamma}$ and 
		$f\in \mathcal{S}$,
		where $\beta$ is given by \eqref{eq:secondfunctionalcalculusforlsdkjlfa}.

		We first 
		consider the case 
		where 
		$f(\theta)=e^{-\theta^{2}}$.
		Then for all
		$\gamma\in\Gamma$, $z\in {Z}$ and $v_{0} \in \mathscr{X}_{\Gamma}$,
		\begin{eqnarray*}
			\Big(\mathcal{U}^{\scalebox{0.6}{$(\gamma,z)$}}\big(
			\beta_{v_0}
			(f)\big)
			\Big)(\theta,v) &=& 	(id\widehat{\otimes}\lambda^{\scalebox{0.6}{$\gamma$}}_*)
			\bigg(
			\beta_{v_0}(f)
			\Big(
			\theta,		U^{\scalebox{0.5}{$(\gamma^{-1},\gamma{z})$}}(v)
			\Big)	
			\bigg)
			\\
			&=&
			e^{-\theta^{2}}
			\widehat{\otimes}~ 
			e^{
				-\Big\|
				\varPhi 
				\big( \lambda^{\scalebox{0.6}{$\gamma^{-1}$}}(v)+\zeta^{\scalebox{0.6}{$(\gamma^{-1},\gamma{z})$}}-v_{0}
				\big)
				\Big\|^2
			}
			\\
			&=&
			e^{-\theta^{2}}
			\widehat{\otimes}~ 
			e^{
				-\bigg\|
				\varPhi 
				\Big( \lambda^{\scalebox{0.6}{$\gamma^{-1}$}}
				\big(v+\lambda^{\scalebox{0.6}{$\gamma$}}
				(\zeta^{\scalebox{0.6}{$(\gamma^{-1},\gamma{z})$}}-v_{0})\big)
				\Big)
				\bigg\|^2
			}
			\\
			&=&
			e^{-\theta^{2}}
			\widehat{\otimes}~ 
			e^{
				-\bigg\|
				\lambda^{\scalebox{0.6}{$\gamma^{-1}$}}
				\Big( \varPhi  
				\big(v+\lambda^{\scalebox{0.6}{$\gamma$}}
				(\zeta^{\scalebox{0.6}{$(\gamma^{-1},\gamma{z})$}}-v_{0})\big)
				\Big)
				\bigg\|^2
			}
			\\
			&=&
			e^{-\theta^{2}}
			\widehat{\otimes}~ 
			e^{
				-\Big\|
				\varPhi  \big(v+\lambda^{\scalebox{0.6}{$\gamma$}}
				(\zeta^{\scalebox{0.6}{$(\gamma^{-1},\gamma{z})$}}-v_{0})\big)
				\Big\|^{2}
			}
			\\
			&=& 
			\big(\beta_{w_0}(f)\big)(\theta,v),
		\end{eqnarray*}
		where $w_{0}=\lambda^{\scalebox{0.7}{$\gamma$}}\big(v_{0}-\zeta^{\scalebox{0.6}{$(\gamma^{-1},\gamma{z})$}}\big)=U^{\scalebox{0.6}{$(\gamma,z)$}}
		(v_{0})$.
		Next, for the function 
		$f(\theta)=\theta\cdot{e}^{-\theta^{2}}$, we have
		\begin{eqnarray*}
			&&\Big(
			\mathcal{U}^{\scalebox{0.6}{$(\gamma,z)$}}
			\big(
			\beta_{v_0}(f)
			\big)
			\Big)(\theta,v) 
			\\
			&=& 	(id\widehat{\otimes}\lambda^{\scalebox{0.7}{$\gamma$}}_*)
			\bigg(
			\beta_{v_0}(f)
			\Big(
			\theta,		U^{\scalebox{0.6}{$(\gamma^{-1},\gamma{z})$}}(v)
			\Big)	
			\bigg)
			\\
			&=&
			(id\widehat{\otimes}\lambda^{\scalebox{0.7}{$\gamma$}}_*)
			\Bigg(\theta\cdot
			e^{-\theta^{2}}
			\widehat{\otimes}~ 
			e^{
				-\Big\|
				\varPhi 
				\big( U^{\scalebox{0.6}{$(\gamma^{-1},\gamma{z})$}}(v)-v_{0} 
				\big)
				\Big\|^2
			}
			+
			e^{-\theta^{2}}\widehat{\otimes}~ 
			e^{
				-\Big\|
				\varPhi 
				\big( U^{\scalebox{0.6}{$(\gamma^{-1},\gamma{z})$}}(v)-v_{0} 
				\big)
				\Big\|^2
			} \cdot 
			\mathfrak{B}
			\Big(
			U^{\scalebox{0.6}{$(\gamma^{-1},\gamma{z})$}}(v)-v_{0}
			\Big)
			\Bigg)
			\\
			&=&
			\theta\cdot	e^{-\theta^{2}}
			\widehat{\otimes}~ 
			e^{
				-\Big\|
				\varPhi 
				\big( U^{\scalebox{0.6}{$(\gamma^{-1},\gamma{z})$}}(v)-v_{0} 
				\big)
				\Big\|^2
			}+
			e^{-\theta^{2}}\widehat{\otimes}~ 
			e^{
				-\Big\|
				\varPhi 
				\big( U^{\scalebox{0.6}{$(\gamma^{-1},\gamma{z})$}}(v)-v_{0} 
				\big)
				\Big\|^2
			} \cdot 
			\mathfrak{B}
			\Big(
			\lambda^{\gamma}
			\big(
			U^{\scalebox{0.6}{$(\gamma^{-1},\gamma{z})$}}(v)-v_{0}
			\big)
			\Big)
			\\
			&=&
			\theta\cdot	e^{-\theta^{2}}\widehat{\otimes} 
			e^{
				-\Big\|
				\varPhi  \big(v+\lambda^{\gamma}(\zeta^{\scalebox{0.6}{$(\gamma^{-1},\gamma{z})$}}-v_{0}
				)\big)
				\Big\|^2
			}
			+
			e^{-\theta^{2}}
			\widehat{\otimes} 
			e^{
				-\Big\|
				\varPhi  \big(v+\lambda^{\gamma}(\zeta^{\scalebox{0.6}{$(\gamma^{-1},\gamma{z})$}}-v_{0}
				)\big)
				\Big\|^2
			} \cdot	
			\mathfrak{B}
			\Big(
			v+\lambda^{\gamma}\big(
			\zeta^{\scalebox{0.6}{$(\gamma^{-1},\gamma{z})$}}-v_{0}
			\big)
			\Big)
			\\
			&=& 
			\big(\beta_{w_{0}}(f)\big) 
			(\theta,v),
		\end{eqnarray*}
		where $w_{0}=\lambda^{\gamma}\big(v_{0}-\zeta^{\scalebox{0.6}{$(\gamma^{-1},\gamma{z})$}}\big)=U^{\scalebox{0.6}{$(\gamma,z)$}}
		(v_{0})$.

		Since the $C^*$-algebra $\mathcal{S}$ is generated by  $e^{-\theta^{2}}$ and $\theta\cdot{e}^{-\theta^{2}}$,
		\eqref{eq:rulesforisometricforalgebras} holds for all $f\in \mathcal{S}$.
	\end{proof}

	The above lemma allows us to restrict the $*$-isomorphisms 
	$\mathcal{U}^{\scalebox{0.6}{$(\gamma,z)$}}$  to  $\mathcal{A}(\mathscr{X}_{\Gamma})$, 
	and we obtain a family of
	$*$-isomorphisms 
	$$
	\Big\{	\mathcal{U}^{\scalebox{0.6}{$(\gamma,z)$}}: \mathcal{A}(\mathscr{X}_{\Gamma})
	\longrightarrow
	\mathcal{A}(\mathscr{X}_{\Gamma})
	\Big\}_ 
	{\gamma\in\Gamma, z\in {Z}}.
	$$
	For each 
	$\gamma\in\Gamma$ 
	and each
	$f\in 
	\mathcal{A}
	\big(
	C_{b}(Z, \mathscr{X}_{\Gamma})
	\big)$,
	we define 	
	$$
	\gamma\cdot {f}: Z \longrightarrow \mathcal{A}(\mathscr{X}_{\Gamma})
	$$
	by
	\begin{equation}
		\label{eq:actionsinducedbyisometrics}
		(\gamma\cdot {f})(z)=
		\mathcal{U}^{\scalebox{0.6}{$(\gamma,\gamma^{-1}z)$}}
		\big(
		f(\gamma^{-1}z)
		\big) 
	\end{equation}
	for 
	$z\in{Z}$.

	\begin{lem} \label{zenmenengmeiyoulakdjlkfanedkfd}
		$
		\gamma\cdot f
		\in 
		\mathcal{A}
		\big(
		C_{b}(Z, \mathscr{X}_{\Gamma})
		\big)
		$
		for all 
		$\gamma\in\Gamma$ 
		and 
		$f\in 
		\mathcal{A}
		\big(
		C_{b}(Z, \mathscr{X}_{\Gamma})
		\big)$.
	\end{lem}
	\begin{proof}
		It suffices to verify that
		$$
		\gamma \cdot \beta_{\xi}
		(f)\in 
		\mathcal{A}
		\big(
		C_{b}(Z, \mathscr{X}_{\Gamma})
		\big)
		$$
		for all $\gamma\in\Gamma$,
		$f\in\mathcal{S}$ 
		and 
		$\xi\in 
		C_{b}(Z, \mathscr{X}_{\Gamma})$.
		Using  \eqref{eq:actionsinducedbyisometrics},
		\eqref{eq:secondfunctionalcalculusforsections}, \eqref{eq:rulesforisometricforalgebras} and \eqref{eq:theactionsofgammaonsections},
		we compute that for every $z\in Z$,
		\begin{eqnarray*}
			\big( 
			\gamma \cdot \beta_{\xi}
			(f)
			\big)  
			(z) &=& 
			\mathcal{U}^{\scalebox{0.6}{$(\gamma,\gamma^{-1}z)$}}
			\Big(
			\beta_{\xi}
			(f)
			(\gamma^{-1}z)
			\Big)
			\\
			&=& 
			\mathcal{U}^{\scalebox{0.6}{$(\gamma,\gamma^{-1}z)$}}
			\Big(
			\beta_{\xi(\gamma^{-1}z)}
			(f)
			\Big)
			\\
			&=& 
			\beta_{
				U^{\scalebox{0.6}{$(\gamma,\gamma^{-1}z)$}}
				\big(\xi(\gamma^{-1}z)\big)
			}
			(f)  
			\\
			&=&
			\beta_{(\gamma \cdot \xi)(z)}
			(f)
			\\
			&=&
			\big( 
			\beta_{\gamma \cdot \xi}
			(f)
			\big)(z). 
		\end{eqnarray*}
		Hence, for all $\gamma\in\Gamma$,
		$f\in\mathcal{S}$ 
		and 
		$\xi\in 
		C_{b}(Z, \mathscr{X}_{\Gamma})$,
		we have
		\begin{equation}
			\label{eq:rulesforgammaactionforelementary}
			\gamma \cdot \beta_{\xi}
			(f)
			= 
			\beta_{\gamma \cdot \xi}
			(f).
		\end{equation}
		Combining this with the fact that
		$\gamma\cdot\xi
		\in 
		C_{b}(Z, \mathscr{X}_{\Gamma})$,
		we obtain 
		$\gamma \cdot \beta_{\xi}
		(f)\in 
		\mathcal{A}
		\big(
		C_{b}(Z, \mathscr{X}_{\Gamma})
		\big)$.
	\end{proof}

	By Lemma \ref{affineisometryproperty22} and Lemma \ref{zenmenengmeiyoulakdjlkfanedkfd}, 
	\eqref{eq:actionsinducedbyisometrics} gives rise to a $\Gamma$-action on
	$\mathcal{A}
	\big(
	C_{b}(Z, \mathscr{X}_{\Gamma})
	\big)$.
	We conclude this section by showing that  
	$\mathcal{A}
	\big(
	C_{b}(Z, \mathscr{X}_{\Gamma})
	\big)$
	is a  
	proper $\Gamma$-$C^*$-algebra
	whenever 
	$h$
	is a coarse embedding.
	It should be noted that our construction of proper algebras is more concise than the one in \cite{KY}.

	\begin{lem}
		[{\cite[Lemma 6.6]{GWY}\label{propermethod}}] Let $W$ be a locally compact Hausdorff space and let $\Gamma$ be a countable discrete group. Let $\Gamma\curvearrowright_{\alpha} W$ be an action by homeomorphisms and let $\alpha_*$ be the induced action on $C_{0}(W)$. Then the action $\alpha$ is (topologically) proper
		if and only if for any $f\in C_{0}(W)$,
		$$
		\lim\limits_{\gamma\rightarrow\infty}
		\Big\|
		\big(\alpha_*(\gamma)f\big)\cdot f\Big\|=0.
		$$
	\end{lem}

	Recall from \cite{GHT,K1988} 
	that a $\Gamma$-$C^*$-algebra 
	$A$ 
	is \emph{proper} 
	if there is
	a locally compact, 
	proper 
	$\Gamma$-space $X$ 
	and an equivariant
	$*$-homomorphism from 
	$C_0(X)$ 
	into the center of the multiplier algebra of $A$
	such that 
	$C_0(X)\cdot A$
	is dense in $A$.

	\begin{prop}\label{babycase}
		Suppose that
		$h:\Gamma\rightarrow\mathscr{X}$
		is a coarse embedding,  where  $\mathscr{X}$ is a Property $(H)$ Banach space. 
		Then
		$\mathcal{A}
		\big(
		C_{b}(Z, \mathscr{X}_{\Gamma})
		\big)$
		is a proper
		$\Gamma$-$C^*$-algebra.
	\end{prop}
	\begin{proof}
		As in Definition \ref{proper-algebra-by-geng-most-key}, 
		let
		$\mathcal{Z}
		\big(
		C_{b}(Z,\mathscr{X}_{\Gamma})
		\big)$
		be  the graded $C^*$-subalgebra of
		$\mathcal{A}
		\big(
		C_{b}(Z,\mathscr{X}_{\Gamma})
		\big)$
		generated by the set
		\begin{eqnarray*}
			\Big\{
			\beta_{\xi}
			(f) ~\Big|~ f\in\mathcal{S}_{\rm even}, ~ 
			\xi\in 
			C_{b}(Z, \mathscr{X}_{\Gamma}) 
			\Big\},
		\end{eqnarray*}
		where $\beta_{\xi}$ is given by \eqref{eq:secondfunctionalcalculusforsections}.	
		It is a $\Gamma$-invariant subalgebra
		contained in the center of
		$\mathcal{A}
		\big(
		C_{b}(Z, \mathscr{X}_{\Gamma})
		\big)$.
		Moreover,
		$\mathcal{Z}
		\big(
		C_{b}(Z, \mathscr{X}_{\Gamma})
		\big)\cdot 
		\mathcal{A}
		\big(
		C_{b}(Z, \mathscr{X}_{\Gamma})
		\big)$ 
		is dense in 
		$\mathcal{A}
		\big(
		C_{b}(Z, \mathscr{X}_{\Gamma})
		\big)$.
		The Gelfand spectrum
		$$\Delta:={\rm Spec}
		\Big(
		\mathcal{Z}
		\big(
		C_{b}(Z, \mathscr{X}_{\Gamma})
		\big)\Big)
		$$ 
		is a locally compact Hausdorff  space equipped with the $\Gamma$-action induced from that on 
		$\mathcal{Z}
		\big(
		C_{b}(Z, \mathscr{X}_{\Gamma})
		\big)$. 
		The Gelfand transform 
		yields a canonical $*$-isomorphism
		$$
		\mathcal{Z}
		\big(
		C_{b}(Z, \mathscr{X}_{\Gamma})
		\big)
		\cong 
		C_{0}(\Delta),
		$$
		which is equivariant with respect to the induced $\Gamma$-action.
		Hence, by Lemma \ref{propermethod}, 
		to obtain the properness of the $\Gamma$-action on 
		$\mathcal{A}
		\big(
		C_{b}(Z, \mathscr{X}_{\Gamma})
		\big)$,
		it suffices to verify that 
		$$
		\lim\limits_{\gamma\rightarrow\infty}
		\big\|
		(\gamma\cdot{f}){f}
		\big\|=0
		$$	
		for every
		$f\in 
		\mathcal{Z}
		\big(
		C_{b}(Z, \mathscr{X}_{\Gamma})
		\big)$.

		For each $R>0$,
		we set
		$$
		\mathcal{Z}_{R}
		\big(
		C_{b}(Z, \mathscr{X}_{\Gamma})
		\big) = \mathcal{Z}
		\big(
		C_{b}(Z, \mathscr{X}_{\Gamma})
		\big) \cap 
		\mathcal{A}_{R}
		\big(
		C_{b}(Z, \mathscr{X}_{\Gamma})
		\big).
		$$
		For every 
		$f\in 
		\mathcal{Z}_{R}
		\big(
		C_{b}(Z, \mathscr{X}_{\Gamma})
		\big)$,
		it follows from \eqref{eq:actionsinducedbyisometrics}, \eqref{eq:isometricsforalgebras} and \eqref{eq:affineisometricesgeneralone} that
		\begin{equation*}
			\label{eq:thegammaactionsforelementsinalgebrasthekthlevel}
			\begin{array}{rcl}
				\big(
				(\gamma\cdot {f})(z)
				\big)(\theta,v)
				&=&
				\Big( \mathcal{U}^{\scalebox{0.6}{$(\gamma,\gamma^{-1}z)$}}
				\big(
				f(\gamma^{-1}z)
				\big)
				\Big)
				(\theta,v)
				\\
				&=&
				(id\widehat{\otimes}\lambda^{\scalebox{0.6}{$\gamma$}}_{*})
				\bigg(
				f(\gamma^{-1}z) 
				\Big(
				\theta,		U^{\scalebox{0.6}{$(\gamma^{-1},z)$}}(v)
				\Big)	
				\bigg)
				\\
				&=&
				(id\widehat{\otimes}\lambda^{\scalebox{0.6}{$\gamma$}}_{*})
				\bigg(
				f(\gamma^{-1}z) 
				\Big(
				\theta,		\lambda^{\scalebox{0.6}{$\gamma^{-1}$}}(v)+\zeta^{\scalebox{0.6}{$(\gamma^{-1},z)$}}
				\Big)	
				\bigg)
			\end{array}
		\end{equation*}
		for all 
		$\gamma\in\Gamma$, 
		$z\in {Z}$ 
		and 
		$(\theta,v)\in \mathbb{R} \times \mathscr{X}_{\Gamma}$.
		Combining this with
		Proposition \ref{strongproperness}, 
		we obtain 
		$(\gamma\cdot{f}){f}=0$ for  all
		$\gamma$ large enough.
		Since
		\begin{eqnarray*}
			\mathcal{Z}
			\big(
			C_{b}(Z, \mathscr{X}_{\Gamma})
			\big)=
			\overline{\bigcup_{R>0} \mathcal{Z}_{R}
				\big(
				C_{b}(Z, \mathscr{X}_{\Gamma})
				\big)},
		\end{eqnarray*}
		the general case follows by a standard approximation argument.
	\end{proof}

	\subsection{Case for the direct sum of  infinitely many  Property 
		$(H)$
		Banach spaces}
	
Based on the results of  Section 5.1, 
we construct the proper 
	$\Gamma$-$C^*$-algebra for a countable discrete group that 
	admits a
	coarse embedding   with finite complexity 
	into an
	$\ell^2$-direct sum of infinitely many  Property 
	$(H)$
	Banach spaces. 
The second condition in the definition of coarse embedding  with finite complexity  (Definition \ref{mostkeyidea1})  
 will play a crucial role in establishing  the properness of the $\Gamma$-action on the $C^*$-algebra (see Proposition \ref{properalgebra}).

	Suppose that
	$h:\Gamma\to\mathscr{E}=
	\bigoplus_{p=1}^{\infty}
	\mathscr{X}_p$ 
	is a coarse embedding  with finite complexity, 
	where each
	$\mathscr{X}_p$
	is a Property 
	$(H)$
	Banach space with a Property $(H)$ map 
	$\phi_p: 
	S(\mathscr{X}_p) \to S(\mathscr{H})$.
	The third condition of Definition \ref{mostkeyidea1}  imposes no requirements on the first finitely many maps  $\phi_p$. 
	In other words,  all but finitely many of the maps $\phi_p$ are required to be Lipschitz.
	We may assume, without loss of generality,  
	that  
	$\phi_p$ is Lipschitz for all $p\geq 2$
	since  the direct sum of finitely many Property 
	$(H)$
	Banach spaces also has Property 
	$(H)$.

	Denote by 
	$\mathscr{E}_{k}$
	the 
	$\ell^2$-direct sum 
	$\bigoplus_{p=1}^{k}
	\mathscr{X}_p$
	and by 
	$\mathscr{H}_{k}$
	the corresponding  
	$\ell^2$-direct sum 
	$\bigoplus_{p=1}^{k}
	\mathscr{H}$.	
	Let $\mathfrak{h}_{k}: \Gamma \to \mathscr{E}_{k}$  be the projection of $h$ onto  $\mathscr{E}_{k}$,
	then
	$$
	\mathfrak{h}_{k}=\oplus_{p=1}^{k} h_{p}.
	$$
	As in Section 3, 
	we write
	\begin{center}
		$\mathscr{X}_{p,\Gamma}=\ell^2(\Gamma, \mathscr{X}_p)$, \   
		$\mathscr{H}_{\Gamma}=\ell^2(\Gamma, \mathscr{H})$, \    $\mathscr{E}_{k,\Gamma}=\ell^2(\Gamma, \mathscr{E}_k)$, \      
		$\mathscr{H}_{k,\Gamma}=\ell^2(\Gamma, \mathscr{H}_k)$.
	\end{center}
	Since each 
	$\mathfrak{h}_{k}: \Gamma \to \mathscr{E}_{k}$
	is  bornologous, 
	the construction in  Section 3 
	produces  
	a family of affine isometries on $\mathscr{E}_{k,\Gamma}$,
	indexed by
	$\gamma \in \Gamma$ and $z\in Z$.
	More precisely, 
	for every 
	$k \geq 1$,
	$\gamma\in\Gamma$ and 
	$z \in {Z}$,
	we have
	\begin{center}
		$\zeta_{k}^{\scalebox{0.6}{$(\gamma,z)$}} \in \mathscr{E}_{k,\Gamma}$, \quad 
		$\lambda_{k}^{\scalebox{0.6}{$\gamma$}}: \mathscr{E}_{k,\Gamma}\longrightarrow \mathscr{E}_{k,\Gamma}$, \quad 
		and \quad  
		$U_{k}^{\scalebox{0.6}{$(\gamma,z)$}}: 
		\mathscr{E}_{k,\Gamma}
		\longrightarrow 
		\mathscr{E}_{k,\Gamma}$,
	\end{center} 
	which are given
	by
	\begin{equation}
		\label{eq:cocyclesforthekthlevel}
		\zeta_{k}^{\scalebox{0.6}{$(\gamma,z)$}}(\tilde{\gamma})=	
		\sqrt{z_{\scalebox{0.6}{$\gamma^{-1}\tilde{\gamma}$}}}
		\cdot 
		\Big( \mathfrak{h}_{k}(\gamma^{-1}\tilde{\gamma})-
		\mathfrak{h}_{k}(\tilde{\gamma})
		\Big),
	\end{equation}
	\begin{equation}
		\label{eq:hahjkkkkdjoikjxkluiowejio}
		\big( \lambda_{k}^{\scalebox{0.6}{$\gamma$}}(f)
		\big)(\tilde{\gamma})=
		f(\gamma^{-1}\tilde{\gamma}),
	\end{equation}
	and
	\begin{equation}
		\label{eq:affineisometriesforthekthlevel}
		U_{k}^{\scalebox{0.6}{$(\gamma,z)$}}(f) = 	\lambda_{k}^{\scalebox{0.6}{$\gamma$}}(f)+\zeta_{k}^{\scalebox{0.6}{$(\gamma,z)$}}
	\end{equation}	
	for 
	$\tilde{\gamma}\in\Gamma$ and 
	$f\in\mathscr{E}_{k,\Gamma}$.
	Furthermore, for every 
	$k \geq 1$,
	the family of affine isometries 
	$$\Big\{ U_{k}^{\scalebox{0.6}{$(\gamma,z)$}} ~\Big|~ \gamma\in\Gamma, z\in  Z  
	\Big\}$$
	induces a $\Gamma$-action 
	on
	$C_{b}(Z, \mathscr{E}_{k,\Gamma})$
	via 
	\begin{equation}
		\label{eq:Gammaactionsforthekthlevel}
		\quad  (\gamma \cdot {f})
		(z)
		=	U_{k}^{\scalebox{0.6}{$(\gamma,\gamma^{-1}z)$}}
		\big(
		f(\gamma^{-1}z)
		\big), \quad 
		\gamma\in\Gamma,\ 
		f\in
		C_{b}(Z, \mathscr{E}_{k,\Gamma}),\ 
		z\in {Z}.
	\end{equation}

	Since each
	$\phi_p: 
	S(\mathscr{X}_p) \to S(\mathscr{H})$
	is a Property $(H)$ map,
	it follows from Proposition \ref{PropertyHmapforLpspaces}
	that  
	$$
	\mathcal{M}_{\phi_p}:
	S\big( 
	\mathscr{X}_{p,\Gamma}
	\big) \longrightarrow 
	S\big( 
	\mathscr{H}_{\Gamma}
	\big)
	$$ 	
	is a Property $(H)$ map for 
	$\mathscr{X}_{p,\Gamma}$.
	For $p\geq2$, 
	the map $\phi_p$ is Lipschitz.  
	Hence, by Proposition \ref{LipofextenedMazurmap},
	$\mathcal{M}_{\phi_p}$ 
	is
	Lipschitz with Lipschitz constant at most
	$2 \cdot  L_{\phi_p}+1$. 
	Moreover,  by Lemma \ref{3lip}, 
	the Lipschitz constant of its scalar extension $\mathcal{M}_{\phi_p}^{I}$   
	is  at most
	$4 \cdot  L_{\phi_p}+ 3$.
	For $p=1$, 
	let 
	$\mathcal{M}_{\phi_1}^{\eta}:
	\mathscr{X}_{1,\Gamma}
	\to
	\mathscr{H}_{\Gamma}$ 
	be an $\eta$-scalar proper extension of $\mathcal{M}_{\phi_1}$
	given by Corollary \ref{finitemodulusofcontinuity}.
	Denote by $\varPsi_k$
	the composition of the following maps
	\begin{equation}
		\label{eq:ThepropertyHmapforthekthlevel}
		\begin{array}{rcl}
			\mathscr{E}_{k,\Gamma} & \xlongrightarrow{\cong} & \bigoplus_{p=1}^{k}
			\mathscr{X}_{p,\Gamma} 
			\\
			& &    \  \  \  \  \  \  \
			\Bigg\downarrow \mathcal{M}_{\phi_1}^{\eta} \oplus \cdots \oplus \mathcal{M}_{\phi_k}^{I}  
			\\
			\mathscr{H}_{k,\Gamma} & \xlongleftarrow{\cong} &
			\bigoplus_{p=1}^{k}
			\mathscr{H}_{\Gamma},
		\end{array}
	\end{equation}
	where the isometric isomorphisms are canonical, and 
	$\mathcal{M}_{\phi_1}^{\eta} \oplus \cdots \oplus \mathcal{M}_{\phi_k}^{I}$ denotes the direct sum of
	$\mathcal{M}_{\phi_1}^{\eta}$
	and 
	$\mathcal{M}_{\phi_p}^{I}$ ($2\leq p \leq k$). 
	By Remark \ref{otheidrectsumofetascalarextensionsslaoproper},  
	$\mathcal{M}_{\phi_1}^{\eta} \oplus \cdots \oplus \mathcal{M}_{\phi_k}^{I}$
	is a proper  extension of the  Property $(H)$ map 
	$$
	\big(
	\mathcal{M}_{\phi_1}^{\eta} \oplus \cdots \oplus \mathcal{M}_{\phi_k}^{I}
	\big)\Big|_{\scalebox{0.7}{$S\big(\bigoplus_{p=1}^{k}
			\mathscr{X}_{p,\Gamma}\big)$}}.
	$$ 
	The same conclusion holds for
	$\varPsi_k$.

	The following lemma plays a key role in the construction of the $\Gamma$-equivariant asymptotic morphism in Proposition \ref{asymptoticGammathemostkeypoint1}.

	\begin{lem}\label{inequalityforPsi1}
		For every $k\geq1$
		and 
		$v,w
		\in
		\mathscr{E}_{k,\Gamma}$, 
		we have 	
		\begin{eqnarray*}
			&& 
			\Big\| 
			\varPsi_{k}
			(v+w)
			-
			\varPsi_{k}
			(v)
			\Big\|
			\leq
			\sqrt{	
				\bigg(
				\omega_{\scalebox{0.6}{$
						\mathcal{M}_{\phi_1}^{\eta}
						$}}
				\Big(
				\big\|
				\varrho_{k,1}(w)
				\big\|
				\Big)
				\bigg)^2	
				+
				16 \cdot \sum\limits_{p=2}^{k}
				(L_{\phi_p}+1)^2 \cdot
				\big\|
				\varrho_{k,p}(w)
				\big\|^2
			},
		\end{eqnarray*} 
		where 
		$\varrho_{k}: \mathscr{E}_{k,\Gamma} 
		\to  
		\bigoplus_{p=1}^{k}
		\mathscr{X}_{p,\Gamma}$ is the isometric isomorphism in \eqref{eq:ThepropertyHmapforthekthlevel}
		and 
		$\varrho_{k,p}: 
		\mathscr{E}_{k,\Gamma}  
		\xlongrightarrow{\varrho_{k}}  
		\bigoplus_{p=1}^{k}
		\mathscr{X}_{p,\Gamma}
		\to 
		\mathscr{X}_{p,\Gamma}$ 
		is the projection onto the $p$-th component.
	\end{lem}
	\begin{proof}
		For  every $k\geq1$ and 
		$v,w
		\in
		\mathscr{E}_{k,\Gamma}$,   	
		\begin{eqnarray*}
			&& 
			\Big\| 
			\varPsi_{k}
			(v+w)
			-
			\varPsi_{k}
			(v)
			\Big\|^2 
			\\
			&=&
			\bigg\|
			\Big(
			\mathcal{M}_{\phi_1}^{\eta} \oplus \cdots \oplus \mathcal{M}_{\phi_k}^{I}
			\Big) 
			\big( 
			\varrho_{k}(v+w)
			\big) 
			- \Big(
			\mathcal{M}_{\phi_1}^{\eta} \oplus \cdots \oplus \mathcal{M}_{\phi_k}^{I}
			\Big) 
			\big( 
			\varrho_{k}(v)
			\big) 
			\bigg\|^2
			\\
			&=&
			\Big\| 
			\mathcal{M}_{\phi_1}^{\eta}	
			\big( 
			\varrho_{k,1}(v)+\varrho_{k,1}(w)
			\big) 
			-
			\mathcal{M}_{\phi_1}^{\eta}
			\big( 
			\varrho_{k,1}(v)
			\big) 
			\Big\|^2 
			+
			\sum\limits_{p=2}^{k}
			\Big\|
			\mathcal{M}_{\phi_p}^{I}
			\big( 
			\varrho_{k,p}(v)+\varrho_{k,p}(w)
			\big) 
			-
			\mathcal{M}_{\phi_p}^{I}
			\big( 
			\varrho_{k,p}(v)
			\big) 
			\Big\|^2
			\\
			&\leq& 
			\bigg(
			\omega_{\scalebox{0.6}{$
					\mathcal{M}_{\phi_1}^{\eta}	
					$}}
			\Big(
			\big\|
			\varrho_{k,1}(w)
			\big\|
			\Big)
			\bigg)^2	
			+
			16 \cdot \sum\limits_{p=2}^{k}
			(L_{\phi_p}+1)^2 \cdot
			\big\|
			\varrho_{k,p}(w)
			\big\|^2.
		\end{eqnarray*} 
		This completes the proof. 
	\end{proof}

	\begin{defn}\label{Cliffordmap}
		We define
		$\mathfrak{B}_{k} :\mathscr{E}_{k,\Gamma} \rightarrow
		\text{Cliff}_{\mathbb{C}}
		(\mathscr{H}_{k,\Gamma})$
		by 
		$v\mapsto  C\big(\varPsi_k(v)\big)$. 
		Moreover,
		for any  
		$v_{0}\in \mathscr{E}_{k,\Gamma}$,  we define 
		$\mathfrak{B}_{k}^{v_{0}}: \mathscr{E}_{k,\Gamma} \rightarrow
		\text{Cliff}_{\mathbb{C}}
		(\mathscr{H}_{k,\Gamma})$ 
		by 
		$\mathfrak{B}_{k}^{v_{0}}(v) = \mathfrak{B}_{k}(v-v_{0})$.
	\end{defn}

	By \eqref{eq:functionalcalculusone} and \eqref{eq:functionalcalculustwo},
	we obtain the graded $*$-homomorphisms 
	\begin{equation}
		\label{eq:thethirdfunctionalcalculusfortheKthlevel}
		\begin{array}{rcl}
			\beta_{k,v}: ~ \mathcal{S} &\longrightarrow &
			\mathcal{S}\widehat{\otimes}{C}_{0}
			\big(
			\mathscr{E}_{k,\Gamma},
			\text{Cliff}_{\mathbb{C}}
			(\mathscr{H}_{k,\Gamma})
			\big)
			\\
			f &\longmapsto & 
			f\big(
			\varTheta \widehat{\otimes} 1+1\widehat{\otimes} 	\mathfrak{B}_{k}^{v}
			\big),
		\end{array}
	\end{equation}
	and 
	\begin{equation}
		\label{eq:thethirdfunctionalcalculusfortheKthlevelforgeneralsections}
		\begin{array}{rcl}
			\beta_{k,\xi}:~ \mathcal{S} &\longrightarrow &
			\mathcal{F}
			\big(
			Z, ~ 
			\mathcal{A}(\mathscr{E}_{k,\Gamma})
			\big)
			\\
			f &\longmapsto & 
			\Big\{
			z\mapsto 
			\beta_{k,\xi(z)}(f)
			\Big\},
		\end{array}
	\end{equation}
	where
	$k\in\mathbb{N}$,
	$v\in\mathscr{E}_{k,\Gamma}$,
	$\xi\in 
	C_{b}(Z, \mathscr{E}_{k,\Gamma})$,
	and 
	$Z$ is the metric space defined in \eqref{eq:simplicityforthebasespaceZ}.
	For each 
	$k\geq1$,
	following  Section 4.4,  
	we have the graded $C^*$-algebra
	$\mathcal{A}
	\big(
	C_{b}(Z, \mathscr{E}_{k,\Gamma})
	\big)$.
	Following Section 5.1, 
	the family of affine isometries
	$$\Big\{ U_{k}^{\scalebox{0.6}{$(\gamma,z)$}}~\Big|~ \gamma\in\Gamma, z\in Z
	\Big\}$$ 
	defined in \eqref{eq:affineisometriesforthekthlevel}
	induces a
	family  of $*$-isomorphisms
	$$
	\Big\{	\mathcal{U}_{k}^{\scalebox{0.6}{$(\gamma,z)$}}: 	
	\mathcal{A}(\mathscr{E}_{k,\Gamma})
	\longrightarrow 
	\mathcal{A}(\mathscr{E}_{k,\Gamma})
	\Big\}_ 
	{\gamma\in\Gamma, z\in {Z}},
	$$
	which in turn induces a $\Gamma$-action on 
	$\mathcal{A}
	\big(
	C_{b}(Z, \mathscr{E}_{k,\Gamma})
	\big)$.
	More precisely, 
	the 
	$*$-isomorphism 
	$\mathcal{U}_{k}^{\scalebox{0.6}{$(\gamma,z)$}}$
	is given by 
	\begin{equation}
		\label{eq:isometricsforthekthlevelalgebras}
		\big(\mathcal{U}^{\scalebox{0.6}{$(\gamma,z)$}}_{k}
		(f)
		\big)(\theta,v)
		=
		(id\widehat{\otimes}\lambda^{\scalebox{0.6}{$\gamma$}}_{k,*})
		\bigg(
		f\Big(
		\theta,		U^{\scalebox{0.6}{$(\gamma^{-1},\gamma{z})$}}_{k}(v)
		\Big)	
		\bigg)
	\end{equation}
	for 
	$f\in \mathcal{A}(\mathscr{E}_{k,\Gamma})$ 
	and 
	$(\theta,v)\in \mathbb{R} \times \mathscr{E}_{k,\Gamma}$,
	where $\lambda^{\scalebox{0.6}{$\gamma$}}_{k,*}:\text{Cliff}_{\mathbb{C}}
	(\mathscr{H}_{k,\Gamma}) \to \text{Cliff}_{\mathbb{C}}
	(\mathscr{H}_{k,\Gamma})$
	is the $*$-isomorphism  as in \eqref{eq:commutativeofriowithcliffalgebrariojdklk}.
	The $\Gamma$-action on 
	$\mathcal{A}
	\big(
	C_{b}(Z, \mathscr{E}_{k,\Gamma})
	\big)$ 
	is given by
	\begin{equation}
		\label{eq:GammaactionsfortheKthlevelalgebras}
		(\gamma\cdot {f})(z)=
		\mathcal{U}_{k}^{\scalebox{0.6}{$(\gamma,\gamma^{-1}z)$}}
		\big(
		f(\gamma^{-1}z)
		\big)
	\end{equation}
	for 
	$\gamma\in\Gamma$, 
	$f\in\mathcal{A}
	\big(
	C_{b}(Z, \mathscr{E}_{k,\Gamma})
	\big)$ 
	and 
	$z\in{Z}$.

	Note that each
	$\mathcal{A}
	\big(
	C_{b}(Z, \mathscr{E}_{k,\Gamma})
	\big)$
	may not be a proper $\Gamma$-$C^*$-algebra since $\mathfrak{h}_{k}: \Gamma \to \mathscr{E}_{k}$  is only bornologous.
	The properness will be ensured in the limit as
	$k \to +\infty$.
	We now introduce the limiting construction.

	For a family of $C^*$-algebras 
	$\{A_k\}_{k\in\mathbb{N}}$,
	let 
	$$\prod\limits_{k=1}^{\infty}
	A_k$$
	denote the $C^*$-algebra of all bounded sequences 
	$(a_1,a_2,\cdots)$
	with $a_k\in A_k$,
	equipped with the supremum norm.
	Moreover, let
	$$\bigoplus\limits_{k=1}^{\infty}
	A_k \subseteq \prod\limits_{k=1}^{\infty}
	A_k$$ 
	be the $C^*$-subalgebra  
	consisting of sequences $(a_1,a_2,\cdots)$  such that  $\lim_{k\to\infty}\|a_{k}\|=0$.
	Recall from Definition \ref{proper-algebra-by-geng-most-key} that, for $k\geq1$ and $R>0$, 
	$\mathcal{A}_{R}
	\big(
	C_{b}(Z, \mathscr{E}_{k,\Gamma})
	\big)$ 
	is the graded
	$*$-subalgebra of
	$\mathcal{A}
	\big(
	C_{b}(Z, \mathscr{E}_{k,\Gamma})
	\big)$ 
	consisting of those  
	$f$ 
	such that 
	${\rm Supp} \big(f(z)\big) \subseteq 
	B_{\mathbb{R}\times\mathscr{E}_{k,\Gamma}}(0,R)$ 
	for all 
	$z\in Z$.

	\begin{defn}\label{A-prod}
		We define  
	$\mathcal{A}^{\scalebox{0.5}{$\prod$}}(Z,\mathscr{E})$
	to be the quotient 
	$C^*$-algebra 
	$$
	\dfrac{\prod\limits_{k=1}^{\infty}
		\mathcal{A}
		\big(
		C_{b}(Z, \mathscr{E}_{k,\Gamma})
		\big)
	}{\bigoplus\limits_{k=1}^{\infty}
		\mathcal{A}
		\big(
		C_{b}(Z, \mathscr{E}_{k,\Gamma})
		\big)
	}.
	$$
Moreover, 
we define  
$\mathcal{A}^{\flat}(Z,\mathscr{E})$
to be the $C^*$-subalgebra of  
$\mathcal{A}^{\scalebox{0.5}{$\prod$}}(Z,\mathscr{E})$
generated by the set
	\begin{eqnarray*}
			\Big\{ 
			[(f_{1},f_{2},\cdots)]
			~\Big|~ 
			(f_{1},f_{2},\cdots)\in\prod\limits_{k=1}^{\infty}
			\mathcal{A}_{R}
			\big(
			C_{b}(Z, \mathscr{E}_{k,\Gamma})
			\big)  
			\text{ for some } R>0
			\Big\},
		\end{eqnarray*}
		where $[\cdot]$ denotes the equivalence class of an element in the quotient $C^*$-algebra. 
	\end{defn}	
	
We use the superscript  $\flat$  to indicate that  the algebra $\mathcal{A}^{\flat}(Z,\mathscr{E})$ 
is generated by elements with uniformly bounded supports.
The $\Gamma$-actions on the family of $C^*$-algebras
	$$\Big\{
	\mathcal{A}
	\big(
	C_{b}(Z, \mathscr{E}_{k,\Gamma})
	\big)
	\Big\}_{k\in\mathbb{N}},
	$$  
	given by \eqref{eq:GammaactionsfortheKthlevelalgebras},
	induce a $\Gamma$-action on  $\mathcal{A}^{\scalebox{0.5}{$\prod$}}(Z,\mathscr{E})$
	via
	\begin{eqnarray*}
		\gamma\cdot[(f_{1},f_{2},\cdots)]=\big[(\gamma\cdot f_{1},~\gamma\cdot f_{2},~\cdots)\big].
	\end{eqnarray*}
	The following result shows that this action restricts to a $\Gamma$-action on the $C^*$-subalgebra  	
	$\mathcal{A}^{\flat}(Z,\mathscr{E})$.
	Hence, $\mathcal{A}^{\flat}(Z,\mathscr{E})$
	is a $\Gamma$-$C^*$-algebra.
	We emphasize that  $\mathcal{A}^{\scalebox{0.5}{$\prod$}}(Z,\mathscr{E})$  need not be a proper
	$\Gamma$-$C^*$-algebra, 
	whereas its subalgebra $\mathcal{A}^{\flat}(Z,\mathscr{E})$ is proper.

	\begin{lem}\label{Gammaalgebraforthespecialcase}
		$\gamma\cdot [(f_{1},f_{2},\cdots)]\in 
		\mathcal{A}^{\flat}(Z,\mathscr{E})$ 
		for all 
		$\gamma\in\Gamma$ 
		and 
		$[(f_{1},f_{2},\cdots)]\in 
		\mathcal{A}^{\flat}(Z,\mathscr{E})$.	 
	\end{lem}
	\begin{proof}
		It suffices to show that for every $R>0$,  every
		$(f_{1},f_{2},\cdots)\in\prod\limits_{k=1}^{\infty}
		\mathcal{A}_{R}
		\big(
		C_{b}(Z, \mathscr{E}_{k,\Gamma})
		\big)$, and 
		every $\gamma\in\Gamma$,
		$$
		\gamma\cdot [(f_{1},f_{2},\cdots)]\in 
		\mathcal{A}^{\flat}(Z,
		\mathscr{E}).
		$$
		It follows from \eqref{eq:GammaactionsfortheKthlevelalgebras}, \eqref{eq:isometricsforthekthlevelalgebras} and \eqref{eq:affineisometriesforthekthlevel}
		that
		\begin{equation}
			\label{eq:thegammaactionsforelementsinalgebrasthekthlevel}
			\begin{array}{rcl}
				\big(
				(\gamma\cdot {f_k})(z)
				\big)(\theta,v)
				&=&
				\Big( \mathcal{U}^{\scalebox{0.6}{$(\gamma,\gamma^{-1}z)$}}_k
				\big(
				f_k(\gamma^{-1}z)
				\big)
				\Big)
				(\theta,v)
				\\
				&=&
				(id\widehat{\otimes}\lambda^{\scalebox{0.6}{$\gamma$}}_{k,*})
				\bigg(
				f_k(\gamma^{-1}z) 
				\Big(
				\theta,		U_k^{\scalebox{0.6}{$(\gamma^{-1},z)$}}(v)
				\Big)	
				\bigg)
				\\
				&=&
				(id\widehat{\otimes}\lambda^{\scalebox{0.6}{$\gamma$}}_{k,*})
				\bigg(
				f_k(\gamma^{-1}z) 
				\Big(
				\theta,		\lambda^{\scalebox{0.6}{$\gamma^{-1}$}}_k(v)+\zeta^{\scalebox{0.6}{$(\gamma^{-1},z)$}}_k
				\Big)	
				\bigg)
			\end{array}
		\end{equation}
		for all 
		$k\geq1$,
		$z\in Z$,
		and 
		$(\theta,v)\in \mathbb{R} \times \mathscr{E}_{k,\Gamma}$.
		Moreover,
		by \eqref{eq:cocyclesforthekthlevel} and Definition \ref{mostkeyidea1}, 
		we obtain	
		\begin{equation}
			\label{eq:theupperboundcalculationsforcocycles}
			\begin{array}{rcl}
				\Big\|
				\zeta^{\scalebox{0.6}{$(\gamma^{-1},z)$}}_k
				\Big\|
				&=&	
				\sqrt{ 
					\sum\limits_{\tilde{\gamma}\in\Gamma} 
					\Big\|
					\zeta_{k}^{\scalebox{0.6}{$(\gamma^{-1},z)$}}(\tilde{\gamma})
					\Big\|^2 ~}
				\\
				&=&
				\sqrt{
					\sum\limits_{\tilde{\gamma}\in\Gamma} 
					\bigg\|
					\sqrt{z_{\scalebox{0.6}{$\gamma\tilde{\gamma}$}}}
					\cdot 
					\Big( \mathfrak{h}_{k}(\gamma\tilde{\gamma})-
					\mathfrak{h}_{k}(\tilde{\gamma})
					\Big)
					\bigg\|^2~
				}
				\\
				&=&
				\sqrt{
					\sum\limits_{\tilde{\gamma}\in\Gamma}z_{\scalebox{0.6}{$\gamma\tilde{\gamma}$}}
					\cdot 
					\sum\limits_{p=1}^{k} 
					\big\|
					h_{p}(\gamma\tilde{\gamma})-
					h_{p}(\tilde{\gamma})
					\big\|^2~
				}
				\\
				&\leq& 
				\sqrt{
					\sum\limits_{\tilde{\gamma}\in\Gamma}z_{\scalebox{0.6}{$\gamma\tilde{\gamma}$}}
					\cdot 
					\sum\limits_{p=1}^k 
					\Big[
					\rho_{p}^{+}\big(d(\gamma,e)\big)\Big]^{2}~
				}
				\\
				&=& 
				\sqrt{
					\sum\limits_{p=1}^k 
					\Big[
					\rho_{p}^{+}\big(d(\gamma,e)\big)\Big]^{2}~
				}
				\\
				&\leq&
				\sqrt{
					\sum\limits_{p=1}^{\infty} 
					\Big[
					\rho_{p}^{+}\big(d(\gamma,e)\big)\Big]^{2}~
				}
				<+\infty
			\end{array}
		\end{equation}
		for all
		$\gamma\in\Gamma$, 
		$z \in {Z}$ 
		and 
		$k\geq1$.
		Set
		$$
		R_{\gamma}=\sqrt{
			\sum\limits_{p=1}^{\infty} 
			\Big[
			\rho_{p}^{+}\big(d(\gamma,e)\big)\Big]^{2}~
		}.$$
		Then, combining \eqref{eq:thegammaactionsforelementsinalgebrasthekthlevel} and \eqref{eq:theupperboundcalculationsforcocycles}, 
		we deduce that 
		$$\text{Supp} 
		\Big(
		(\gamma\cdot{f}_{k})(z)
		\Big) \subseteq 
		B_{\scalebox{0.6}{$\mathbb{R} \times \mathscr{E}_{k,\Gamma}$}}(0,R_{\gamma}+R)$$
		for all $z\in Z$ and $k\geq1$,
		that is, 
		$$
		(\gamma\cdot f_{1},\gamma\cdot f_{2},\cdots)\in\prod\limits_{k=1}^{\infty}
		\mathcal{A}_{R_{\gamma}+R}
		\big(
		C_{b}(Z, \mathscr{E}_{k,\Gamma})
		\big).
		$$
		This finishes the proof.
	\end{proof}

	We conclude this section by showing that  the $\Gamma$-$C^*$-algebra $\mathcal{A}^{\flat}(Z,\mathscr{E})$
	is proper.
	
	\begin{prop}\label{properalgebra}
		Suppose that $\Gamma$ admits a coarse embedding  with finite complexity into
		$\mathscr{E}=
		\bigoplus_{p=1}^{\infty}
		\mathscr{X}_p$. 
		Then 
		$\mathcal{A}^{\flat}(Z,\mathscr{E})$
		is a proper $\Gamma$-$C^*$-algebra.
	\end{prop}
	\begin{proof}
		For  
		$k\geq1$,
		let 
		$\mathcal{Z}
		\big(
		C_{b}(Z, \mathscr{E}_{k,\Gamma})
		\big)$
		denote  the graded $C^*$-subalgebra of
		$\mathcal{A}
		\big(
		C_{b} (Z, 
		\mathscr{E}_{k,\Gamma})
		\big)$
		generated by 
		$$\Big\{
		\beta_{k,\xi}
		(f) ~\Big|~ f\in\mathcal{S}_{\rm even}, ~ 
		\xi\in 
		C_{b}(Z, \mathscr{E}_{k,\Gamma}) 
		\Big\},
		$$ 
		where  $\beta_{k,\xi}$ is given by \eqref{eq:thethirdfunctionalcalculusfortheKthlevelforgeneralsections}.
		It lies in the center of
		$\mathcal{A}
		\big(
		C_{b}(Z, \mathscr{E}_{k,\Gamma})
		\big)$.
		We 
		define 
		$\mathcal{Z}^{\flat}(Z,\mathscr{E})$
		to be the graded $C^*$-subalgebra 
		of  
	$\mathcal{A}^{\scalebox{0.5}{$\prod$}}(Z,\mathscr{E})$
		generated by the set
		\begin{equation*}
			\Big\{ 
			[(f_{1},f_{2},\cdots)]
			~\Big|~ 
			(f_{1},f_{2},\cdots)\in\prod\limits_{k=1}^{\infty}
			\mathcal{Z}_{R}
			\big(
			C_{b}(Z, \mathscr{E}_{k,\Gamma})
			\big) 
			\text{ for some } R>0
			\Big\},
		\end{equation*}
		where 
		$$
		\mathcal{Z}_{R}
		\big(
		C_{b}(Z, \mathscr{E}_{k,\Gamma})
		\big) := \mathcal{Z}
		\big(
		C_{b}(Z, \mathscr{E}_{k,\Gamma})
		\big) \cap \mathcal{A}_{R}
		\big(
		C_{b}(Z, \mathscr{E}_{k,\Gamma})
		\big).
		$$
		This algebra is commutative and lies in the center of 
		$\mathcal{A}^{\flat}(Z,\mathscr{E})$. 
		Furthermore, 
		$\mathcal{Z}^{\flat}(Z,\mathscr{E})
		\cdot 
		\mathcal{A}^{\flat}(Z,\mathscr{E})$
		is dense in
		$\mathcal{A}^{\flat}(Z,\mathscr{E})$.
		Note that  the $\Gamma$-action on 
		$\mathcal{A}^{\flat}(Z,\mathscr{E})$ induces a $\Gamma$-action on
		$\mathcal{Z}^{\flat}(Z,\mathscr{E})$
		by restriction.
		Thus, 
		the Gelfand spectrum
		$$\Delta:={\rm Spec}
		\Big(
		\mathcal{Z}^{\flat}(Z,\mathscr{E})\Big)
		$$ 
		admits  
		an induced $\Gamma$-action.
		The Gelfand transform 
		yields a canonical $*$-isomorphism
		$$
		\mathcal{Z}^{\flat}(Z,\mathscr{E})
		\cong 
		C_{0}(\Delta),
		$$
		which is equivariant with respect to the $\Gamma$-actions.
		Hence,
		by Lemma \ref{propermethod}, 
		to obtain the properness of the $\Gamma$-action on 
		$\mathcal{A}^{\flat}(Z,\mathscr{E})$,
		it suffices to verify that 
		$$
		\lim\limits_{\gamma\rightarrow\infty}
		\big\|
		(\gamma\cdot{x}){x}
		\big\|=0
		$$	
		for every
		$x \in \mathcal{Z}^{\flat}(Z,\mathscr{E})$.

		We first consider the case where  $x=\big[(f_{1},f_{2},\cdots)\big]$
		with $(f_{1},f_{2},\cdots)\in\prod\limits_{k=1}^{\infty}
		\mathcal{Z}_{R}
		\big(
		C_{b}(Z, \mathscr{E}_{k,\Gamma})
		\big)$
		for $R>0$. 
		For such an $R>0$,	
		by the second condition of Definition \ref{mostkeyidea1},
		there exists 
		$c>0$ 
		such that 
		\begin{equation*}
			\sum\limits_{p=1}^{\infty}\Big[\rho_{p}^{-}(t)\Big]^2\geq 4R^{2}+1 \quad \text{for all}\   
			t\geq c.
		\end{equation*}
		Moreover, 
		for each $\gamma\in\Gamma$ with $d(\gamma,e)\geq c$, 
		there  exists a positive integer $M_{\gamma}$ such that 
		\begin{equation*}
			\sum\limits_{p=1}^{k}\Big[\rho_{p}^{-}\big(d(\gamma,e)\big)\Big]^2 > 4R^{2} \quad 
			\text{for all}\  
			k\geq 
			M_{\gamma}.
		\end{equation*}
		Thus, for each
		$\gamma\in\Gamma$ with $d(\gamma,e)\geq c$, 
		it follows from \eqref{eq:cocyclesforthekthlevel} and Definition \ref{mostkeyidea1} that 
		\begin{equation}
			\label{eq:thelowerboundcalculationforcocyclesthekthlevel}
			\begin{array}{rcl}
				\Big\|
				\zeta^{\scalebox{0.6}{$(\gamma^{-1},z)$}}_k
				\Big\|
				&=&	
				\sqrt{ 
					\sum\limits_{\tilde{\gamma}\in\Gamma} 
					\Big\|
					\zeta_{k}^{\scalebox{0.6}{$(\gamma^{-1},z)$}}(\tilde{\gamma})
					\Big\|^2 ~}
				\\
				&=&
				\sqrt{
					\sum\limits_{\tilde{\gamma}\in\Gamma} 
					\bigg\|
					\sqrt{z_{\scalebox{0.6}{$\gamma\tilde{\gamma}$}}}
					\cdot 
					\Big( \mathfrak{h}_{k}(\gamma\tilde{\gamma})-
					\mathfrak{h}_{k}(\tilde{\gamma})
					\Big)
					\bigg\|^2~
				}
				\\
				&=&
				\sqrt{
					\sum\limits_{\tilde{\gamma}\in\Gamma}z_{\scalebox{0.6}{$\gamma\tilde{\gamma}$}}
					\cdot 
					\sum\limits_{p=1}^{k} 
					\big\|
					h_{p}(\gamma\tilde{\gamma})-
					h_{p}(\tilde{\gamma})
					\big\|^2~
				}
				\\
				&\geq& 
				\sqrt{
					\sum\limits_{\tilde{\gamma}\in\Gamma}z_{\scalebox{0.6}{$\gamma\tilde{\gamma}$}}
					\cdot 
					\sum\limits_{p=1}^k 
					\Big[
					\rho_{p}^{-}\big(d(\gamma,e)\big)\Big]^{2}~
				}
				\\
				&=& 
				\sqrt{
					\sum\limits_{p=1}^k 
					\Big[
					\rho_{p}^{-}\big(d(\gamma,e)\big)\Big]^{2}~
				}
				\\
				&>& 2R
			\end{array}
		\end{equation}
		for all  $k\geq 
		M_{\gamma}$ and 
		all $z \in {Z}$.   
		Combining  
		\eqref{eq:thegammaactionsforelementsinalgebrasthekthlevel}
		and 
		\eqref{eq:thelowerboundcalculationforcocyclesthekthlevel},
		we obtain that,
		for every $\gamma\in\Gamma$ satisfying $d(\gamma,e)\geq c$,
		\begin{equation*}
			\text{Supp} 
			\Big(
			(\gamma\cdot{f}_{k})(z)
			\Big) 
			\cap {B}_{\scalebox{0.7}{$\mathbb{R}\times\mathscr{E}_{k,\Gamma}$}}(0, R)=\emptyset \quad 
			\text{for
				all} \ 
			k\geq 
			M_{\gamma}\  
			\text{and all}\  
			z\in{Z}.
		\end{equation*}
		This implies 
		$(\gamma\cdot {f}_{k})
		{f}_{k}=0$ 
		for all 
		$k\geq 
		M_{\gamma}$, 
		and hence
		$(\gamma\cdot{x}) {x}=0$
		for every 
		$\gamma\in\Gamma$ 
		with 
		$d(\gamma,e)\geq c$.

		The general case follows by a standard approximation argument.
	\end{proof}

	\section{The Bott element for $\mathcal{A}^{\flat}(Z,\mathscr{E})$}

Let 	
$\mathcal{A}^{\flat}(Z,\mathscr{E})$
be the proper $\Gamma$-$C^*$-algebra 
associated  to  a coarse embedding  with finite complexity 
$h: \Gamma \to \mathscr{E}=
	\bigoplus_{p=1}^{\infty}
	\mathscr{X}_p$,
constructed in Section 5.2.
In this section,
	we construct the Bott element	
	$\mathfrak{b}\in 
	KK_{0}^{\Gamma}
	\big( 
	\mathcal{S},~
	\mathcal{A}^{\flat}(Z,\mathscr{E})
	\big)$.
	We describe this element in terms of asymptotic morphisms. The notion of a special $\Gamma$-equivariant asymptotic morphism  used throughout this section is recalled in Construction~\ref{constr:KK-facts-asymptotic} in the Appendix.
	As explained in the Appendix,  the
	Bott element 
	$\mathfrak{b}$  can induce  a homomorphism  
	from $KK$-theory  with coefficients in $\mathcal{S}$ to   
	 $KK$-theory  with coefficients in
	 the proper $\Gamma$-$C^*$-algebra 
	 $\mathcal{A}^{\flat}(Z,\mathscr{E})$.
	The third condition in the definition of  coarse embedding  with finite complexity  (Definition \ref{mostkeyidea1})  	
plays a  crucial role  in establishing the $\Gamma$-equivariance in Proposition \ref{asymptoticGammathemostkeypoint1}.

	Suppose that
	$h:\Gamma\rightarrow\mathscr{E}=
	\bigoplus_{p=1}^{\infty}
	\mathscr{X}_p$ 
	is a coarse embedding  with finite complexity.  
	For each $t\geq1$,
	we define a map
	\begin{equation}
		\label{eq:thebottmapforelementaryalgebras}
		\mathfrak{T}_{t}: 
		\mathcal{S}
		\longrightarrow  
		\mathcal{A}^{\flat}(Z,\mathscr{E})
	\end{equation}
	by
	$$
	\mathfrak{T}_{t}(f) =
	\Big[\big(~ 
	\beta_{1,{\bf 0}}(f_t),
	~\beta_{2,{\bf 0}}(f_t),
	~\cdots ~  
	\big)\Big],
	$$
	where 
	$f_t(x)=f(\frac{x}{t})$,
	${\bf 0}$ is the zero function in 
	$C_{b}(Z, \mathscr{E}_{k,\Gamma})$
	and 
	$\beta$ is given by 
	\eqref{eq:thethirdfunctionalcalculusfortheKthlevelforgeneralsections}.
	By definition,
	each $\mathfrak{T}_{t}$ is
	a $*$-homomorphism.
	
	\begin{prop}\label{asymptoticGammathemostkeypoint1}
		For every 
		$\gamma\in\Gamma$ 
		and 
		$f\in\mathcal{S}$,  
		we have
		$$
		\lim\limits_{t\rightarrow \infty}
		\Big\|
		\gamma
		\cdot
		\big(
		\mathfrak{T}_{t}(f)
		\big)-
		\mathfrak{T}_{t}(f)
		\Big\|
		=0.
		$$
	\end{prop}
	\begin{proof}
		By the Stone--Weierstrass 
		theorem, it suffices to prove the proposition
		for the function
		$f(x) = (x + i)^{-1}$.

		For every $\gamma\in\Gamma$, $k\in\mathbb{N}$ and $t\geq1$, 
		combining \eqref{eq:rulesforgammaactionforelementary},   
		\eqref{eq:affineisometriesforthekthlevel},  
		\eqref{eq:Gammaactionsforthekthlevel},
		\eqref{eq:thethirdfunctionalcalculusfortheKthlevel},  
		\eqref{eq:thethirdfunctionalcalculusfortheKthlevelforgeneralsections},    
		Lemma \ref{inequalityforPsi1}, 
		\eqref{eq:cocyclesforthekthlevel}
		and \eqref{eq:upperboundforthecocycles},
		we obtain
		\begin{eqnarray*}
			&& 
			\Big\|
			\gamma
			\cdot
			\beta_{k,{\bf 0}}(f_t)  
			-
			\beta_{k,{\bf 0}}(f_t)  
			\Big\|
			\\
			&=& 
			\Big\|
			\beta_{k,\gamma
				\cdot{\bf 0}}(f_t)  
			-
			\beta_{k,{\bf 0}}(f_t)  
			\Big\|
			\\
			&=&
			\sup_{\scalebox{0.7}{$z\in{Z}$}}
			\biggg\| 
			\Bigg[
			\frac{1}{t}
			\bigg(
			\varTheta \widehat{\otimes} 1+1\widehat{\otimes}\mathfrak{B}_{k}^{\scalebox{0.7}{$\zeta^{\scalebox{0.6}{$(\gamma,\gamma^{-1}z)$}}_{k}$}}
			\bigg)+
			i\Bigg]^{-1} 
			-    
			\bigg[
			\frac{1}{t}
			\big(
			\varTheta \widehat{\otimes} 1+1\widehat{\otimes}\mathfrak{B}_{k}
			\big)+
			i\bigg]^{-1} 
			\biggg\|
			\\
			&\leq&
			\sup_{\scalebox{0.7}{$z\in{Z}$}}
			\Biggg\{
			\biggg\| 
			\Bigg[
			\frac{1}{t}
			\bigg(
			\varTheta \widehat{\otimes} 1+1\widehat{\otimes}\mathfrak{B}_{k}^{\scalebox{0.7}{$\zeta^{\scalebox{0.6}{$(\gamma,\gamma^{-1}z)$}}_{k}$}}
			\bigg)+
			i\Bigg]^{-1}
			\biggg\| 
			\cdot 
			\Bigg\|    
			\bigg[
			\frac{1}{t}
			\big(
			\varTheta \widehat{\otimes} 1+1\widehat{\otimes}\mathfrak{B}_{k}
			\big)+
			i\bigg]^{-1} 
			\Bigg\|
			\cdot 
			\Bigg\| 
			\frac{
				\mathfrak{B}_{k}^{\scalebox{0.7}{$\zeta^{\scalebox{0.6}{$(\gamma,\gamma^{-1}z)$}}_{k}$}}-
				\mathfrak{B}_{k} 
			}{t} 
			\Bigg\|
			\Biggg\}
			\\
			&\leq&
			\frac{1}{t}\cdot\sup_{\scalebox{0.7}{$z\in{Z}$}} 
			\Big\| \mathfrak{B}_{k}^{\scalebox{0.7}{$\zeta^{\scalebox{0.6}{$(\gamma,\gamma^{-1}z)$}}_{k}$}}-
			\mathfrak{B}_{k}
			\Big\|
			\\
			&=&
			\frac{1}{t}\cdot\sup_{\scalebox{0.7}{$z\in{Z}$}} \sup_{\scalebox{0.7}{$v\in \mathscr{E}_{k,\Gamma}$}}
			\Big\| 
			\varPsi_{k}
			\big(
			v-
			\zeta^{\scalebox{0.6}{$(\gamma,\gamma^{-1}z)$}}_{k}
			\big)
			-
			\varPsi_{k}
			(v)
			\Big\|
			\\
			&\leq&
			\frac{1}{t}\cdot\sup_{\scalebox{0.7}{$z\in{Z}$}} 		\sqrt{	
				\bigg[
				\omega_{\scalebox{0.6}{$
						\mathcal{M}_{\phi_1}^{\eta}
						$}}
				\Big(
				\big\|
				\varrho_{k,1}\big(\zeta^{\scalebox{0.6}{$(\gamma,\gamma^{-1}z)$}}_{k}\big)
				\big\|
				\Big)
				\bigg]^2	
				+
				16\cdot \sum\limits_{p=2}^{k}
				\big(L_{\phi_p}+1\big)^2 
				\cdot
				\Big\|
				\varrho_{k,p}\big(\zeta^{\scalebox{0.6}{$(\gamma,\gamma^{-1}z)$}}_{k}\big)
				\Big\|^2	
			}
			\\
			&\leq&
			\frac{1}{t}\cdot	\sqrt{	
				\bigg[
				\omega_{\scalebox{0.6}{$
						\mathcal{M}_{\phi_1}^{\eta}
						$}}
				\Big(
				\rho_{1}^{+}\big(d(\gamma^{-1},e)
				\big)
				\Big)
				\bigg]^2	
				+
				16\cdot \sum\limits_{p=2}^{k}
				\Big( 
				(L_{\phi_p}+1)
				\cdot
				\rho_{p}^{+}\big(d(\gamma^{-1},e)
				\big)
				\Big) ^2
			}
			\\
			&\leq&
			\frac{1}{t}\cdot	\sqrt{	
				\bigg[
				\omega_{\scalebox{0.6}{$
						\mathcal{M}_{\phi_1}^{\eta}
						$}}
				\Big(
				\rho_{1}^{+}\big(d(\gamma^{-1},e)
				\big)
				\Big)
				\bigg]^2	
				+
				16\cdot \sum\limits_{p=2}^{\infty}
				\Big( 
				(L_{\phi_p}+1)
				\cdot
				\rho_{p}^{+}\big(d(\gamma^{-1},e)
				\big)
				\Big) ^2
			}.
		\end{eqnarray*} 
		Note that 
		$h:\Gamma\to\mathscr{E}=
		\bigoplus_{p=1}^{\infty}
		\mathscr{X}_p$ 
		is a coarse embedding  with finite complexity.
		It follows from the third condition of Definition \ref{mostkeyidea1}  that 
		\begin{eqnarray*}
			\sum\limits_{p=2}^{\infty}
			\Big( 
			(L_{\phi_p}+1)
			\cdot
			\rho_{p}^{+}\big(d(\gamma^{-1},e)
			\big)
			\Big) ^2<+\infty \quad  \text{for every} \ 
			\gamma\in\Gamma.
		\end{eqnarray*}
		On the other hand, 
		by Corollary \ref{finitemodulusofcontinuity}, 
		$$	
		\omega_{\scalebox{0.6}{$
				\mathcal{M}_{\phi_1}^{\eta}
				$}}
		\Big(
		\rho_{1}^{+}\big(d(\gamma^{-1},e)
		\big)
		\Big)
		<+\infty \quad \text{for every}\ 	\gamma\in\Gamma.
		$$
		Therefore, for every 
		$\gamma\in\Gamma$,   
		\begin{eqnarray*}
			&&
			\lim\limits_{t\rightarrow \infty}
			\Big\|
			\gamma
			\cdot
			\big(
			\mathfrak{T}_{t}(f)
			\big)-
			\mathfrak{T}_{t}(f)
			\Big\|
			\\
			&=& 	
			\lim\limits_{t\rightarrow \infty}
			\bigg\|
			\Big[
			\Big( 
			\gamma
			\cdot
			\beta_{1,{\bf 0}}(f_t)  
			-
			\beta_{1,{\bf 0}}(f_t),~
			\gamma
			\cdot
			\beta_{2,{\bf 0}}(f_t)  
			-
			\beta_{2,{\bf 0}}(f_t),~
			\cdots ~
			\Big)
			\Big]
			\bigg\|
			\\
			&\leq&
			\lim\limits_{t\rightarrow \infty}
			\sup_{k\in\mathbb{N}}
			\Big\|
			\gamma
			\cdot
			\beta_{k,{\bf 0}}(f_t)  
			-
			\beta_{k,{\bf 0}}(f_t) 
			\Big\|
			\\
			&=&
			0.
		\end{eqnarray*}
		This completes the proof.
	\end{proof}

	By Proposition \ref{asymptoticGammathemostkeypoint1} and Construction \ref{constr:KK-facts-asymptotic} in the Appendix, 
	we have  
	$$\big[(
	\mathfrak{T}_{t}
	)_{t\in[1,\infty)}\big] 
	\in 
	KK^{\Gamma}_{0} 
	\big(
	\mathcal{S}, ~
	\mathcal{A}^{\flat}(Z,\mathscr{E})
	\big).$$

	\begin{rmk}\label{Ithiknitisaintemediatepartforthelatersections}
		The results of the preceding sections are sufficient to prove  the rational injectivity of the assembly map (Theorem \ref{main-result1}) 
		in a special case where each map
		$$\mathcal{S}\longrightarrow  \mathcal{A}(\mathscr{E}_{k,\Gamma}),\quad  f\mapsto \beta_{k,0}(f),$$ 
		given by \eqref{eq:thethirdfunctionalcalculusfortheKthlevel}, 
		induces an isomorphism in $K$-theory.   
		We now briefly outline the main idea of the proof.
		
		It follows from 
		$\big[(
		\mathfrak{T}_{t}
		)_{t\in[1,\infty)}\big] 
		\in 
		KK^{\Gamma}_{0} 
		\big(
		\mathcal{S}, ~ 
		\mathcal{A}^{\flat}(Z,\mathscr{E})
		\big)$ and
		\eqref{eq:BC-assembly-natural} 
		that
		the diagram 
		\begin{equation*}\label{babycaseformainresult}
			\xymatrix@R=1.3cm@C=0.8cm{
				KK_{*}^{\Gamma}(E\Gamma,\mathcal{S}) \ar[r]^{\pi_*} \ar[d]^{\big[(
					\mathfrak{T}_{t}
					)_{t\in[1,\infty)}\big] }  & KK^\Gamma_*(\underline{E}\Gamma, \mathcal{S}) \ar[r]^\mu \ar[d]^{\big[(
					\mathfrak{T}_{t}
					)_{t\in[1,\infty)}\big] } & K_*(\mathcal{S} \rtimes_{\operatorname{r}} \Gamma) \ar[d]^{
					\big[(
					\mathfrak{T}_{t}
					)_{t\in[1,\infty)}\big]  \rtimes_{\operatorname{r}} \Gamma} \\
				KK_{*}^{\Gamma}
				\big(E\Gamma, ~	\mathcal{A}^{\flat}(Z,\mathscr{E})\big) 
				\ar[r]^{\pi_*} & KK^\Gamma_*
				\big(
				\underline{E}\Gamma, ~ \mathcal{A}^{\flat}(Z,\mathscr{E})
				\big) 
				\ar[r]^\mu & K_*\big(
				\mathcal{A}^{\flat}(Z,\mathscr{E}) \rtimes_{\operatorname{r}} \Gamma
				\big)
			}
		\end{equation*}	
		commutes,
		where 
		$E\Gamma$ is the universal space for proper and free actions of the group $\Gamma$, $\underline{E}\Gamma$ is the universal space for proper actions of the group $\Gamma$, and 
		$\pi_*$ is induced by the natural $\Gamma$-equivariant continuous map $\pi: E\Gamma \to \underline{E}\Gamma$.
		The vertical maps on the left and in the middle are  usually called  \emph{Bott maps}.
		Combining Theorem \ref{thm:proper-GHT} and Proposition \ref{properalgebra},
		we obtain that 
		the bottom map $\mu$ is an isomorphism. On the other hand, the bottom $\pi_*$ is rationally injective 
		by Lemma \ref{jdskljfaljsdkfjsqqq} in the Appendix.
		If the left Bott map 
		is injective,
		then the commutative diagram implies that the composition of the maps in the top row is rationally injective, as desired. 
		The injectivity of the left Bott map can be established in the special case by 
		combining Lemma \ref{isomorphismbetweenKtheory-babycase},
		Lemma \ref{isomorphismasmallkeyforlaterresults} and Lemma \ref{smalltrick}.

		In general,
		each map 
		$\mathcal{S} \to \mathcal{A}(\mathscr{E}_{k,\Gamma})$ 
		induces only an injective homomorphism in $K$-theory,
		which does not ensure the injectivity of the left Bott map. 
		To overcome this difficulty, we enlarge the algebra
		$\mathcal{A}^{\flat}(Z,\mathscr{E})$ 
		and employ a deformation argument. 
		A detailed proof of the injectivity in the general case will be given in
		Section 9.
		
		The material in Sections 5 and 6  serves as a prototype for constructing a larger proper $\Gamma$-algebra and  its associated Bott element, which will be developed in the subsequent sections.
	\end{rmk}

	\section{Deformation of affine isometries and larger proper $\Gamma$-$C^*$-algebras} 
	
	In order to obtain a $C^*$-algebra  in which the proper $\Gamma$-action can be continuously deformed to an almost trivial $\Gamma$-action (that is, an action which only moves points in the base space and acts trivially on the fibres), 
	we enlarge the proper $\Gamma$-algebra constructed in Section 5. 
	This deformation of group actions will allow us to establish the injectivity of the Bott map (see Section 9).
	Suppose that 
	$\Gamma$  
	admits a
	coarse embedding  with finite complexity
	into an
	$\ell^2$-direct sum of infinitely many Property 
	$(H)$ 
	Banach spaces. 
	As in Section 5,
	we decompose the infinite direct sum into an increasing filtration by finite partial sums. 
	Each term in this filtration is a finite direct sum of Property 
	$(H)$ Banach spaces and hence also has Property $(H)$.
	For a single Banach space with Property 
	$(H)$, 
	an enlarged $\Gamma$-$C^*$-algebra will be constructed in Section  7.2. 
	We then take the limit of these algebras to obtain the enlarged   $\Gamma$-$C^*$-algebra in the general setting. 
	This will be carried out in Section  7.3.

	\subsection{Affine isometries for the larger space}
	
	Suppose that
$h: \Gamma \to \mathscr{X}$ 
is bornologous, 
where  $\mathscr{X}$ is a  Banach space.
A $\Gamma$-action on the function space $C_{b}(Z, \mathscr{X}_{\Gamma})$
has already been constructed  in Section 3.
In this subsection,
we replace $\mathscr{X}_{\Gamma}$  by a larger space  $L^2 ([0,1], \mathscr{X}_{\Gamma})$,
and  construct a family (parametrized by $s\in[0,1]$) 
of $\Gamma$-actions  
on the function space
$C_{b}
\big(
Z, ~
L^2 ([0,1], \mathscr{X}_{\Gamma})
\big)$.
This setting will be used in the construction of the enlarged algebras in the following two subsections.

We first construct a family of affine isometries on 
	$L^2 ([0,1],\mathscr{X}_{\Gamma})$
	parametrized by 
	$\gamma\in\Gamma$, $z\in Z$ and $s\in[0,1]$.
	Based on the construction in Section 3, 
	for each  $\gamma\in\Gamma$,  $z\in Z$ and  $s\in[0,1]$,
	there exists an 
	affine isometry  
	$$
	U^{\scalebox{0.7}{$(\gamma,z,s)$}}:
	L^2 ([0,1], \mathscr{X}_{\Gamma})
	\longrightarrow 
	L^2 ([0,1], \mathscr{X}_{\Gamma})
	$$
	defined by
	\begin{eqnarray}
		\label{eq:theaffineisometricsforenlargedspaceelementaryone}
		\big(
		U^{\scalebox{0.7}{$(\gamma,z,s)$}}(f)
		\big)(t)= 
		\begin{cases}
			U^{\scalebox{0.7}{$(\gamma,z)$}}
			\big(
			f(t)
			\big), & t\in[0,s],
			\\
			f(t), & t\in(s,1]
		\end{cases}
	\end{eqnarray}
	for 
	$f\in 
	L^2 ([0,1], \mathscr{X}_{\Gamma})$.

	\begin{rmk}\label{degeneratesdfkjalsdjkfls}
		For $s=1$,  $U^{\scalebox{0.7}{$(\gamma,z,1)$}}$ encodes the original affine isometry  $U^{\scalebox{0.6}{$(\gamma,z)$}}$. 
		In contrast, for $s=0$, 
		$U^{\scalebox{0.7}{$(\gamma,z,0)$}}$ is the identity operator for all $\gamma\in\Gamma$ and $z\in {Z}$. 
	\end{rmk}

	It is convenient to rewrite 
	$U^{\scalebox{0.7}{$(\gamma,z,s)$}}$
	in the form of \eqref{eq:affineisometricesgeneralone}.
	To this end, 
	for  $\gamma\in\Gamma$, $z\in Z$ and $s\in[0,1]$,
	define 	$\lambda^{\scalebox{0.7}{$(\gamma,s)$}}: L^2 ([0,1], \mathscr{X}_{\Gamma})
	\to 
	L^2 ([0,1], \mathscr{X}_{\Gamma})$
	by 
	\begin{eqnarray*}
		\big(
		\lambda^{\scalebox{0.7}{$(\gamma,s)$}}(f)
		\big)(t)= 
		\begin{cases}
			\lambda^{\scalebox{0.7}{$\gamma$}}
			\big(
			f(t)
			\big),& t\in[0,s],
			\\
			f(t), & t\in(s,1],
		\end{cases}
	\end{eqnarray*}
	and
	define 
	$\zeta^{\scalebox{0.7}{$(\gamma,z,s)$}}\in 
	L^2 ([0,1], \mathscr{X}_{\Gamma})$
	by 
	\begin{eqnarray}
		\label{eq:thecocyclesfortheenlargedspaceelementaryone}
		\zeta^{\scalebox{0.7}{$(\gamma,z,s)$}}(t)
		= 
		\begin{cases}
			\zeta^{\scalebox{0.7}{$(\gamma,z)$}},& t\in[0,s],
			\\
			0, & t\in(s,1],
		\end{cases}
	\end{eqnarray}
	where
	$\zeta^{\scalebox{0.7}{$(\gamma,z)$}}$
	is given by \eqref{eq:cocycles}.
	Then we have
	\begin{eqnarray}
		\label{eq:theaffineisometriesfortheenlargedspaces}
		U^{\scalebox{0.7}{$(\gamma,z,s)$}}(f)
		= 
		\lambda^{\scalebox{0.7}{$(\gamma,s)$}}(f) + \zeta^{\scalebox{0.7}{$(\gamma,z,s)$}}
	\end{eqnarray}
	for all 
	$f\in 
	L^2 ([0,1], \mathscr{X}_{\Gamma})$.

	The following result is immediate from Lemma \ref{affineisometryproperty1}. 
	\begin{lem}
		For each  $s\in[0,1]$,
		the following statements hold: 
		\begin{itemize}
			\item [(1)]	$U^{\scalebox{0.7}{$(e,z,s)$}}$
			is the 
			identity operator 
			for all 
			$z\in{Z}$;
			\item [(2)]	$U^{\scalebox{0.7}{$(\gamma, z^\prime,s)$}} U^{\scalebox{0.7}{$(\gamma^\prime,z,s)$}}
			=
			U^{\scalebox{0.7}{$(\gamma\gamma^\prime,z,s)$}}$
			for all 
			$\gamma, \gamma^\prime\in\Gamma$ and $z,z^\prime\in {Z}$ 
			such that
			$z^\prime=\gamma^\prime z$. 
		\end{itemize}	
	\end{lem}

	As in Section 3,
	for each $s\in[0,1]$,
	the family of affine isometries 
	$$\Big\{	
	U^{\scalebox{0.7}{$(\gamma,z,s)$}} ~\Big|~ \gamma\in\Gamma, z\in {Z}
	\Big\}$$
	induces a  $\Gamma$-action  
	on 
	$C_{b}
	\big(
	Z, ~
	L^2 ([0,1], \mathscr{X}_{\Gamma})
	\big)$
	via
	\begin{eqnarray}
		\label{eq:theationsforenlargedspaces}
		\big(
		\gamma\cdot_{\alpha_{s}} f
		\big)
		(z)
		=	U^{\scalebox{0.7}{$(\gamma,\gamma^{-1}z,s)$}}
		\big(
		f(\gamma^{-1}z)
		\big),
	\end{eqnarray}
	where $\gamma\in\Gamma$, $z\in {Z}$ and $f\in C_{b}
	\big(
	Z, ~
	L^2 ([0,1], \mathscr{X}_{\Gamma})
	\big)$.
	To specify the action that depends on $s\in[0,1]$,
	we write a subscript   
	$\alpha_s$ in \eqref{eq:theationsforenlargedspaces}.
	
	\begin{rmk}\label{bubiandehomotopyoftheendsofactiondlkfj}
		Note that for $s=1$, the $\Gamma$-$\alpha_1$-action on  $C_{b}
		\big(
		Z, ~
		L^2 ([0,1], \mathscr{X}_{\Gamma})
		\big)$ encodes the original $\Gamma$-action on 
		$C_{b}(
		Z, \mathscr{X}_{\Gamma})$. 
		In contrast, for $s=0$, the $\Gamma$-$\alpha_0$-action on  $C_{b}
		\big(
		Z, ~
		L^2 ([0,1], \mathscr{X}_{\Gamma})
		\big)$ is induced only by precomposition with the action of $\Gamma$ on $Z$, while leaving the $L^2 ([0,1], \mathscr{X}_{\Gamma})$-values of the functions unchanged.
	\end{rmk}

\subsection{The larger
		$\Gamma$-$C^*$-algebra generated by uniform  simple functions}

In this subsection, 
	we construct a larger $\Gamma$-$C^*$-algebra associated to a bornologous map $h: \Gamma \to \mathscr{X}$,
	where  $\mathscr{X}$ is a Property $(H)$ Banach space.
	This construction is not a straightforward extension of that in Section 5.1, 
since the algebra to be constructed here must support a continuous deformation of the $\Gamma$-$\alpha_s$-actions. 
The $C^*$-algebra 
generated by the entire function space
$C_{b}
\big(Z, ~
L^2([0,1], \mathscr{X}_{\Gamma}) 
\big)$ 
is too large to ensure the continuity of this deformation.
We therefore need to choose a suitable subspace of
$C_{b}
\big(Z, ~
L^2([0,1], \mathscr{X}_{\Gamma}) 
\big)$ 
and consider  the 
$C^*$-subalgebra  generated by it 
(see Definition \ref{algebrasgeneratedbysimplefunctions}).

	Suppose that
	$h: \Gamma \to \mathscr{X}$ 
	is bornologous, where  $\mathscr{X}$ is a Property $(H)$ Banach space with a Property 
	$(H)$
	map 
	$\phi: S(\mathscr{X})\to S(\mathscr{H})$.
	As noted in Remark \ref{properHmapsforiteratedconstructionkdjlkjl}, 
	applying Proposition \ref{PropertyHmapforLpspaces} twice yields a
	Property $(H)$ map
	$$
	\mathcal{M}_{\mathcal{M}_{\phi}}: 
	S\big( 
	L^2 ([0,1], \mathscr{X}_{\Gamma})
	\big) 
	\longrightarrow 
	S\big( 
	L^2 ([0,1], \mathscr{H}_{\Gamma}) 
	\big)
	$$ 	
	for the Property $(H)$ Banach space
	$L^2 ([0,1], \mathscr{X}_{\Gamma})$.

	We set, for simplicity, 
	$$
	\phi_{\scalebox{0.6}{$[0,1]$}} = \mathcal{M}_{\mathcal{M}_{\phi}}.
	$$
	Let 
	$\phi_{\scalebox{0.6}{$[0,1]$}}^{\eta}:
	L^2 ([0,1], \mathscr{X}_{\Gamma})
	\to 
	L^2 ([0,1], \mathscr{H}_{\Gamma})$
	be an  
	$\eta$-scalar proper extension of $\phi_{\scalebox{0.6}{$[0,1]$}}$
	given by Corollary \ref{finitemodulusofcontinuity}.

	\begin{defn}\label{Cliffordmap}
		We define
		$\mathfrak{B}_{\scalebox{0.5}{$[0,1]$}}:
		L^2 ([0,1], \mathscr{X}_{\Gamma}) 
		\rightarrow
		\text{Cliff}_{\mathbb{C}}
		\big(
		L^2 ([0,1], \mathscr{H}_{\Gamma}) 
		\big)$
		by
		$v\mapsto C\big(\phi_{\scalebox{0.6}{$[0,1]$}}^{\eta}(v)\big)$. 
		Moreover,
		for any  
		$v_{0}\in L^2 ([0,1], \mathscr{X}_{\Gamma})$,  
		we define $\mathfrak{B}_{\scalebox{0.5}{$[0,1]$}}^{v_{0}}: L^2 ([0,1], \mathscr{X}_{\Gamma}) 
		\rightarrow
		\text{Cliff}_{\mathbb{C}}
		\big(
		L^2 ([0,1], \mathscr{H}_{\Gamma}) 
		\big)$ 
		by 
		$\mathfrak{B}_{\scalebox{0.5}{$[0,1]$}}^{v_{0}}(v) = \mathfrak{B}_{\scalebox{0.5}{$[0,1]$}}(v-v_{0})$.
	\end{defn}

	By  \eqref{eq:functionalcalculusone} and \eqref{eq:functionalcalculustwo}, 
	we obtain the  graded $*$-homomorphisms 
	\begin{equation}
		\label{eq:thefunctionalcalculusforenlargeddpacefirstone}
		\begin{array}{rcl}
			\beta_{v}:~ \mathcal{S} &\longrightarrow&
			\mathcal{S}\widehat{\otimes} 
			C_{0}
			\Big(
			L^2 ([0,1], \mathscr{X}_{\Gamma}),~ 
			\text{Cliff}_{\mathbb{C}}
			\big(
			L^2 ([0,1], \mathscr{H}_{\Gamma}) 
			\big)
			\Big)
			\\
			f & \longmapsto & 
			f \big(
			\varTheta \widehat{\otimes} 1+1\widehat{\otimes} 		\mathfrak{B}_{\scalebox{0.5}{$[0,1]$}}^{v}
			\big),
		\end{array}
	\end{equation}
	and 
	\begin{equation}
		\label{eq:thefunctionalcalculusforenlargedspacesecondoneforsections}
		\begin{array}{rcl}
			\beta_{\xi}:~ \mathcal{S} &\longrightarrow &
			\mathcal{F}
			\Big(
			Z, ~
			\mathcal{A}
			\big(
			L^2 ([0,1], \mathscr{X}_{\Gamma})
			\big)
			\Big)
			\\
			f&\longmapsto & 
			\Big\{
			z\mapsto 
			\beta_{\xi(z)}(f)
			\Big\},
		\end{array}
	\end{equation}
	where 
	$v\in L^2 ([0,1], \mathscr{X}_{\Gamma})$,
	$\xi\in 
	C_{b}
	\big(
	Z, ~
	L^2 ([0,1], \mathscr{X}_{\Gamma})
	\big)$, 
	and $Z$ is the metric space defined in \eqref{eq:simplicityforthebasespaceZ}.

	Following  Section 4.4,   
	we have the graded $C^*$-algebra 
	$$\mathcal{A}
	\Big(
	C_{b}
	\big(Z, ~
	L^2([0,1], \mathscr{X}_{\Gamma}) 
	\big)
	\Big).
	$$
	For each $s\in[0,1]$,
	following  Section 5.1, 
	the family of affine isometries
	$$
	\Big\{ 
	U^{\scalebox{0.7}{$(\gamma,z,s)$}}
	~\Big|~ 
	\gamma\in\Gamma, z\in Z
	\Big\}
	$$ 
	defined in \eqref{eq:theaffineisometricsforenlargedspaceelementaryone}
	induces a
	family  of $*$-isomorphisms
	$$
	\bigg\{	
	\mathcal{U}^{\scalebox{0.7}{$(\gamma,z,s)$}}: 	
	\mathcal{A}
	\big(
	L^2 ([0,1], \mathscr{X}_{\Gamma})
	\big)
	\longrightarrow 
	\mathcal{A}
	\big(
	L^2 ([0,1], \mathscr{X}_{\Gamma})
	\big)
	\bigg\}_ 
	{z\in {Z}, \gamma\in\Gamma},
	$$
	which in turn  induces a $\Gamma$-$\alpha_s$-action on 
	$\mathcal{A}
	\Big(
	C_{b}
	\big(
	Z, ~
	L^2 ([0,1], \mathscr{X}_{\Gamma})
	\big)
	\Big)$.
	More precisely, 
	the 
	$*$-isomorphism 
	$\mathcal{U}^{\scalebox{0.7}{$(\gamma,z,s)$}}$
	is given by  
	\begin{eqnarray}
		\label{eq:isometricsbetweenenlargedalgebras}
		& \big(\mathcal{U}^{\scalebox{0.7}{$(\gamma,z,s)$}}(f)
		\big)(\theta,v)
		=
		(id\widehat{\otimes}\lambda^{\scalebox{0.7}{$(\gamma,s)$}}_*)
		\bigg(
		f\Big(
		\theta,		U^{\scalebox{0.7}{$(\gamma^{-1},\gamma{z},s)$}}(v)
		\Big)	
		\bigg)
	\end{eqnarray}
	for 
	$f\in\mathcal{A}
	\big(
	L^2 ([0,1], \mathscr{X}_{\Gamma})
	\big)$ 
	and
	$(\theta,v)\in \mathbb{R} \times 
	L^2 ([0,1], \mathscr{X}_{\Gamma})$,
	where the $*$-isomorphism  $\lambda^{\scalebox{0.7}{$(\gamma,s)$}}_*$ 
	is constructed as in \eqref{eq:commutativeofriowithcliffalgebrariojdklk}.
	The $\Gamma$-$\alpha_s$-action is given by
	\begin{eqnarray}
		\label{eq:theGammaactionsontheenlargedalgebraselementraycase}
		& 
		(\gamma \cdot_{\alpha_s} f)(z)=
		\mathcal{U}^{\scalebox{0.7}{$(\gamma,\gamma^{-1}z,s)$}}
		\big(
		f(\gamma^{-1}z)
		\big) 
	\end{eqnarray}
	for 
	$\gamma\in\Gamma$,
	$z\in {Z}$ and
	$f\in 
	\mathcal{A}
	\Big(
	C_{b}
	\big(
	Z, ~
	L^2 ([0,1], \mathscr{X}_{\Gamma})
	\big)
	\Big)$.
	We insert the symbol $\alpha_s$  above to specify the action depending on  $s\in[0,1]$.
	
	\begin{rmk}\label{taizongyaolewoqunitayejdksjfklsd}
		For $s=1$,  $\mathcal{U}^{\scalebox{0.7}{$(\gamma,z,1)$}}$ 
		encodes the original map  $\mathcal{U}^{\scalebox{0.6}{$(\gamma,z)$}}$.
		In contrast, for $s=0$, 
		$\mathcal{U}^{\scalebox{0.7}{$(\gamma,z,0)$}}$ is the identity map for all $\gamma\in\Gamma$ and $z\in {Z}$.      
		Hence, 
		for $s=1$, the  $\Gamma$-$\alpha_1$-action on 
		$\mathcal{A}
		\Big(
		C_{b}
		\big(
		Z, ~
		L^2 ([0,1], \mathscr{X}_{\Gamma})
		\big)
		\Big)$ 
		encodes the original $\Gamma$-action on 
		$\mathcal{A}
		\big(
		C_{b}(Z, \mathscr{X}_{\Gamma})
		\big)$.
		In contrast, for $s=0$, the $\Gamma$-$\alpha_0$-action on  
		$\mathcal{A}
		\Big(
		C_{b}
		\big(
		Z, ~
		L^2 ([0,1], \mathscr{X}_{\Gamma})
		\big)
		\Big)$ 
		is induced only by precomposition with the action of $\Gamma$ on $Z$, while leaving the $\mathcal{A}\big(L^2 ([0,1], \mathscr{X}_{\Gamma})\big)$-values of the functions unchanged.
	\end{rmk}

	We now point out the following important observation.
	\begin{rmk}
		The family of $\Gamma$-$\alpha_s$-actions on $\mathcal{A}
		\Big(
		C_{b}
		\big(
		Z, ~
		L^2 ([0,1], \mathscr{X}_{\Gamma})
		\big)
		\Big)$
		is not continuous with respect to $s\in[0,1]$, in the sense that
		there exist
		$f \in 
		\mathcal{A}
		\Big(
		C_{b}
		\big(
		Z, ~
		L^2 ([0,1], \mathscr{X}_{\Gamma})
		\big)
		\Big)$  
		and $\gamma\in\Gamma$ 
		such that the map $s\mapsto\gamma\cdot_{\alpha_s} f$
		is not continuous on  
		$[0,1]$. 
		This failure reflects the fact  that
		$C_{b}
		\big(
		Z, ~
		L^2 ([0,1], \mathscr{X}_{\Gamma})
		\big)$ 
		is too large to support a continuous deformation  of group  actions.
		To remedy this problem,
		we shall construct a $\Gamma$-$C^*$-subalgebra of 
		$\mathcal{A}
		\Big(
		C_{b}
		\big(
		Z, ~
		L^2 ([0,1], \mathscr{X}_{\Gamma})
		\big)
		\Big)$ 
		generated by a suitably chosen (not necessarily closed) subspace of 
		$C_{b}
		\big(
		Z, ~
		L^2 ([0,1],\mathscr{X}_{\Gamma})
		\big)$. 
	\end{rmk}

	\begin{defn}\label{Simplefunctionsjdlksjafkjsdf}
		Given finitely many functions  	
		$f_{1}, \cdots, f_{k}\in 
		C_{b}(Z,\mathscr{X}_{\Gamma})$ 
		and a finite measurable partition 
		$\{E_1, \cdots, E_k\}$ of 
		$[0,1]$,
		we define a function
		$$
		F_{_{\scalebox{0.6}{$(f_{1},\cdots,f_{k}; E_{1},\cdots,E_{k})$}}}: Z\longrightarrow 
		L^2 ([0,1], \mathscr{X}_{\Gamma})
		$$
		by
		\[
		\Big(
		F_{_{\scalebox{0.6}{$(f_{1},\cdots, f_{k}; E_{1},\cdots,E_{k})$}}}
		(z)
		\Big)(t)= 
		\begin{cases}
			f_{1}(z), & t\in E_{1}, \\
			f_{2}(z), & t\in E_{2}, \\
			\hspace{0.4cm} \vdots    &  
			\hspace{0.4cm} \vdots \\
			f_{k}(z), & t\in E_{k}.
		\end{cases}
		\]
		We call it  a
		\emph{uniform simple function generated by
			$f_{1},  \cdots, f_{k}$ over 
			$E_{1},  \cdots, E_{k}$}.  
	\end{defn}

	By definition,   
	$$
	F_{_{\scalebox{0.6}{$(f_{1},\cdots, f_{k}; E_{1},\cdots,E_{k})$}}}
	\in 
	C_{b}
	\big(
	Z,~
	L^2 ([0,1], \mathscr{X}_{\Gamma})
	\big).
	$$
	We denote by 
	$C_{\rm Sim}
	\big(
	Z,~
	L^2 ([0,1], \mathscr{X}_{\Gamma})
	\big)$ 
	the subspace of 
	$C_{b}
	\big(
	Z,~
	L^2 ([0,1], \mathscr{X}_{\Gamma})
	\big)$ 
	consisting of all uniform simple functions.  
	Note that it is not a closed subspace.

	\begin{rmk}\label{alpha-s-invariant}
		It follows directly from the definition that 
		$C_{\rm Sim}
		\big(
		Z,~
		L^2 ([0,1], \mathscr{X}_{\Gamma})
		\big)$
		is $\Gamma$-$\alpha_s$-invariant for all 
		$s\in [0,1]$, i.e.,
		\[ \quad 
		\gamma\cdot_{\alpha_s}f 
		\in  C_{\rm Sim}
		\big(
		Z,~
		L^2 ([0,1], \mathscr{X}_{\Gamma})
		\big) \quad 
		\text{for all } 
		s\in[0,1],\ 
		\gamma\in\Gamma,\  
		f\in 
		C_{\rm Sim}
		\big(
		Z,~
		L^2 ([0,1], \mathscr{X}_{\Gamma})
		\big).
		\] 
	\end{rmk}

	\begin{defn}\label{algebrasgeneratedbysimplefunctions}
		Define 
		$\mathcal{A}
		\Big(
		C_{\rm Sim}
		\big(
		Z, ~
		L^2 ([0,1], \mathscr{X}_{\Gamma})
		\big)
		\Big)$
		to be the $C^*$-subalgebra of 
		$C_{b}
		\Big(
		Z, ~
		\mathcal{A}\big( 
		L^2 ([0,1], \mathscr{X}_{\Gamma})
		\big)
		\Big)$
		generated by
		\begin{eqnarray*}
			\Big\{
			\beta_{\xi}(f) 
			~ \Big| ~ 
			f\in\mathcal{S}, 
			~ \xi\in 
			C_{\rm Sim}
			\big(
			Z, ~
			L^2 ([0,1], \mathscr{X}_{\Gamma})
			\big)
			\Big\},
		\end{eqnarray*}
		where $\beta_{\xi}$ is given by \eqref{eq:thefunctionalcalculusforenlargedspacesecondoneforsections}.
		Moreover,  for $R>0$, 
		we set	
		$$
		\mathcal{A}_{R}
		\Big(
		C_{\rm Sim}
		\big(
		Z, ~
		L^2 ([0,1], \mathscr{X}_{\Gamma})
		\big)
		\Big) = \mathcal{A}
		\Big(
		C_{\rm Sim}
		\big(
		Z, ~
		L^2 ([0,1], \mathscr{X}_{\Gamma})
		\big)
		\Big)
		\cap
		\mathcal{A}_{R}
		\Big(
		C_{b}
		\big(
		Z, ~
		L^2 ([0,1], \mathscr{X}_{\Gamma})
		\big)
		\Big).
		$$
	\end{defn}

	By definition,   
	\begin{eqnarray*}
		\mathcal{A}
		\Big(
		C_{\rm Sim}
		\big(
		Z, ~
		L^2 ([0,1], \mathscr{X}_{\Gamma})
		\big)
		\Big) =
		\overline{\bigcup_{R>0} 
			\mathcal{A}_{R}
			\Big(
			C_{\rm Sim}
			\big(
			Z, ~
			L^2 ([0,1], \mathscr{X}_{\Gamma})
			\big)
			\Big)}.
	\end{eqnarray*}
	
	It follows from 
	\eqref{eq:rulesforgammaactionforelementary}  and Remark \ref{alpha-s-invariant} 
	that, for each $s\in[0,1]$, 
	the  
	$\Gamma$-$\alpha_s$-action on 
	\[\mathcal{A}
	\Big(
	C_{b}
	\big(
	Z, ~
	L^2 ([0,1], \mathscr{X}_{\Gamma})
	\big)
	\Big)\]
	given by \eqref{eq:theGammaactionsontheenlargedalgebraselementraycase} 
	restricts to a $\Gamma$-action on
	the $C^*$-subalgebra
	\[\mathcal{A}
	\Big(
	C_{\rm Sim}
	\big(
	Z, ~
	L^2 ([0,1], \mathscr{X}_{\Gamma})
	\big)
	\Big).
	\]
	The restricted action will also be called the
	$\Gamma$-$\alpha_s$-action.
	The arguments of Remark \ref{taizongyaolewoqunitayejdksjfklsd} remain valid for this subalgebra.

	\begin{rmk}\label{aneasyembed}
		The $C^*$-algebra 
		$\mathcal{A}
		\big(
		L^2 ([0,1], \mathscr{X}_{\Gamma})
		\big)$
		embeds canonically into
		$C_{b}
		\Big(
		Z, ~
		\mathcal{A}\big( 
		L^2 ([0,1], \mathscr{X}_{\Gamma})
		\big)
		\Big)$
		as the subalgebra of constant functions.
		Since simple functions are dense in  
		$L^2 ([0,1], \mathscr{X}_{\Gamma})$,  
		the  algebra 
		$\mathcal{A}
		\big(
		L^2 ([0,1], \mathscr{X}_{\Gamma})
		\big)$
		can be generated by the set
		\begin{eqnarray*}
			\Big\{
			\beta_{g}
			(f) ~ \Big| ~ 
			f\in\mathcal{S}, ~
			g\in L^2 ([0,1], \mathscr{X}_{\Gamma}) \text{ is a simple function}
			\Big\},
		\end{eqnarray*}
		where $\beta$ is given by \eqref{eq:thefunctionalcalculusforenlargeddpacefirstone}.
		Consequently, under the above embedding, every element of  $\mathcal{A}
		\big(
		L^2 ([0,1], \mathscr{X}_{\Gamma})
		\big)$
		is mapped into 
		$$\mathcal{A}
		\Big(
		C_{\rm Sim}
		\big(
		Z, ~
		L^2 ([0,1], \mathscr{X}_{\Gamma})
		\big)
		\Big),$$ 
		and we thereby obtain a well-defined embedding
		$$
		\iota: 
		\mathcal{A}
		\big(
		L^2 ([0,1], \mathscr{X}_{\Gamma})
		\big)
		\longrightarrow 
		\mathcal{A}
		\Big(
		C_{\rm Sim}
		\big(
		Z, ~
		L^2 ([0,1], \mathscr{X}_{\Gamma})
		\big)
		\Big).
		$$  
		By Remark \ref{taizongyaolewoqunitayejdksjfklsd},
		the $\Gamma$-$\alpha_0$-action on  
		$\mathcal{A}
		\Big(
		C_{\rm Sim}
		\big(
		Z, ~
		L^2 ([0,1], \mathscr{X}_{\Gamma})
		\big)
		\Big)$ 
		is induced only by precomposition with the action of $\Gamma$ on $Z$, while leaving the $\mathcal{A}\big(L^2 ([0,1], \mathscr{X}_{\Gamma})\big)$-values of the functions unchanged.
		Hence, $\iota$ is a $\Gamma$-equivariant $*$-homomorphism with respect to the trivial $\Gamma$-action on 
		$\mathcal{A}
		\big(
		L^2 ([0,1], \mathscr{X}_{\Gamma})
		\big)$ 
		and the  $\Gamma$-$\alpha_0$-action on $\mathcal{A}
		\Big(
		C_{\rm Sim}
		\big(
		Z, ~
		L^2 ([0,1], \mathscr{X}_{\Gamma})
		\big)
		\Big)$. 
	\end{rmk}

	The $K$-theory of this new algebra is also computable.
	The following result is
	an analogue of  
	Lemma \ref{isomorphismbetweenKtheory-babycase}.

	\begin{prop}\label{isomorphismnewone}
		For any Property $(H)$ Banach space $\mathscr{X}$,	
		the embedding  \begin{eqnarray*}
			\iota:
			\mathcal{A}
			\big(
			L^2 ([0,1], \mathscr{X}_{\Gamma})
			\big)
			&\longrightarrow&
			\mathcal{A}
			\Big(
			C_{\rm Sim}
			\big(
			Z, ~
			L^2 ([0,1], \mathscr{X}_{\Gamma})
			\big)
			\Big),
		\end{eqnarray*} 
		given in Remark \ref{aneasyembed},
		induces an isomorphism
		\begin{eqnarray*}
			\iota_*:
			K_*\Big(
			\mathcal{A}
			\big(
			L^2 ([0,1], \mathscr{X}_{\Gamma})
			\big)
			\Big)
			&\longrightarrow&
			K_*\bigg(
			\mathcal{A}
			\Big(
			C_{\rm Sim}
			\big(
			Z, ~
			L^2 ([0,1], \mathscr{X}_{\Gamma})
			\big)
			\Big)
			\bigg).
		\end{eqnarray*} 
	\end{prop}
	\begin{proof}
		The proof of Lemma \ref{isomorphismbetweenKtheory-babycase} 	 may be repeated verbatim, but with 
		$\mathcal{A}
		(\mathscr{X})$ 
		replaced
		by 
		$\mathcal{A}
		\big(
		L^2 ([0,1], \mathscr{X}_{\Gamma})
		\big)$, 
		and
		$\mathcal{A}
		\big(
		C_{b}(Y,\mathscr{X})
		\big)$
		replaced  by
		$\mathcal{A}
		\Big(
		C_{\rm Sim}
		\big(
		Z, ~
		L^2 ([0,1], \mathscr{X}_{\Gamma})
		\big)
		\Big)$.
	\end{proof}

	We now consider the algebra that  
	contains the information about continuous deformation of $\Gamma$-$\alpha_{s}$-actions.

	\begin{lem}\label{continuityofmodifiedgroupaction} 
		For all 
		$\gamma\in\Gamma$ 
		and
		$\varphi\in\mathcal{A}
		\Big(
		C_{\rm Sim}
		\big(
		Z, ~
		L^2 ([0,1], \mathscr{X}_{\Gamma})
		\big)
		\Big)$,
		the function
		$$
		[0,1]\ni {s} \ \longmapsto \  \gamma\cdot_{\alpha_{s}} \varphi
		\in\mathcal{A}
		\Big(
		C_{\rm Sim}
		\big(
		Z, ~
		L^2 ([0,1], \mathscr{X}_{\Gamma})
		\big)
		\Big)
		$$
		is continuous on $[0,1]$.
	\end{lem}
	\begin{proof}
		By a standard approximation argument, 
		it suffices to verify the assertion
		for 
		$\varphi=\beta_{\xi}(f)$,
		where 
		$f\in\mathcal{S}$ 
		and 
		$\xi\in 
		C_{\rm Sim}
		\big(
		Z, ~
		L^2 ([0,1], \mathscr{X}_{\Gamma})
		\big)$.
		We may write $\xi=F_{_{\scalebox{0.6}{$(f_{1},\cdots, f_{k}; E_{1},\cdots,E_{k})$}}}$,
		where 
		$f_{1},f_{2},\cdots, f_{k}\in 
		C_{b}(Z,\mathscr{X}_{\Gamma})$ and
		$E_1,E_2,\cdots, E_k\subseteq[0,1]$
		are pairwise disjoint measurable subsets satisfying
		$\bigcup_{j=1}^{k}E_j=[0,1]$.
		
		Fix $\gamma\in\Gamma$.
		Since each
		$f_j$
		is bounded,
		we have 
		\begin{eqnarray*}
			M:=
			\sup_{z\in Z}
			\max_{1\leq j\leq k}
			\Big\|
			f_j(\gamma^{-1}z)-U^{\scalebox{0.7}{$(\gamma,\gamma^{-1}z)$}}\big(
			f_j(\gamma^{-1}z)
			\big) 
			\Big\|
			<+\infty.
		\end{eqnarray*}
		In view of \eqref{eq:theationsforenlargedspaces},
		\eqref{eq:theaffineisometricsforenlargedspaceelementaryone} and Definition \ref{Simplefunctionsjdlksjafkjsdf}, we obtain 
		\begin{equation}
			\label{eq:woqunidayedekdlkjalkdjk}
			\begin{array}{cll} 
				&& ~ \Big\| \big(\gamma\cdot_{\alpha_s}\xi\big)(z)-\big(\gamma\cdot_{\alpha_t}\xi\big)(z)
				\Big\|
				\\
				&=& 
				\bigg\|
				U^{\scalebox{0.7}{$(\gamma,\gamma^{-1}z,s)$}}
				\Big(F_{_{\scalebox{0.7}{$(f_{1},\cdots, f_{k}; E_{1},\cdots,E_{k})$}}}(\gamma^{-1}z)
				\Big)-
				U^{\scalebox{0.7}{$(\gamma,\gamma^{-1}z,t)$}}
				\Big( F_{_{\scalebox{0.7}{$(f_{1},\cdots, f_{k}; E_{1},\cdots,E_{k})$}}}(\gamma^{-1}z)
				\Big)
				\bigg\|
				\\
				&\leq&
				|s-t|^{\frac{1}{2}}
				\cdot 	
				\max\limits_{1\leq j\leq k}
				\Big\|
				f_j(\gamma^{-1}z)-U^{\scalebox{0.7}{$(\gamma,\gamma^{-1}z)$}}\big(
				f_j(\gamma^{-1}z)
				\big) 
				\Big\|
				\\
				&\leq&
				M\cdot |s-t|^{\frac{1}{2}}
			\end{array}
		\end{equation}
		for all $z\in Z$ 
		and $s,t\in[0,1]$. 
		Note that 
		$\Big\{ 
		\big(\gamma\cdot_{\alpha_s}\xi\big)(z)
		~\Big|~ z\in Z, ~ s\in [0,1]
		\Big\}$ 
		is a bounded subset in 
		$L^2 ([0,1], \mathscr{X}_{\Gamma})$.
		Hence,
		the assertion follows from  \eqref{eq:woqunidayedekdlkjalkdjk}, 
		Proposition \ref{localuniformlycontinuous} and
		\begin{eqnarray*}
			\Big\|
			\gamma\cdot_{\alpha_s} \beta_{\xi}(f)-
			\gamma\cdot_{\alpha_t} \beta_{\xi}(f)
			\Big\|
			&=& 
			\Big\|
			\beta_{\gamma\cdot_{\alpha_s}\xi}(f)-
			\beta_{\gamma\cdot_{\alpha_t}\xi}(f)
			\Big\|
			\\ 
			&=&
			\sup_{z\in Z} 	
			\Big\|
			\beta_{
				\scalebox{0.7}{$\big(\gamma\cdot_{\alpha_s}\xi\big)(z)$}
			}(f)-
			\beta_{
				\scalebox{0.7}{$\big(\gamma\cdot_{\alpha_t}\xi\big)(z)$}
			}(f)
			\Big\|.
		\end{eqnarray*}	
		This completes the proof.
	\end{proof}

	For  every
	$\gamma\in\Gamma$
	and
	$f\in 
	C\bigg(
	[0,1],~
	\mathcal{A}
	\Big(
	C_{\rm Sim}
	\big(
	Z, ~
	L^2 ([0,1], \mathscr{X}_{\Gamma})
	\big)
	\Big)
	\bigg)$, 
	define 
	$$
	\gamma\cdot_{\alpha_{[0,1]}} f: [0,1] \longrightarrow \mathcal{A}
	\Big(
	C_{\rm Sim}
	\big(
	Z, ~
	L^2 ([0,1], \mathscr{X}_{\Gamma})
	\big)
	\Big)
	$$
	by 
	\begin{eqnarray}
		\label{eq:theparamitrizedactionsoncontinuousalgebra}
		(\gamma\cdot_{\alpha_{[0,1]}} f)(t) 
		= \gamma\cdot_{\alpha_{t}} f(t)
	\end{eqnarray}
	for 
	$t\in[0,1]$.

	\begin{lem}\label{keylemmafordeformation}
		For all 
		$\gamma\in\Gamma$
		and
		$f\in 
		C\bigg(
		[0,1],~
		\mathcal{A}
		\Big(
		C_{\rm Sim}
		\big(
		Z, ~
		L^2 ([0,1], \mathscr{X}_{\Gamma})
		\big)
		\Big)
		\bigg)$, 
		we have
		$$\gamma\cdot_{\alpha_{[0,1]}} f \in  C\bigg(
		[0,1],~
		\mathcal{A}
		\Big(
		C_{\rm Sim}
		\big(
		Z, ~
		L^2 ([0,1], \mathscr{X}_{\Gamma})
		\big)
		\Big)
		\bigg).
		$$
	\end{lem}	
	\begin{proof}
		For  any
		$s,t\in[0,1]$,	
		\begin{eqnarray*}
			\Big\|
			\big(
			\gamma\cdot_{\alpha_{[0,1]}} f
			\big)(s)
			-\big(
			\gamma\cdot_{\alpha_{[0,1]}} f
			\big)(t)
			\Big\|
			&=&
			\Big\| 
			\gamma\cdot_{\alpha_s} f(s)
			-\gamma\cdot_{\alpha_t} f(t)
			\Big\|
			\\
			&\leq&
			\Big\| 
			\gamma\cdot_{\alpha_s} f(s)
			-	\gamma\cdot_{\alpha_t} f(s)
			\Big\| 
			+
			\Big\| 
			\gamma\cdot_{\alpha_t} f(s)
			-\gamma\cdot_{\alpha_t} f(t)
			\Big\|
			\\
			&=&
			\Big\| 
			\gamma\cdot_{\alpha_s} f(s)
			-	\gamma\cdot_{\alpha_t} f(s)
			\Big\| 
			+
			\Big\| 
			f(s)
			- f(t)
			\Big\|.
		\end{eqnarray*}		
		This inequality, combined with Lemma \ref{continuityofmodifiedgroupaction}  and the continuity of $f$,  
		implies  that	
		$\gamma\cdot_{\alpha_{[0,1]}} f$
		is continuous at  $s\in[0,1]$.
	\end{proof}

	Combined with Lemma \ref{keylemmafordeformation},  \eqref{eq:theparamitrizedactionsoncontinuousalgebra} gives rise to a  $\Gamma$-$\alpha_{[0,1]}$-action on the $C^*$-algebra
	$$C\bigg(
	[0,1],~
	\mathcal{A}
	\Big(
	C_{\rm Sim}
	\big(
	Z, ~
	L^2 ([0,1], \mathscr{X}_{\Gamma})
	\big)
	\Big)
	\bigg).$$
	Moreover, for each $s\in[0,1]$,  
	the  evaluation map  
	\begin{equation}
		\label{eq:vlationsatsforthespecialcaselkjrek}
		{\rm ev}_s:  
		C\bigg(
		[0,1],~
		\mathcal{A}
		\Big(
		C_{\rm Sim}
		\big(
		Z, ~
		L^2 ([0,1], \mathscr{X}_{\Gamma})
		\big)
		\Big)
		\bigg) 
		\longrightarrow  
		\mathcal{A}
		\Big(
		C_{\rm Sim}
		\big(
		Z, ~
		L^2 ([0,1], \mathscr{X}_{\Gamma})
		\big)
		\Big),
	\end{equation}
	which sends $f$ to $f(s)$, 
	intertwines the $\Gamma$-$\alpha_{[0,1]}$-action  
	and
	$\Gamma$-$\alpha_{s}$-action.

We remark that 
the algebra
$$C\bigg(
[0,1],~
\mathcal{A}
\Big(
C_{\rm Sim}
\big(
Z, ~
L^2 ([0,1], \mathscr{X}_{\Gamma})
\big)
\Big)
\bigg)$$ 
encodes the information of a continuous deformation of group actions. 
The evaluation maps in \eqref{eq:vlationsatsforthespecialcaselkjrek}, 
which  intertwine  the 
$\Gamma$-$\alpha_{[0,1]}$-action  
with each
$\Gamma$-$\alpha_{s}$-action,
will play a crucial role in Section 9.

	\subsection{The larger proper $\Gamma$-$C^*$-algebra}
	
	Based on the results of the previous two subsections, 
	we construct the larger proper 
	$\Gamma$-$C^*$-algebra for a countable discrete group that 
	admits a
	coarse embedding   with finite complexity 
	into an
	$\ell^2$-direct sum of infinitely many  Property 
	$(H)$
	Banach spaces.

	Suppose that
	$h:\Gamma\to\mathscr{E}=
	\bigoplus_{p=1}^{\infty}
	\mathscr{X}_p$ 
	is a coarse embedding  with finite complexity, 
	where each
	$\mathscr{X}_p$
	is a Property 
	$(H)$
	Banach space with a Property $(H)$ map 
	$\phi_p: 
	S(\mathscr{X}_p) \to S(\mathscr{H})$.
	As discussed at the beginning of Section 5.2,
	we assume, without loss of generality, that $\phi_p$ is Lipschitz for all 
	$p \geq 2$.
	
	We now apply the results of the previous two subsections to the present setting.
	Since each 
	$\mathfrak{h}_{k}: \Gamma \to \mathscr{E}_{k}$
	is  bornologous, 
	the construction in Section 7.1  
	produces  
	a family of affine isometries on 
	$L^2 ([0,1], \mathscr{E}_{k,\Gamma})$,
	parametrized by 
	$\gamma\in\Gamma, z\in Z$ and $s\in[0,1]$.
	More precisely, 
	for every 
	$k \geq 1$,
	$\gamma\in\Gamma$, 
	$z \in {Z}$ and $s\in[0,1]$,
	we have 
	\begin{center}
		$\zeta_{k}^{\scalebox{0.7}{$(\gamma,z,s)$}}\in 
		L^2 ([0,1],\mathscr{E}_{k,\Gamma})$, \quad 
		$\lambda_{k}^{\scalebox{0.7}{$(\gamma,s)$}}: 	
		L^2 ([0,1], \mathscr{E}_{k,\Gamma})
		\longrightarrow 
		L^2 ([0,1], \mathscr{E}_{k,\Gamma})$, \quad  and 
		$U_{k}^{\scalebox{0.7}{$(\gamma,z,s)$}}:
		L^2 
		([0,1], \mathscr{E}_{k,\Gamma})
		\longrightarrow 
		L^2 
		([0,1], \mathscr{E}_{k,\Gamma})$,
	\end{center}
	which are given by 
	\begin{equation}
		\label{eq:thecocyclesfortheenlargedspacethekthlevels}
		\zeta_{k}^{\scalebox{0.7}{$(\gamma,z,s)$}}(t)
		= 
		\begin{cases}
			\zeta_{k}^{\scalebox{0.7}{$(\gamma,z)$}},& t\in[0,s],
			\\
			0, & t\in(s,1],
		\end{cases}
	\end{equation}
	\begin{equation*}
		\big(
		\lambda_{k}^{\scalebox{0.7}{$(\gamma,s)$}}(f)
		\big)(t)= 
		\begin{cases}
			\lambda_{k}^{\scalebox{0.7}{$\gamma$}}
			\big(
			f(t)
			\big),& t\in[0,s],
			\\
			f(t), & t\in(s,1],
		\end{cases}
	\end{equation*}
	and
	\begin{equation}
		\label{eq:thekthlevelaffineisometriesforenlargeddspaces}
		\big(
		U_{k}^{\scalebox{0.7}{$(\gamma,z,s)$}}(f)
		\big)(t)= 
		\begin{cases}
			U_{k}^{\scalebox{0.7}{$(\gamma,z)$}}
			\big(
			f(t)
			\big), & t\in[0,s],
			\\
			f(t), & t\in(s,1]
		\end{cases}
	\end{equation}
	for 
	$f\in 
	L^2 ([0,1], \mathscr{E}_{k,\Gamma})$. 
	It follows from \eqref{eq:affineisometriesforthekthlevel}  that 
	\begin{eqnarray}
		\label{eq:thekthlevelaffineisometriesforenlargedspaces}
		U_{k}^{\scalebox{0.7}{$(\gamma,z,s)$}}(f)
		= 
		\lambda_{k}^{\scalebox{0.7}{$(\gamma,s)$}}(f) + \zeta_{k}^{\scalebox{0.7}{$(\gamma,z,s)$}}
	\end{eqnarray}
	for all 
	$f\in 
	L^2 ([0,1], \mathscr{E}_{k,\Gamma})$. 
	By Remark \ref{alpha-s-invariant}, 
	for every $k\geq1$ and $s\in[0,1]$,
	the family of affine isometries 
	$$
	\Big\{	U_{k}^{\scalebox{0.7}{$(\gamma,z,s)$}} ~\Big|~ \gamma\in\Gamma, z\in {Z}
	\Big\}
	$$
	induces a $\Gamma$-$\alpha_s$-action
	on 
	$C_{\rm Sim}
	\big(
	Z, ~
	L^2 ([0,1], \mathscr{E}_{k,\Gamma})
	\big)$
	via 
	\begin{eqnarray}
		\label{eq:theGammasctionsparametrizedbysforthekthlevelsections}
		\big(
		\gamma\cdot_{\alpha_{s}} f
		\big)
		(z)
		=	U^{\scalebox{0.7}{$(\gamma,\gamma^{-1}z,s)$}}_{k}
		\big(
		f(\gamma^{-1}z)
		\big),
	\end{eqnarray}
	where $\gamma\in\Gamma$, $z\in {Z}$ and 
	$f\in C_{\rm Sim}
	\big(
	Z, ~
	L^2 ([0,1], \mathscr{E}_{k,\Gamma})
	\big)$.

	\begin{rmk}\label{bubiandehomotopyoftheendsofactiondlkfjkkkiejoiuxoksaleuaslkrso}
		As explained in Remarks \ref{degeneratesdfkjalsdjkfls} and \ref{bubiandehomotopyoftheendsofactiondlkfj}, 
		for $s=1$, the $\Gamma$-$\alpha_1$-action on  
		$C_{\rm Sim}
		\big(
		Z, ~
		L^2 ([0,1], \mathscr{E}_{k,\Gamma})
		\big)$ 
		encodes the original $\Gamma$-action on 
		$C_{b}(
		Z, \mathscr{E}_{k,\Gamma})$. 
		For $s=0$, the $\Gamma$-$\alpha_0$-action on 
		$C_{\rm Sim}
		\big(
		Z, ~
		L^2 ([0,1], \mathscr{E}_{k,\Gamma})
		\big)$  
		is induced only by precomposition with the action of $\Gamma$ on $Z$, while leaving the $L^2 ([0,1], \mathscr{E}_{k,\Gamma})$-values of the functions unchanged.
	\end{rmk}

	As noted in Remark \ref{properHmapsforiteratedconstructionkdjlkjl}, for each $p\geq1$,
	applying Proposition \ref{PropertyHmapforLpspaces} twice yields a
	Property $(H)$ map
	$$
	\mathcal{M}_{\mathcal{M}_{\phi_p}}: 
	S\big( 
	L^2 ([0,1], \mathscr{X}_{p,\Gamma})
	\big) 
	\longrightarrow 
	S\big( 
	L^2 ([0,1], \mathscr{H}_{\Gamma}) 
	\big)
	$$ 
	for the Property $(H)$ Banach space
	$L^2 ([0,1], \mathscr{X}_{p,\Gamma})$.

	We set 
	$$\phi_{p,\scalebox{0.6}{$[0,1]$}} = \mathcal{M}_{\mathcal{M}_{\phi_p}}.$$
	For $p\geq2$,  the map
	$\phi_p$ is Lipschitz.  
	Thus, by applying Proposition \ref{LipofextenedMazurmap} twice,
	we know that 
	$\phi_{p,\scalebox{0.6}{$[0,1]$}}$
	is Lipschitz with Lipschitz constant at most
	$4 \cdot  L_{\phi_p}+3$. 
	Moreover, by Lemma \ref{3lip}, 
	the Lipschitz constant of its scalar extension $\phi_{p,\scalebox{0.6}{$[0,1]$}}^I$ 
	is at most
	$8 \cdot  L_{\phi_p}+ 7$.
	For $p=1$, 
	let 
	$\phi_{\scalebox{0.6}{$1,[0,1]$}}^{\eta}:
	L^2 ([0,1], \mathscr{X}_{1,\Gamma})
	\to 
	L^2 ([0,1], \mathscr{H}_{\Gamma})$
	be an  
	$\eta$-scalar proper extension of $\phi_{\scalebox{0.6}{$1,[0,1]$}}$
	given by Corollary \ref{finitemodulusofcontinuity}.
	Denote by $\varPsi_{\scalebox{0.6}{$k,[0,1]$}}$
	the composition of the following maps
	\begin{equation}
		\label{eq:thepropertyHmapfortheKthlevelfortheenlargedspacesssss}
		\begin{array}{ccc}
			L^2 ([0,1], \mathscr{E}_{k,\Gamma}) & \xlongrightarrow{\cong}
			L^2 
			\big( 
			[0,1], \bigoplus_{p=1}^{k}
			\mathscr{X}_{p,\Gamma} 
			\big)& \xlongrightarrow{\cong}  \bigoplus_{p=1}^{k}
			L^2 ([0,1], \mathscr{X}_{p,\Gamma})
			\\
			&& \ \ \ \    \ \ \  
			\Bigg\downarrow \phi_{\scalebox{0.6}{$1,[0,1]$}}^{\eta} \oplus \cdots \oplus \phi_{\scalebox{0.6}{$k,[0,1]$}}^I 
			\\
			L^2 ([0,1],\mathscr{H}_{k,\Gamma})& \xlongleftarrow{\cong} 
			L^2 
			\big( 
			[0,1],
			\bigoplus_{p=1}^{k}
			\mathscr{H}_{\Gamma} 
			\big)& \xlongleftarrow{\cong} 
			\bigoplus_{p=1}^{k}
			L^2 ([0,1], \mathscr{H}_{\Gamma}),
		\end{array}
	\end{equation}
	where the isometric isomorphisms are canonical, and 
	$\phi_{\scalebox{0.6}{$1,[0,1]$}}^{\eta} \oplus \cdots \oplus \phi_{\scalebox{0.6}{$k,[0,1]$}}^I$ is the direct sum of
	$\phi_{\scalebox{0.6}{$1,[0,1]$}}^{\eta}$
	and 
	$\phi_{p,\scalebox{0.6}{$[0,1]$}}^I$ ($2\leq p \leq k$). 
	As in the argument  above Lemma \ref{inequalityforPsi1}, 
	$\varPsi_{\scalebox{0.6}{$k, [0,1]$}}$ is a proper extension of the Property $(H)$ map 
	$$
	\varPsi_{\scalebox{0.6}{$k, [0,1]$}}\Big|_{\scalebox{0.7}{$S\big(L^2([0,1],\mathscr{E}_{k,\Gamma})\big)$}}.
	$$

	The following lemma plays a key role in the construction of the $\Gamma$-equivariant asymptotic morphism in Proposition \ref{asymptoticGammathemostkeypoint}.

	\begin{lem}\label{inequalityforPsi}
		For every $k\geq1$
		and  
		$v,w
		\in
		L^2 ([0,1], \mathscr{E}_{k,\Gamma})$, 	
		\begin{eqnarray*}
			&& 
			\Big\| \varPsi_{\scalebox{0.6}{$k, [0,1]$}}
			(v+w)
			-
			\varPsi_{\scalebox{0.6}{$k, [0,1]$}}
			(v)
			\Big\|
			\leq
			\sqrt{	
				\bigg(
				\omega_{\scalebox{0.7}{$
						\phi_{\scalebox{0.6}{$1,[0,1]$}}^{\eta}
						$}}
				\Big(
				\big\|
				\varrho_{k,1}(w)
				\big\|
				\Big)
				\bigg)^2	
				+
				64 \cdot \sum\limits_{p=2}^{k}
				(L_{\phi_p}+1)^2
				\cdot
				\big\|
				\varrho_{k,p}(w)
				\big\|^2
			}\  ,
		\end{eqnarray*} 
		where 
		$\varrho_{k}: 
		L^2 ([0,1], \mathscr{E}_{k,\Gamma}) 
		\to  
		\bigoplus_{p=1}^{k}
		L^2 ([0,1], \mathscr{X}_{p,\Gamma})$ 
		is the isometric isomorphism in \eqref{eq:thepropertyHmapfortheKthlevelfortheenlargedspacesssss}
		and 
		$\varrho_{k,p}: 
		L^2 ([0,1], \mathscr{E}_{k,\Gamma}) 
		\xlongrightarrow{\varrho_{k}}  
		\bigoplus_{p=1}^{k}
		L^2 ([0,1], \mathscr{X}_{p,\Gamma})
		\to 
		L^2 ([0,1], \mathscr{X}_{p,\Gamma})$ 
		is the projection onto the $p$-th component.
	\end{lem}
	\begin{proof}
		For every $k\geq1$ 
		and 
		$v,w
		\in
		L^2 ([0,1], \mathscr{E}_{k,\Gamma})$, 	
		\begin{eqnarray*}
			&& 
			\Big\| \varPsi_{\scalebox{0.6}{$k, [0,1]$}}
			(v+w)
			-
			\varPsi_{\scalebox{0.6}{$k, [0,1]$}}
			(v)
			\Big\|^2 
			\\
			&=&
			\bigg\|
			\Big(
			\phi_{\scalebox{0.6}{$1,[0,1]$}}^{\eta} \oplus \cdots \oplus \phi_{\scalebox{0.6}{$k,[0,1]$}}^I 
			\Big) 
			\big( 
			\varrho_{k}(v+w)
			\big) 
			- \Big(
			\phi_{\scalebox{0.6}{$1,[0,1]$}}^{\eta} \oplus \cdots \oplus \phi_{\scalebox{0.6}{$k,[0,1]$}}^I 
			\Big) 
			\big( 
			\varrho_{k}(v)
			\big) 
			\bigg\|^2
			\\
			&=&
			\Big\| 
			\phi_{\scalebox{0.6}{$1,[0,1]$}}^{\eta}	\big( 
			\varrho_{k,1}(v)+\varrho_{k,1}(w)
			\big) 
			-
			\phi_{\scalebox{0.6}{$1,[0,1]$}}^{\eta}
			\big( 
			\varrho_{k,1}(v)
			\big) 
			\Big\|^2 
			+
			\sum\limits_{p=2}^{k}
			\Big\|
			\phi_{p,\scalebox{0.6}{$[0,1]$}}^I 
			\big( 
			\varrho_{k,p}(v)+\varrho_{k,p}(w)
			\big) 
			-
			\phi_{p,\scalebox{0.6}{$[0,1]$}}^I 
			\big( 
			\varrho_{k,p}(v)
			\big) 
			\Big\|^2
			\\
			&\leq& 
			\bigg(
			\omega_{\scalebox{0.7}{$
					\phi_{\scalebox{0.6}{$1,[0,1]$}}^{\eta}
					$}}
			\Big(
			\big\|
			\varrho_{k,1}(w)
			\big\|
			\Big)
			\bigg)^2	
			+
			64 \cdot \sum\limits_{p=2}^{k}
			(L_{\phi_p}+1)^2
			\cdot
			\big\|
			\varrho_{k,p}(w)
			\big\|^2.
		\end{eqnarray*} 
		This completes the proof.
	\end{proof}

	\begin{defn}\label{Cliffordmap}
		We define
		$\mathfrak{B}_{\scalebox{0.6}{$k,[0,1]$}} :
		L^2 ([0,1], \mathscr{E}_{k,\Gamma})
		\rightarrow
		\text{Cliff}_{\mathbb{C}}
		\big(
		L^2 ([0,1], \mathscr{H}_{k,\Gamma}) 
		\big)$
		by 
		$v\mapsto 
		C\big(
		\varPsi_{\scalebox{0.6}{$k, [0,1]$}}
		(v)
		\big)$. 
		Moreover,
		for any  
		$v_{0}\in 
		L^2 ([0,1], \mathscr{E}_{k,\Gamma})$,  
		we define $\mathfrak{B}_{\scalebox{0.6}{$k,[0,1]$}}^{v_{0}}: L^2 ([0,1], \mathscr{E}_{k,\Gamma})
		\rightarrow
		\text{Cliff}_{\mathbb{C}}
		\big(
		L^2 ([0,1], \mathscr{H}_{k,\Gamma}) 
		\big)$ 
		by
		$\mathfrak{B}_{\scalebox{0.6}{$k,[0,1]$}}^{v_{0}}(v) = \mathfrak{B}_{\scalebox{0.6}{$k,[0,1]$}}(v-v_{0})$. 
	\end{defn}

	By  \eqref{eq:functionalcalculusone} and \eqref{eq:functionalcalculustwo},  
	we obtain the graded $*$-homomorphisms 
	\begin{equation}
		\label{eq:theenlargedfunctionalcalculusforthekthlevellsdjfk}
		\begin{array}{rcl}
			\beta_{k,v}: ~ \mathcal{S}&\longrightarrow &
			\mathcal{S}\widehat{\otimes} 
			C_{0}
			\Big(
			L^2 ([0,1], \mathscr{E}_{k,\Gamma}),~ 
			\text{Cliff}_{\mathbb{C}}
			\big(
			L^2 ([0,1], \mathscr{H}_{k,\Gamma}) 
			\big)
			\Big)
			\\
			f & \longmapsto & 
			f \big(
			\varTheta \widehat{\otimes} 1+1\widehat{\otimes} 		\mathfrak{B}_{\scalebox{0.6}{$k,[0,1]$}}^{v}
			\big),
		\end{array}
	\end{equation}
	and 
	\begin{equation}
		\label{eq:theenlargedfunctionalcalculusforthekthlevelsectionsttds}
		\begin{array}{rcl}
			\beta_{k,\xi}: ~ \mathcal{S} &\longrightarrow &
			\mathcal{F}
			\Big(
			Z, ~
			\mathcal{A}
			\big(
			L^2 ([0,1], \mathscr{E}_{k,\Gamma})
			\big)
			\Big)
			\\
			f &\longmapsto & 
			\Big\{
			z\mapsto 
			\beta_{k,\xi(z)}(f)
			\Big\},
		\end{array}
	\end{equation}
	where 
	$v\in L^2 ([0,1], \mathscr{E}_{k,\Gamma})$,
	$\xi\in 
	C_{b}
	\big(
	Z, ~
	L^2 ([0,1], \mathscr{E}_{k,\Gamma})
	\big)$,
	and 
	$Z$ is the metric space defined in \eqref{eq:simplicityforthebasespaceZ}.
	For $v=0$, we simply write $\beta_{k}$ in place of $\beta_{k,v}$.
	Likewise, if $\xi$ is the zero function, we write 
	$\beta_{k,{\bf 0}}$ in place of $\beta_{k,\xi}$.
	
	For each  $k \geq 1$, 
	following Section 7.2,
	we have the
	graded 
	$C^*$-algebras
	\begin{center}
		$\mathcal{A}
		\Big(
		C_{\rm Sim}
		\big(
		Z, ~
		L^2 ([0,1], \mathscr{E}_{k,\Gamma})
		\big)
		\Big)$ \quad 
		and \quad 
		$C\bigg(
		[0,1],~
		\mathcal{A}
		\Big(
		C_{\rm Sim}
		\big(
		Z, ~
		L^2 ([0,1], \mathscr{E}_{k,\Gamma})
		\big)
		\Big)
		\bigg)$.
	\end{center}
	Moreover,
	for every $s\in[0,1]$ and $k\geq1$,
	the family of affine isometries
	$$\Big\{ U_{k}^{\scalebox{0.7}{$(\gamma,z,s)$}}
	~\Big|~
	\gamma\in\Gamma, z\in Z
	\Big\}$$ 
	given by \eqref{eq:thekthlevelaffineisometriesforenlargeddspaces} or \eqref{eq:thekthlevelaffineisometriesforenlargedspaces}
	induces a
	family  of $*$-isomorphisms
	$$
	\Big\{	
	\mathcal{U}_{k}^{\scalebox{0.7}{$(\gamma,z,s)$}}: 	
	\mathcal{A}
	\big(
	L^2 ([0,1], \mathscr{E}_{k,\Gamma})
	\big)
	\longrightarrow 
	\mathcal{A}
	\big(
	L^2 ([0,1], \mathscr{E}_{k,\Gamma})
	\big)
	\Big\}_ 
	{\gamma\in\Gamma, z\in {Z}},
	$$
	which in turn induces a $\Gamma$-$\alpha_s$-action on 
	$\mathcal{A}
	\Big(
	C_{\rm Sim}
	\big(
	Z, ~
	L^2 ([0,1], \mathscr{E}_{k,\Gamma})
	\big)
	\Big)$. 
	More precisely, 
	the 
	$*$-isomorphism 
	$\mathcal{U}_{k}^{\scalebox{0.7}{$(\gamma,z,s)$}}$
	is given by  
	\begin{equation}
		\label{eq:theparametrizedactionsonthekthlevelalgebrashahah}
		\big(\mathcal{U}^{\scalebox{0.7}{$(\gamma,z,s)$}}_{k}(f)
		\big)(\theta,v)
		=
		(id\widehat{\otimes}\lambda^{\scalebox{0.7}{$(\gamma,s)$}}_{k,*})
		\bigg(
		f\Big(
		\theta,		U^{\scalebox{0.7}{$(\gamma^{-1},\gamma{z},s)$}}_{k}(v)
		\Big)	
		\bigg)
	\end{equation}
	for 
	$f\in\mathcal{A}
	\big(
	L^2 ([0,1], \mathscr{E}_{k,\Gamma})
	\big)$ 
	and 
	$(\theta,v)\in \mathbb{R} \times 
	L^2 ([0,1], \mathscr{E}_{k,\Gamma})$,
	where $\lambda^{\scalebox{0.7}{$(\gamma,s)$}}_{k,*}:\text{Cliff}_{\mathbb{C}}
	\big(
	L^2 ([0,1], \mathscr{H}_{k,\Gamma}) 
	\big)\to \text{Cliff}_{\mathbb{C}}
	\big(
	L^2 ([0,1], \mathscr{H}_{k,\Gamma}) 
	\big)$ 
	is the $*$-isomorphism  as in \eqref{eq:commutativeofriowithcliffalgebrariojdklk}.
	The $\Gamma$-$\alpha_s$-action is given by
	\begin{equation}
		\label{eq:theparemetrizedgammaactions} 
		(\gamma \cdot_{\alpha_s} f)(z)=
		\mathcal{U}_k^{\scalebox{0.7}{$(\gamma,\gamma^{-1}z,s)$}}
		\big(
		f(\gamma^{-1}z)
		\big)
	\end{equation}
	for 
	$\gamma\in\Gamma$,
	$z\in{Z}$ and
	$f\in 
	\mathcal{A}
	\Big(
	C_{\rm Sim}
	\big(
	Z, ~
	L^2 ([0,1], \mathscr{E}_{k,\Gamma})
	\big)
	\Big)$.
	
	\begin{rmk}\label{taizongyaolewoqunitayejdksjfklsdhahawoqunidkfjaoksd}
		The arguments in Remark \ref{taizongyaolewoqunitayejdksjfklsd} remain valid here. 
		For each  $k \geq 1$,  $\mathcal{U}_{k}^{\scalebox{0.7}{$(\gamma,z,1)$}}$ 
		encodes the original map  $\mathcal{U}^{\scalebox{0.6}{$(\gamma,z)$}}_{k}$.
		In contrast, for each  $k \geq 1$, 
		$\mathcal{U}_{k}^{\scalebox{0.7}{$(\gamma,z,0)$}}$  is the identity map for all $\gamma\in\Gamma$ and $z\in {Z}$.      
		Hence, 
		the  $\Gamma$-$\alpha_1$-action on 
		$\mathcal{A}
		\Big(
		C_{\rm Sim}
		\big(
		Z, ~
		L^2 ([0,1], \mathscr{E}_{k,\Gamma})
		\big)
		\Big)$ 
		encodes the original $\Gamma$-action on 
		$\mathcal{A}
		\big(
		C_{b}
		(Z, \mathscr{E}_{k,\Gamma})
		\big)$. 
		In contrast,  the $\Gamma$-$\alpha_0$-action on  
		$\mathcal{A}
		\Big(
		C_{\rm Sim}
		\big(
		Z, ~
		L^2 ([0,1], \mathscr{E}_{k,\Gamma})
		\big)
		\Big)$ 
		is induced only by precomposition with the action of $\Gamma$ on $Z$, while leaving the $\mathcal{A}\big(L^2 ([0,1], \mathscr{E}_{k,\Gamma})\big)$-values of the functions unchanged.
		Thus, by Remark \ref{aneasyembed},
		we obtain a family of embeddings 
		\begin{equation*}
			\bigg\{ 
			\iota_{\scalebox{0.55}{$k$}}:
			\mathcal{A}
			\big(
			L^2 
			([0,1], \mathscr{E}_{k,\Gamma})
			\big) \longrightarrow 
			\mathcal{A}
			\Big(
			C_{\rm Sim}
			\big(
			Z, ~
			L^2 ([0,1], \mathscr{E}_{k,\Gamma})
			\big)
			\Big)\bigg\}_{k\in\mathbb{N}},
		\end{equation*}
		each of which  intertwines the trivial $\Gamma$-action 
		and the
		$\Gamma$-$\alpha_{0}$-action.
		This will help us to prove  Proposition \ref{dklajlksdjlfkjasdkkkkkkkk}.
	\end{rmk}

	In addition, for each $k \geq 1$, 
	the $\Gamma$-$\alpha_{[0,1]}$-action on $C\bigg(
	[0,1],~
	\mathcal{A}
	\Big(
	C_{\rm Sim}
	\big(
	Z, ~
	L^2 ([0,1], \mathscr{E}_{k,\Gamma})
	\big)
	\Big)
	\bigg)$ 
	is given by
	\begin{equation}\label{eq:thekthlevelofthingkdiewlaskjfljsad}
		(\gamma\cdot_{\alpha_{[0,1]}} f)(t) 
		= \gamma\cdot_{\alpha_{t}} f(t)
	\end{equation}
	for  
	$\gamma\in\Gamma$, $t\in[0,1]$ 
	and 
	$f\in 
	C\bigg(
	[0,1],~
	\mathcal{A}
	\Big(
	C_{\rm Sim}
	\big(
	Z, ~
	L^2 ([0,1], \mathscr{E}_{k,\Gamma})
	\big)
	\Big)
	\bigg)$.

	We are now ready to introduce the proper algebra.

	\begin{defn}\label{A-infty-Simp}
		Let $\mathcal{K}$ be the $C^*$-algebra of compact operators on a separable infinite-dimensional complex Hilbert space.
		We define  
		$\mathcal{A}^{\scalebox{0.5}{$\prod$}}_{\rm Sim}(Z,[0,1],\mathscr{E})$
		to be the quotient 
		$C^*$-algebra 
		$$
		\dfrac{\prod\limits_{k=1}^{\infty}
			\hspace{0.3cm} 
			\mathcal{A}
			\Big(
			C_{\rm Sim}
			\big(
			Z, ~
			L^2 ([0,1], \mathscr{E}_{k,\Gamma})
			\big)
			\Big)
			{\otimes}\mathcal{K}}{\bigoplus\limits_{k=1}^{\infty}
			\hspace{0.3cm}
			\mathcal{A}
			\Big(
			C_{\rm Sim}
			\big(
			Z, ~
			L^2 ([0,1], \mathscr{E}_{k,\Gamma})
			\big)
			\Big)
			{\otimes}\mathcal{K}}.
		$$
		Moreover, 
		we define  
		$\mathcal{A}^{\flat}_{\rm Sim}
		(Z,[0,1],\mathscr{E})$
		to be the $C^*$-subalgebra of  
		$\mathcal{A}^{\scalebox{0.5}{$\prod$}}_{\rm Sim}
		(Z,[0,1],\mathscr{E})$
		generated by the set
		\begin{equation*}
			\bigg\{ 
			\big[(a_{1},a_{2},\cdots)\big]
			~\bigg|~ 
			(a_{1},a_{2},\cdots)\in\prod\limits_{k=1}^{\infty} 
			\mathcal{A}_{R}
			\Big(
			C_{\rm Sim}
			\big(
			Z, ~
			L^2 ([0,1], \mathscr{E}_{k,\Gamma})
			\big)
			\Big)
			{\odot} 
			\mathcal{K}
			\text{ for some } R>0
			\bigg\},
		\end{equation*}
		where $\odot$
		denotes 
		the algebraic tensor product.
	\end{defn}

We use the superscript  $\flat$  to indicate that  the algebra 	
$\mathcal{A}^{\flat}_{\rm Sim}
(Z,[0,1],\mathscr{E})$
is generated by elements with uniformly bounded supports.
We emphasize that  	$\mathcal{A}^{\scalebox{0.5}{$\prod$}}_{\rm Sim}(Z,[0,1],\mathscr{E})$ 
need not be a proper  $\Gamma$-$C^*$-algebra, 
	whereas its  subalgebra $\mathcal{A}^{\flat}_{\rm Sim}
	(Z,[0,1],\mathscr{E})$ is proper (Proposition \ref{properalgebrabySimple}). 	
Both	
$\mathcal{A}^{\scalebox{0.5}{$\prod$}}_{\rm Sim}(Z,[0,1],\mathscr{E})$  and the algebra $\mathcal{A}^{\scalebox{0.5}{$\prod$}}_{{\rm Sim}, [0,1]}
(Z,[0,1],\mathscr{E})$
to be introduced below  will play key roles in establishing the injectivity of the Bott map in Sections 8 and 9.

	\begin{defn}\label{Key-idea-deformation}
		Define  
		$\mathcal{A}^{\scalebox{0.5}{$\prod$}}_{{\rm Sim}, [0,1]}
		(Z,[0,1],\mathscr{E})$
		to be the quotient 
		$C^*$-algebra 
		$$
		\dfrac{\prod\limits_{k=1}^{\infty}
			\hspace{0.3cm}
			C\bigg(
			[0,1],~
			\mathcal{A}
			\Big(
			C_{\rm Sim}
			\big(
			Z, ~
			L^2 ([0,1], \mathscr{E}_{k,\Gamma})
			\big)
			\Big)
			\bigg)
			{\otimes}\mathcal{K}}{\bigoplus\limits_{k=1}^{\infty}
			\hspace{0.3cm}
			C\bigg(
			[0,1],~
			\mathcal{A}
			\Big(
			C_{\rm Sim}
			\big(
			Z, ~
			L^2 ([0,1], \mathscr{E}_{k,\Gamma})
			\big)
			\Big)
			\bigg)
			{\otimes}\mathcal{K}}.
		$$
	\end{defn}

	\begin{rmk}\label{evaluationswichisgammaequivalentwithrestoeachaction}
		The algebra 	
		$\mathcal{K}$ 
		introduced above is useful in the computation of $K$-theory in Section 9.
		For every $k \geq 1$ and  $s\in[0,1]$,  
		the  $C^*$-algebra
		$\mathcal{A}
		\Big(
		C_{\rm Sim}
		\big(
		Z, ~
		L^2 ([0,1], \mathscr{E}_{k,\Gamma})
		\big)
		\Big)
		{\otimes}\mathcal{K}$  
		admits a $\Gamma$-$\alpha_s$-action given diagonally by 
		\[
		\gamma\cdot_{\alpha_s} (f{\otimes}K)=(\gamma\cdot_{\alpha_s} f){\otimes}K.
		\]
		This in turn
		induces a  $\Gamma$-$\alpha_s$-action on 
		$\mathcal{A}^{\scalebox{0.5}{$\prod$}}_{\rm Sim}
		(Z,[0,1],\mathscr{E})$  via
		\begin{eqnarray*}
			\gamma\cdot_{\alpha_s}[(a_{1},a_{2},\cdots)]=\big[(\gamma\cdot_{\alpha_s}a_{1},~\gamma\cdot_{\alpha_s}a_{2},~\cdots)\big],
		\end{eqnarray*}
		where  
		$a_{k}\in\mathcal{A}
		\Big(
		C_{\rm Sim}
		\big(
		Z, ~
		L^2 ([0,1], \mathscr{E}_{k,\Gamma})
		\big)
		\Big)
		{\otimes}\mathcal{K}$.
		Similarly, there is a $\Gamma$-$\alpha_{[0,1]}$-action on each
		$$
		C\bigg(
		[0,1],~
		\mathcal{A}
		\Big(
		C_{\rm Sim}
		\big(
		Z, ~
		L^2 ([0,1], \mathscr{E}_{k,\Gamma})
		\big)
		\Big)
		\bigg)
		{\otimes}\mathcal{K},
		$$ 
		which induces a $\Gamma$-$\alpha_{[0,1]}$-action on
		$\mathcal{A}^{\scalebox{0.5}{$\prod$}}_{{\rm Sim}, [0,1]}
		(Z,[0,1],\mathscr{E})$.
		It is clear that the family of evaluation maps 
		at $s\in[0,1]$
		\begin{equation*}
			\Bigg\{ 
			{\rm ev}_{k,s}: 
			C\bigg(
			[0,1],~
			\mathcal{A}
			\Big(
			C_{\rm Sim}
			\big(
			Z, ~
			L^2 ([0,1], \mathscr{E}_{k,\Gamma})
			\big)
			\Big)
			\bigg)
			\longrightarrow 
			\mathcal{A}
			\Big(
			C_{\rm Sim}
			\big(
			Z, ~
			L^2 ([0,1], \mathscr{E}_{k,\Gamma})
			\big)
			\Big)\Bigg\}_{k\in\mathbb{N}}
		\end{equation*}
		induces  a  $*$-homomorphism  
		\begin{equation}
			\label{eq:theevaluationsmapfromsimtosim}
			{\rm ev}_s: 
			\mathcal{A}^{\scalebox{0.5}{$\prod$}}_{{\rm Sim}, [0,1]}
			(Z,[0,1],\mathscr{E})
			\longrightarrow 
			\mathcal{A}^{\scalebox{0.5}{$\prod$}}_{\rm Sim}
			(Z,[0,1],\mathscr{E}).
		\end{equation}
		By \eqref{eq:thekthlevelofthingkdiewlaskjfljsad}, 
		both of them 
		intertwine 
		 the 
		$\Gamma$-$\alpha_{[0,1]}$-action  
		and the
		$\Gamma$-$\alpha_{s}$-action.
		We emphasize that 
		the algebra
		$\mathcal{A}^{\scalebox{0.5}{$\prod$}}_{{\rm Sim}, [0,1]}
	(Z,[0,1],\mathscr{E})$
		encodes the information of a continuous deformation of group actions. 
		The homomorphisms in \eqref{eq:theevaluationsmapfromsimtosim}, 
		which  intertwine  the 
		$\Gamma$-$\alpha_{[0,1]}$-action  
		with each
		$\Gamma$-$\alpha_{s}$-action,
		will play a crucial role in Section 9.

	\end{rmk}

	The following result shows that for each $s\in[0,1]$, 
	the  
	$\Gamma$-$\alpha_s$-action on  
	$\mathcal{A}^{\scalebox{0.5}{$\prod$}}_{\rm Sim}
	(Z,[0,1],\mathscr{E})$  
	restricts to an action on 
	$\mathcal{A}^{\flat}_{\rm Sim}
	(Z,[0,1],\mathscr{E})$. 
	Hence, 
	$\mathcal{A}^{\flat}_{\rm Sim}
	(Z,[0,1],\mathscr{E})$  becomes a $\Gamma$-$\alpha_s$-$C^*$-algebra for every 
	$s\in[0,1]$.

	\begin{lem}\label{Gammaalgebraforthespecialcase123}
		For all $s\in[0,1]$, 
		$\gamma\in\Gamma$ 
		and 
		$[(a_{1},a_{2},\cdots)]\in 
		\mathcal{A}^{\flat}_{\rm Sim}
		(Z,[0,1],\mathscr{E})$,
		we have 	
		\begin{eqnarray*}
			\gamma\cdot_{\alpha_s} [(a_{1},a_{2},\cdots)]
			\in 
			\mathcal{A}^{\flat}_{\rm Sim}
			(Z,[0,1],\mathscr{E}).
		\end{eqnarray*} 
	\end{lem}
	\begin{proof}
		Set
		$$R_{\gamma,s}=\sqrt{s}\cdot\sqrt{
			\sum\limits_{p=1}^{\infty} 
			\Big[
			\rho_{p}^{+}\big(d(\gamma,e)\big)\Big]^{2}
		}$$
		for 
		$\gamma\in\Gamma$ and  
		$s\in[0,1]$.
		We first show that for every $k\geq1$, 
		\begin{equation*}
			f\in\mathcal{A}_{R}
			\Big(
			C_{\rm Sim}
			\big(
			Z, ~
			L^2 ([0,1], \mathscr{E}_{k,\Gamma})
			\big)
			\Big) \ 
			\Longrightarrow \    
			\gamma\cdot_{\alpha_s} f \in 	\mathcal{A}_{R_{\gamma,s}+R}
			\Big(
			C_{\rm Sim}
			\big(
			Z, ~
			L^2 ([0,1], \mathscr{E}_{k,\Gamma})
			\big)
			\Big).
		\end{equation*}
		For such $f$, 
		it follows from \eqref{eq:theparemetrizedgammaactions}, \eqref{eq:theparametrizedactionsonthekthlevelalgebrashahah} and  \eqref{eq:thekthlevelaffineisometriesforenlargedspaces}  
		that
		\begin{eqnarray}
			\label{eq:theparametrizedactionsfotthekthleveleverykdjljlk}
			\begin{array}{rcl}
				\big( 
				(\gamma\cdot_{\alpha_s} f)(z)
				\big)
				(\theta,v) 
				& = & 
				\Big( \mathcal{U}_{k}^{\scalebox{0.7}{$(\gamma,\gamma^{-1}z,s)$}}
				\big(
				f (\gamma^{-1}z)
				\big)
				\Big)
				(\theta,v) 
				\\
				& = &
				(id\widehat{\otimes}\lambda^{\scalebox{0.7}{$(\gamma,s)$}}_{k,*})
				\bigg(
				f(\gamma^{-1}z) 
				\Big(
				\theta,	U_k^{\scalebox{0.7}{$(\gamma^{-1},z,s)$}}(v)
				\Big)	
				\bigg)
				\\
				&= &
				(id\widehat{\otimes}\lambda^{\scalebox{0.7}{$(\gamma,s)$}}_{k,*})
				\bigg(
				f(\gamma^{-1}z) 
				\Big(
				\theta,		\lambda^{\scalebox{0.7}{$(\gamma^{-1},s)$}}_k(v)+\zeta^{\scalebox{0.7}{$(\gamma^{-1},z,s)$}}_k
				\Big)	
				\bigg)
			\end{array}
		\end{eqnarray}
		for all  
		$z\in{Z}$ 
		and 
		$(\theta,v) 
		\in 
		\mathbb{R} \times 
		L^2 ([0,1], \mathscr{E}_{k,\Gamma})$.
		Moreover,
		by \eqref{eq:theupperboundcalculationsforcocycles} and \eqref{eq:thecocyclesfortheenlargedspacethekthlevels},
		we obtain 	
		\begin{equation}
			\label{eq:whatisthakkkkijasdfosmmm}
			\begin{array}{c} 
				\bigg\|
				\zeta^{\scalebox{0.7}{$(\gamma^{-1},z,s)$}}_k
				\bigg\|
				=
				\sqrt{s}\cdot
				\Big\|
				\zeta^{\scalebox{0.7}{$(\gamma^{-1},z)$}}_k
				\Big\|
				\leq
				R_{\gamma,s}
			\end{array}
		\end{equation}
		for all $k\geq1$, $s\in[0,1]$, $\gamma\in\Gamma$, and 
		$z \in {Z}$.	
		Then, combining \eqref{eq:theparametrizedactionsfotthekthleveleverykdjljlk} and \eqref{eq:whatisthakkkkijasdfosmmm}, 
		we deduce that  
		$$\text{Supp} 
		\Big(
		(\gamma\cdot_{\alpha_s} f)(z)
		\Big) \subseteq 
		B_{\scalebox{0.55}[0.6]{$\mathbb{R} \times 
				L^2 
				\big( 
				[0,1], 
				\mathscr{E}_{k,\Gamma}
				\big)$}}(0,R_{\gamma,s}+R)$$
		for all  $s\in[0,1]$, $\gamma\in\Gamma$, and 
		$z \in {Z}$, that is,
		$$\gamma\cdot_{\alpha_s} f\in 	\mathcal{A}_{R_{\gamma,s}+R}
		\Big(
		C_{\rm Sim}
		\big(
		Z, ~
		L^2 ([0,1], \mathscr{E}_{k,\Gamma})
		\big)
		\Big).$$ 
		
		It follows from the above fact that,
		for every $k\geq1$, 
		\begin{equation*}
			a\in\mathcal{A}_{R}
			\Big(
			C_{\rm Sim}
			\big(
			Z, ~
			L^2 ([0,1], \mathscr{E}_{k,\Gamma})
			\big)
			\Big){\odot}\mathcal{K} \
			\Longrightarrow\   
			\gamma\cdot_{\alpha_s} a\in 	\mathcal{A}_{R_{\gamma,s}+R}
			\Big(
			C_{\rm Sim}
			\big(
			Z, ~
			L^2 ([0,1], \mathscr{E}_{k,\Gamma})
			\big)
			\Big){\odot}\mathcal{K}
		\end{equation*}
		and hence 
		\begin{equation*}
			\begin{array}{r}
				(a_{1},a_{2},\cdots)\in\prod\limits_{k=1}^{\infty} 
				\mathcal{A}_{R}
				\Big(
				C_{\rm Sim}
				\big(
				Z, ~
				L^2 ([0,1], \mathscr{E}_{k,\Gamma})
				\big)
				\Big)
				{\odot}\mathcal{K}\hspace{8cm} 
				\\
				\Longrightarrow
				(\gamma\cdot_{\alpha_s} a_{1},~\gamma\cdot_{\alpha_s} a_{2},~\cdots) \in
				\prod\limits_{k=1}^{\infty}
				\mathcal{A}_{R_{\gamma,s}+R}
				\Big(
				C_{\rm Sim}
				\big(
				Z, ~
				L^2 ([0,1], \mathscr{E}_{k,\Gamma})
				\big)
				\Big)
				{\odot}\mathcal{K}.
			\end{array}
		\end{equation*}
		This completes the proof.
	\end{proof}

	As explained in  Remark \ref{taizongyaolewoqunitayejdksjfklsdhahawoqunidkfjaoksd}, the  $\Gamma$-$\alpha_0$-action on
	$\mathcal{A}^{\flat}_{\rm Sim}
	(Z,[0,1],\mathscr{E})$
	is induced only by precomposition with the action of $\Gamma$ on $Z$, while leaving the values unchanged.
	On the other hand, 
	the $\Gamma$-$\alpha_1$-action on
	$\mathcal{A}^{\flat}_{\rm Sim}
	(Z,[0,1],\mathscr{E})$
	encodes the original $\Gamma$-action on $\mathcal{A}^{\flat}(Z,\mathscr{E})$, and is therefore proper. 
	We now provide the details of the proof.
	Note that the proof follows essentially the same strategy as that of Proposition \ref{properalgebra}; the only additional ingredient is the algebra of compact operators $\mathcal{K}$,
which makes the argument slightly more involved.
The second condition in the definition of coarse embedding  with finite complexity
(Definition \ref{mostkeyidea1})   
plays  a crucial role in establishing  the properness.

	\begin{prop}\label{properalgebrabySimple}
		Suppose that $\Gamma$ admits a coarse embedding  with finite complexity into
		$\mathscr{E}=
		\bigoplus_{p=1}^{\infty}
		\mathscr{X}_p$.
		Then 
		$\mathcal{A}^{\flat}_{\rm Sim}
		(Z,[0,1],\mathscr{E})$
		is a proper $\Gamma$-$\alpha_1$-$C^*$-algebra.
	\end{prop}
	\begin{proof}
		For 
		$k\geq1$,
		let 
		\[\mathcal{Z}
		\Big(
		C_{\rm Sim}
		\big(
		Z, ~
		L^2 ([0,1], \mathscr{E}_{k,\Gamma})
		\big)
		\Big)\]
		be the graded $C^*$-subalgebra of
		\[\mathcal{A}
		\Big(
		C_{\rm Sim}
		\big(
		Z, ~
		L^2 ([0,1], \mathscr{E}_{k,\Gamma})
		\big)
		\Big)\]
		generated by 
		\begin{equation*}
			\Big\{
			\beta_{k,\xi}(f) 
			~ \Big| ~ 
			f\in\mathcal{S}_{\rm even}, 
			~ \xi\in 
			C_{\rm Sim}
			\big(
			Z, ~
			L^2 ([0,1], \mathscr{E}_{k,\Gamma})
			\big)
			\Big\},
		\end{equation*}
		where $\beta_{k,\xi}$ is given by \eqref{eq:theenlargedfunctionalcalculusforthekthlevelsectionsttds}.
		It is in the center of 
		$\mathcal{A}
		\Big(
		C_{\rm Sim}
		\big(
		Z, ~
		L^2 ([0,1], \mathscr{E}_{k,\Gamma})
		\big)
		\Big)$.
		We 
		define 
		$\mathcal{Z}^{\flat}_{\rm Sim}
		(Z,[0,1],\mathscr{E})$
		to be the graded $C^*$-subalgebra 
		of  
		$$
		\dfrac{\prod\limits_{k=1}^{\infty}
			~\mathcal{A}
			\Big(
			C_{\rm Sim}
			\big(
			Z, ~
			L^2 ([0,1], \mathscr{E}_{k,\Gamma})
			\big)
			\Big)
		}{\bigoplus\limits_{k=1}^{\infty}
			~\mathcal{A}
			\Big(
			C_{\rm Sim}
			\big(
			Z, ~
			L^2 ([0,1], \mathscr{E}_{k,\Gamma})
			\big)
			\Big)
		}
		$$
		generated by the set
		\begin{eqnarray*}
			\Big\{ 
			\big[(f_{1},f_{2},\cdots)\big]
			~\Big|~ 
			(f_{1},f_{2},\cdots)\in\prod\limits_{k=1}^{\infty}
			\mathcal{Z}_{R}
			\Big(
			C_{\rm Sim}
			\big(
			Z, ~
			L^2 ([0,1], \mathscr{E}_{k,\Gamma})
			\big)
			\Big)
			\text{ for some } R>0
			\Big\},
		\end{eqnarray*}
		where 
		$$
		\mathcal{Z}_{R}
		\Big(
		C_{\rm Sim}
		\big(
		Z, ~
		L^2 ([0,1], \mathscr{E}_{k,\Gamma})
		\big)
		\Big) := \mathcal{Z}
		\Big(
		C_{\rm Sim}
		\big(
		Z, ~
		L^2 ([0,1], \mathscr{E}_{k,\Gamma})
		\big)
		\Big) \cap  \mathcal{A}_{R}
		\Big(
		C_{\rm Sim}
		\big(
		Z, ~
		L^2 ([0,1], \mathscr{E}_{k,\Gamma})
		\big)
		\Big).
		$$

		Let 
		$\mathcal{L}
		\big(
		\mathcal{A}^{\flat}_{\rm Sim}
		(Z,[0,1],\mathscr{E})
		\big)$
		denote the set of all
		adjointable module homomorphisms on 
		$\mathcal{A}^{\flat}_{\rm Sim}
		(Z,[0,1],\mathscr{E})$.
		This is the multiplier algebra of 
		$\mathcal{A}^{\flat}_{\rm Sim}
		(Z,[0,1],\mathscr{E})$.
		Define a $*$-homomorphism
		$$
		\varpi: \mathcal{Z}^{\flat}_{\rm Sim}
		(Z,
		[0,1],
		\mathscr{E}) \longrightarrow \mathcal{L}
		\big(
		\mathcal{A}^{\flat}_{\rm Sim}
		(Z,[0,1],\mathscr{E})
		\big)
		$$
		by
		\begin{equation*}
			\begin{array}{ccc} 
				\varpi
				\big([(f_{1},f_{2},\cdots)]\big)
				\Big([(a_1,a_2,\cdots)]\Big) =
				[ (~ (f_{1}{\otimes}\mathrm{I})\cdot a_{1},~ (f_{2}{\otimes}\mathrm{I})\cdot  a_{2},~ \cdots ~ ) ],  
			\end{array}
		\end{equation*}
		where $\mathrm{I}$ is the identity operator.  
		Viewing 
		$\mathcal{A}^{\flat}_{\rm Sim}
		(Z,
		[0,1],
		\mathscr{E})$ as an essential ideal of 
		$\mathcal{L}
		\big(
		\mathcal{A}^{\flat}_{\rm Sim}
		(Z,[0,1],\mathscr{E})
		\big)$,
		we see that the image of 
		$\varpi$ 
		lies
		in the center of 
		$\mathcal{L}
		\big(
		\mathcal{A}^{\flat}_{\rm Sim}
		(Z,[0,1],\mathscr{E})
		\big)$,
		since each $\varpi(x)$
		commutes with  every element of
		$\mathcal{A}^{\flat}_{\rm Sim}
		(Z,
		[0,1],
		\mathscr{E})$.
		Moreover, 
		$\varpi \big( 
		\mathcal{Z}^{\flat}_{\rm Sim}
		(Z,
		[0,1],
		\mathscr{E}) \big) 
		\cdot \mathcal{A}^{\flat}_{\rm Sim}
		(Z,[0,1],\mathscr{E})$
		is dense in
		$\mathcal{A}^{\flat}_{\rm Sim}
		(Z,[0,1],\mathscr{E})$.
		For each $s\in[0,1]$, the $\Gamma$-$\alpha_s$-actions on 
		the family of $C^*$-algebras
		\[\bigg\{
		\mathcal{Z}
		\Big(
		C_{\rm Sim}
		\big(
		Z, ~
		L^2 ([0,1], \mathscr{E}_{k,\Gamma})
		\big)
		\Big)
		\bigg\}_{k\in\mathbb{N}}\]
		induce, as in Lemma \ref{Gammaalgebraforthespecialcase123}, a $\Gamma$-$\alpha_s$-action  on 
		$\mathcal{Z}^{\flat}_{\rm Sim}
		(Z,[0,1],\mathscr{E})$
		that is 
		compatible with the $\Gamma$-$\alpha_s$-action on 
		$\mathcal{A}^{\flat}_{\rm Sim}
		(Z,[0,1],\mathscr{E})$.
		Thus, 
		$\varpi$ is  $\Gamma$-$\alpha_s$-equivariant
		for all $s\in[0,1]$.
		The Gelfand spectrum
		$$\Delta:={\rm Spec}\Big(\mathcal{Z}^{\flat}_{\rm Sim}
		(Z,
		[0,1],
		\mathscr{E}
		)\Big)
		$$ 
		is a locally compact Hausdorff  space equipped with the $\Gamma$-$\alpha_s$-action induced from that on 
		$\mathcal{Z}^{\flat}_{\rm Sim}
		(Z,[0,1],\mathscr{E})$. 
		The Gelfand transform 
		yields a canonical $*$-isomorphism
		$$\mathcal{Z}^{\flat}_{\rm Sim}
		(Z,[0,1],\mathscr{E})
		\cong 
		C_{0}(\Delta),
		$$
		which is equivariant with respect to the induced actions.
		Hence, by Lemma \ref{propermethod}, 
		to obtain the properness of the $\Gamma$-$\alpha_1$-action on 
		$\mathcal{A}^{\flat}_{\rm Sim}
		(Z,[0,1],\mathscr{E})$,
		it suffices to verify that  
		$$
		\lim\limits_{\gamma\rightarrow\infty}
		\big\|
		(\gamma\cdot_{\alpha_1}x) {x}
		\big\|
		=0
		$$
		for every 
		$x \in \mathcal{Z}^{\flat}_{\rm Sim}
		(Z,[0,1],\mathscr{E})$.

		We first consider the case where  $x=[(f_{1},f_{2},\cdots)]$
		with 
		$$
		(f_{1},f_{2},\cdots)\in\prod\limits_{k=1}^{\infty}
		\mathcal{Z}_{R}
		\Big(
		C_{\rm Sim}
		\big(
		Z, ~
		L^2 ([0,1], \mathscr{E}_{k,\Gamma})
		\big)
		\Big)
		$$
		for  $R>0$. 
		For such an $R>0$,	
		by the second condition of Definition \ref{mostkeyidea1},
		there exists 
		$c>0$ 
		such that 
		\begin{equation*}
			\sum\limits_{p=1}^{\infty}\Big[\rho_{p}^{-}(t)\Big]^2\geq 4R^{2}+1 \quad \text{for all}\   
			t\geq c.
		\end{equation*}
		Moreover, 
		for each $\gamma\in\Gamma$ with $d(\gamma,e)\geq c$, 
		there  exists a positive integer $M_{\gamma}$ such that 
		\begin{equation*}
			\sum\limits_{p=1}^{k}\Big[\rho_{p}^{-}\big(d(\gamma,e)\big)\Big]^2 > 4R^{2} \quad 
			\text{for all}\  
			k\geq 
			M_{\gamma}.
		\end{equation*}
		Thus, for each
		$\gamma\in\Gamma$ with $d(\gamma,e)\geq c$, 
		it follows from  \eqref{eq:thecocyclesfortheenlargedspacethekthlevels} and  \eqref{eq:thelowerboundcalculationforcocyclesthekthlevel}
		that
		\begin{equation}
			\label{eq:theolwerboundcaldulationforthekthlevel}
			\Big\|\zeta^{\scalebox{0.7}{$(\gamma^{-1},z,1)$}}_k
			\Big\|=
			\Big\|
			\zeta^{\scalebox{0.7}{$(\gamma^{-1},z)$}}_k
			\Big\|
			\geq
			\sqrt{
				\sum\limits_{p=1}^k 
				\Big[
				\rho_{p}^{-}\big(d(\gamma,e)\big)\Big]^{2}
			}\ 
			> 2R \quad 
			\text{for all }  
			k\geq 
			M_{\gamma} \text{ and all } 
			z \in {Z}. 
		\end{equation}
		Combining  
		\eqref{eq:theparametrizedactionsfotthekthleveleverykdjljlk} and
		\eqref{eq:theolwerboundcaldulationforthekthlevel},
		we obtain that, 
		for every $\gamma\in\Gamma$ satisfying $d(\gamma,e)\geq c$,
		\begin{equation*}
			\text{Supp} 
			\Big(
			(\gamma\cdot_{\alpha_1}f_{k})(z)
			\Big) 
			\cap {B}_{\scalebox{0.7}{$\mathbb{R} \times 
					L^2 
					\big( 
					[0,1], 
					\mathscr{E}_{k,\Gamma}
					\big)$}}(0, R)=\emptyset \quad 
			\text{for all}\   k\geq 
			M_{\gamma}\  \text{and 
				all}\  z \in {Z}.	
		\end{equation*}
		This implies 
		$(\gamma \cdot_{\alpha_1} f_{k})
		f_{k}=0$ 
		for all 
		$k\geq 
		M_{\gamma}$, 
		and hence
		$(\gamma\cdot_{\alpha_1}x) x=0$
		for every 
		$\gamma\in\Gamma$ 
		with 
		$d(\gamma,e)\geq c$.

		The general case follows by a standard approximation argument.
	\end{proof}

\section{Bott maps}
	
Let 
$\mathcal{A}^{\flat}_{\rm Sim}
(Z,[0,1],\mathscr{E})$
be  the  $C^*$-algebra  associated to
a coarse embedding  with finite complexity 
$h: \Gamma \to \mathscr{E}=
\bigoplus_{p=1}^{\infty}
\mathscr{X}_p$, 
constructed in the previous section.
In this section, 
we construct the Bott element	
	$\mathfrak{b}\in 
	KK_{0}^{\Gamma}
	\big( 
	\mathcal{S}, ~
	\mathcal{A}^{\flat}_{\rm Sim}
	(Z,
	[0,1],
	\mathscr{E})
	\big)$ 
	together with  the associated Bott maps.

	\subsection{The Bott element for the larger algebra} 
	As in Section 6, 
	we describe  the Bott element		$\mathfrak{b}\in 
	KK_{0}^{\Gamma}
	\big( 
	\mathcal{S}, ~
	\mathcal{A}^{\flat}_{\rm Sim}
	(Z,
	[0,1],
	\mathscr{E})
	\big)$ 
	in terms of  a special $\Gamma$-equivariant asymptotic morphism   recalled in Construction~\ref{constr:KK-facts-asymptotic} in the Appendix.
	The third condition in the definition of 
	coarse embedding  with finite complexity (Definition \ref{mostkeyidea1})  	
plays a  crucial role  in establishing 
the $\Gamma$-equivariance  in Proposition \ref{asymptoticGammathemostkeypoint}.

	Fix a rank-one projection $p_0\in\mathcal{K}$.	For each $t\geq1$,
	we define a map
	\begin{equation}
		\label{eq:theBottmapfottheenlargedalgebraselementarycasekdk}
		\mathfrak{T}_{t}^{\rm Sim}: 
		\mathcal{S}
		\longrightarrow  
		\mathcal{A}^{\flat}_{\rm Sim}
		(Z,[0,1],\mathscr{E})
		\subseteq 
		\mathcal{A}^{\scalebox{0.5}{$\prod$}}_{\rm Sim}
		(Z,[0,1],\mathscr{E})
	\end{equation}
	by
	$$	
	\mathfrak{T}_{t}^{\rm Sim} (f) =
	\Big[\big(~ 
	\beta_{1,{\bf 0}}(f_t) {\otimes} p_{0},~
	\beta_{2,{\bf 0}}(f_t) {\otimes} p_{0},~
	\cdots ~
	\big)\Big],
	$$
	where 
	$f_t(x)=f(\frac{x}{t})$,
	${\bf 0}$ is the zero function in 
	$C_{\rm Sim}
	\big(
	Z, ~
	L^2 ([0,1], \mathscr{E}_{k,\Gamma})
	\big)$  and 
	$\beta$ is given by  
	\eqref{eq:theenlargedfunctionalcalculusforthekthlevelsectionsttds}.
	By definition, 
	each $\mathfrak{T}_{t}^{\rm Sim}$ 
	is a $*$-homomorphism.

	\begin{prop}\label{asymptoticGammathemostkeypoint}
		For every 
		$\gamma\in\Gamma$ 
		and 
		$f\in\mathcal{S}$,  
		$$
		\lim\limits_{t\rightarrow\infty}
		\sup\limits_{\substack{s\in[0,1] \\  k\in\mathbb{N}}}
		\Big\|
		\gamma
		\cdot_{\alpha_{s}}
		\big(
		\beta_{k,{\bf 0}}(f_t)
		\big)-
		\beta_{k,{\bf 0}}(f_t)
		\Big\|=0.
		$$
		Moreover, for every
		$\gamma\in\Gamma$ 
		and 
		$f\in\mathcal{S}$,
		$$
		\lim\limits_{t\rightarrow \infty}
		\sup_{s\in[0,1]}
		\Big\|
		\gamma
		\cdot_{\alpha_{s}}
		\big(
		\mathfrak{T}_{t}^{\rm Sim}(f)
		\big)-
		\mathfrak{T}_{t}^{\rm Sim}(f)
		\Big\|
		=0.
		$$
	\end{prop}
	\begin{proof}
		
		By the Stone--Weierstrass 
		theorem, it suffices to prove the proposition
		for the function
		$f(x) = (x + i)^{-1}$.
		
		For 
		$s\in[0,1]$,
		$\gamma\in\Gamma$,
		$k\in\mathbb{N}$
		and
		$z \in {Z}$, 
		recall from  \eqref{eq:thecocyclesfortheenlargedspacethekthlevels}
		that 
		the function $\zeta_{k}^{\scalebox{0.7}{$(\gamma,\gamma^{-1}z,s)$}}\in 
		L^2 ([0,1], \mathscr{E}_{k,\Gamma})$
		has the form:
		\begin{eqnarray*}
			\zeta_{k}^{\scalebox{0.7}{$(\gamma,\gamma^{-1}z,s)$}}(t)
			= 
			\begin{cases}
				\zeta_{k}^{\scalebox{0.7}{$(\gamma,\gamma^{-1}z)$}},& t\in[0,s],
				\\
				0, & t\in(s,1],
			\end{cases}
		\end{eqnarray*}
		where 
		$\zeta_{k}^{\scalebox{0.7}{$(\gamma,\gamma^{-1}z)$}}\in\mathscr{E}_{k,\Gamma}$
		is given by
		\begin{eqnarray*}
			\zeta^{\scalebox{0.7}{$(\gamma,\gamma^{-1}z)$}}_{k}(\tilde{\gamma})
			=	
			\sqrt{z_{\tilde{\gamma}}}
			\cdot 
			\Big( \mathfrak{h}_{k}(\gamma^{-1}\tilde{\gamma})-
			\mathfrak{h}_{k}(\tilde{\gamma})
			\Big), \quad  \tilde{\gamma}\in\Gamma.
		\end{eqnarray*}
		For each
		$p\geq 1$, 
		define 
		$\xi^{\scalebox{0.7}{$(\gamma,\gamma^{-1}z)$}}_{p}
		\in 
		\mathscr{X}_{p,\Gamma}$
		by
		\begin{eqnarray*}
			\xi^{\scalebox{0.7}{$(\gamma,\gamma^{-1}z)$}}_{p}(\tilde{\gamma})
			=	
			\sqrt{z_{\tilde{\gamma}}}
			\cdot 
			\Big( h_{p}(\gamma^{-1}\tilde{\gamma})-
			h_{p}(\tilde{\gamma})
			\Big), \quad  \tilde{\gamma}\in\Gamma,
		\end{eqnarray*}
		and define
		$\xi^{\scalebox{0.7}{$(\gamma,\gamma^{-1}z,s)$}}_{p}\in
		L^2 ([0,1], \mathscr{X}_{p,\Gamma})$ 
		by 
		\begin{eqnarray*}
			\xi^{\scalebox{0.7}{$(\gamma,\gamma^{-1}z,s)$}}_{p}(t)
			= 
			\begin{cases}
				\xi^{\scalebox{0.7}{$(\gamma,\gamma^{-1}z)$}}_{p},& t\in[0,s],
				\\
				0, & t\in(s,1].
			\end{cases}
		\end{eqnarray*}
		By construction,   
		$$
		\varrho_{k,p}\Big(\zeta_{k}^{\scalebox{0.7}{$(\gamma,\gamma^{-1}z,s)$}}\Big)=\xi^{\scalebox{0.7}{$(\gamma,\gamma^{-1}z,s)$}}_{p} \quad \text{for every } 
		1\leq p\leq k,
		$$  
		where 
		$\varrho_{k,p}$ is the map given in Lemma \ref{inequalityforPsi}. 
		Note that, by \eqref{eq:upperboundforthecocycles},
		\begin{equation}\label{eq:asmallestimatesonthevectiorjdlkjfklkkdkjflajkds}
			\big\|
			\xi^{\scalebox{0.7}{$(\gamma,\gamma^{-1}z)$}}_{p}
			\big\|  
			\leq 
			\rho^{+}_{p}
			\big(
			d(\gamma^{-1},e)
			\big)
		\end{equation}
		for all $p\geq1$, 
		$\gamma\in\Gamma$ 
		and 
		$z\in Z$.

		For every $\gamma\in\Gamma$, $s\in[0,1]$,  $k\in\mathbb{N}$  and $t\geq1$, 
		combining 
		\eqref{eq:rulesforgammaactionforelementary}, 
		\eqref{eq:thekthlevelaffineisometriesforenlargedspaces},  \eqref{eq:theGammasctionsparametrizedbysforthekthlevelsections},   \eqref{eq:theenlargedfunctionalcalculusforthekthlevellsdjfk},  \eqref{eq:theenlargedfunctionalcalculusforthekthlevelsectionsttds},   
		Lemma \ref{inequalityforPsi} and \eqref{eq:asmallestimatesonthevectiorjdlkjfklkkdkjflajkds},
		we obtain
		\begin{eqnarray*}
			&& 
			\Big\|
			\gamma
			\cdot_{\alpha_{s}}
			\big( \beta_{k,{\bf 0}}(f_t) \big)  
			-
			\beta_{k,{\bf 0}}(f_t)  
			\Big\|
			\\
			&=& 
			\Big\|
			\beta_{k, \gamma
				\cdot_{\alpha_{s}}{\bf 0}}(f_t)  
			-
			\beta_{k,{\bf 0}}(f_t)  
			\Big\|
			\\
			&=&
			\sup_{\scalebox{0.7}{$z\in{Z}$}}
			\bigg\| 
			f_{t}\Big(
			\varTheta \widehat{\otimes} 1+1\widehat{\otimes} 		\mathfrak{B}_{\scalebox{0.6}{$k,[0,1]$}}^{\scalebox{0.7}{$\zeta^{\scalebox{0.7}{$(\gamma,\gamma^{-1}z,s)$}}_{k}$}}
			\Big) 
			-  
			f_{t}\big(
			\varTheta \widehat{\otimes} 1+1\widehat{\otimes} 		\mathfrak{B}_{\scalebox{0.6}{$k,[0,1]$}}
			\big)
			\bigg\|
			\\
			&=&
			\sup_{\scalebox{0.7}{$z\in{Z}$}}
			\biggg\| 
			\Bigg[
			\frac{1}{t}
			\bigg(
			\varTheta \widehat{\otimes} 1+1\widehat{\otimes}\mathfrak{B}_{\scalebox{0.6}{$k,[0,1]$}}^{\scalebox{0.7}{$\zeta^{\scalebox{0.7}{$(\gamma,\gamma^{-1}z,s)$}}_{k}$}}
			\bigg)+
			i\Bigg]^{-1} 
			-    
			\bigg[
			\frac{1}{t}
			\big(
			\varTheta \widehat{\otimes} 1+1\widehat{\otimes}\mathfrak{B}_{\scalebox{0.6}{$k,[0,1]$}}
			\big)+
			i\bigg]^{-1} 
			\biggg\|
			\\
			&\leq&
			\sup_{\scalebox{0.7}{$z\in{Z}$}}
			\Biggg\{~
			\biggg\| 
			\Bigg[
			\frac{1}{t}
			\bigg(
			\varTheta \widehat{\otimes} 1+1\widehat{\otimes}\mathfrak{B}_{\scalebox{0.6}{$k,[0,1]$}}^{\scalebox{0.7}{$\zeta^{\scalebox{0.7}{$(\gamma,\gamma^{-1}z,s)$}}_{k}$}}
			\bigg)+
			i\Bigg]^{-1}
			\biggg\| 
			\cdot 
			\Bigg\|    
			\bigg[
			\frac{1}{t}
			\big(
			\varTheta \widehat{\otimes} 1+1\widehat{\otimes}\mathfrak{B}_{\scalebox{0.6}{$k,[0,1]$}}
			\big)+
			i\bigg]^{-1} 
			\Bigg\|
			\\
			&&
			\cdot 
			\Bigg\| 
			\frac{
				\mathfrak{B}_{\scalebox{0.6}{$k,[0,1]$}}^{^{\scalebox{0.7}{$\zeta^{\scalebox{0.7}{$(\gamma,\gamma^{-1}z,s)$}}_{k}$}}}-
				\mathfrak{B}_{\scalebox{0.6}{$k,[0,1]$}} 
			}{t} 
			\Bigg\|
			~\Biggg\}
			\\
			&\leq&
			\frac{1}{t}\cdot\sup_{\scalebox{0.7}{$z\in{Z}$}} 
			\Big\| \mathfrak{B}_{\scalebox{0.6}{$k,[0,1]$}}^{\scalebox{0.7}{$\zeta^{\scalebox{0.7}{$(\gamma,\gamma^{-1}z,s)$}}_{k}$}}-
			\mathfrak{B}_{\scalebox{0.6}{$k,[0,1]$}}\Big\|
			\\
			&=&
			\frac{1}{t}\cdot\sup_{\scalebox{0.7}{$z\in{Z}$}} \sup_{\scalebox{0.7}{$v\in
					L^2 
					\big( 
					[0,1], \mathscr{E}_{k,\Gamma}\big)
					$}}
			\bigg\| \varPsi_{\scalebox{0.6}{$k, [0,1]$}}
			\Big(
			v-
			\zeta^{\scalebox{0.7}{$(\gamma,\gamma^{-1}z,s)$}}_{k}
			\Big)
			-
			\varPsi_{\scalebox{0.6}{$k, [0,1]$}}
			(v)
			\bigg\|
			\\
			&\leq&
			\frac{1}{t}\cdot\sup_{\scalebox{0.7}{$z\in{Z}$}} 		\sqrt{	
				\bigg[
				\omega_{\scalebox{0.7}{$
						\phi_{\scalebox{0.6}{$1,[0,1]$}}^{\eta}
						$}}
				\Big(
				\big\|
				\xi^{\scalebox{0.7}{$(\gamma,\gamma^{-1}z,s)$}}_{1}
				\big\|
				\Big)
				\bigg]^2	
				+
				64 \cdot \sum\limits_{p=2}^{k}
				\Big(
				(L_{\phi_p}+1) 
				\cdot
				\big\|
				\xi^{\scalebox{0.7}{$(\gamma,\gamma^{-1}z,s)$}}_{p}
				\big\|
				\Big)^2
			}
			\\
			&\leq&
			\frac{1}{t}\cdot\sup_{\scalebox{0.7}{$z\in{Z}$}} 		\sqrt{	
				\bigg[
				\omega_{\scalebox{0.7}{$
						\phi_{\scalebox{0.6}{$1,[0,1]$}}^{\eta}
						$}}
				\Big(
				\big\|
				\xi^{\scalebox{0.7}{$(\gamma,\gamma^{-1}z)$}}_{1}
				\big\|
				\Big)
				\bigg]^2	
				+
				64 \cdot \sum\limits_{p=2}^{k}
				\Big(
				(L_{\phi_p}+1) 
				\cdot
				\big\|
				\xi^{\scalebox{0.7}{$(\gamma,\gamma^{-1}z)$}}_{p}
				\big\|
				\Big)^2
			}
			\\
			&\leq&
			\frac{1}{t}\cdot	\sqrt{	
				\bigg[
				\omega_{\scalebox{0.7}{$
						\phi_{\scalebox{0.6}{$1,[0,1]$}}^{\eta}
						$}}
				\Big(
				\rho_{1}^{+}\big(d(\gamma^{-1},e)
				\big)
				\Big)
				\bigg]^2	
				+
				64 \cdot \sum\limits_{p=2}^{k}
				\Big( 
				(L_{\phi_p}+1)
				\cdot
				\rho_{p}^{+}\big(d(\gamma^{-1},e)
				\big)
				\Big) ^2
			}
			\\
			&\leq&
			\frac{1}{t}\cdot	\sqrt{	
				\bigg[
				\omega_{\scalebox{0.7}{$
						\phi_{\scalebox{0.6}{$1,[0,1]$}}^{\eta}
						$}}
				\Big(
				\rho_{1}^{+}\big(d(\gamma^{-1},e)
				\big)
				\Big)
				\bigg]^2	
				+
				64 \cdot \sum\limits_{p=2}^{\infty}
				\Big( 
				(L_{\phi_p}+1)
				\cdot
				\rho_{p}^{+}\big(d(\gamma^{-1},e)
				\big)
				\Big) ^2
			}.
		\end{eqnarray*} 
		Note that $h:\Gamma\to\mathscr{E}=
		\bigoplus_{p=1}^{\infty}
		\mathscr{X}_p$ 
		is a coarse embedding  with finite complexity.
		It follows from the third condition of Definition \ref{mostkeyidea1}  that  
		\begin{eqnarray*}
			\sum\limits_{p=2}^{\infty}
			\Big( 
			(L_{\phi_p}+1)
			\cdot
			\rho_{p}^{+}\big(d(\gamma^{-1},e)
			\big)
			\Big) ^2 <+\infty \quad  \text{for every} \ 
			\gamma\in\Gamma.
		\end{eqnarray*}
		On the other hand, 
		by Corollary \ref{finitemodulusofcontinuity}, 
		$$	
		\omega_{\scalebox{0.7}{$
				\phi_{\scalebox{0.6}{$1,[0,1]$}}^{\eta}
				$}}
		\Big(
		\rho_{1}^{+}\big(d(\gamma^{-1},e)
		\big)
		\Big)
		<+\infty \quad  \text{for every} \ 
		\gamma\in\Gamma.
		$$
		Therefore,  
		$$
		\lim\limits_{t\rightarrow\infty}
		\sup\limits_{\substack{s\in[0,1] \\  k\in\mathbb{N}}}
		\Big\|
		\gamma
		\cdot_{\alpha_{s}}
		\big( \beta_{k,{\bf 0}}(f_t) \big)  
		-
		\beta_{k,{\bf 0}}(f_t)  
		\Big\|=0 \quad \text{for each} \ 
		\gamma\in\Gamma.
		$$

		Furthermore, for every 
		$\gamma\in\Gamma$,    
		\begin{equation*}
			\lim\limits_{t\rightarrow \infty}
			\sup_{s\in[0,1]}
			\Big\|
			\gamma
			\cdot_{\alpha_{s}}
			\big(
			\mathfrak{T}_{t}^{\rm Sim}(f)
			\big)-
			\mathfrak{T}_{t}^{\rm Sim}(f)
			\Big\|
			~ \leq ~
			\lim\limits_{t\rightarrow\infty}
			\sup\limits_{\substack{s\in[0,1] \\  k\in\mathbb{N}}}
			\Big\|
			\gamma
			\cdot_{\alpha_{s}}
			\big( \beta_{k,{\bf 0}}(f_t) \big)  
			-
			\beta_{k,{\bf 0}}(f_t)  
			\Big\| ~ = ~ 0.
		\end{equation*}
		This completes the proof.
	\end{proof}

	From now on,
	we insert the superscripts $\alpha_s$ into the notation for equivariant $KK$-groups (see the Appendix)
	to specify the $\Gamma$-$\alpha_s$-actions.
	When the action is trivial, we suppress the superscript.

	By  Proposition \ref{asymptoticGammathemostkeypoint} 
	and 
	Construction \ref{constr:KK-facts-asymptotic} 
	in the Appendix, 
	we have  	 	
	$$
	\big[\big(
	\mathfrak{T}_{t}^{\rm Sim}
	\big)_{t\in[1,\infty)}\big]
	\in KK^{\Gamma,\alpha_s}_{0} 
	\big(
	\mathcal{S}, ~
	\mathcal{A}^{\flat}_{\rm Sim}
	(Z,[0,1],\mathscr{E})
	\big)
	$$
	and 
\begin{equation}
	\label{eq:thebottmawoqunimaklkdsjflksjlkfjlkkkkofghhjjo}
	\big[\big(
	\mathfrak{T}_{t}^{\rm Sim}
	\big)_{t\in[1,\infty)}\big]
	\in KK^{\Gamma,\alpha_s}_{0} 
	\big(
	\mathcal{S}, ~
	\mathcal{A}^{\scalebox{0.5}{$\prod$}}_{\rm Sim}
	(Z,[0,1],\mathscr{E})
	\big)
\end{equation}
	for each $s\in[0,1]$.
The first of these two elements induces the following Bott map, and the second will be used to establish the injectivity of the Bott map. 

Let $E\Gamma$ be the universal space for proper and free actions of the group $\Gamma$ and let $\underline{E}\Gamma$ be the universal space for proper actions  of the group $\Gamma$. The group homomorphisms ($s\in[0,1]$)
\begin{eqnarray}
	\label{eq:thebottmapforfreeactionofGamma}
	&	KK^{\Gamma}_{*}
	(E\Gamma,
	\mathcal{S}) 
	\longrightarrow 
	KK^{\Gamma, \alpha_{s}}_{*}
	\Big( 
	E\Gamma,~
	\mathcal{A}^{\flat}_{\rm Sim}
	(Z,[0,1],\mathscr{E})
	\Big)
\end{eqnarray}
and 
\begin{eqnarray}
	\label{eq:thebottmapforproperactionfofgammaassdf}
	&    KK^{\Gamma}_{*}
	(\underline{E}\Gamma,
	\mathcal{S}) 
	\longrightarrow 
	KK^{\Gamma, \alpha_{s}}_{*}
	\Big( 
	\underline{E}\Gamma,~
	\mathcal{A}^{\flat}_{\rm Sim}
	(Z,[0,1],\mathscr{E})
	\Big)
\end{eqnarray}
induced by 
taking the  Kasparov product with  
$\big[\big(
\mathfrak{T}_{t}^{\rm Sim}
\big)_{t\in[1,\infty)}\big]
\in KK^{\Gamma,\alpha_s}_{0} 
\big(
\mathcal{S}, ~
\mathcal{A}^{\flat}_{\rm Sim}
(Z,[0,1],\mathscr{E})
\big)$
are called
\emph{Bott maps}.

Note that 
for each fixed  $t^{\prime}\geq1$,  \eqref{eq:rulesforgammaactionforelementary} and Remark \ref{bubiandehomotopyoftheendsofactiondlkfjkkkiejoiuxoksaleuaslkrso}
	imply that
	$\mathfrak{T}_{t^{\prime}}^{\rm Sim}$ 
	is a $\Gamma$-equivariant $*$-homomorphism with respect to the $\Gamma$-$\alpha_0$-action on 
	$\mathcal{A}^{\scalebox{0.5}{$\prod$}}_{\rm Sim}
	(Z,[0,1],\mathscr{E})$.
	It follows that
	each 	
	$\mathfrak{T}_{t^{\prime}}^{\rm Sim}$ 
	determines an element
		\begin{equation*}
		\label{eq:theequalityforasymptoticandonetwoqunidayekdjkdskldsjflm}
		\big[
		\mathfrak{T}_{t^\prime}^{\rm Sim}
		\big]
		\in 
		KK^{\Gamma,\alpha_0}_{0} 
		\big(
		\mathcal{S},~ 
		\mathcal{A}^{\scalebox{0.5}{$\prod$}}_{\rm Sim}
		(Z,[0,1],\mathscr{E})
		\big).
	\end{equation*}
Moreover,	
	\begin{equation}
		\label{eq:theequalityforasymptoticandoneterm}
		\big[\big(
		\mathfrak{T}_{t}^{\rm Sim}
		\big)_{t\in[1,\infty)}\big]
		=\big[
		\mathfrak{T}_{t^\prime}^{\rm Sim}
		\big]
		\quad \text{in }
		KK^{\Gamma,\alpha_0}_{0} 
		\big(
		\mathcal{S},~ 
		\mathcal{A}^{\scalebox{0.5}{$\prod$}}_{\rm Sim}
		(Z,[0,1],\mathscr{E})
		\big)
	\end{equation}
	for any 
	$t^\prime\in[1,\infty)$.
This fact will be useful in the proof of Proposition \ref{injectivityofbottmapdklsjfalksdjlfkjas}.

\subsection{Factorizations of the asymptotic morphism} 
	
As mentioned in the previous subsection, 
the asymptotic morphism in \eqref{eq:thebottmawoqunimaklkdsjflksjlkfjlkkkkofghhjjo} 
will be used to establish  the injectivity of the Bott map.
In this subsection, 
we construct factorizations of this asymptotic morphism  which will  play a crucial role in proving
the injectivity of the Bott map 
in Section 9.

	For each $t \geq 1$, in analogy with \eqref{eq:theBottmapfottheenlargedalgebraselementarycasekdk},
	we obtain a $*$-homomorphism
	\begin{eqnarray*}
		\mathfrak{T}_{[0,1],t}^{\rm Sim}: 
		\mathcal{S}
		& \longrightarrow & 
		\mathcal{A}^{\scalebox{0.5}{$\prod$}}_{{\rm Sim}, [0,1]}
		(Z,[0,1],\mathscr{E})
		\\
		f & \longmapsto &
		\big[( 
		F_{1,t} {\otimes} p_{0},~ F_{2,t} {\otimes} p_{0},~ 
		\cdots ~)\big],
	\end{eqnarray*}
	where each $F_{k,t}$ is a constant function
	on $[0,1]$ with value
	$\beta_{k,{\bf 0}}(f_t) \in \mathcal{A}
	\Big(
	C_{\rm Sim}
	\big(
	Z, ~
	L^2 ([0,1], \mathscr{E}_{k,\Gamma})
	\big)
	\Big)$.

	We remark that Proposition \ref{asymptoticGammathemostkeypoint} 
	implies that
	$$
	\lim\limits_{t\rightarrow \infty}
	\Big\|
	\gamma
	\cdot_{\alpha_{[0,1]}}
	\big(
	\mathfrak{T}_{[0,1],t}^{\rm Sim}(f)
	\big)-
	\mathfrak{T}_{[0,1],t}^{\rm Sim}(f)
	\Big\|
	=0
	$$
	for all $\gamma\in\Gamma$ and $f\in\mathcal{S}$.
	Combined with Construction \ref{constr:KK-facts-asymptotic},
	this yields   
	$$
	\big[	
	\big(
	\mathfrak{T}_{[0,1],t}^{\rm Sim}
	\big)_{t\in[1,\infty)}\big]\in KK^{\Gamma,\alpha_{[0,1]}}_{0} 
	\big(
	\mathcal{S}, ~
	\mathcal{A}^{\scalebox{0.5}{$\prod$}}_{{\rm Sim}, [0,1]}
	(Z,[0,1],\mathscr{E})
	\big).
	$$
	Moreover, for every $t\geq1$ and $s\in[0,1]$, 
	the following diagram commutes:
	\[\begin{tikzcd}[row sep=large]
		\mathcal{A}^{\scalebox{0.5}{$\prod$}}_{{\rm Sim}, [0,1]}
		(Z,[0,1],\mathscr{E})
		\arrow[r, "{\rm ev}_s"] &
		\mathcal{A}^{\scalebox{0.5}{$\prod$}}_{\rm Sim}
		(Z,[0,1],\mathscr{E}),
		\\
		\mathcal{S} 
		\arrow[u, "\mathfrak{T}_{[0,1],t}^{\rm Sim}" '] \arrow[ur, "\mathfrak{T}_{t}^{\rm Sim}" '] 
		&
	\end{tikzcd}\] 
	where ${\rm ev}_s$ is the $*$-homomorphism  given in Remark \ref{evaluationswichisgammaequivalentwithrestoeachaction}. 
	Hence, 
	\begin{eqnarray}
		\label{eq:Kasparovsproductforevaluationmaps}
		\big[	
		\big(
		\mathfrak{T}_{[0,1],t}^{\rm Sim}
		\big)_{t\in[1,\infty)}\big]
		\otimes 
		[{\rm ev}_s]
		= \big[\big(
		\mathfrak{T}_{t}^{\rm Sim}
		\big)_{t\in[1,\infty)}\big],
	\end{eqnarray}
	where 
	$\big[\big(
	\mathfrak{T}_{t}^{\rm Sim}
	\big)_{t\in[1,\infty)}\big]
	\in KK^{\Gamma,\alpha_s}_{0} 
	\big(
	\mathcal{S}, ~
	\mathcal{A}^{\scalebox{0.5}{$\prod$}}_{\rm Sim}
	(Z,[0,1],\mathscr{E})
	\big)$ and
	$\otimes$ is the Kasparov product in \eqref{eq:thekasparovproductorinalpapers}.

	\begin{defn}\label{withoutZalgebra} 
		We define $\mathcal{A}^{\scalebox{0.5}{$\prod$}}
		([0,1],	\mathscr{E})$
		to be the quotient 
		$C^*$-algebra 
		\begin{eqnarray*}
			\dfrac{\prod\limits_{k=1}^{\infty}
				\hspace{0.3cm}\mathcal{A}
				\big(
				L^2 
				([0,1], \mathscr{E}_{k,\Gamma})
				\big)
				{\otimes}\mathcal{K}}{\bigoplus\limits_{k=1}^{\infty}
				\hspace{0.3cm}\mathcal{A}
				\big(
				L^2 
				([0,1], \mathscr{E}_{k,\Gamma})
				\big)
				{\otimes}\mathcal{K}},
		\end{eqnarray*} 
		where each  
		$\mathcal{A}
		\big(
		L^2 
		([0,1], \mathscr{E}_{k,\Gamma})
		\big)$ 
		is obtained from Definition \ref{basicproperalgebra-geng} by using the 
		$*$-homomorphism \eqref{eq:theenlargedfunctionalcalculusforthekthlevellsdjfk}.
	\end{defn}
	We endow $\mathcal{A}^{\scalebox{0.5}{$\prod$}}
	([0,1],	\mathscr{E})$ 
	with the trivial  $\Gamma$-action.
	This algebra will be useful in the proof of Proposition \ref{klsdjkkkkjksljflksdjf},
	since  the $*$-homomorphisms 
	$\mathfrak{T}_{t}^{\rm Sim}$ 
	from 
	$\mathcal{S}$  to	$\mathcal{A}^{\scalebox{0.5}{$\prod$}}_{\rm Sim}
	(Z,[0,1],\mathscr{E})$
	factor through  
	$\mathcal{A}^{\scalebox{0.5}{$\prod$}}
	([0,1],	\mathscr{E})$. 
	We now introduce it.

	For each $t \geq 1$,  
	in analogy with \eqref{eq:theBottmapfottheenlargedalgebraselementarycasekdk},
	we obtain a $*$-homomorphism
	\begin{equation}
		\label{eq:thebottmapforthequiteemelemtarycasewithoutgammaactions}
		\begin{array}{rcl}
			\mathfrak{T}_{t}: ~ \mathcal{S} 
			& \longrightarrow & 
			\mathcal{A}^{\scalebox{0.5}{$\prod$}}
			([0,1],	\mathscr{E})
			\\
			f & \longmapsto &
			\big[( ~
			\beta_{1}(f_t) {\otimes} p_{0},~
			\beta_{2}(f_t) {\otimes} p_{0},~
			\cdots ~)
			\big],
		\end{array}
	\end{equation}
	where 
	$\beta_{k}$ is given by 
	\eqref{eq:theenlargedfunctionalcalculusforthekthlevellsdjfk}.
	We use the same symbol as in \eqref{eq:thebottmapforelementaryalgebras},
	which causes no confusion, since the map in
	\eqref{eq:thebottmapforelementaryalgebras}  is not used after Section 6.

	The family of embeddings 
	$\{\iota_{\scalebox{0.55}{$k$}}\}_{k\in\mathbb{N}}$ 
	in Remark \ref{taizongyaolewoqunitayejdksjfklsdhahawoqunidkfjaoksd}
	induces an embedding	
	\begin{equation}
		\label{eq:theembeddingmapforkthlevelalgebraskjl}
		\iota: 
		\mathcal{A}^{\scalebox{0.5}{$\prod$}}
		([0,1],	\mathscr{E})
		\longrightarrow
		\mathcal{A}^{\scalebox{0.5}{$\prod$}}_{\rm Sim}
		(Z,[0,1],\mathscr{E}),
	\end{equation}
	which intertwines the 
	trivial $\Gamma$-action on 
	$\mathcal{A}^{\scalebox{0.5}{$\prod$}}
	([0,1],\mathscr{E})$
	and the
	$\Gamma$-$\alpha_0$-action on 
	$\mathcal{A}^{\scalebox{0.5}{$\prod$}}_{\rm Sim}
	(Z,[0,1],\mathscr{E})$.
	Since,
	for every $t\geq1$,
	the diagram
	\[\begin{tikzcd}[row sep=large]
		&
		\mathcal{A}^{\scalebox{0.5}{$\prod$}}_{\rm Sim}
		(Z,[0,1],\mathscr{E})
		\\
		\mathcal{S} 
		\arrow[ur, "\mathfrak{T}_{t}^{\rm Sim}" ] 
		\arrow[r, "\mathfrak{T}_{t}" ']&
		\mathcal{A}^{\scalebox{0.5}{$\prod$}}
		([0,1],\mathscr{E}) 
		\arrow[u, "\iota" ']
	\end{tikzcd}\]
	commutes,
	we obtain 
	\begin{eqnarray}
		\label{eq:thekasparocproductforembeddingmap}
		&& \big[	
		\mathfrak{T}_{t}\big]
		\otimes 
		[\iota]
		= \big[
		\mathfrak{T}_{t}^{\rm Sim}
		\big] \quad \text{for each } 
		t \geq 1,
	\end{eqnarray}
	where 
	$\big[	
	\mathfrak{T}_{t}\big]\in 
	KK_{0}\big(\mathcal{S}, ~
	\mathcal{A}^{\scalebox{0.5}{$\prod$}}
	([0,1],\mathscr{E})\big)$
	and
	$\big[
	\mathfrak{T}_{t}^{\rm Sim}
	\big]
	\in KK^{\Gamma,\alpha_0}_{0} 
	\big(
	\mathcal{S}, ~ 
	\mathcal{A}^{\scalebox{0.5}{$\prod$}}_{\rm Sim}
	(Z,[0,1],\mathscr{E})
	\big)$.

	\section{Rational injectivity of the Bott map}

	In this section,
	we show that the Bott map constructed in the previous section is rationally injective. 
	More precisely, 
	for each $s\in[0,1]$,
	we  prove 
	the composition of the group homomorphisms  
	\[\begin{tikzcd}[row sep=1cm, column sep=0.6cm]
		KK^{\Gamma}_{*}
		(E\Gamma,
		\mathcal{S}) 
		\otimes_{\mathbb{Z}} \mathbb{Q}	\arrow[r] &		KK^{\Gamma, \alpha_{s}}_{*}
		\Big( 
		E\Gamma,~
		\mathcal{A}^{\flat}_{\rm Sim}
		(Z,[0,1],\mathscr{E})
		\Big)
		\otimes_{\mathbb{Z}} \mathbb{Q} \arrow[r] &
		KK_{\mathbb{R}, *}^{\Gamma,\alpha_{s}}
		\Big( 
		E\Gamma,~
		\mathcal{A}^{\flat}_{\rm Sim}
		(Z,[0,1],\mathscr{E})
		\Big)
	\end{tikzcd}\]
	is injective,
	where
	the first map is given by 
	\eqref{eq:thebottmapforfreeactionofGamma} 
	and the second one is given by the natural map mentioned in Construction \ref{AASdakljdksjf}.
	For simplicity,
	we denote this composition by 
	\[
	\begin{matrix}
		KK^{\Gamma}_{*}
		(E\Gamma,
		\mathcal{S}) 
		\otimes_{\mathbb{Z}} \mathbb{Q}
		\xlongrightarrow{\big[\big(
			\mathfrak{T}_{t}^{\rm Sim}
			\big)_{t\in[1,\infty)}\big]}
		KK_{\mathbb{R}, *}^{\Gamma,\alpha_{s}}
		\Big( 
		E\Gamma,~
		\mathcal{A}^{\flat}_{\rm Sim}
		(Z,[0,1],\mathscr{E})
		\Big),
	\end{matrix}
	\] 
	and all the other similar maps in this section will be simplified in the same way.

	It suffices to prove the injectivity of the group homomorphism
	\begin{eqnarray}
		\label{eq:therationallybottmapforfreeactionofgammas}
		& KK^{\Gamma}_{*}
		(E\Gamma,
		\mathcal{S}) 
		\otimes_{\mathbb{Z}} \mathbb{Q}
		\xlongrightarrow{\big[\big(
			\mathfrak{T}_{t}^{\rm Sim}
			\big)_{t\in[1,\infty)}\big]}
		KK_{\mathbb{R}, *}^{\Gamma,\alpha_{s}}
		\Big( 
		E\Gamma,~
		\mathcal{A}^{\scalebox{0.5}{$\prod$}}_{\rm Sim}
		(Z,[0,1],\mathscr{E})
		\Big)
	\end{eqnarray} 
	due to the commuting diagram
	\begin{eqnarray}
		\label{eq:thecommutativedigramsafortwoalgebrasd}
		\begin{tikzcd}[row sep=0.6cm, column sep=3cm]
			KK^{\Gamma}_{*}
			(E\Gamma,
			\mathcal{S}) 
			\otimes_{\mathbb{Z}} \mathbb{Q} \arrow[r,"{\big[\big(
				\mathfrak{T}_{t}^{\rm Sim}
				\big)_{t\in[1,\infty)}\big]}"] \arrow[dr,"{\big[\big(
				\mathfrak{T}_{t}^{\rm Sim}
				\big)_{t\in[1,\infty)}\big]}" ',pos=0.8] &
			KK_{\mathbb{R}, *}^{\Gamma,\alpha_{s}}
			\Big( 
			E\Gamma,~
			\mathcal{A}^{\flat}_{\rm Sim}
			(Z,[0,1],\mathscr{E})
			\Big)\arrow[d]
			\\
			&KK_{\mathbb{R}, *}^{\Gamma,\alpha_{s}}
			\Big( 
			E\Gamma,~
			\mathcal{A}^{\scalebox{0.5}{$\prod$}}_{\rm Sim}
			(Z,[0,1],\mathscr{E})
			\Big),
		\end{tikzcd}
	\end{eqnarray}
	where the right vertical map is induced by the inclusion  
	$\mathcal{A}^{\flat}_{\rm Sim}
	(Z,[0,1],\mathscr{E})
	\subseteq 
	\mathcal{A}^{\scalebox{0.5}{$\prod$}}_{\rm Sim}
	(Z,[0,1],\mathscr{E})$.

	We briefly outline the strategy of the proof.  
	Consider the following commuting diagram:
	{\small 
		\[\begin{tikzcd}[column sep=0.3cm, row sep=1.8cm]
			&KK_{*}^{\Gamma}(E\Gamma, 	\mathcal{S}) 	
			\otimes_{\mathbb{Z}} \mathbb{Q}
			\arrow[ddl, bend right, "{\big[\big(
				\mathfrak{T}_{t}^{\rm Sim}
				\big)_{t\in[1,\infty)}\big]}" '] 
			\arrow[ddr, bend left, "{\big[\big(
				\mathfrak{T}_{t}^{\rm Sim}
				\big)_{t\in[1,\infty)}\big]}" ] 
			\arrow[d, "{\big[	
				\big(
				\mathfrak{T}_{[0,1],t}^{\rm Sim}
				\big)_{t\in[1,\infty)}\big]
			}" ] &
			\\
			&KK_{\mathbb{R},*}^{\Gamma, \alpha_{[0,1]}}
			\Big(E\Gamma, ~  
			\mathcal{A}^{\scalebox{0.5}{$\prod$}}_{{\rm Sim}, [0,1]}
			(Z,[0,1],\mathscr{E})
			\Big)
			\arrow[dl, "{\otimes 
				[{\rm ev}_0]
			} "]  \arrow[dr, "{\otimes 
				[{\rm ev}_s]
			}" '] &
			\\
			KK_{\mathbb{R}, *}^{\Gamma,\alpha_{0}}
			\Big( 
			E\Gamma, ~
			\mathcal{A}^{\scalebox{0.5}{$\prod$}}_{\rm Sim}
			(Z,[0,1],\mathscr{E})
			\Big)
			&& KK_{\mathbb{R}, *}^{\Gamma,\alpha_{s}}
			\Big( 
			E\Gamma, ~
			\mathcal{A}^{\scalebox{0.5}{$\prod$}}_{\rm Sim}
			(Z,[0,1],\mathscr{E})
			\Big),
		\end{tikzcd}
		\]}
	
	\noindent 
	where  $s \in (0,1]$.
	We first show that 
	the  map 
	$\otimes 
	[{\rm ev}_s]$ 
	is an isomorphism (see Proposition \ref{isomorphismbyevt})
	for every $s\in[0,1]$. 
	Then we prove that         
	the upper-left map is injective 
	(see Proposition \ref{klsdjkkkkjksljflksdjf}) 
	by observing that the $\Gamma$-$\alpha_0$-action on $\mathcal{A}^{\scalebox{0.5}{$\prod$}}_{\rm Sim}
	(Z,[0,1],\mathscr{E})$ is almost trivial.     
	Thus, 
	$\big[	
	\big(
	\mathfrak{T}_{[0,1],t}^{\rm Sim}
	\big)_{t\in[1,\infty)}\big]$
	is injective and hence yields the injectivity of 
	the upper-right map.

	\begin{lem}\label{isomorphismasmallkeyforlaterresults}
		Let 
		$\big\{ 
		\sigma_k: A_k \to  B_k \big\}_{k\in\mathbb{N}}$
		be a family of $*$-homomorphisms 
		between $C^*$-algebras.
		If each $\sigma_k$ induces an isomorphism in $K$-theory, then 
		so does the induced $*$-homomorphism
		\begin{equation*}
			\begin{array}{ccc}
				\dfrac{\prod\limits_{k=1}^{\infty}
					A_k{\otimes}\mathcal{K}}{\bigoplus\limits_{k=1}^{\infty}
					A_k{\otimes}\mathcal{K}} 
				&  \longrightarrow  & 
				\dfrac{\prod\limits_{k=1}^{\infty}
					B_k{\otimes}\mathcal{K}}{\bigoplus\limits_{k=1}^{\infty}
					B_k{\otimes}\mathcal{K}}\ 
				.
			\end{array} 
		\end{equation*} 
	\end{lem}
	\begin{proof}
		We  consider the following commuting diagram of short exact sequences of $C^*$-algebras:
		\[
		\begin{tikzcd}[row sep=0.8cm, column sep=1cm]
			0 \arrow[r] &  \bigoplus\limits_{k=1}^{\infty}
			A_k{\otimes}\mathcal{K} \arrow[r] \arrow[d,"\bigoplus_k\sigma_k{\otimes} {\rm id}"] &  \prod\limits_{k=1}^{\infty}
			A_k{\otimes}\mathcal{K} \arrow[r] \arrow[d, "\prod_k\sigma_k{\otimes} {\rm id}"] &  \dfrac{\prod\limits_{k=1}^{\infty}
				A_k{\otimes}\mathcal{K}}{\bigoplus\limits_{k=1}^{\infty}
				A_k{\otimes}\mathcal{K}} \arrow[r] \arrow[d] & 0 
			\\
			0 \arrow[r] &  \bigoplus\limits_{k=1}^{\infty}
			B_k{\otimes}\mathcal{K} \arrow[r] &  \prod\limits_{k=1}^{\infty}
			B_k{\otimes}\mathcal{K} \arrow[r]  &  \dfrac{\prod\limits_{k=1}^{\infty}
				B_k{\otimes}\mathcal{K}}{\bigoplus\limits_{k=1}^{\infty}
				B_k{\otimes}\mathcal{K}} \arrow[r]  & 0. 
		\end{tikzcd}
		\]
		By the Mayer--Vietoris  and Five lemma arguments, it suffices to show that $\bigoplus_k\sigma_k{\otimes} {\rm id}$ 
		and $\prod_k\sigma_k{\otimes} {\rm id}$ induce isomorphisms at the level of $K$-theory. 
		
		Since every $\sigma_k$ induces an isomorphism in $K$-theory, 	
		it follows that
		for every  $n \geq 1$, 
		\[
		\begin{array}{cccc}
			\bigoplus\limits_{k=1}^{n}\sigma_k{\otimes} {\rm id}:& \bigoplus\limits_{k=1}^{n}
			A_k{\otimes}\mathcal{K}
			& \longrightarrow & \bigoplus\limits_{k=1}^{n}
			B_k{\otimes}\mathcal{K}
		\end{array} 
		\] 
		induces an isomorphism in $K$-theory. 
		Moreover, 
		\[
		\bigoplus\limits_{k=1}^{\infty}
		A_k{\otimes}\mathcal{K}=\varinjlim_{n} \bigoplus\limits_{k=1}^{n}
		A_k{\otimes}\mathcal{K}
		\quad \text{and} \quad  
		\bigoplus\limits_{k=1}^{\infty}
		B_k{\otimes}\mathcal{K}=\varinjlim_{n} \bigoplus\limits_{k=1}^{n}
		B_k{\otimes}\mathcal{K}.
		\]
		Hence, 
		$\bigoplus_k\sigma_k{\otimes} {\rm id}$ 
		induces  an isomorphism in $K$-theory. 
		
		Finally, the map $\prod_k\sigma_k{\otimes} {\rm id}$ 
		induces an isomorphism in $K$-theory by  Proposition 2.7.12 in \cite{WY}.
	\end{proof}

	\begin{lem}\label{jdsklfajsdfklas}
		For each $s\in[0,1]$,
		the $*$-homomorphism  
		\begin{eqnarray*}
			{\rm ev}_{s}:
			\mathcal{A}^{\scalebox{0.5}{$\prod$}}_{{\rm Sim}, [0,1]}
			(Z,[0,1],\mathscr{E})
			\longrightarrow
			\mathcal{A}^{\scalebox{0.5}{$\prod$}}_{\rm Sim}
			(Z,[0,1],\mathscr{E}),
		\end{eqnarray*}
		given by \eqref{eq:theevaluationsmapfromsimtosim},
		induces an isomorphism
		\begin{eqnarray*}
			({\rm ev}_{s})_*:
			K_*\Big(
			\mathcal{A}^{\scalebox{0.5}{$\prod$}}_{{\rm Sim}, [0,1]}
			(Z,[0,1],\mathscr{E})
			\Big) 
			\longrightarrow
			K_*\Big(
			\mathcal{A}^{\scalebox{0.5}{$\prod$}}_{\rm Sim}
			(Z,[0,1],\mathscr{E})
			\Big). 
		\end{eqnarray*}
	\end{lem}
	\begin{proof}
		By definition, we obtain the following commuting diagram of short exact sequences of $C^*$-algebras:
		\[
		\begin{tikzcd}[row sep=0.8cm, column sep=2cm]
			0 \arrow[d] & 0 \arrow[d]
			\\
			\bigoplus\limits_{k=1}^{\infty}
			C\bigg(
			[0,1],~
			\mathcal{A}
			\Big(
			C_{\rm Sim}
			\big(
			Z, ~
			L^2 ([0,1], \mathscr{E}_{k,\Gamma})
			\big)
			\Big)
			\bigg)
			{\otimes}\mathcal{K} \arrow[r,"\bigoplus_k {\rm ev}_{k,s} {\otimes} {\rm id}"] \arrow[d] &   
			\bigoplus\limits_{k=1}^{\infty}
			\mathcal{A}
			\Big(
			C_{\rm Sim}
			\big(
			Z, ~
			L^2 ([0,1], \mathscr{E}_{k,\Gamma})
			\big)
			\Big)
			{\otimes}\mathcal{K} \arrow[d]
			\\      
			\prod\limits_{k=1}^{\infty}
			C\bigg(
			[0,1],~
			\mathcal{A}
			\Big(
			C_{\rm Sim}
			\big(
			Z, ~
			L^2 ([0,1], \mathscr{E}_{k,\Gamma})
			\big)
			\Big)
			\bigg)
			{\otimes}\mathcal{K} \arrow[r,"\prod_k {\rm ev}_{k,s} {\otimes} {\rm id}"] \arrow[d] & 
			\prod\limits_{k=1}^{\infty}
			\mathcal{A}
			\Big(
			C_{\rm Sim}
			\big(
			Z, ~
			L^2 ([0,1], \mathscr{E}_{k,\Gamma})
			\big)
			\Big)
			{\otimes}\mathcal{K} \arrow[d]
			\\     
			\mathcal{A}^{\scalebox{0.5}{$\prod$}}_{{\rm Sim}, [0,1]}
			(Z,[0,1],\mathscr{E}) \arrow[r,"{\rm ev}_{s}"] \arrow[d] &
			\mathcal{A}^{\scalebox{0.5}{$\prod$}}_{\rm Sim}
			(Z,[0,1],\mathscr{E}) \arrow[d]
			\\ 
			0 & \ 
			0.
		\end{tikzcd}
		\]
		For every $s\in[0,1]$ and $k\in\mathbb{N}$,
		the evaluation map
		\begin{equation*}
			{\rm ev}_{k,s}: 
			C\bigg(
			[0,1],~
			\mathcal{A}
			\Big(
			C_{\rm Sim}
			\big(
			Z, ~
			L^2 ([0,1], \mathscr{E}_{k,\Gamma})
			\big)
			\Big)
			\bigg)
			\longrightarrow 
			\mathcal{A}
			\Big(
			C_{\rm Sim}
			\big(
			Z, ~
			L^2 ([0,1], \mathscr{E}_{k,\Gamma})
			\big)
			\Big)
		\end{equation*}
		induces an isomorphism
		at the level of $K$-theory. 
		The result then follows from Lemma \ref{isomorphismasmallkeyforlaterresults}.
	\end{proof}

	Given a countable discrete group $\Gamma$, following \cite{GHT},  we use the term \emph{proper
		$\Gamma$-space} for a metrizable space $X$ with a proper $\Gamma$-action such that the quotient space
	is again metrizable. It is called  a \emph{free and proper $\Gamma$-space} if the action is, in addition, free.

	\begin{prop}\label{isomorphismbyevt}
		Let $\Delta$ be a free and proper $\Gamma$-space, then for each 
		$s\in[0,1]$, 
		the $*$-homomorphism   
		\begin{eqnarray*}
			{\rm ev}_{s}:
			\mathcal{A}^{\scalebox{0.5}{$\prod$}}_{{\rm Sim}, [0,1]}
			(Z,[0,1],\mathscr{E})
			\longrightarrow
			\mathcal{A}^{\scalebox{0.5}{$\prod$}}_{\rm Sim}
			(Z,[0,1],\mathscr{E})
		\end{eqnarray*}
		induces an isomorphism
		$$
		\otimes[{\rm ev}_{s}]:  
		KK^{\Gamma, \alpha_{[0,1]}}_{*}
		\Big( 
		\Delta, ~
		\mathcal{A}^{\scalebox{0.5}{$\prod$}}_{{\rm Sim}, [0,1]}
		(Z,[0,1],\mathscr{E})
		\Big)
		\longrightarrow 
		KK^{\Gamma, \alpha_{s}}_{*}
		\Big( 
		\Delta, ~
		\mathcal{A}^{\scalebox{0.5}{$\prod$}}_{\rm Sim}
		(Z,[0,1],\mathscr{E})
		\Big),
		$$
		where $\otimes$ is given by \eqref{eq:theproductmapforhomomorphismshahahahah}.
	\end{prop}
	\begin{proof}
		Combining Lemma \ref{jdsklfajsdfklas} 
		with the standard Mayer--Vietoris and Five Lemma arguments, we obtain the result. 
	\end{proof}

	\begin{lem}\label{embeddingisomorphsm} 
		The embedding 
		\begin{eqnarray*} 
			\iota: 
			\mathcal{A}^{\scalebox{0.5}{$\prod$}}
			([0,1],\mathscr{E})
			&\longrightarrow&
			\mathcal{A}^{\scalebox{0.5}{$\prod$}}_{\rm Sim}
			(Z,[0,1],\mathscr{E}),
		\end{eqnarray*} 
		given by \eqref{eq:theembeddingmapforkthlevelalgebraskjl},	
		induces an isomorphism 
		\begin{eqnarray*} 
			\iota_*: 
			K_*\Big( 
			\mathcal{A}^{\scalebox{0.5}{$\prod$}}
			([0,1],\mathscr{E})
			\Big) 
			&\longrightarrow&
			K_*\Big( 
			\mathcal{A}^{\scalebox{0.5}{$\prod$}}_{\rm Sim}
			(Z,[0,1],\mathscr{E})
			\Big).
		\end{eqnarray*} 
	\end{lem}
	\begin{proof}
		By definition, we have the following commuting diagram of short exact sequences of $C^*$-algebras:			
		\[\begin{tikzcd}[row sep=0.8cm, column sep=1.7cm]
			0 \arrow[d] & 0 \arrow[d]
			\\
			\bigoplus\limits_{k=1}^{\infty}
			\mathcal{A}
			\big(
			L^2 ([0,1], \mathscr{E}_{k,\Gamma})
			\big)
			{\otimes}\mathcal{K} \arrow[r,"\bigoplus_k \iota_k {\otimes} {\rm id}"] \arrow[d] &   
			\bigoplus\limits_{k=1}^{\infty}
			\mathcal{A}
			\Big(
			C_{\rm Sim}
			\big(
			Z, ~
			L^2 ([0,1], \mathscr{E}_{k,\Gamma})
			\big)
			\Big)
			{\otimes}\mathcal{K} \arrow[d]
			\\      
			\prod\limits_{k=1}^{\infty}
			\mathcal{A}
			\big(
			L^2 ([0,1], \mathscr{E}_{k,\Gamma})
			\big)
			{\otimes}\mathcal{K} \arrow[r,"\prod_k \iota_k {\otimes} {\rm id}"] \arrow[d] & 
			\prod\limits_{k=1}^{\infty}
			\mathcal{A}
			\Big(
			C_{\rm Sim}
			\big(
			Z, ~
			L^2 ([0,1], \mathscr{E}_{k,\Gamma})
			\big)
			\Big)
			{\otimes}\mathcal{K} \arrow[d]
			\\     
			\mathcal{A}^{\scalebox{0.5}{$\prod$}}
			([0,1],\mathscr{E}) \arrow[r,"\iota"] \arrow[d] &
			\mathcal{A}^{\scalebox{0.5}{$\prod$}}_{\rm Sim}
			(Z,[0,1],\mathscr{E}) \arrow[d]
			\\ 
			0 & \ 
			0.
		\end{tikzcd}
		\]
		The result follows from  Proposition \ref{isomorphismnewone} and Lemma \ref{isomorphismasmallkeyforlaterresults}. 
	\end{proof}

	\begin{prop}\label{dklajlksdjlfkjasdkkkkkkkk}
		Let $\Delta$ be a free and proper $\Gamma$-space, then the embedding	\begin{eqnarray*} 
			\iota: 
			\mathcal{A}^{\scalebox{0.5}{$\prod$}}
			([0,1],\mathscr{E})
			&\longrightarrow&
			\mathcal{A}^{\scalebox{0.5}{$\prod$}}_{\rm Sim}
			(Z,[0,1],\mathscr{E})
		\end{eqnarray*} 
		induces an isomorphism
		\begin{eqnarray*} 
			\otimes[\iota]: 
			KK^{\Gamma}_{*}
			\Big(
			\Delta, ~
			\mathcal{A}^{\scalebox{0.5}{$\prod$}}
			([0,1],\mathscr{E})
			\Big) 
			\longrightarrow 
			KK^{\Gamma, \alpha_{0}}_{*}
			\Big( 
			\Delta, ~
			\mathcal{A}^{\scalebox{0.5}{$\prod$}}_{\rm Sim}
			(Z,[0,1],\mathscr{E})
			\Big),
		\end{eqnarray*}
		where  
		$\mathcal{A}^{\scalebox{0.5}{$\prod$}}
		([0,1],\mathscr{E})$
		is equipped with the trivial $\Gamma$-action.
	\end{prop}
	\begin{proof}
		Combining Lemma  \ref{embeddingisomorphsm} 
		with the standard Mayer--Vietoris and Five Lemma arguments, we obtain the result.   
	\end{proof}

	Next, we show that for each fixed $t\geq1$, 
	the $*$-homomorphism
	$\mathfrak{T}_{t}: \mathcal{S} 
	\to
	\mathcal{A}^{\scalebox{0.5}{$\prod$}}
	([0,1],\mathscr{E})$
	given by \eqref{eq:thebottmapforthequiteemelemtarycasewithoutgammaactions}
	induces an injective homomorphism in $K$-theory.

	\begin{lem}\label{smalltrick} 
		Let
		$p_{0}$
		be a rank-one projection in 
		$\mathcal{K}$,
		then the $*$-homomorphism
		\begin{eqnarray*}
			\varsigma:\mathcal{S}
			\longrightarrow
			\dfrac{
				\prod\limits_{1}^{\infty} \mathcal{S}{\otimes}\mathcal{K}
			}{
				\bigoplus\limits_{1}^{\infty} \mathcal{S}{\otimes}\mathcal{K}
			},
		\end{eqnarray*} 
		which sends $f$ to
		$[(f{\otimes}p_{0},~ f{\otimes}p_{0},~ \cdots)]$, 
		induces an injective homomorphism in $K$-theory.
	\end{lem}
	\begin{proof}
		
		The short exact sequence of $C^*$-algebras
		\begin{eqnarray*}
			\begin{matrix}
				0\longrightarrow 	\bigoplus\limits_{1}^{\scalebox{0.7}{$\infty$}} \mathcal{S}{\otimes}\mathcal{K}
				\stackrel{\imath}{\longrightarrow}	\prod\limits_{1}^{\scalebox{0.7}{$\infty$}} \mathcal{S}{\otimes}\mathcal{K}
				\stackrel{\jmath}{\longrightarrow} 
				\dfrac{
					\prod\limits_{1}^{\infty} \mathcal{S}{\otimes}\mathcal{K}
				}{
					\bigoplus\limits_{1}^{\infty} \mathcal{S}{\otimes}\mathcal{K}
				}
				\longrightarrow 0
			\end{matrix}
		\end{eqnarray*}	
		induces the following exact sequence in $K$-theory:
		\[\begin{matrix}
			K_*\Big(
			\bigoplus\limits_{1}^{\scalebox{0.7}{$\infty$}} \mathcal{S}{\otimes}\mathcal{K}
			\Big)
			\stackrel{\imath_*}{\longrightarrow}
			K_*\Big(
			\prod\limits_{1}^{\scalebox{0.7}{$\infty$}} \mathcal{S}{\otimes}\mathcal{K}
			\Big)
			\stackrel{\jmath_*}{\longrightarrow} 
			K_*\Bigg(
			\dfrac{
				\prod\limits_{1}^{\infty} \mathcal{S}{\otimes}\mathcal{K}
			}{
				\bigoplus\limits_{1}^{\infty} \mathcal{S}{\otimes}\mathcal{K}
			}
			\Bigg).
		\end{matrix}\]
		Thus, we obtain the following commutative diagram:
		\[\begin{tikzcd}[row sep=0.5cm, column sep=1cm]
			\bigoplus\limits_{1}^{\scalebox{0.7}{$\infty$}} 
			K_*(\mathcal{S}) 	\arrow[r, hook] &	\prod\limits_{1}^{\scalebox{0.7}{$\infty$}} 
			K_*(\mathcal{S}) 
			\\
			K_*\Big(
			\bigoplus\limits_{1}^{\scalebox{0.7}{$\infty$}} \mathcal{S}{\otimes}\mathcal{K}
			\Big)
			\arrow[r, "\imath_*"] \arrow[u, "\cong" '] &
			K_*\Big(
			\prod\limits_{1}^{\scalebox{0.7}{$\infty$}} \mathcal{S}{\otimes}\mathcal{K}
			\Big)
			\arrow[r, "\jmath_*"] \arrow[u, "\cong" ']&
			K_*\Bigg(
			\dfrac{
				\prod\limits_{1}^{\infty} \mathcal{S}{\otimes}\mathcal{K}
			}{
				\bigoplus\limits_{1}^{\infty} \mathcal{S}{\otimes}\mathcal{K}
			}
			\Bigg),
			\\
			& K_*(\mathcal{S}) \arrow[u, "\sigma_*" '] \arrow[ur, "\varsigma_*" '] 
		\end{tikzcd}
		\]
		where $\sigma$ is defined by
		$\sigma(f)=(f {\otimes} p_{0},~ f {\otimes} p_{0},~ \cdots)$
		and the top horizontal map is the canonical inclusion.

		Let 
		$x\in	
		K_{*}(\mathcal{S})$
		with
		$\varsigma_*(x)=0$.
		Then 
		$\jmath_*\big(\sigma_*(x)\big)=0$.
		By exactness, 
		there exists 
		$y\in 
		K_*\Big(
		\bigoplus\limits_{1}^{\scalebox{0.7}{$\infty$}} \mathcal{S}{\otimes}\mathcal{K}
		\Big)$
		such that 
		$\imath_*(y)=\sigma_*(x)$.
		Under the standard identifications
		\[
		K_*\Big(
		\bigoplus\limits_{1}^{\scalebox{0.7}{$\infty$}} \mathcal{S}{\otimes}\mathcal{K}
		\Big) \cong \bigoplus\limits_{1}^{\scalebox{0.7}{$\infty$}} 
		K_*(\mathcal{S})  \quad  
		\text{and} \quad  
		K_*\Big(
		\prod\limits_{1}^{\scalebox{0.7}{$\infty$}} \mathcal{S}{\otimes}\mathcal{K}
		\Big) \cong 
		\prod\limits_{1}^{\scalebox{0.7}{$\infty$}} 
		K_*(\mathcal{S}), 
		\]
		we may write
		\[
		y=(*,\cdots,*,0,0,\cdots), \quad 
		\sigma_*(x)= (x,x,\cdots, x, \cdots).
		\]
		It follows that
		$x=0$, and hence
		$\varsigma_*$ is injective.
	\end{proof}

	\begin{lem}\label{InjectiveofKtheory} 
		For each  $t \geq 1$,
		the $*$-homomorphism $\mathfrak{T}_{t}$, given by \eqref{eq:thebottmapforthequiteemelemtarycasewithoutgammaactions}, 
		induces an injective homomorphism
		in $K$-theory, i.e., 
		\[
		\begin{matrix}
			(\mathfrak{T}_{t})_{*}:
			K_{*}(\mathcal{S})
			\xlongrightarrow{}
			K_{*}	
			\big(
			\mathcal{A}^{\scalebox{0.5}{$\prod$}}
			([0,1],\mathscr{E})
			\big)
		\end{matrix}
		\] 
		is injective.
	\end{lem}
	\begin{proof}
		
		Recall from Definition \ref{withoutZalgebra} that   
		\begin{eqnarray*}
			\mathcal{A}^{\scalebox{0.5}{$\prod$}}
			([0,1],\mathscr{E})=	\dfrac{\prod\limits_{k=1}^{\infty}
				\hspace{0.2cm}\mathcal{A}
				\big(
				L^2 ([0,1], \mathscr{E}_{k,\Gamma})
				\big)
				{\otimes}\mathcal{K}}{\bigoplus\limits_{k=1}^{\infty}
				\hspace{0.2cm}\mathcal{A}
				\big(
				L^2 ([0,1], \mathscr{E}_{k,\Gamma})
				\big)
				{\otimes}\mathcal{K}}.
		\end{eqnarray*}
		For each $k\geq1$,
		since 
		$L^2 
		([0,1],\mathscr{E}_{k,\Gamma})$
		is a Property $(H)$
		Banach space with respect to the Property $(H)$ map
		$$
		\varPsi_{\scalebox{0.6}{$k, [0,1]$}}\Big|_{\scalebox{0.7}{$S\big(L^2([0,1],\mathscr{E}_{k,\Gamma})\big)$}},
		$$
		there exists
		a paving 
		$\{\mathscr{V}_{k,n}\}_{n\in\mathbb{N}}$
		of 
		$L^2 
		([0,1],\mathscr{E}_{k,\Gamma})$
		and a paving 
		$\{\mathscr{W}_{k,n}\}_{n\in\mathbb{N}}$
		of 
		$L^2 ([0,1],\mathscr{H}_{k,\Gamma})$
		such that the restriction of $\varPsi_{\scalebox{0.6}{$k, [0,1]$}}$ 
		to  $S(\mathscr{V}_{k,n})$
		is a homeomorphism onto $S(\mathscr{W}_{k,n})$ for every $n\in\mathbb{N}$.
		Moreover, we may require that
		$\varPsi_{\scalebox{0.6}{$k, [0,1]$}}$ 
		maps $\mathscr{V}_{k,n}$
		into $\mathscr{W}_{k,n}$ 
		for every $n\in\mathbb{N}$.
		As in \eqref{eq:functionalcalculusone}, one can 
		obtain a family of graded $*$-homomorphisms 
		\begin{equation}\label{eq:finitedimensionalspackejkldjfioe}
			\beta_{k,n,v}: 
			\mathcal{S} \longrightarrow
			\mathcal{S}\widehat{\otimes} {C}_{0}
			\big(\mathscr{V}_{k,n}, 
			\text{Cliff}_{\mathbb{C}}
			(\mathscr{W}_{k,n})
			\big)
		\end{equation}
		by means of the maps 
		$\varPsi_{\scalebox{0.6}{$k, [0,1]$}}\big|_{\mathscr{V}_{k,n}}$.
		So by Definition \ref{basicproperalgebra-geng},
		we have the $C^*$-algebras 
		$\mathcal{A}(\mathscr{V}_{k,n})$.
		Note that there exists the restriction $*$-homomorphism from  
		$\mathcal{A}
		\Big(
		L^2 ([0,1],\mathscr{E}_{k,\Gamma}), ~
		\mathscr{V}_{k,n}
		\Big)$
		to $\mathcal{A}(\mathscr{V}_{k,n})$.

		Let 
		$$
		\mathcal{P}=
		\big\{
		(n_1,n_2,n_3,\cdots) 
		~\big|~ n_j\in\mathbb{N} \text{ for all } j\geq1
		\big\}.
		$$
		We equip $\mathcal{P}$ with the partial order $\preccurlyeq$ defined by 
		\begin{eqnarray*}
			(n_1,n_2,n_3,\cdots)
			\preccurlyeq (m_1,m_2,m_3,\cdots) \quad 
			\Longleftrightarrow 
			\quad  n_j\leq m_j \text{ for all }
			j\geq1.
		\end{eqnarray*}
		Then $(\mathcal{P},\preccurlyeq)$
		is a partially ordered set.

		Fix $t \geq 1$.
		For every  
		$\vec{n} 
		=
		(n_1,n_2,n_3,\cdots)
		\in\mathcal{P}$, 
		we have the following commutative diagram:
		\[\begin{tikzcd}[row sep=1cm, column sep=3cm]
			& \prod\limits_{k=1}^{\infty}
			\mathcal{A}
			\Big(
			L^2 ([0,1], \mathscr{E}_{k,\Gamma}), ~
			\mathscr{V}_{k,n_k}
			\Big)
			{\otimes}\mathcal{K}
			\arrow[d,"\pi_{\vec{n}}"] 
			\\
			\mathcal{S} \arrow[r,"	\mathfrak{T}_{t}^{\vec{n}}"]  \arrow[dr,"	\mathfrak{T}_{t}"] &	\dfrac{\prod\limits_{k=1}^{\infty}
				\mathcal{A}
				\Big(
				L^2 ([0,1],\mathscr{E}_{k,\Gamma}), ~
				\mathscr{V}_{k,n_k}
				\Big)
				{\otimes}\mathcal{K}}{\bigoplus
				\limits_{k=1}^{\infty}
				\mathcal{A}
				\Big(
				L^2 
				([0,1],\mathscr{E}_{k,\Gamma}), ~
				\mathscr{V}_{k,n_k}
				\Big)
				{\otimes}\mathcal{K}} 
			\arrow[d,"\imath_{\vec{n}}"] 
			\\
			&\dfrac{\prod\limits_{k=1}^{\infty}
				\mathcal{A}
				\big(
				L^2 
				([0,1],\mathscr{E}_{k,\Gamma})
				\big)
				{\otimes}\mathcal{K}}{\bigoplus\limits_{k=1}^{\infty}
				\mathcal{A}
				\big(
				L^2 ([0,1],\mathscr{E}_{k,\Gamma})
				\big)
				{\otimes}\mathcal{K}}\ ,
		\end{tikzcd}\]
		where  
		$\pi_{\vec{n}}$ 
		is the quotient map, $\imath_{\vec{n}}$ is the map induced by the inclusions
		and
		$\mathfrak{T}_{t}^{\vec{n}}$ is defined analogously to \eqref{eq:thebottmapforthequiteemelemtarycasewithoutgammaactions}.
		It is clear that 
		$\imath_{\vec{n}}$ 
		is injective. 
		Define $\mathcal{A}_{\vec{n}}$ 
		to be the image of  $\imath_{\vec{n}}$.
		Then both
		\begin{center}
			$\big\{
			\mathcal{A}_{\vec{n}}
			\big\}_{\vec{n} 
				\in\mathcal{P}}$ \quad 
			and \quad 
			$\Big\{~
			\prod\limits_{k=1}^{\infty}
			\mathcal{A}
			\Big(
			L^2 
			([0,1],\mathscr{E}_{k,\Gamma}), ~
			\mathscr{V}_{k,n_k}
			\Big)
			{\otimes}\mathcal{K}~
			\Big\}_{\vec{n} 
				\in\mathcal{P}}$
		\end{center}
		form inductive systems of $C^*$-algebras over the partially ordered set  $(\mathcal{P},\preccurlyeq)$.
		Since for each $k\in\mathbb{N}$, 
		the union 
		$\cup_{n\in\mathbb{N}} \mathscr{V}_{k,n}$
		is dense in  
		$L^2 
		([0,1],\mathscr{E}_{k,\Gamma})$,
		it follows that 
		$$\mathcal{A}
		\big(
		L^2 ([0,1], \mathscr{E}_{k,\Gamma})
		\big)
		=
		\varinjlim\limits_{n}
		\mathcal{A}
		\Big(
		L^2 ([0,1],\mathscr{E}_{k,\Gamma}), ~
		\mathscr{V}_{k,n}
		\Big).$$
		Hence,
		\begin{eqnarray*}
			\prod\limits_{k=1}^{\infty}
			\mathcal{A}
			\big(
			L^2 ([0,1],\mathscr{E}_{k,\Gamma})
			\big)
			{\otimes}\mathcal{K}
			=
			\varinjlim_{
				\vec{n}\in\mathcal{P}}
			\prod\limits_{k=1}^{\infty}
			\mathcal{A}
			\Big(
			L^2 
			([0,1],\mathscr{E}_{k,\Gamma}), ~
			\mathscr{V}_{k,n_k}
			\Big)
			{\otimes}\mathcal{K},
		\end{eqnarray*}
		which in turn implies that
		\begin{eqnarray*}
			\varinjlim_{
				\vec{n}\in\mathcal{P}}
			\mathcal{A}_{\vec{n}}
			=
			\dfrac{\prod\limits_{k=1}^{\infty}
				\mathcal{A}
				\big(
				L^2 ([0,1], \mathscr{E}_{k,\Gamma})
				\big)
				{\otimes}\mathcal{K}}{\bigoplus\limits_{k=1}^{\infty}
				\mathcal{A}
				\big(
				L^2 ([0,1], \mathscr{E}_{k,\Gamma})
				\big)
				{\otimes}\mathcal{K}}.
		\end{eqnarray*}
		Note that  
		$\imath_{\vec{n}}$  
		is a $*$-isomorphism onto 
		$\mathcal{A}_{\vec{n}}$.
		To obtain the desired result,
		it suffices to show that 
		$\mathfrak{T}_{t}^{\vec{n}}$
		induces an injective homomorphism 
		at the level of $K$-theory
		for each 
		$\vec{n} 
		\in\mathcal{P}$.
		
		We consider the following commuting diagram:
		\begin{equation}\label{eq:kdjaflkdkkkkdkjlaksdjf}
			\begin{tikzcd}[row sep=1cm, column sep=1.5cm]
				\mathcal{S}  \arrow[r,"	\mathfrak{T}_{t}^{\vec{n}}"] \arrow[d,"\varsigma"]	& \dfrac{\prod\limits_{k=1}^{\infty}
					\mathcal{A}
					\Big(
					L^2([0,1],\mathscr{E}_{k,\Gamma}), ~
					\mathscr{V}_{k,n_k}
					\Big)
					{\otimes}\mathcal{K}}{\bigoplus
					\limits_{k=1}^{\infty}
					\mathcal{A}
					\Big(
					L^2([0,1],\mathscr{E}_{k,\Gamma}), ~
					\mathscr{V}_{k,n_k}
					\Big)
					{\otimes}\mathcal{K}}  \arrow[d,"q"]
				\\
				\dfrac{
					\prod\limits_{1}^{\infty} \mathcal{S}{\otimes}\mathcal{K}
				}{
					\bigoplus\limits_{1}^{\infty} \mathcal{S}{\otimes}\mathcal{K}
				} \arrow[r, "\chi_{t}"] &
				\dfrac{\prod\limits_{k=1}^{\infty}
					\mathcal{A}(\mathscr{V}_{k,n_k})
					{\otimes}\mathcal{K}}{\bigoplus
					\limits_{k=1}^{\infty}
					\mathcal{A}(\mathscr{V}_{k,n_k})
					{\otimes}\mathcal{K}} \  ,
			\end{tikzcd}
		\end{equation}
		where $\varsigma$ 
		is given in Lemma \ref{smalltrick}, 
		$q$ is the restriction $*$-homomorphism and 
		$\chi_{t}$
		is the map induced by the family of $*$-homomorphisms 
		$\{\beta_{k,n_k,0}\}$ 
		in \eqref{eq:finitedimensionalspackejkldjfioe}.
		Note that for all 
		$k,n \in \mathbb{N}$,
		$$K_*(\mathcal{S}) 
		\cong 
		K_*\big(
		\mathcal{A}(\mathscr{V}_{k,n})
		\big).
		$$
		Thus,
		it follows from Lemma \ref{isomorphismasmallkeyforlaterresults}  
		that  
		$\chi_{t}$
		induces 
		an isomorphism in $K$-theory. 
		Combining Lemma 
		\ref{smalltrick} with the commutativity of  diagram \eqref{eq:kdjaflkdkkkkdkjlaksdjf}, we conclude that   $\mathfrak{T}_{t}^{\vec{n}}$
		induces an injective homomorphism
		at the level of $K$-theory
		for each 
		$\vec{n} 
		\in\mathcal{P}$.
	\end{proof}

	\begin{prop}\label{shuodianshenmene}
		For each $t\geq1$,
		the group homomorphism
		\[
		\begin{matrix}
			KK^{\Gamma}_{*}
			(E\Gamma,
			\mathcal{S}) 
			\otimes_{\mathbb{Z}} \mathbb{Q}
			\xlongrightarrow{[\mathfrak{T}_{t}]} 
			KK^{\Gamma}_{\mathbb{R},*}
			\Big(
			E\Gamma,~
			\mathcal{A}^{\scalebox{0.5}{$\prod$}}
			([0,1],\mathscr{E})
			\Big)
		\end{matrix}
		\]
		is injective,
		where $[\mathfrak{T}_{t}]\in 
		KK_{0}\big(\mathcal{S},~
		\mathcal{A}^{\scalebox{0.5}{$\prod$}}
		([0,1],\mathscr{E})\big)$.
	\end{prop}
	\begin{proof}
		
		Since the $\Gamma$-action on 
		$\mathcal{A}^{\scalebox{0.5}{$\prod$}}
		([0,1],\mathscr{E})$
		is trivial,
		we obtain the following commuting diagram:	
		\[
		\begin{tikzcd}[row sep=0.8cm, column sep=1cm]
			KK_{*}^{\Gamma}(E\Gamma, \mathcal{S}) \otimes_{\mathbb{Z}} \mathbb{Q} \arrow[r] \arrow[d]	& KK_{\mathbb{R}, *}^{\Gamma}
			\Big(
			E\Gamma, ~ \mathcal{A}^{\scalebox{0.5}{$\prod$}}
			([0,1],\mathscr{E})
			\Big)  \arrow[d]
			\\
			KK_{*}(B\Gamma, \mathcal{S}) \otimes_{\mathbb{Z}} \mathbb{Q} \arrow[r] &
			KK_{\mathbb{R}, *}
			\Big(
			B\Gamma, ~ \mathcal{A}^{\scalebox{0.5}{$\prod$}}
			([0,1],\mathscr{E})
			\Big)
			\\
			\bigoplus_{j \in \mathbb{Z}/2\mathbb{Z}} K_{*-j}(B\Gamma) \otimes_{\mathbb{Z}} K_j(\mathcal{S}) \otimes_{\mathbb{Z}} \mathbb{Q} \arrow[r] \arrow[u]& 
			\bigoplus_{j \in \mathbb{Z}/2\mathbb{Z}} K_{*-j}(B\Gamma) \otimes_{\mathbb{Z}} K_{\mathbb{R}, j}	
			\big(
			\mathcal{A}^{\scalebox{0.5}{$\prod$}}
			([0,1],\mathscr{E})
			\big) \arrow[u] ,
		\end{tikzcd}
		\]
		where the upper vertical maps are the natural isomorphisms given by Remark 
		\ref{KKimportantcase}(1),
		the lower vertical maps are the natural isomorphisms given by Lemma \ref{lem:KK-separate-variables},  
		and the
		horizontal maps are induced by taking the Kasparov product with
		$[\mathfrak{T}_{t}]\in 
		KK_{0}\big(\mathcal{S}, ~
		\mathcal{A}^{\scalebox{0.5}{$\prod$}}
		([0,1],\mathscr{E})\big)$
		and by  the change-of-coefficient
		homomorphisms in Construction \ref{AASdakljdksjf}.

		It suffices to show the bottom horizontal map is injective.
		Since this is a homomorphism between $\mathbb{Q}$-vector spaces, it suffices to show the maps on the second tensor components, i.e., the compositions
		\[
		\begin{matrix}
			K_j(\mathcal{S}) \otimes_{\mathbb{Z}} \mathbb{Q} & \longrightarrow &
			K_j
			\big(
			\mathcal{A}^{\scalebox{0.5}{$\prod$}}
			([0,1],\mathscr{E})
			\big)
			\otimes_{\mathbb{Z}} \mathbb{Q}
			& \longrightarrow &
			K_{\mathbb{R}, j}	
			\big(
			\mathcal{A}^{\scalebox{0.5}{$\prod$}}
			([0,1],\mathscr{E})
			\big)
		\end{matrix}
		\]
		for $j=0,1$,
		are injective.
		This is clear for $j=0$ 
		since $K_0(\mathcal{S})\cong 0$.
		For $j=1$,
		we  rewrite the composition as 
		\[
		\begin{matrix}
			\mathbb{Q}\cong  K_1(\mathcal{S}) \otimes_{\mathbb{Z}} \mathbb{Q}
			\hookrightarrow
			K_1(\mathcal{S}) \otimes_{\mathbb{Z}} \mathbb{R}
			\cong 
			K_{\mathbb{R}, 1}(\mathcal{S})
			\xlongrightarrow{}
			K_{\mathbb{R}, 1}	
			\big(
			\mathcal{A}^{\scalebox{0.5}{$\prod$}}
			([0,1],\mathscr{E})
			\big)
		\end{matrix}
		\]
		which is injective by Lemma \ref{InjectiveofKtheory}, as desired.
	\end{proof}

	\begin{prop}\label{klsdjkkkkjksljflksdjf}
		For each $t\geq1$,
		the group homomorphism  
		\[
		\begin{matrix}
			KK^{\Gamma}_{*}
			(E\Gamma,
			\mathcal{S}) 
			\otimes_{\mathbb{Z}} \mathbb{Q}
			\xlongrightarrow{\big[
				\mathfrak{T}_{t}^{\rm Sim}
				\big]} 
			KK_{\mathbb{R}, *}^{\Gamma,\alpha_0}
			\Big( 
			E\Gamma,~
			\mathcal{A}^{\scalebox{0.5}{$\prod$}}_{\rm Sim}
			(Z,[0,1],\mathscr{E})
			\Big)
		\end{matrix}
		\] 
		is injective,
		where 
		$\big[
		\mathfrak{T}_{t}^{\rm Sim}
		\big]
		\in KK^{\Gamma,\alpha_0}_{0} 
		\big(
		\mathcal{S}, ~ 
		\mathcal{A}^{\scalebox{0.5}{$\prod$}}_{\rm Sim}
		(Z,[0,1],\mathscr{E})
		\big)$.
	\end{prop}
	\begin{proof}
		It follows from \eqref{eq:thekasparocproductforembeddingmap}  that, for each $t\geq1$,
		the following diagram commutes:	
		\[
		\begin{tikzcd}[column sep=1cm, row sep=1cm]
			&
			KK_{\mathbb{R}, *}^{\Gamma,\alpha_0}
			\Big( 
			E\Gamma, ~
			\mathcal{A}^{\scalebox{0.5}{$\prod$}}_{\rm Sim}
			(Z,[0,1],\mathscr{E})
			\Big)
			\\
			KK^{\Gamma}_{*}
			(E\Gamma,
			\mathcal{S}) 
			\otimes_{\mathbb{Z}} \mathbb{Q}
			\arrow[r, "{[	
				\mathfrak{T}_{t}]
			}"] 
			\arrow[ur, "{\big[
				\mathfrak{T}_{t}^{\rm Sim}
				\big]}"] 
			&
			KK^{\Gamma}_{\mathbb{R},*}
			\Big(
			E\Gamma, ~
			\mathcal{A}^{\scalebox{0.5}{$\prod$}}
			([0,1],\mathscr{E})
			\Big)
			\arrow[u, " {\otimes[\iota]}" ']. 
		\end{tikzcd}
		\]	
		The desired result therefore  follows from
		Construction \ref{AASdakljdksjf},
		Proposition \ref{dklajlksdjlfkjasdkkkkkkkk} 
		and Proposition \ref{shuodianshenmene}.	
	\end{proof}

	We now prove the injectivity of the map in \eqref{eq:therationallybottmapforfreeactionofgammas}.

	\begin{prop}\label{injectivityofbottmapdklsjfalksdjlfkjas}
		For each $s\in[0,1]$,
		the group homomorphism  
		\[
		\begin{matrix}
			KK^{\Gamma}_{*}
			(E\Gamma,
			\mathcal{S}) 
			\otimes_{\mathbb{Z}} \mathbb{Q}
			\xlongrightarrow{\big[\big(
				\mathfrak{T}_{t}^{\rm Sim}
				\big)_{t\in[1,\infty)}\big]}
			KK_{\mathbb{R}, *}^{\Gamma,\alpha_{s}}
			\Big( 
			E\Gamma, ~
			\mathcal{A}^{\scalebox{0.5}{$\prod$}}_{\rm Sim}
			(Z,[0,1],\mathscr{E})
			\Big)
		\end{matrix}
		\] 
		is injective,
		where 
		$\big[\big(
		\mathfrak{T}_{t}^{\rm Sim}
		\big)_{t\in[1,\infty)}\big]
		\in KK^{\Gamma,\alpha_s}_{0} 
		\big(
		\mathcal{S}, ~
		\mathcal{A}^{\scalebox{0.5}{$\prod$}}_{\rm Sim}
		(Z,[0,1],\mathscr{E})
		\big)$. 
	\end{prop}
	\begin{proof}
		For every $s\in[0,1]$,	
		it follows from \eqref{eq:Kasparovsproductforevaluationmaps} 
		that the diagram below commutes:	
		\[\begin{tikzcd}[column sep=2.8cm, row sep=1cm]
			KK_{*}^{\Gamma}(E\Gamma, 	\mathcal{S}) 	
			\otimes_{\mathbb{Z}} \mathbb{Q}
			\arrow[r, "{\big[	
				\big(
				\mathfrak{T}_{[0,1],t}^{\rm Sim}
				\big)_{t\in[1,\infty)}\big]
			}" ]
			\arrow[dr, "{\big[\big(
				\mathfrak{T}_{t}^{\rm Sim}
				\big)_{t\in[1,\infty)}\big]}" ',pos=0.7] &
			KK_{\mathbb{R},*}^{\Gamma, \alpha_{[0,1]}}
			\Big(E\Gamma, ~  
			\mathcal{A}^{\scalebox{0.5}{$\prod$}}_{{\rm Sim}, [0,1]}
			(Z,[0,1],\mathscr{E})
			\Big)
			\arrow[d, "{\otimes 
				[{\rm ev}_s]
			} "] 
			\\
			& KK_{\mathbb{R}, *}^{\Gamma,\alpha_{s}}
			\Big( 
			E\Gamma, ~
			\mathcal{A}^{\scalebox{0.5}{$\prod$}}_{\rm Sim}
			(Z,[0,1],\mathscr{E})
			\Big). 
		\end{tikzcd}
		\]
		For each $s\in[0,1]$,	
		the right vertical map is an isomorphism by Proposition \ref{isomorphismbyevt} and 
		Construction \ref{AASdakljdksjf}.
		Hence,  
		it suffices to show that the top map is injective.

		Now we consider the diagram for $s=0$. 
		Combining  \eqref{eq:theequalityforasymptoticandoneterm} 
		and Proposition \ref{klsdjkkkkjksljflksdjf}, 
		we see that the bottom map is injective.
		The injectivity of the top map then follows from commutativity of the diagram.		
	\end{proof}

	Combining Proposition \ref{injectivityofbottmapdklsjfalksdjlfkjas}  with the commutativity of the diagram \eqref{eq:thecommutativedigramsafortwoalgebrasd}, 
	we obtain the following result.

	\begin{cor}\label{finalcorollaryhahahahahwoyinggiahikjlll}
		For each $s\in[0,1]$,
		the group homomorphism  
		\[
		\begin{matrix}
			KK^{\Gamma}_{*}
			(E\Gamma,
			\mathcal{S}) 
			\otimes_{\mathbb{Z}} \mathbb{Q}
			\xlongrightarrow{\big[\big(
				\mathfrak{T}_{t}^{\rm Sim}
				\big)_{t\in[1,\infty)}\big]}
			KK_{\mathbb{R}, *}^{\Gamma,\alpha_{s}}
			\Big( 
			E\Gamma,~
			\mathcal{A}^{\flat}_{\rm Sim}
			(Z,[0,1],\mathscr{E})
			\Big)
		\end{matrix}
		\] 
		is injective,
		where 
		$\big[\big(
		\mathfrak{T}_{t}^{\rm Sim}
		\big)_{t\in[1,\infty)}\big]
		\in KK^{\Gamma,\alpha_s}_{0} 
		\big(
		\mathcal{S}, ~
		\mathcal{A}^{\flat}_{\rm Sim}
		(Z,[0,1],\mathscr{E})
		\big)$. 
	\end{cor}

	\section{The proof of the main theorem}
	
	Having established all the necessary ingredients, we now prove the main theorem of this article.
	\begin{proof}[Proof of Theorem 1.1]
		We consider the following commuting diagram:
		{\small  
			\[
			\begin{tikzcd}[row sep=1.8cm, column sep=0.8cm]
				K_{*+1}^{\Gamma}(E\Gamma) 	\otimes_{\mathbb{Z}} \mathbb{Q} 
				\arrow[r, "\pi_*"] \arrow[d, "{\big[\big(
					\mathfrak{T}_{t}^{\rm Sim}
					\big)_{t\in[1,\infty)}\big]}"]
				& 	K_{*+1}^{\Gamma}(\underline{E}\Gamma) 	\otimes_{\mathbb{Z}} \mathbb{Q}
				\arrow[r, "\mu"]  \arrow[d, "{\big[\big(
					\mathfrak{T}_{t}^{\rm Sim}
					\big)_{t\in[1,\infty)}\big]}"]
				& K_{*+1}(C_{\rm r}^*\Gamma) 	\otimes_{\mathbb{Z}} \mathbb{Q} \arrow[d, "{\big[\big(
					\mathfrak{T}_{t}^{\rm Sim}
					\big)_{t\in[1,\infty)}\big]}\rtimes_{\rm r}\Gamma"]
				\\
				KK_{\mathbb{R},*}^{\Gamma,\alpha_{1}}
				\Big(E\Gamma, ~ 
				\mathcal{A}^{\flat}_{\rm Sim}
				(Z,[0,1],\mathscr{E})
				\Big) \arrow[r, "\pi_*"] 
				& 	KK_{\mathbb{R},*}^{\Gamma,\alpha_{1}}
				\Big(\underline{E}\Gamma, ~  
				\mathcal{A}^{\flat}_{\rm Sim}
				(Z,[0,1],\mathscr{E})
				\Big) \arrow[r, "\mu"] 
				& K_{\mathbb{R},*}
				\Big(
				\mathcal{A}^{\flat}_{\rm Sim}
				(Z,[0,1],\mathscr{E})
				\rtimes_{\rm r} \Gamma
				\Big),
			\end{tikzcd}
			\]}	
		
		\noindent 
		where the left horizontal maps are  induced by the natural $\Gamma$-equivariant continuous map $\pi: E\Gamma \to \underline{E}\Gamma$,
		the right horizontal maps are the reduced Baum--Connes assembly maps, 
		and the vertical maps are induced by
		$\big[\big(
		\mathfrak{T}_{t}^{\rm Sim}
		\big)_{t\in[1,\infty)}\big]
		\in KK^{\Gamma,\alpha_1}_{0} 
		\big(
		\mathcal{S}, ~ 
		\mathcal{A}^{\flat}_{\rm Sim}
		(Z,[0,1],\mathscr{E})
		\big)$. 
		
		The map 
		$\mu$ in the bottom row is an isomorphism  by Proposition \ref{properalgebrabySimple} and 
		Theorem \ref{thm:proper-GHT}.
		The map $\pi_*$  in the bottom row is injective by Lemma \ref{jdskljfaljsdkfjsqqq}.
		Moreover, the left vertical map
		is injective by Corollary \ref{finalcorollaryhahahahahwoyinggiahikjlll}.
		Therefore, the composition of the maps in the top row is injective,
		as desired.
	\end{proof}

	\section{Appendix on $KK$-theory}

	In this appendix, we collect several elementary facts from $KK$-theory that are used throughout the paper. A convenient introduction to this material can be found in \cite{GWY}; more comprehensive treatments are given in \cite{GHT,K1988}.

	An extremely potent tool in noncommutative geometry, particularly in relation with the Novikov conjecture, is Kasparov's equivariant $KK$-theory \cite{K1988}, 
	which associates to a locally compact and $\sigma$-compact group 
	$\Gamma$ 
	and two  
	$\Gamma$-$C^*$-algebras 
	$A$ 
	and 
	$B$ 
	(meaning that 
	$\Gamma$ 
	acts on them) the abelian group $KK^\Gamma(A, B)$. 
	It is contravariant in 
	$A$ 
	and covariant in 
	$B$, 
	both with respect to equivariant 
	$*$-homomorphisms. 
	It is equivariantly homotopy-invariant, stably invariant, preserves equivariant split exact sequences and satisfies Bott periodicity, 
	i.e. there are natural isomorphisms
	\[
	KK^\Gamma(A, B) \cong KK^\Gamma(\Sigma^2 A, B) \cong KK^\Gamma(\Sigma A, \Sigma B)  \cong KK^\Gamma(A, \Sigma^2 B), 
	\]
	where 
	$\Sigma^k A$ 
	stands for 
	$C_0(\mathbb{R}^k,A)$ 
	with 
	$k\in\mathbb{N}$ 
	and 
	$\Gamma$ 
	acts trivially on 
	$\mathbb{R}$. 
	These properties ensure that 
	a short exact sequence 
	$0 \to J \to E \to A \to 0$ 
	of 
	$\Gamma$-$C^*$-algebras and 
	equivariant $*$-homomorphisms induces a six-term exact sequence in the second variable, 
	and with extra conditions such as that 
	$E$ is a nuclear 
	(in particular commutative) 
	proper $\Gamma$-$C^*$-algebra, 
	it also induces a six-term exact sequence in the first variable (see \cite{GHT}). 
	When one of the two variables is $\mathbb{C}$, 
	equivariant $KK$-theory recovers 
	\begin{itemize}
		\item equivariant $K$-theory: $KK^\Gamma(\mathbb{C}, B) 
		\cong 
		K^\Gamma_0(B)$;
		\item equivariant $K$-homology: $KK^\Gamma(A, \mathbb{C}) 
		\cong 
		K_\Gamma^0(A)$. 
	\end{itemize}
	
	\begin{rmk}
		The definition of equivariant $KK$-theory is usually tailored to the theory of 
		\emph{graded $C^*$-algebras}. However in this paper, 
		when we consider the equivariant 
		$KK$-groups of graded $C^*$-algebras,  
		we \emph{disregard their gradings} and treat them as trivially graded. 
	\end{rmk}

	The most striking feature that gives
	equivariant $KK$-theory its power is the \emph{Kasparov product}, 
	which gives a group homomorphism
	\begin{equation}
		\label{eq:thekasparovproductorinalpapers}
		KK^\Gamma(A, B) \otimes_{\mathbb{Z}} KK^\Gamma(B, C) \longrightarrow  KK^\Gamma(A, C)
	\end{equation}
	for any three  
	$\Gamma$-$C^*$-algebras 
	$A$, $B$ and $C$. 
	This product is associative. 
	Moreover,  
	the Kasparov product of two elements 
	$x \in KK^\Gamma(A, B)$ 
	and 
	$y \in KK^\Gamma(B, C)$ 
	is often denoted by 
	$x \otimes  y$. 
	In particular, 
	for any $\Gamma$-equivariant $*$-homomorphism 
	$\varphi: A \to B$, 
	the Kasparov products give rise to
	the group homomorphisms 
	$$
	[\varphi] \otimes : KK^\Gamma(B, C) \longrightarrow  KK^\Gamma(A, C)
	$$  
	and 
	\begin{equation}
		\label{eq:theproductmapforhomomorphismshahahahah}
		\otimes [\varphi] : KK^\Gamma(D, A) \longrightarrow   KK^\Gamma(D, B).
	\end{equation}

	When the acting group $\Gamma$ is the trivial group, 
	we simply write $KK(A,B)$ for $KK^\Gamma(A, B)$ and drop the word ``equivariant'' everywhere. 
	There is a forgetful functor from $KK^\Gamma$ to $KK$.

	\begin{rmk}\label{KKimportantcase}
		In some important special cases, 
		we can turn an equivariant $KK$-group $KK^\Gamma(A, B)$ into a related non-equivariant $KK$-group, 
		which is often much easier to study. 
		We summarize them as follows:
		\begin{enumerate}
			\item\label{rmk:KK-facts-de-equivariantize:trivial-B} 
			When $\Gamma$ is a countable discrete group and its action on $B$ is trivial, it is immediate from the definition that there is a natural isomorphism 
			$$
			KK^\Gamma(A, B) \cong KK(A \rtimes \Gamma, B),
			$$ 
			where $A \rtimes \Gamma$ is the maximal crossed product. 
			In particular if 
			$A = C_0(X)$ 
			for a locally compact second countable space $X$ and $\Gamma$ acts freely and properly on $X$, then since 
			$C_0(X) \rtimes \Gamma$ 
			is stably isomorphic to 
			$C_0(X/\Gamma)$, 
			we have a natural isomorphism 
			$$
			KK^\Gamma(C_0(X), B) \cong KK(C_0(X / \Gamma), B).
			$$ 
			\item\label{rmk:KK-facts-de-equivariantize:translation-A} When $\Gamma$ is a countable discrete group and 
			$A = C_0(\Gamma, D)$ 
			with an action of $\Gamma$ by translation on the domain $\Gamma$, 
			there is a natural isomorphism 
			$$
			KK^\Gamma(C_0(\Gamma, D), B) \overset{\cong}{\longrightarrow} KK(D, B)
			$$ 
			which is given by first applying the forgetful functor and then composing with the embedding 
			$D \cong C(\{1_{\Gamma} \}, D) \hookrightarrow C_0(\Gamma, D)$. 
		\end{enumerate}
	\end{rmk}

	In this paper, we will focus on the case when the first variable $A$ in $KK^\Gamma(A,B)$ is commutative and view the theory as a homological theory on the spectrum of $A$. 
	In fact, we will need a variant of it that may be thought of as homology with $\Gamma$-compact support. 
	Recall that a subset of a topological space $X$, on which $\Gamma$ acts, is called \emph{$\Gamma$-compact} if it is contained in $\{\gamma\cdot x ~\big| ~ \gamma \in \Gamma, ~ x \in K \}$ for some compact subset $K$ in $X$.

	\begin{defn}\label{defn:KK-Gam-compact}
		Given a countable discrete group $\Gamma$, 
		a Hausdorff space $X$ with a $\Gamma$-action, a $\Gamma$-$C^*$-algebra $B$ and $n\in\mathbb{Z}$, 
		we write 
		$KK^\Gamma_n(X, B)$ 
		for the inductive limit of the equivariant $KK$-groups 
		$KK^\Gamma 
		\big(
		C_0(Y), C_0(\mathbb{R}^n, A) 
		\big)$, 
		where $Y$ ranges over $\Gamma$-invariant and $\Gamma$-compact subsets of $X$ and $A$ ranges over $\Gamma$-invariant separable $C^*$-subalgebras of $B$, 
		both directed by inclusion. 
		
		We write $K^\Gamma_n(X)$ for $KK^\Gamma_n(X, \mathbb{C})$ 
		and call it the \emph{$\Gamma$-equivariant $K$-homology of $X$ with $\Gamma$-compact supports}. 
	\end{defn}
	
	\begin{rmk}\label{sldkjfljs22222}
		In some cases,
		we need to specify the actions of $\Gamma$ on the $C^*$-algebra $B$ (denoted by $\Gamma\curvearrowright_{\alpha}B$), 
		and then we write 
		$KK^{\Gamma,\alpha}_*(X, B)$ for the equivariant $KK$-group.
	\end{rmk}
	
	It is clear from Bott periodicity that there is a natural isomorphism $KK^\Gamma_n(X, B) \cong 
	KK^\Gamma_{n+2}(X, B)$. 
	Thus we can view the index $n$ as an element of $\mathbb{Z}/2\mathbb{Z}$. Also note that this construction is covariant both in $X$ with respect to continuous maps and in $B$ with respect to equivariant $*$-homomorphisms. Partially generalizing the functoriality in the second variable, the Kasparov product gives us a natural product $KK^\Gamma_n(X, B) \otimes_\mathbb{Z} KK^\Gamma(B,C) \to KK^\Gamma_n(X, C)$ for any 
	$\Gamma$-$C^*$-algebras 
	$B$ and $C$. 
	
	We may think of $KK^\Gamma_n(-, B)$ as an extraordinary homology theory in the sense of Eilenberg--Steenrod. 
	In the non-equivariant case, 
	the coefficient algebra $B$ plays a rather minor role in this picture. 
	
	\begin{lem}\label{lem:KK-separate-variables}
		For any CW-complex $X$, any $C^*$-algebra $B$ and any $n\in\mathbb{Z}/2\mathbb{Z}$, 
		there is a natural isomorphism
		\[
		KK_n(X, B) \otimes_{\mathbb{Z}}\mathbb{Q} 
		\cong \bigoplus_{j \in \mathbb{Z}/2\mathbb{Z}} K_j(X) \otimes_{\mathbb{Z}} K_{n-j}(B)  \otimes_{\mathbb{Z}} \mathbb{Q}.
		\]
	\end{lem}
	
	This follows from a version of the K\"{u}nneth Theorem.
	One can find the details in 
	\cite{GWY}.

	Let $\underline{E}\Gamma$ be the universal space for proper actions  of the group $\Gamma$. 
	The \emph{reduced Baum--Connes assembly map} for a countable discrete group $\Gamma$ and a $\Gamma$-$C^*$-algebra $B$ is a group homomorphism 
	\[
	\mu: KK^\Gamma_*(\underline{E}\Gamma, B) \longrightarrow  K_*(B \rtimes_{\rm r} \Gamma).
	\]
	It is natural in $B$ with respect to $\Gamma$-equivariant $*$-homomorphisms or more generally with respect to taking Kasparov products, in the sense that any element 
	$b\in KK^\Gamma(B,C)$ 
	induces a commutative diagram
	\begin{equation}
		\label{eq:BC-assembly-natural}
		\begin{array}{ccr}
			KK^\Gamma_*(\underline{E}\Gamma, B) & \xlongrightarrow{\mu} & K_*(B \rtimes_{\operatorname{r}} \Gamma) 
			\\
			\Big\downarrow  b & & \Big\downarrow b\rtimes_{\operatorname{r}} \Gamma 
			\\
			KK^\Gamma_*(\underline{E}\Gamma, C) & \xlongrightarrow{\mu} & K_*(C \rtimes_{\operatorname{r}} \Gamma)
		\end{array}
	\end{equation}
	for an induced group homomorphism
	$b \rtimes_{\rm r} \Gamma$.

	The case when $B=\mathbb{C}$ is of special interest. 
	Let  $E\Gamma$ be the universal space for proper and free actions of the group $\Gamma$.
	The \emph{rational strong Novikov conjecture} asserts that the composition
	\[
	K^\Gamma_*(E\Gamma) \overset{\pi_*}{\longrightarrow}
	K^\Gamma_*(\underline{E}\Gamma) \overset{\mu}{\longrightarrow} K_*(C^*_{\rm r} \Gamma) 
	\] 
	is injective after tensoring each term by $\mathbb{Q}$,
	where  $\pi_*$ is induced by the natural 
	$\Gamma$-equivariant continuous map $\pi: E\Gamma \to \underline{E}\Gamma$. 
	It implies the Novikov conjecture,
	the Gromov--Lawson--Rosenberg  conjecture on the nonexistence of positive scalar curvature for aspherical manifolds and Gromov's zero-in-the-spectrum conjecture.
	
	On the other hand, it has proven extremely useful to have the flexibility of a general $\Gamma$-algebra $B$ in the picture, largely due to the following key observation, which is based on a theorem of Green and Julg,
	and an equivariant cutting-and-pasting argument on $B$. 
	
	\begin{thm}[{\cite[Theorem~13.1]{GHT}, \cite[Proposition~ 5.11]{KS2003}}]\label{thm:proper-GHT}
		For any countable discrete group $\Gamma$  and a $\Gamma$-$C^*$-algebra $B$, 
		if $B$ is a proper $\Gamma$-$C^*$-algebra, 
		then the reduced Baum--Connes assembly map 
		\[
		\mu: KK^\Gamma_*(\underline{E}\Gamma, B) \longrightarrow 
		K_*(B \rtimes_{\rm r} \Gamma)
		\]
		is an isomorphism.
	\end{thm}
	
	This is the basis of the 
	\emph{Dirac-dual-Dirac} method, 
	which was applied very successfully to the study of the Baum--Connes assembly map. It is based on the construction of a proper $\Gamma$-$C^*$-algebra $B$ together with $KK$-elements 
	$\mathfrak{a}\in KK_*^\Gamma(B,\mathbb{C})$ 
	and $\mathfrak{b}\in KK_*^\Gamma(\mathbb{C}, B)$ such that $\mathfrak{b} \otimes \mathfrak{a}$ 
	is equal to the identity element in $KK^\Gamma(\mathbb{C},\mathbb{C})$. 
	When this is possible, Theorem \ref{thm:proper-GHT} allows us to conclude that the Baum--Connes assembly map for $\Gamma$ is an isomorphism and the rational strong Novikov conjecture for $\Gamma$ follows. 
	Although we do not directly apply this method to prove Theorem \ref{main-result1}, our strategy still calls for a proper $\Gamma$-$C^*$-algebra $B$ and a $KK$-element 
	$\mathfrak{b} \in KK_*^\Gamma(\mathbb{C}, B)$. 
	The following construction is quite useful for applications. 
	
	\begin{constr}	\label{constr:KK-facts-asymptotic} 	
		Let $B$ be a $\Gamma$-$C^*$-algebra and let $\varphi_t: C_0(\mathbb{R}) \to B$ be a family of $*$-homomorphisms indexed by $t\in [1, \infty)$, 
		such that 
		\begin{enumerate}
			\item 
			\emph{pointwise continuous}, i.e., $t \mapsto \varphi_t(f)$ is continuous for any 
			$f\in C_0(\mathbb{R})$; 
			and 
			\item 
			\emph{asymptotically $\Gamma$-invariant}, 
			i.e., $\lim\limits_{t\to\infty} 
			\big\| 
			\gamma \cdot \left(\varphi_t(f)\right) - \varphi_t(f) 
			\big\| = 0$ 
			for any $f \in C_0(\mathbb{R})$ and any $\gamma\in\Gamma$. 
		\end{enumerate}
		Obviously, the family of $*$-homomorphisms forms a $\Gamma$-equivariant asymptotic morphism in the sense of \cite{HK}.
		By \cite[Definition~7.4]{HK}, 
		there is an element 
		\[
		\left[ (\varphi_t)_{t\in[1,\infty)} \right] \in KK^\Gamma_0 
		\big(C_0(\mathbb{R}), B \big) 
		\cong  
		KK^\Gamma_1(\mathbb{C}, B)
		\]
		whose image under the forgetful map 
		\[
		KK^\Gamma_0 
		\big(C_0(\mathbb{R}), B \big) \longrightarrow   KK(C_0(\mathbb{R}), B)
		\]
		is equal to the element $[\varphi_{t^\prime}]$ induced by the $*$-homomorphism 
		$\varphi_{t^\prime}$  for  any
		$t^\prime \in [1, \infty)$.
		
		When 
		we need to specify the action of $\Gamma$ on the $C^*$-algebra $B$ (denoted by $\Gamma\curvearrowright_{\alpha}B$), 
		we write 
		\[
		\left[ (\varphi_t)_{t\in[1,\infty)} \right] \in KK^{\Gamma,\alpha}_0 
		\big(C_0(\mathbb{R}), B \big).
		\] 
	\end{constr}

	At the end of this section, 
	we recall the construction of equivariant $KK$-theory with real coefficients introduced by Antonini, Azzali and Skandalis. 
	
	\begin{constr} [\cite{AAS2016}] \label{AASdakljdksjf}
		The \emph{equivariant $KK$-theory with real coefficients}  is a bivariant theory that associates, to each pair $(A,B)$ of $\Gamma$-$C^*$-algebras, the group  
		\[
		KK^\Gamma_{\mathbb{R}} (A, B) = \varinjlim_{\mathcal{N}} KK^\Gamma (A, B \otimes \mathcal{N}),
		\]
		where $\otimes$ stands for the minimal tensor product and the inductive limit is taken over all II$_1$-factors $\mathcal{N}$ with unital 
		$*$-homomorphisms as connecting maps. This theory is contravariant in the first variable and covariant in the second, and there is a natural map from 
		$KK^\Gamma(A, B) \otimes_{\mathbb{Z}} \mathbb{R}$ 
		to $KK^\Gamma_{\mathbb{R}} (A, B)$ since $K_0(\mathcal{N}) \cong \mathbb{R}$ 
		for any II$_1$-factor $\mathcal{N}$. 
		Moreover the Kasparov product extends to this theory. 
		
		Given a countable discrete group $\Gamma$, a Hausdorff space $X$ with a $\Gamma$-action, and a $\Gamma$-$C^*$-algebra $B$, 
		we define $KK^\Gamma_{\mathbb{R},*}(X, B)$ 
		in the same way as in Definition \ref{defn:KK-Gam-compact}. 
		Then the universal coefficient theorem allows us to identify $KK_{\mathbb{R}, *} (X, \mathbb{C})$ with $K_*(X) \otimes_{\mathbb{Z}} \mathbb{R}$ in a natural way. 
	\end{constr}

	The key reason we consider $KK$-theory with real coefficients is the following convenient fact. 
	
	\begin{lem} [{\cite{AAS2020,GWY}}] \label{jdskljfaljsdkfjsqqq}
		For any countable discrete group $\Gamma$ 
		and $\Gamma$-$C^*$-algebra $A$, the homomorphism 
		\[
		\pi_*: 
		KK^\Gamma_{\mathbb{R},*} (E\Gamma, A) \longrightarrow 
		KK^\Gamma_{\mathbb{R},*} (\underline{E}\Gamma, A),
		\]
		which is induced by the natural $\Gamma$-equivariant continuous map $\pi: E\Gamma \to \underline{E}\Gamma$, 
		is injective. 
	\end{lem}

\section{Non-sofic groups and Hilbert--Hadamard spaces } 

This final section presents two results. The first gives a negative answer to Pestov's question. The second shows that countable groups admitting coarse embeddings into the principal  Hilbert--Hadamard spaces studied in \cite{GWY} fall within the scope of our  main theorem and hence satisfy the Novikov conjecture.

\subsection{Non-sofic groups}	
	
In this subsection,  
we  show that the non-sofic groups constructed by
Kun and Thom (\cite{KunThom})  have Property A and hence admit arbitrarily
small complexity rates. 
These groups provide the first examples of
non-sofic groups with Property A,  thereby resolving Pestov's question
(\cite[Open Question 9.6]{VladimirGPestov2008})  in the negative.

In \cite{KunThom},  the authors constructed a family of non-sofic groups in terms of generalized
wreath products and amalgamated free products.   
Let $\Gamma$ be a subgroup of $G$. 
$\Gamma$ is called  an  \emph{infranormal subgroup}  of $G$  if the set  
$\big\{ g\in G ~\big|~ g\Gamma g^{-1}\leq\Gamma \big\}$
generates $G$ as a group. 
They showed that if  $\Gamma$  is an infranormal but not normal subgroup of $G$, and both $\Gamma$ and $G$  have  Kazhdan's Property (T), then the generalized
wreath product  
$(\mathbb{Z}/2\mathbb{Z}) \wr_{G/\Gamma} G$  and the amalgamated free product $G*_{\Gamma}G$  are non-sofic.
In particular,  they considered the following explicit examples.

Let $q$ be a prime power and let  $r,d \geq 3$.
Set
\[
R_{+}=\mathbb{F}_{q}[x_{1},\ldots,x_{d}]
\quad\text{and}\quad
R=\mathbb{F}_{q}[x_{1}^{\pm1},\ldots,x_{d}^{\pm1}],
\]
and let $\operatorname{SL}_{d}(\mathbb{Z})$ act on $R$ by monomial substitutions.
The groups
\begin{equation}\label{eq:nonsoficgroupsbykunthom}
	\Gamma:=\operatorname{EL}_{r}(R_{+})
	\quad\text{and}\quad
	G:=\operatorname{EL}_{r}(R)\rtimes\operatorname{SL}_{d}(\mathbb{Z})
\end{equation}
are residually finite and have Kazhdan's Property (T). 
Moreover,
$\Gamma$ is an  infranormal but not normal subgroup  of $G$.
Thus, the groups 
$$ 
(\mathbb{Z}/2\mathbb{Z}) \wr_{G/\Gamma} G \quad \text{and} \quad 
G*_{\Gamma}G 
$$  
are non-sofic.
Note that both $\Gamma$ and $G$ are finitely generated groups.

The following result shows that these non-sofic groups have Property A and hence have finite complexity rates.

\begin{prop}
	Let $\Gamma$ and $G$ be the finitely generated groups defined in \eqref{eq:nonsoficgroupsbykunthom}.  Then  $(\mathbb{Z}/2\mathbb{Z}) \wr_{G/\Gamma} G$  and $G*_{\Gamma}G$  have Property A.
\end{prop}
\begin{proof}
	We first note that exactness is equivalent to Property A for every countable discrete group \cite{Ozawa2000}.
	
Recall that
	\[
	(\mathbb Z/2\mathbb Z)\wr_{G/\Gamma}G
	=\left(\bigoplus_{G/\Gamma}\mathbb Z/2\mathbb Z\right)\rtimes G, 
	\]
	where $G$ acts by permuting the coordinates of the direct sum.
	Thus, we have a short exact sequence of countable discrete  groups
	\[
	1\longrightarrow
	\bigoplus_{G/\Gamma}\mathbb Z/2\mathbb Z
	\longrightarrow
	(\mathbb Z/2\mathbb Z)\wr_{G/\Gamma}G
	\longrightarrow
	G
	\longrightarrow 1.
	\] 
	Since  $\bigoplus_{G/\Gamma}\mathbb Z/2\mathbb Z$ is amenable, 
	by Theorem 5.1 in \cite{KirchbergWassermann}, 
	it suffices to show that  $G$  is exact. 
	
	Since
	$R=\mathbb F_q[x_1^{\pm1},\ldots,x_d^{\pm1}]$
	is an integral domain, it embeds naturally into its field of fractions
	\[
	K=\mathbb F_q(x_1,\ldots,x_d).
	\]
	Hence,
	\[
	\mathrm{EL}_r(R)
	\leq \mathrm{GL}_r(R)
	\leq \mathrm{GL}_r(K).
	\]
By the main theorem in \cite{GuentnerHigsonWeinberger},
	the subgroup $\mathrm{EL}_r(R)$  of $\mathrm{GL}_r(K)$  is exact.
	The same theorem also implies that $\mathrm{SL}_d(\mathbb Z)$ is exact.	
Now the group
$G=\mathrm{EL}_r(R)\rtimes \mathrm{SL}_d(\mathbb Z)$	
fits into the short exact sequence 
\[
	1\longrightarrow
	\mathrm{EL}_r(R)
	\longrightarrow
	G
	\longrightarrow
	\mathrm{SL}_d(\mathbb Z)
	\longrightarrow 1.
	\]
Another application of Theorem 5.1 in \cite{KirchbergWassermann}
shows that
$G$	is exact.
  		
This completes the proof for the generalized wreath product $(\mathbb{Z}/2\mathbb{Z}) \wr_{G/\Gamma} G$.

For the amalgamated free product  $G*_{\Gamma}G$,   	
the conclusion follows directly from Dykema's  theorem (\cite{Dykema2004}).
\end{proof}

\subsection{Coarse embeddings of Hilbert--Hadamard spaces into $L^p$-spaces}

In this subsection, we show that the principal Hilbert--Hadamard spaces studied in \cite{GWY} admit coarse embeddings into $L^p$-spaces. Consequently, our main theorem applies to all countable discrete groups admitting coarse embeddings into these spaces. 
These  Hilbert--Hadamard spaces
play a central role in the study of subgroups of diffeomorphism groups of smooth manifolds.

We first recall the notion of continuum product.
\begin{defn}\cite{GWY} 
	Let $(X,d_X)$ be a separable metric space, and let  $(\Omega,\mu)$ be a measure
	space with $\mu(\Omega)<\infty$. 
	For $p\in[1, \infty)$,
	the \emph{$L^p$-continuum product} of $X$
	over $(\Omega,\mu)$ is the space
	\[
	L^p(\Omega,\mu,X)
	\]
	of equivalence classes of  $\mu$-measurable maps $f:\Omega\to X$ satisfying
	\[
	\int_{\Omega} d_X\bigl(f(v),x_0\bigr)^p\,d\mu<\infty,
	\]
	where $x_0$ is a fixed point in $X$, and two maps are identified
	if they differ only on a measure-zero subset of $\Omega$.
	
	It follows from the triangle inequality that the above condition
	does not depend on the choice of $x_0$. Moreover, the Minkowski
	inequality implies that the formula
	\[
	d_{L^p}(f,g)
	=
	\left(
	\int_{\Omega} d_X\bigl(f(v),g(v)\bigr)^p\,d\mu
	\right)^{1/p}
	\]
	defines a metric on $L^p(\Omega,\mu,X)$.
\end{defn}

To apply this construction to subgroups of diffeomorphism groups, we need to take $\Omega$ to be a smooth manifold $M$ and $X$ to be the space of positive-definite symmetric matrices equipped with its standard Riemannian metric.  We  recall some standard facts about this space.
Let $n\geq1$ be an integer, and let $\mathcal{P}(n)$ be the
space of real symmetric positive-definite $n\times n$ matrices.
The standard Riemannian metric on $\mathcal{P}(n)$ is given by
\[
\langle A,B\rangle_Q
=
\operatorname{Tr}\bigl(Q^{-1}AQ^{-1}B\bigr),
\qquad
A,B\in T_Q\mathcal{P}(n)
\cong\operatorname{Sym}_n(\mathbb R).
\]
Denote by
$d_{\mathcal{P}(n)}$ 
the corresponding Riemannian distance on $\mathcal{P}(n)$.
With this metric, 
$\mathcal{P}(n)$  is a complete, simply connected
Riemannian manifold of nonpositive sectional curvature.
Moreover, the  action of $\operatorname{GL}(n,\mathbb R)$ on $\mathcal{P}(n)$,
given by
\[
A\cdot Q=AQA^{\mathsf T},\quad  A\in\operatorname{GL}(n,\mathbb R),\, Q\in \mathcal{P}(n),
\]
is transitive and isometric.
Thus $\mathcal{P}(n)$ is a homogeneous Hadamard manifold
(see \cite{BridsonHaefliger1999,GWY}).
We denote by $\mathcal{P}_1(n)$ the subspace of $\mathcal{P}(n)$
consisting of matrices with determinant $1$.
The $L^2$-continuum product
$L^2\bigl(M,\mu,\mathcal{P}_1(n)\bigr)$ and the larger one
$L^2\bigl(M,\mu,\mathcal{P}(n)\bigr)$
are both Hilbert--Hadamard spaces in the terminology of \cite{GWY}.

Let $M$ be a closed smooth manifold of dimension $n$, equipped
with a smooth positive density $\mu$.
Gong, Wu, and Yu \cite{GWY} proved the rational strong Novikov
conjecture for geometrically discrete subgroups of the
volume-preserving diffeomorphism group $\operatorname{Diff}(M,\mu)$.
The cornerstone of their approach is that geometrically discrete
subgroups of  $\operatorname{Diff}(M,\mu)$  act properly and
isometrically on the  Hilbert--Hadamard space
\[
L^2\bigl(M,\mu,\mathcal{P}_1(n)\bigr).
\]

This subsection is devoted to extending their result from
proper isometric actions to the more general setting of
coarse embeddings, while also removing the volume-preserving
assumption.
More precisely, 
for any finite measure $\mu$
on $M$ and any
$p\in[1,\infty)$,
the class of countable discrete groups admitting
coarse embeddings into the $L^p$-continuum product
\[
L^p\bigl(M,\mu,\mathcal{P}(n)\bigr)
\]
falls within the scope of our main theorem
by virtue of  Proposition \ref{snowflakemetricbundleembedding}  below, 
and hence the Novikov
conjecture holds for all such groups.
This is a  generalization in two directions. 
First, we do not require $p=2$: for $1<p<\infty$ with $p\neq2$,
the target space is generally not a Hilbert--Hadamard space,
but is weakly bolic in the sense of Kasparov and Skandalis
\cite{KS2003}.
The framework of \cite{GWY} does not apply directly
to these non-Hilbert--Hadamard spaces.
Second, we relax the requirement of a proper isometric action
to the weaker condition of a coarse embedding.
For example,
any countably infinite subgroup
$\Gamma<\operatorname{Diff}(M)$ preserving a Riemannian metric
$\textsl{g}$ fails the geometric discreteness condition in
\cite{GWY}.
Nevertheless, $\Gamma$ admits a coarse embedding into
$L^2\bigl(M,\mu,\mathcal{P}(n)\bigr)$
for any non-atomic probability measure $\mu$ on $M$,
by virtue of the result of Guentner, Higson, and Weinberger
\cite{GuentnerHigsonWeinberger}.

We now establish a bridge between the $L^p$-continuum product and the classical $L^p$-spaces.

\begin{prop}\label{snowflakemetricbundleembedding}
	Let $(\Omega,\mu)$ be a finite measure space.
	For every $p\in[1,\infty)$ and every integer $n\geq1$,
	there exists a measure space $(\Omega',\mu')$ such that the
	$L^p$-continuum product
	\[
	L^p\bigl(\Omega,\mu,\mathcal{P}(n)\bigr)
	\]
	admits a coarse embedding into the Banach space
	$L^{2p}(\Omega',\mu')$.
\end{prop}

\begin{proof}
	Fix $p\in[1,\infty)$ and an integer $n\geq1$.
	
	Since $(\mathcal{P}(n),d_{\mathcal{P}(n)})$ is a homogeneous
	Hadamard manifold, it has finite Nagata dimension by
	\cite[Theorem 3.6]{LangSchlichenmaier}. 
	It therefore follows from
	\cite[Lemma 5.1 and Theorem 5.2]{NaorSilberman2011}
	that $\big(\mathcal{P}(n),d_{\mathcal{P}(n)}^{\,\theta}\big)$  admits a
	bi-Lipschitz embedding into a Hilbert space for every
	$\theta\in(0,1)$.
	Since $\mathcal{P}(n)$ is separable, the target Hilbert space
	may be taken to be  $\ell^2(\mathbb{N})$.
	More precisely, for each $\theta\in(0,1)$, there exist a map
	\[
	\iota:\mathcal{P}(n)\longrightarrow\ell^2(\mathbb{N})
	\]
	and constants $0<c\leq C<\infty$ such that
	\begin{equation}\label{eq:whtaisinequoskldhhhkdjlkfj}
		c\,d_{\mathcal{P}(n)}(A,B)^\theta
		\leq
		\bigl\|\iota(A)-\iota(B)\bigr\|
		\leq
		C\,d_{\mathcal{P}(n)}(A,B)^\theta,
		\qquad A,B\in\mathcal{P}(n).
	\end{equation}
	
	Take $\theta=1/2$ and fix $A_0\in\mathcal{P}(n)$.
	Set
	\[
	\mathscr{Y}=L^{2p}[0,1].
	\]
	The Gaussian construction
	(see \cite[Proposition~6.4.12]{AlbiacKalton2016})
	yields a linear isometric embedding
	\[
	\tau:\ell^2(\mathbb{N})\longrightarrow\mathscr{Y}.
	\]
	For $f\in L^p\bigl(\Omega,\mu,\mathcal{P}(n)\bigr)$,
	define $\widetilde{f}:\Omega\to\mathscr{Y}$ by
	\[
	\widetilde{f}(v)
	=\tau\bigl(\iota(f(v))-\iota(A_0)\bigr),
	\qquad v\in\Omega.
	\]
	Since $\iota$ and $\tau$ are continuous, $\widetilde{f}$ is
	$\mu$-measurable.
	The separability of $\mathscr{Y}$ then implies that
	$\widetilde{f}$ is Bochner measurable
	(see \cite[Theorem~1.1.20]{HNVW2016}).
	Moreover, \eqref{eq:whtaisinequoskldhhhkdjlkfj} gives
	\[
	\begin{aligned}
		\int_\Omega\bigl\|\widetilde{f}(v)\bigr\|^{2p}\,d\mu
		&=
		\int_\Omega
		\bigl\|\iota(f(v))-\iota(A_0)\bigr\|^{2p}\,d\mu\\
		&\leq
		C^{2p}\int_\Omega
		d_{\mathcal{P}(n)}\bigl(f(v),A_0\bigr)^p\,d\mu
		<\infty.
	\end{aligned}
	\]
	Thus $\widetilde{f}\in L^{2p}(\Omega,\mu,\mathscr{Y})$.
	If $f=g$ $\mu$-almost everywhere, then
	$\widetilde{f}=\widetilde{g}$ $\mu$-almost everywhere.
	Hence the map
	\[
	\psi:
	L^p\bigl(\Omega,\mu,\mathcal{P}(n)\bigr)
	\longrightarrow L^{2p}(\Omega,\mu,\mathscr{Y}),
	\qquad
	\psi(f):=\widetilde{f},
	\]
	is well-defined.
	
	Since $\tau$ is a linear isometry, for any
	$f,g\in L^p\bigl(\Omega,\mu,\mathcal{P}(n)\bigr)$ we have
	\[
	\begin{aligned}
		\bigl\|\psi(f)-\psi(g)\bigr\|^{2p}
		&=
		\int_\Omega
		\bigl\|\widetilde{f}(v)-\widetilde{g}(v)\bigr\|^{2p}
		\,d\mu\\
		&=
		\int_\Omega
		\bigl\|\iota(f(v))-\iota(g(v))\bigr\|^{2p}\,d\mu.
	\end{aligned}
	\]
	Applying \eqref{eq:whtaisinequoskldhhhkdjlkfj} and taking
	$2p$-th roots, we obtain
	\[
	c\,d_{L^p}(f,g)^{1/2}
	\leq
	\bigl\|\psi(f)-\psi(g)\bigr\|
	\leq
	C\,d_{L^p}(f,g)^{1/2}.
	\]
	Thus $\psi$ is a coarse embedding.
	
	Finally, set
	\[
	\Omega'=\Omega\times[0,1]
	\qquad\text{and}\qquad
	\mu'=\mu\otimes m,
	\]
	where $m$ denotes Lebesgue measure on $[0,1]$.
	By the Fubini isomorphism for Lebesgue--Bochner spaces
	(see \cite[Proposition~1.2.24]{HNVW2016}), there is a canonical
	isometric isomorphism
	\[
	L^{2p}(\Omega,\mu,\mathscr{Y})
	\xlongrightarrow{\cong}
	L^{2p}(\Omega',\mu').
	\]
	Composing $\psi$ with this isomorphism yields the desired
	coarse embedding.
\end{proof}

\begin{rmk}
	In fact, the most natural application of our main theorem is to study the countable groups that admit coarse embeddings  with finite complexity
	into the infinite direct sum  
	$$
	\bigoplus_{p=1}^{\infty} L^p\bigl(M,\mu_p,\mathcal{P}(p)\bigr).
	$$ 
	This framework can capture a wide range of subgroups of $\operatorname{Diff}(M)$ and will be developed in a forthcoming paper.
\end{rmk}

	\bibliographystyle{amsplain}

\end{document}